\documentclass[12pt]{article}

\usepackage[utf8]{inputenc}
\usepackage{placeins}
\usepackage{graphicx}
\usepackage{float}
\usepackage{ragged2e}
\usepackage{amsthm}
\usepackage{amsmath}
\usepackage{mathtools}
\usepackage{amsfonts}
\usepackage{amssymb}
\usepackage[normalem]{ulem}
\usepackage{bbm}
\usepackage{bm}
\usepackage{tikz}
\usepackage{algpseudocode}
\usepackage{breqn}

\usepackage{natbib}
\usepackage{nicefrac}
\usepackage[multiple]{footmisc}
\usepackage{enumitem}
\setenumerate{wide}
\usepackage{thmtools}
\usepackage{thm-restate}
\usepackage[nameinlink, capitalize]{cleveref}

\usepackage{cprotect}
\usepackage{etoc}
\usepackage{etaremune}

\usepackage[T1]{fontenc}    %
\usepackage{booktabs}       %
\usepackage{microtype}      %
\floatstyle{ruled}
\usepackage{xcolor}

\usepackage{caption}
\usepackage{comment}
\usepackage{subcaption}

\usetikzlibrary{arrows,automata, calc, positioning}
\usepackage{adjustbox}
\usepackage{wrapfig}
\DeclareCaptionLabelSeparator{custom}{---}

\usepackage{multirow}
\usepackage{accents}

\usepackage[subtle]{savetrees}
\usepackage{apptools}

\usepackage[nodisplayskipstretch]{setspace}
\def\spacingset#1{\renewcommand{\baselinestretch}%
{#1}\small\normalsize} \spacingset{1}

\DeclareRobustCommand{\bb}[1]{\mathbb{#1}}
\DeclareRobustCommand{\c}[1]{\mathcal{#1}}

\DeclareRobustCommand{\mr}[1]{\mathrm{#1}}

\DeclareRobustCommand{\beef}{\vphantom{\sum}}

\DeclareRobustCommand{\HS}{\mathrm{HS}}
\DeclareRobustCommand{\op}{\mathrm{op}}

\DeclareRobustCommand{\H}{H}

\DeclareRobustCommand{\N}{\mathfrak{n}}
\DeclareRobustCommand{\BS}{\mathfrak{B}}
\DeclareRobustCommand{\s}{\mathfrak{s}}
\DeclareRobustCommand{\E}{\mathbb{E}}
\DeclareRobustCommand{\R}{\mathbb{R}}

\DeclareRobustCommand{\a}{\alpha}
\DeclareRobustCommand{\b}{\beta}
\DeclareRobustCommand{\gam}{\gamma}
\DeclareRobustCommand{\d}{\delta}
\DeclareRobustCommand{\D}{\Delta}
\DeclareRobustCommand{\ep}{\varepsilon}
\DeclareRobustCommand{\eps}{\epsilon}

\DeclareRobustCommand{\bX}{\boldsymbol{X}}

\DeclareRobustCommand{\f}[2]{\frac{#1}{#2}}
\DeclareRobustCommand{\set}[1]{\left\{#1 \right\}}
\DeclareRobustCommand{\Set}[2][]{\left\{#1 \, \middle|\, #2 \right\}}
\DeclareRobustCommand{\bk}[2]{\left\langle #1,\,#2\right\rangle}
\DeclareRobustCommand{\sbk}[2]{\langle #1,\,#2\rangle}

\DeclareRobustCommand{\norm}[1]{\left\lVert #1 \right\rVert}
\DeclareRobustCommand{\snorm}[1]{\lVert #1 \rVert}
\DeclareRobustCommand{\inv}[1]{{#1}^{-1}}

\DeclareMathOperator{\tr}{tr}
\DeclareMathOperator{\vspan}{span}
\DeclareMathOperator{\diag}{diag}

\DeclareMathOperator*{\argmin}{\arg\!\min}

\newtheorem{lemma}{Lemma}
\newtheorem{theorem}{Theorem}

\newtheorem{corollary}{Corollary}
\newtheorem{assumption}{Assumption}
\newtheorem{proposition}{Proposition}

\AtAppendix{\counterwithin{lemma}{section}}
\AtAppendix{\counterwithin{theorem}{section}}
\AtAppendix{\counterwithin{algorithm}{section}}
\AtAppendix{\counterwithin{corollary}{section}}
\AtAppendix{\counterwithin{assumption}{section}}
\AtAppendix{\counterwithin{proposition}{section}}

\theoremstyle{definition}
\newtheorem{algorithm}{Algorithm}
\newtheorem{definition}{Definition}

\theoremstyle{remark}
\newtheorem{example}{Example}
\newtheorem{remark}{Remark}

\AtAppendix{\counterwithin{definition}{section}}
\AtAppendix{\counterwithin{example}{section}}
\AtAppendix{\counterwithin{remark}{section}}

\DeclareRobustCommand{\VAN}[2]{#1}  %

\begin{document}

\def\spacingset#1{\renewcommand{\baselinestretch}%
{#1}\small\normalsize} \spacingset{1}

\title{\bf Kernel Ridge Regression Inference \\
with Applications to Preference Data}
\author{Rahul Singh\\
  Society of Fellows and Department of Economics, Harvard University\\
  \and
  Suhas Vijaykumar\thanks{
  This paper was initially circulated in February 2023 \citep{singh2023kernel}. We thank Alberto Abadie, Isaiah Andrews, Joshua Angrist, Victor Chernozhukov, Anna Mikusheva, Whitney Newey, Parag Pathak, Neil Shephard, Vasilis Syrgkanis, and Elie Tamer for helpful discussions. We are particularly grateful to Anna Mikusheva for guidance. Miriam Nelson and Moses Stewart provided excellent research assistance.
  Both authors received support from the Hausman Dissertation Fellowship, and part of this work was done while Rahul Singh visited the Simons Institute.}\hspace{.2cm}\\
  Department of Economics, U.C.~San Diego}
\date{Original draft: February 2023. This draft: July 2026.}
\maketitle

\bigskip
\begin{abstract}
We provide uniform confidence bands for kernel ridge regression (KRR), a widely used nonparametric regression estimator for nonstandard data such as preferences, sequences, and
graphs.
Despite the prevalence of these data\textemdash e.g., student preferences in school matching mechanisms\textemdash the inferential theory of KRR is not fully known.
We construct valid and sharp confidence sets that shrink at nearly the minimax rate, allowing nonstandard regressors.
Our bootstrap procedure uses anti-symmetric multipliers for computational efficiency and for validity under mis-specification.
We use the procedure to develop a test for match effects, i.e. whether students benefit more from the schools they rank highly.
\end{abstract}

\noindent%

\vfill
\newpage
\spacingset{1.9} %

\section{Introduction and related work}

 We study a regularized, nonparametric regression estimator for nonstandard data---such as preferences, sequences, and  graphs---called kernel ridge regression.
 Kernel ridge regression is ubiquitous in data science, and it has a practical, closed form expression \citep{kimeldorf1971some}.
 Several recent works advocate for its use in econometric problems \citep{kasy2018optimal,nie2021quasi,singh2020kernel}.
 However, its inferential theory is not fully known, limiting its role in economic research.

Our interest in KRR is motivated by its effectiveness with nonstandard regressors:
data whose traditional, high-dimensional representation fails to capture their latent, low-dimensional structure.
Important examples of nonstandard regressors in economics are rank-ordered lists that represent preference relations, and binary sequences that represent workers' movements in and out of employment.
Our research question is how to construct uniform confidence bands for KRR that are valid and sharp, with nonstandard regressors.

Preference lists, recorded as rankings over a set of choices, are widely collected by the mechanisms that match students to schools and doctors to hospitals \citep{roth1992two,abdulkadirouglu2005new,abdulkadirouglu2017research,abdulkadirouglu2020parents}.
However, they pose a challenge for traditional econometric methods due to their massive dimension: a saturated regression on preferences over $p$ schools produces $p!$ variables.
Employment histories, recorded as sequences of movements in and out of employment, pose a similar challenge \citep{card1988measuring}.
Traditional methods for uniform inference require a small ambient dimension.
As discussed in Section~\ref{sec:model}, lasso-type methods face issues in this setting since it is impractical to enumerate all regressors.

Our main assumption, which is well studied in prior work on KRR's consistency
\citep{smale2007learning,caponnetto2007optimal,mendelson2010regularization,fischer2020sobolev}, is that the data have a low effective dimension
with respect to
 an implicit, nonlinear basis.
This basis is jointly determined by  the data distribution and by a
user-specified kernel. Well-studied examples of kernels include
 the Gaussian kernel for standard data, a kernel based on Kendall's rank correlation for preference data, and a kernel based on the Hamming distance for employment data.\footnote{Certain kernels for standard data recover the Haar or Hermite bases, often used in series regression.} Kernels exist for many other data types commonly used in empirical economic research, such as networks, images, and text.
Section~\ref{sec:model} formalizes low effective dimension, interprets it as an approximate nonlinear factor model, and assesses it through diagnostic plots.

In the preference example, our main assumption is satisfied if most students' preferences are similar to relatively few preference types. Standard representations of preferences using indicator variables or numeric lists either (i) are highly restrictive, (ii) are computationally intractable, or (iii) fail to reflect a natural notion of similarity between preferences.
By contrast, a nonlinear representation based on Kendall's rank correlation is practical, fully nonparametric, and allows us to exploit latent, low dimensional structure that is not captured by the standard representation: many student's preferences may be similar to each other, in the sense that they make similar pairwise comparisons.

Our primary contribution is uniform inference for KRR in a general setting, allowing nonstandard regressors and various kernels. Our inferential procedure re-uses the artifacts of computing KRR once; it is computationally efficient in this sense. It is a multiplier bootstrap for an empirical process \citep{chernozhukov2014anticoncentration,chernozhukov2016empirical}, but with dependent multipliers.
 Our confidence sets contract at nearly the minimax rate of estimation, with a vanishing incremental factor \citep{andrews2013inference}.

To justify our inferential procedure, we derive nonasymptotic, uniform Gaussian couplings in a reproducing kernel Hilbert space (RKHS), \textcolor{black}{building on the recent and sharp Euclidean couplings of \citet{eldan2020clt}.} These imply a strong form of uniform Gaussian approximation under conditions which often correspond to weaker forms of approximation
\citep{chernozhukov2014gaussian,chernozhukov2016empirical,chernozhuokov2022improved}. We also construct bootstrap couplings to sample from this approximating distribution.
To be useful for uniform inference, the couplings we derive are nonasymptotic, so they can handle the increasing complexity of the KRR estimator as regularization vanishes.
 Although we derive these couplings to analyze KRR, future work may use them to provide inference for other kernel methods.

Finally, we use our inferential procedure to construct a test for match effects \citep{narita2018match,bau2022estimating,angrist2023choice}: do students benefit more from the schools they rank highly? Here, KRR allows us to regress on the $25!$ possible preferences that students may have over
the $25$ high schools in the Boston match
\citep{mania2018kernel}.
Our procedure conducts inference for KRR, and therefore tests for match effects.
In a semi-synthetic exercise calibrated to real student preference data, our test has more power than a currently used approach based on indicators for coarsened preference categories.

Our results complement previous works on Gaussian approximation for KRR. \citet{hable2012asymptotic} studies the setting where regularization is bounded away from zero, which
precludes consistency of KRR.
\citet{shang2013local} study the special case of splines with data that are uniform over $[0,1]$.
\citet{yang2017frequentist} study posterior coverage and sup-norm credible sets in Gaussian process regression under additional assumptions motivated by splines. In particular, uniform boundedness of eigenfunctions
\cite[Assumption B]{yang2017frequentist}
appears to rule out the Gaussian kernel over $\R^2$ \citep{bisiacco2024gaussian} and other widely used kernels and data types.
To our knowledge, our results are the first to provide valid, sharp inference in the nonparametric regime where regularization vanishes, applicable to a wide range of settings where KRR may be applied, e.g., preference data.

Whereas we provide uniform inference for KRR with nonstandard data, a mature literature provides uniform inference for series regression with standard data. See e.g. the review of \cite{chen2007large}. A series procedure is typically tractable, with strong asymptotic guarantees, when the regressor belongs to a low dimensional Euclidean space. By contrast, our procedure remains tractable for preference data.
For standard data settings, the series literature provides results that exceed the scope of this paper, e.g. dependence, instruments, and adaptive hyperparameters \citep{belloni2015some,chen2018optimal,chen2024adaptive}.   Future work may develop analogous results for kernel methods with nonstandard data.

Section~\ref{sec:model} interprets the assumptions that underlie consistency of KRR. The same assumptions will justify inference.
Section~\ref{sec:algo_main} proposes our computationally efficient, valid, and sharp procedure for uniform KRR
inference with nonstandard data. We prove that our confidence sets contract at nearly the optimal rate. Section~\ref{sec:partial} derives our underlying Gaussian and bootstrap couplings. By using anti-symmetric multipliers to cancel bias, we justify inference under mis-specification. Section~\ref{sec:application} applies our procedure as a test for match effects using preference data. In calibrated simulations, it is more powerful than an existing approach.

\section{Model and assumptions}\label{sec:model}

We denote the Euclidean norm in $\mathbb{R}^n$ by $\|\cdot\|_{\mathbb{R}^n}$. For vectors $u,v$ in a Hilbert space $H$ with norm $\|u\|$, we denote by $u \otimes v^*: H \to H$ the tensor product, i.e. the rank one operator with
$(u \otimes v^*)t = \bk{v}{t}u.$ For any $A: H \to H$ we use $\norm{A}_{\op}$ for operator norm, $\norm{A}_{\HS}$ for Hilbert-Schmidt (or Frobenius) norm, and $\tr A$ for trace.
If $A$ is compact and self-adjoint, then $H$ admits an orthonormal basis of $A$-eigenvectors $\{e_1(A), e_2(A), \ldots\}$ and corresponding eigenvalues $\{\nu_1(A), \nu_2(A), \ldots\}$. We suppress the operator $A$ when it is clear from context.

We use $C$, $C'$, et cetera, to denote  sufficiently large, positive universal constants whose value may change across displays. $C(t)$ denotes a large enough number that depends only on the parameter $t$. Similarly, $c$, $c'$, $c(t)$ denote sufficiently small positive quantities. We also use the notation $\lesssim$ (or $\lesssim_t$) to denote an inequality that holds up to a universal constant (or function of $t$). In summary, for $a,b>0$, $a \lesssim b$ means $a \le Cb$, which means $ca \le b.$

Our formal results hold upon event of probability at least $1-\eta$, where $\eta \in (0,1)$. We write $o_p(1)$ for quantities that converge to zero for any fixed $\eta \in (0,1)$ as the sample size grows.

\subsection{Previous work: Closed form estimation}

Our goal is to learn the regression function $f_0(X)=\E(Y|X)$ where the regressors belong to a Polish space, $S$. The regression model $H$ is a space of functions $f:S\rightarrow\R$ with additional structure. In particular, $H$ is derived from a kernel $k:S\times S\rightarrow \R$ chosen by the researcher. The functions $k_x(\cdot)=k(x,\cdot)$ span $H$, and $k(x,x')=\langle k_x,k_{x'} \rangle$ defines an inner product in $H$. These conditions imply the reproducing property: $f(x)=\langle f, k_x\rangle$, for any $f\in H$.\footnote{Formally, $k$ is a positive definite function, so the inner product is well defined. We define the RKHS $\H$ as the closure of $\Set[k_x]{x \in S}$ with respect to $\bk{\cdot}{\cdot}$. We require $S$ to be a separable, complete metric space, so that $H$ is separable. \citet{berlinet2004reproducing} give further background.}

In the final expression, $x\mapsto k_{x}$ may be viewed as a nonlinear basis expansion of a regressor value $x$.\footnote{Sometimes $\phi:x\mapsto k_x$ is called the feature map, where $\phi(x)\in H$ are features of $x$ used to estimate $f_0(x)$.} The kernel function $k(x,x')$ encodes similarity between regressor values $(x,x')$. The implied norm in $H$, given by $\snorm{f}=\bk{f}{f}^{\f{1}{2}}$, reflects a corresponding notion of smoothness: similar regressor values are assigned similar outcome values.  Due to the popularity of kernel methods, there is extensive theoretical and
practical guidance on how to choose $k$, with examples below.

As a weak regularity condition, we maintain that the kernel is bounded, i.e. $k(x,x')\leq \kappa^2$.
Therefore, control of the $H$ norm implies control of the supremum norm: by the Cauchy-Schwarz inequality,
$\sup_{x \in S}|f(x)|=\sup_{x \in S}|\langle f, k_x\rangle|\leq \|f\|\|k_x\|\leq \kappa \|f\|.$
For simplicity, we also maintain that the regression residual $\varepsilon=Y-\E(Y|X)$ satisfies $|\ep| \le \bar\sigma$ and $\E(\ep^2|X)\geq \underline{\sigma}^2$.

KRR extends linear regression to regressors that may be preferences, sequences, or graphs. Consider an independent, identically distributed (i.i.d.) sequence of $n$ observations $(X_i, Y_i)$ in $S \times \mathbb{R}$, which are supported on a probability space $(\bb{P}, \Omega, \c F)$.\footnote{We assume the probability space is sufficiently rich so that we may construct couplings, e.g.~that it supports a countable sequence of i.i.d.~Gaussians independent of the data.} %
The KRR estimator is
$$
\hat f =\argmin_{f \in H} \left[ \E_n \{Y - f(X)\}^2 + \lambda \norm{f}^2 \right]
$$
where $\E_n(\cdot)=\frac{1}{n}\sum_{i=1}^n(\cdot)$ is the average over observations. In the ridge penalty, $\lambda>0$ is the regularization parameter and $\norm{f}$ is the norm in $H$.

The solution to the optimization has a closed form, given in Algorithm~\ref{algo:krr} below, which makes KRR practical even when using nonstandard data $S$ and an infinite basis $x\mapsto k_x$.

\begin{algorithm}[Kernel ridge regression \citep{kimeldorf1971some}]\label{algo:krr}
    Given an i.i.d. sample $D = \{(X_i,Y_i)\}_{i=1}^n$, a kernel $k$, and regularization parameter $\lambda>0$:
    \begin{enumerate}
\item Compute the kernel matrix $K \in \mathbb{R}^{n\times n}$ with entries $K_{ij}=k(X_i,X_j)$, and the kernel vector $K_x\in\mathbb{R}^{1\times n}$ with entries $k(x,X_i)$.
    \item Estimate KRR as $\hat{f}(x)=K_x(K+n\lambda  I )^{-1}Y$, where $Y\in\mathbb{R}^n$ is the vector of outcomes.
    \end{enumerate}
\end{algorithm}

\begin{example}[Linear kernel]\label{ex:linear}
    If $S = \bb{R}^p$, so that the regressors are finite vectors, then $k(x,x') = x^{\top} x'$ recovers linear models: $H$ consists of linear functions $f_\beta(x) = \beta^\top x$ for $\beta \in \bb{R}^p$. Here, $k_x=x$ is a trivial basis expansion, i.e. the identity mapping.

    Let $\bX\in\R^{n\times p}$ be the design matrix. Then $K_x=x^{\top}\bX^{\top}$, $K=\bX\bX^{\top}$, and hence
$$
\hat{f}(x)=x^{\top}\bX^{\top}(\bX\bX^{\top}+n\lambda  I )^{-1}Y=x^{\top}(\bX^{\top}\bX+n\lambda  I )^{-1}\bX^{\top}Y.\footnote{The equality follows from the identity $(UV+aI)U=U(VU+aI)$ for matrices $(U,V)$ and scalar $a$.}
$$
This recovers linear ridge regression, expressed in terms of the Gram matrix $\bX\bX^{\top} \in\R^{n\times n}$ rather than the covariance matrix $\bX^{\top}\bX \in\R^{p\times p}$, which helps when $p$ is large or infinite.
\end{example}

\begin{example}[Polynomial kernel]\label{ex:poly}
   If $S = \bb{R}^2$, then $k(x,x') = (x^\top x'+1)^2$ recovers quadratic models: $H$ consists of nonlinear functions $f_\beta(x) = \beta^\top k_x$, where $k_x=(x_1^2,x_2^2,\sqrt{2}x_1x_2,\sqrt{2}x_1,\sqrt{2}x_2,1)^{\top}$ is a fully interacted quadratic expansion of $x=(x_1,x_2)^{\top}$.

   For quadratic ridge regression, we compute $K_{ij}=(X_i^\top X_j+1)^2$ for each pair of regressors. More generally, for polynomial ridge regression of degree $d$ and $X \in \mathbb{R}^p$, we compute $K_{ij}=(X_i^\top X_j+1)^d$.
Once $K\in\R^{n\times n}$ has been computed, estimation is similar to the linear case.
\end{example}

The kernels in Examples~\ref{ex:linear} and~\ref{ex:poly} correspond to ridge regression with finite basis expansions. Already, in Example \ref{ex:poly}, the polynomial basis has size $p^d$: for large values of $p$ and moderate values of $d$, KRR has the significant practical advantage that it only works with this basis implicitly via the matrix $K\in\R^{n\times n}$.
This intuition extends to kernels that correspond to infinite basis expansions, e.g. Sobolev and Gaussian kernels, which recover classical spline estimators. Most importantly, our results apply to kernels for nonstandard data, as we  discuss below.

\begin{example}[Preference kernel]\label{ex:kendall}
    Suppose that $S$ consists of preferences over $25$ high schools in the Boston match. The space $H$ of functions $f:S \to \bb{R}$ has dimension $25!>10^{25}$. Consider a kernel based on Kendall's rank correlation: $k(x,x')=e^{-N(x,x')}$, where $N(x,x')$ counts the number of pairwise comparisons in which the preferences $x$ and $x'$ disagree \citep{kendall1938new,mallows1957non}.
    The RKHS induced by this kernel recovers $H$ \citep{mania2018kernel}. %

    For preference data regression, we compute $K_{ij}=e^{-N(X_i,X_j)}$ for each pair of preferences.
Once $K\in\R^{n\times n}$ has been computed, estimation is similar to the linear case.
\end{example}

This preference kernel reflects the idea that some preferences are closer than others. For example, with four schools, it asserts that the preference $(A,B,C,D)$ is closer to $(B,A,C,D)$ than $(D,C,B,A)$.
By contrast, using $4!=24$ indicators for each possible preference disregards similarity between preferences. Our results are generic and apply to alternative preference kernels, as long as $k(x,x')$ is a bounded, positive definite function.\footnote{For example, one may place weightings on the relative importance of each position in the preference.}

\begin{example}[Sequence kernel]\label{ex:hamming}
    Suppose that $S$ consists of employment histories, each expressed as a binary sequence over $25$ years \citep{card1988measuring}. The space $H$ of functions $f:S \to \bb{R}$ has dimension $2^{25}>10^{7}$. Consider a kernel based on the Hamming distance: $k(x,x')=e^{-N(x,x')}$, where $N(x,x')$ counts the number of years in which the sequences $x$ and $x'$ disagree \citep{hamming1950error}.
    The RKHS induced by this kernel recovers $H$. %
    For sequence data regression, we compute $K_{ij}=e^{-N(X_i,X_j)}$ for each pair of sequences.
Thereafter, estimation is similar to the linear case.
\end{example}

As before, this sequence kernel reflects the idea that some sequences are closer than others. For example, with four years, it asserts that the employment sequence $(0,0,0,1)$ is closer to $(0,0,1,1)$ than $(1,0,0,0)$.
Again, our results apply to many other sequence kernels.\footnote{Extensions accommodate other aspects of employment histories, e.g. unequal sequence lengths, temporal importance, and subsequence similarity \citep{lodhi2002text}.}

\begin{remark}[Kernel representations versus traditional representations for nonstandard data]
    Examples~\ref{ex:kendall} and~\ref{ex:hamming} demonstrate how KRR remains tractable with nonstandard data, without restricting the variety of preferences or sequences. In both cases, KRR performs estimation via implicit, nonlinear representations that can exploit plausible notions of similarity, e.g. Kendall's rank correlation or the Hamming distance.

A series or lasso method would encounter the problem that it is not possible to regress on $25!$ or $2^{25}$ explicit basis functions. KRR avoids this problem by working with the basis implicitly.

\end{remark}

\subsection{Goal: Valid and sharp confidence sets}

We provide valid and sharp confidence sets for KRR that
(i) apply to nonstandard data types $S$ with various kernels $k$;
(ii) compute as efficiently as computing KRR once;
and (iii) contract at nearly the optimal rate. We now define nonasymptotic validity and sharpness for a sequence of confidence sets $\hat C_n$.\footnote{``Validity'' is also known as ``honesty.'' ``Sharpness'' modifies the usual definition of ``exactness.''}

  \begin{definition}[Validity]\label{def:honest-cb} $\hat C_n$ are $\tau$-valid at level $\a$ if
        $ \bb{P}(f_0 \in \hat{C}_n) \ge 1 - \a - \tau.$
\end{definition}
Validity means that the confidence sets $\hat{C}_n$ cover the true regression function $f_0$ at least at the nominal level, up to a tolerance of $\tau$. \textcolor{black}{In our main results, $\tau=O(n^{-\xi})$ for a sufficiently small fixed $\xi>0$.}

        \begin{definition}[Sharpness]\label{def:exact-cb} $\hat{C}_n$ are $(\d,\tau)$-sharp at level $\a$ if for some $\d, \tau \ge 0$,
        $\bb{P}\{f_0 \in \d \hat f + (1-\delta) \hat{C}_n \} \le 1 - \a + \tau.$
        \end{definition}

        Sharpness means that the confidence sets are not conservative. Our confidence sets $C_n$ are centered at $\hat{f}$, so the convex combination $\d \hat f + (1-\delta) \hat C_n$ is a slight contraction. If the confidence set is contracted by a factor of $1-\delta$, then coverage falls below the nominal level,  up to a tolerance of $\tau$. We will take $\delta=0$ in easy cases, $\delta=1/\log(n)$ in hard cases, \textcolor{black}{and $\tau=O(n^{-\xi})$ in both cases for fixed $\xi>0$.} When $\delta=0$ and $\tau \downarrow 0$, validity and sharpness imply nominal coverage: $\bb{P}(f_0 \in \hat{C}_n) \rightarrow 1 - \a$.

Valid and sharp inference for regression over $25!$ possible regressors is impossible without further structure. In this work, we demonstrate that exactly the same assumptions used to verify the consistency of KRR will justify our inferential procedure.

A bias variance decomposition of KRR illuminates each assumption's role. Let
$
f_{\lambda}= \argmin_{f \in H} [\E \{Y - f(X)\}^2 + \lambda \norm{f}^2]
$ be the pseudo-true parameter. Adding and subtracting,
$$
n^{1/2}(\hat{f}-f_0)=\underbrace{n^{1/2}\{(\hat{f}-f_{\lambda})-\E_n(U)\}}_{\text{residual}}+\underbrace{n^{1/2}\E_n(U)}_{\text{pre-Gaussian}}+\underbrace{n^{1/2}(f_{\lambda}-f_0)}_{\text{bias}},
$$
where $U_i\in H$ is a mean zero random function explicitly defined below.
The first term is a residual that vanishes. The second is the pre-Gaussian term, for which we prove new nonasymptotic Gaussian and bootstrap couplings under a low effective dimension assumption. The third is the bias term, which vanishes by standard approximation arguments under a smoothness assumption. When $f_0\not\in H$, the bias does not vanish, yet we still provide valid inference for $f_{\lambda}$ via our analysis of the residual and pre-Gaussian terms.

\subsection{Main assumption: Low effective dimension}

Our main assumption, which is necessary for consistency of KRR, is that the data, when passed through the kernel, have a low effective dimension: the eigenvalues of the covariance operator $T=\E(k_X\otimes k_X^*)$, satisfying $\langle f,Tg\rangle =\E\{f(X)g(X)\}$, decay at least polynomially. In the special case of linear ridge regression where $p=\dim(X)$ may exceed $n$, we require that the covariance matrix $T=\E(XX^{\top})$ has relatively few important eigenvectors, similar to a factor model assumption in the panel data literature. More generally, we require that $T$ has relatively few important eigenfunctions, though the total number may be infinite.

We quantify low effective dimension using the local width. Let $U = (U_1, U_2, \ldots , U_n)$ be an i.i.d.~sequence of $n$ random variables taking values in $H$, such that  $\bb{E}(U_i)=0$ and $\bb{E}\norm{U_i}^2 < \infty$, and let $\Sigma: H \to H$ defined by
$\Sigma=\bb{E}(U_i \otimes U_i^*)$
denote the associated covariance operator, which is self-adjoint and has finite trace.
The complexity of $\Sigma$ plays a central role in our Gaussian and bootstrap couplings for the empirical process $n^{1/2}\E_n(U)$. By working with the covariance operator, our theory covers settings where the data are high-dimensional and yet possess a latent, low-dimensional structure. In this way, we also avoid imposing explicit regularity conditions on the eigenfunctions, which are challenging to characterize for nonstandard data and general probability distributions.

\begin{assumption}[Local width]\label{assumption:width}
Given $m \ge 0$, the local width
of $\Sigma$ is given by
$\sigma^2(\Sigma, m) = \sum_{s > m} \nu_s(\Sigma),$
where $\{\nu_1(\Sigma), \nu_2(\Sigma), \ldots \}$ are decreasing eigenvalues. We assume that the local width decays: $\sigma^2(\Sigma,m)\lesssim m^{-c}$ for some positive $c>0$ and $m \ge 1$.
\end{assumption}

 The local width is the tail sum of eigenvalues. It quantifies how much of the covariance is not explained by the top $m$ eigenfunctions, i.e. by a nonlinear factor model with $m$ latent factors. It converges to zero for large $m$ when $\bb{E}\snorm{U_i}^2 < \infty$, which is a very weak regularity condition. Our abstract results in Section~\ref{sec:partial} allow $\sigma^2(\Sigma,m) \asymp m^{-c}$ for any positive $c > 0$, corresponding to empirical processes $n^{1/2}\E_n(U)$ indexed by large function classes $H$.
    In this sense, the local width does relatively little to restrict the generality of our results.

To apply our abstract results to KRR, we prove that the pre-Gaussian term for KRR is of the form $n^{1/2}\E_n(U)$, where each summand is given by the random function
 $$
 U_i=(T+\lambda)^{-1}\{(k_{X_i} \otimes k_{X_i}^*-T)(f_0-f_{\lambda}) + \ep_i k_{X_i}\}.$$
For KRR, $n^{1/2}\E_n(U)$ may be viewed as a conditional empirical process \citep{stute1986conditional}.
As $\lambda \downarrow 0$, the Donsker property does not hold because $(T+\lambda)^{-1}$ does not remain bounded.
Therefore we provide nonasymptotic arguments via the local width of $\Sigma$. Then we bound the local width of $\Sigma=\bb{E}(U_i \otimes U_i^*)$ by the local width of $T=\E(k_X\otimes k_X^*)$ via the relation
$$
\sigma^2(\Sigma,m) \le \left(\frac{\kappa\norm{f_0} + \bar\sigma}{\lambda}\right)^2 \sigma^2(T,m).
$$
As such, we ultimately need $\sigma^2(T,m) \downarrow 0$, i.e. spectral decay of the covariance operator $T$.

Spectral decay of the covariance $T$ is reasonable in many cases.
For example, it is satisfied by the Sobolev space of functions with $s$ square integrable derivatives over $[0,1]^p$, as long as $p/s<2$ \citep{fischer2020sobolev}. Intuitively, the effective dimension is increasing in the ambient dimension $p$ and decreasing in the smoothness $s$. It also holds for the RKHS with the Gaussian kernel. For nonstandard yet structured data, such as preferences, we may not always have an analytic justification for the condition. In such cases, we may visualize the eigenvalues through a scree plot. Figure~\ref{fig:eigen} demonstrates that in the Boston Public School setting, while there are $25!$ possible preferences, there are effectively $25$ types of preferences.\footnote{This mirrors other settings with nonstandard data, where the effective dimension has been observed to be much smaller than the ambient dimension \citep{liang2020just}.}

\begin{figure}
\captionsetup[subfigure]{justification=Centering}
\begin{subfigure}[t]{0.48\textwidth}
         \centering
        \resizebox{\textwidth}{!}{%
       \includegraphics[width=\textwidth]{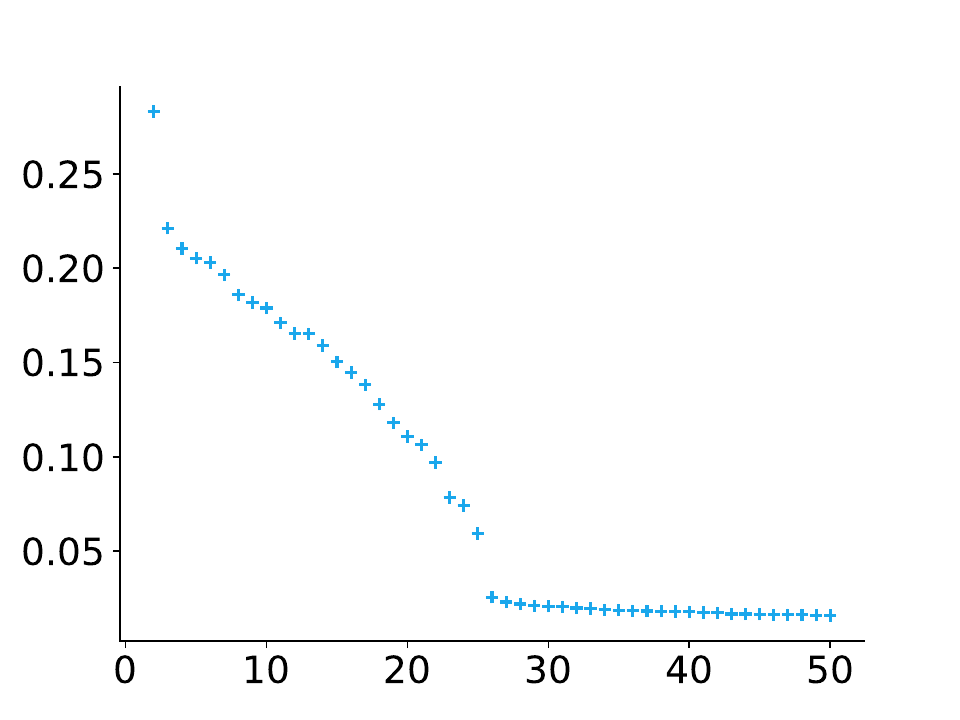}
        }
    \caption{\footnotesize Preferences have a low effective dimension.}\label{fig:eigen}
\end{subfigure}\hspace{\fill} %
\begin{subfigure}[t]{0.48\textwidth}
          \centering
        \resizebox{\textwidth}{!}{%
      \includegraphics[width=\textwidth]{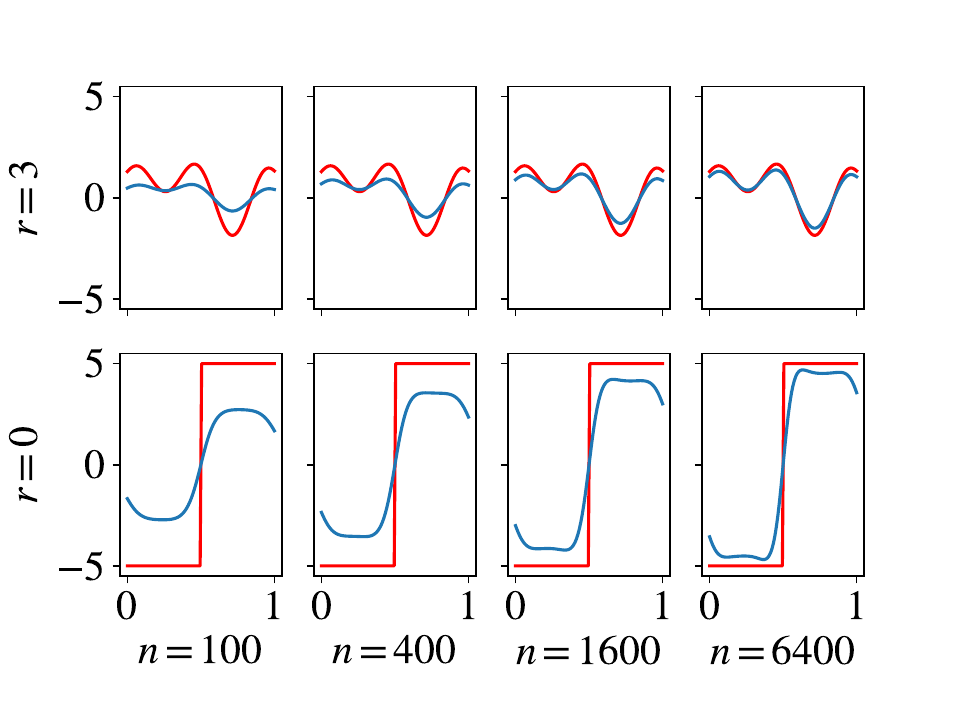}
        }
    \caption{\footnotesize The source condition means faster rates.}\label{fig:source}
\end{subfigure}
\caption{We visualize assumptions underlying consistency of KRR. Figure~\ref{fig:eigen} visualizes spectral decay via the initial $50$ eigenvalues of $\E_n(k_{X}\otimes k_X^*)$, where $X$ is from our application in Section~\ref{sec:application} and $k$ is the preference kernel from Example~\ref{ex:kendall}. Figure~\ref{fig:source} compares $f_{\lambda}$ (blue) to $f_0$ (red). In the upper row, $f_0\in H^3$; in the lower row, $f_0\not\in H$. As we move to the right, we use more observations $(n\uparrow)$ and hence less regularization $(\lambda \asymp n^{-1/2} \downarrow)$.  We see that $f_{\lambda}$ converges to $f_0$ more quickly in the upper row.}
\end{figure}

In Appendix~\ref{sec:spectrum} we further characterize Assumption~\ref{assumption:width}.
We provide concrete upper and lower bounds on local width in leading cases: polynomial and exponential decay of eigenvalues. We also relate it to the entropy number, facilitating comparison with other models. In Appendix~\ref{sec:symbols}, we derive the relation between $\sigma^2(\Sigma,m)$ and $\sigma^2(T,m)$ that we use to derive nonasymptotic inference guarantees for KRR.

\subsection{Additional assumption: Sufficient smoothness}

Our main assumption suffices for valid inference of $f_{\lambda}$. For valid inference of $f_0$, we further require that the bias $\|f_{\lambda}-f_0\|$ vanishes, which is accomplished via the well known source condition \citep{groetsch1984theory}. The source condition imposes that $f_0$ is correctly specified by $H$ and well approximated by the initial eigenfunctions of the covariance operator $T$.
This smoothness assumption appears in
the minimax analysis of KRR, both in $L^2$ and in $H$ \citep{caponnetto2007optimal,fischer2020sobolev}.
Intuitively, KRR is consistent when $f_0$ is sufficiently smooth, with faster rates when $f_0$ is smoother; see Figure~\ref{fig:source}.

\begin{assumption}[Source condition]\label{assumption:source}
        The true regression $f_0$ satisfies $f_0\in H^r$ for some $r\in(1,3]$, where we define $H^r\subseteq H \subseteq L^2$ as
        $
H^r=\left\{f=\sum_{s=1}^{\infty}f_s e_s(T):\;\sum_{s=1}^{\infty} f_s ^2\nu_s^{-r}(T)<\infty\right\}.
        $ Recall $e_s(T)$ are eigenfunctions and $\nu_s(T)$ are eigenvalues of the covariance $T=\E(k_X\otimes k_X^*)$.
    \end{assumption}

  Taking $r=0$ recovers square summability: $\sum_{s=1}^{\infty} f_s ^2<\infty$, which defines $L^2$. Taking $r=1$ gives $\sum_{s=1}^{\infty} f_s ^2 / \nu_s(T) <\infty$, which is equivalent to correct specification. For $r>1$, the smoothness of $f_0$ exceeds the worst case smoothness of $H$; $f_0$ is approximated well by the leading terms in the series $\{e_1(T), e_2(T), \ldots\}$.  This notion of smoothness depends on the kernel and data. For example, two preferences are similar if they make similar pairwise comparisons, and we define smoothness with respect to the induced covariance $T$.

  \begin{figure}
\captionsetup[subfigure]{justification=Centering}
\begin{subfigure}[t]{0.48\textwidth}
         \centering
        \resizebox{\textwidth}{!}{%
       \includegraphics[width=\textwidth]{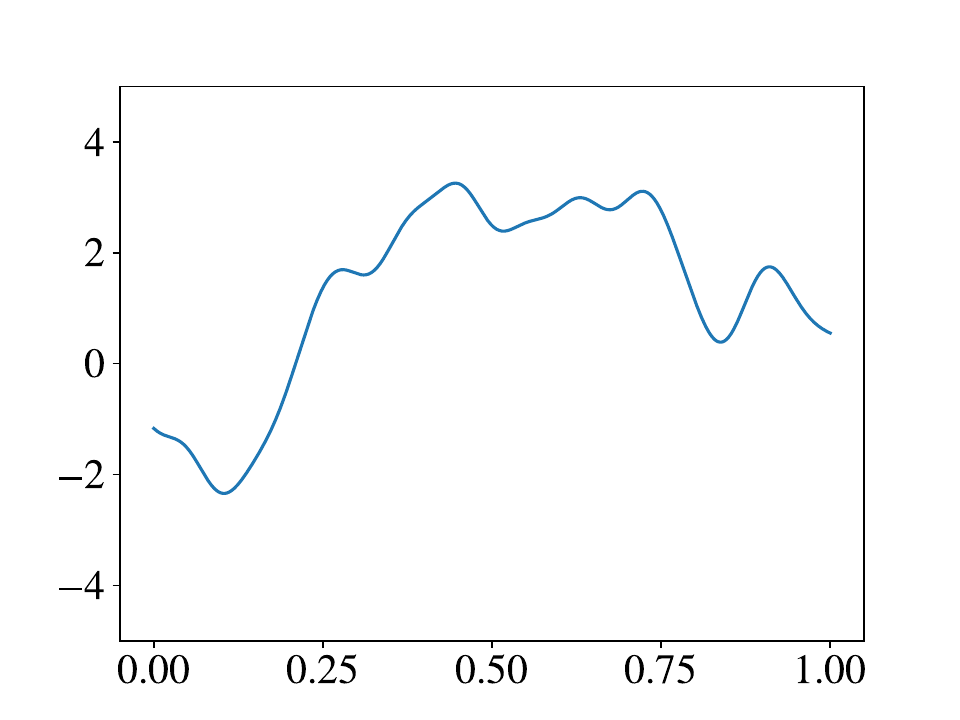}
        }
    \caption{\footnotesize Source condition with $r=1$.}\label{fig:source1}
\end{subfigure}\hspace{\fill} %
\begin{subfigure}[t]{0.48\textwidth}
          \centering
        \resizebox{\textwidth}{!}{%
      \includegraphics[width=\textwidth]{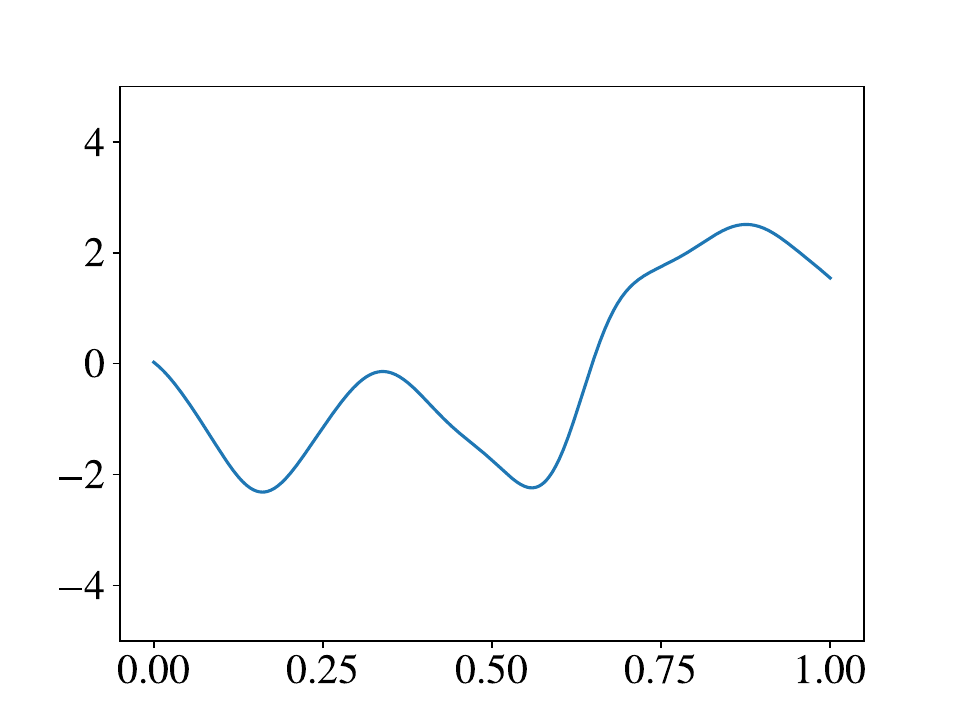}
        }
    \caption{\footnotesize Source condition with $r=3$.}\label{fig:source3}
\end{subfigure}
\caption{The source condition in Sobolev space means the number of square integrable derivatives. We take $H=\mathbb{H}_2^s$. Figures~\ref{fig:source1} and~\ref{fig:source3} visualize typical functions from $H^1$ and $H^3$, respectively.}\label{fig:sources}
\end{figure}

The source condition here
resembles the
one used to analyze ill-posedness in the nonparametric instrumental variable literature; see discussions in e.g. \cite{carrasco2007linear,chen2011rate,singh2020negative}. In Sobolev space, it is satisfied when the true regression $f_0$ has more square integrable derivatives than the regression model $H$ used for estimation.

    \begin{example}[Sobolev space]
        Denote by $\mathbb{H}_2^s$ the Sobolev space with $s>p/2$ square integrable derivatives over $[0,1]^p$. If $H=\mathbb{H}_2^s$ and $f_0\in \mathbb{H}_2^{s_0}$ then $r=s_0/s$ and  $\mathbb{H}_2^{s_0}=(\mathbb{H}_2^{s})^r$. Figure~\ref{fig:sources} illustrates functions with $p=1$, $s=1$, and either $r=1$ or $r=3$.
    \end{example}

\section{Confidence bands with nonstandard data}\label{sec:algo_main}

Inference for KRR presents several challenges.
First, our motivating interest in KRR is its versatility with nonstandard regressors such as preferences.
To preserve this versatility, we derive Gaussian and bootstrap couplings in the RKHS that can exploit low effective dimension.
Importantly, our couplings directly apply to many data types $S$ and kernels $k$.
Second, a computationally intensive inference procedure may undermine the practicality of KRR.
To preserve computational efficiency, we propose an anti-symmetric bootstrap, in closed form, that re-uses the kernel evaluations and matrix inversions of KRR.
Third, to cover $f_0$ we require $\lambda \downarrow 0$, but in this regime asymptotic central limit theorems based upon e.g.~the Donsker property do not apply.\footnote{As $\lambda \downarrow 0$, the resulting stochastic process is not totally bounded in $L^2(\bb{P})$.}
To establish inference without a stable Gaussian limit, we develop a nonasymptotic framework that transfers Gaussian and bootstrap couplings to our inferential procedure.

It is not obvious we can choose a sequence $\lambda \downarrow 0$ that vanishes slowly enough for Gaussian approximation and yet quickly enough to cover $f_0$. We prove that both can be achieved in many settings, often alongside a near-optimal rate of estimation in $H$ norm. Formally, Theorem~\ref{thm:near-minimax-band} provides $H$ norm confidence sets for $f_0$ that are valid, sharp, and nearly minimax.
These imply valid yet conservative $\sup$ norm bands, since $\|f\|_{\infty}\leq \kappa\|f\|$ for all $f\in H$.
As an extension, we show that non-conservative $\sup$ norm bands are possible under an additional assumption.

\subsection{This work: Anti-symmetric bootstrap}\label{sec:algo_bootstrap}

We state our procedure at a high level before filling in details.
\begin{itemize}
    \item For each bootstrap iteration, draw Gaussians and compute the function $\mathfrak{B}$.
    \item Across bootstrap iterations, compute the $(1-\alpha)$-quantile, $\hat{t}_{\alpha}$, of $\|\mathfrak{B}\|$.
    \item Calculate the band $\hat C_\a$ containing $\hat f + \hat{t}_\a \cdot  n^{-1/2}f$ for $\|f\|\leq 1+\delta$.
\end{itemize}

This structure is familiar. As part of our contribution, we express $\mathfrak{B}$ in closed form, with low computational overhead. We propose anti-symmetric multipliers in the definition of $\BS$ to efficiently correct regularization bias. Our proposal departs from e.g.~\citet{yang2017frequentist}, who provide frequentist analysis of the Bayesian posterior.

\begin{algorithm}[Incremental factor confidence band]\label{algo:incremental}
     Given a sample $D = \{(X_i,Y_i)\}_{i=1}^n$, a kernel $k$, a regularization parameter $\lambda>0$, and an incremental factor set to $\delta=0$ or $\delta=1/\log(n)$:
     \begin{enumerate}
         \item Compute the kernel matrix $K \in \mathbb{R}^{n\times n}$ with entries $K_{ij}=k(X_i,X_j)$.
         \item Compute the KRR residuals $\hat{\varepsilon}\in\R^n$ by $\hat{\varepsilon}=Y-K(K+n\lambda  I )^{-1}Y$.
         \item For each bootstrap iteration,
         \begin{enumerate}
             \item draw multipliers $q\in\R^n$ from $\mathcal{N}(0,I-\boldsymbol{1}\boldsymbol{1}^{\top}/n)$, where $\boldsymbol{1} = (1, \ldots, 1)^{\top} \in \bb{R}^n$
             \item compute the vector $\hat{\gamma}=n^{1/2}(K+n\lambda  I )^{-1} \diag(\hat{\varepsilon})q$;
             \item compute the scalar $M=(\hat{\gamma}^{\top}K\hat{\gamma})^{1/2}$.
         \end{enumerate}
         \item Across bootstrap iterations, compute the $(1-\alpha)$-quantile, $\hat{t}_{\alpha}$, of $M$.
         \item Calculate the $H$ norm band $\hat C_\a= \Set[\hat f + \hat{t}_\a \cdot  n^{-1/2}f]{\|f\|\leq 1+\delta}$.
         \item Calculate the uniform band $\hat C_\a(x)=\hat{f}(x) \pm \hat{t}_\a \cdot  n^{-1/2}\kappa (1+\delta)$.
     \end{enumerate}
\end{algorithm}

Implicitly, within each bootstrap iteration, $M=\|\mathfrak{B}\|$ is computed in closed form. The confidence set $\hat C_\a$ is KRR plus the $H$ norm critical value, times $n^{-1/2}$ and an incremental factor $(1+\delta)$. We emphasize that $\delta$ is not a tuning parameter: we recommend taking $\delta = 1/\log(n)$ as a conservative choice. For moderate $n$, the resulting confidence bands are indistinguishable from $\delta=0$ in practice.
Section~\ref{sec:sup-band} below provides an extension with variable width confidence bands.

Inference (Algorithm~\ref{algo:incremental}) is not much costlier than estimation (Algorithm~\ref{algo:krr}) from a computational perspective, as it reuses the regularized inverse $(K+n\lambda  I )^{-1}$. It has no kernel evaluations nor matrix inversions beyond KRR.

The multipliers $q\in\R^n$ can be viewed as anti-symmetric: each element of $\diag(\hat{\varepsilon})q$ is $q_i \hat{\varepsilon}_i=\frac{1}{\sqrt{n}}\sum_{j=1}^n\frac{h_{ij}-h_{ji}}{\sqrt{2}}\hat{\varepsilon}_i$, where $h\in\R^{n\times n}$ is a matrix of i.i.d. standard Gaussians. Here, $\frac{h_{ij}-h_{ji}}{\sqrt{2}}$ is also a standard Gaussian due to independence between $h_{ij}$ and $h_{ji}$, and the matrix $\frac{h-h^{\top}}{\sqrt{2}}$ is anti-symmetric. By taking the difference between $h_{ij}$ and $h_{ji}$, we cancel the bias in $\hat{\varepsilon}_i$. This is equivalent to empirically centering the multipliers when performing the standard Gaussian multiplier bootstrap.

The technique of de-meaning bootstrap residuals has a classical precedent in standard regression, in homoscedastic models with nonrandom regressors \citep{freedman1981bootstrapping}.\footnote{In contrast to the multiplier bootstrap for KRR considered here, \citet{freedman1981bootstrapping} considers re-sampling regression errors, and so does not accommodate heteroscedasticity.}
A related idea appears when testing conditional moment restrictions \citep{sorensen2022testing}. Its application to confidence bands for regularized nonparametric regression appears to be new.

\subsection{Main result: Valid and sharp inference}\label{sec:algo_lead}

We now state our primary theoretical contribution: Theorem~\ref{thm:near-minimax-band} shows that under Assumptions \ref{assumption:width} and \ref{assumption:source},
 Algorithm \ref{algo:incremental} produces valid and sharp confidence sets. Moreover, for an appropriate choice of $\lambda$, the confidence sets contract at a rate arbitrarily close to the minimax rate of estimation in $H$-norm.
Corollary \ref{cor:implied-band} shows that these yield valid uniform confidence bands, which are easy to compute, and which contract at a similar rate. Section~\ref{sec:algo_incremental} below demonstrates that our results hold more generally, for a wide range of regularization parameters $\lambda$.

We focus on an $H$-norm critical value for three key reasons.
First, the $H$-norm $\|\BS\|$ can be computed efficiently in closed form, as detailed in Algorithm \ref{algo:incremental} above.
Second, the rate of convergence of KRR in $H$-norm and its minimax lower bound depend only on the effective dimension and source conditions of Section \ref{sec:model}. This allows us to characterize precisely when valid and sharp inference is possible with nonstandard data, and furthermore when nominal coverage is possible. Third, they imply practical yet valid uniform confidence bands.

Necessarily, a $\sup$-norm critical value will have certain shortcomings for nonstandard data. A practical shortcoming is that with preference or sequence data, bootstrapping the supremum of $|\BS(x)|$ requires searching for the maximum among a large number of elements $x \in S$.  A theoretical shortcoming is that for the $\sup$ norm, minimax rates of estimation have only been worked out in a few special cases, and these are needed for a complete characterization. Nevertheless, Section~\ref{sec:sup-band} below provides and analyzes an extension of our inference procedure with a  $\sup$-norm critical value.

We state our result in two leading cases. In one leading case, the spectrum decays polynomially, i.e. $\nu_s(T) \asymp \omega s^{-\beta}$for some $\beta> 1$.\footnote{The case $\beta \le 1$ is ruled out by the fact that $k$ is bounded, which implies $\sum_{s \ge 1} \nu_s(T) < \infty$. Thus, the assumption of polynomial decay involves only a small loss of generality. The matching lower-bound on eigenvalues of $T$ imposes a form of self-similarity, see \citet{bull2013adaptive}.}  Sobolev kernels belong to this regime. In another leading case, it decays exponentially, i.e. $\nu_s(T) \asymp \omega \exp(-\alpha s^\gamma)$ for some $\gamma \in(0,1)$. The popular Gaussian kernels and inverse multi-quadratic kernels belong to this regime.
Each regime implies that Assumption~\ref{assumption:width} holds for the pre-Gaussian term of KRR.

\textcolor{black}{Throughout, the bounded-kernel and bounded-residual assumptions imply that the KRR summands are bounded, and we state the results directly under this maintained regime.}

\begin{theorem}[Nearly minimax confidence sets]\label{thm:near-minimax-band}
For $\a \in (0,1)$, define $\hat{t}_\a$ by $\bb{P}(\snorm{\mathfrak{B}} > \hat{t}_\a|D) = \a$.\footnote{Conditional upon $D$, $\snorm{\mathfrak{B}}$ has a density \textcolor{black}{or is identically zero}, as it is a quadratic form of a finite dimensional Gaussian.}
\textcolor{black}{Suppose Assumption~\ref{assumption:source} holds with $r>1$}.

\begin{enumerate}[label=(\alph*)]
    \item Under polynomial decay, choose $\d = 1/\log(n)$  %
   and $\lambda \asymp n^{-(\beta + \eps)/(r\beta + 1)}$ for some small $\eps > 0$.
   \textcolor{black}{For any fixed $\xi\in(0,1/(\beta+1))$, $\hat{C}_\a$ in Algorithm~\ref{algo:incremental} is $O(n^{-\xi})$ valid and $\{2/\log(n),O(n^{-\xi})\}$ sharp,} and shrinks at rate $n^{-\rho + \epsilon}$, where $\rho = \beta(r-1)/\{2(r\beta+1)\}$ reflects the minimax rate in $H$ norm.
   \item Under exponential decay, choose $\delta=0$ %
   and $\lambda \asymp n^{-(1+\eps)/r}$ for some small $\eps > 0$. %
   \textcolor{black}{For any fixed $\xi\in(0,1-(1+\eps)/r)$, $\hat{C}_\a$ in Algorithm~\ref{algo:incremental} is $O(n^{-\xi})$ valid and $\{0,O(n^{-\xi})\}$ sharp,}
   and shrinks at rate $n^{-\rho + \epsilon}$, where $\rho = (r-1)/2r$ reflects the minimax rate in $H$ norm. In particular, it obtains nominal coverage: $\bb{P}(f_0 \in \hat{C}_\a) \rightarrow 1 - \a$ as $n \to \infty$.
\end{enumerate}
\end{theorem}

\begin{remark}[Nonzero incremental factor for polynomial decay.]
To obtain the minimax rate of estimation, the regularization parameter $\lambda$ must balance KRR's bias and variance, summarized in the initial two rows of Table~\ref{tab:rate-all}.
With rate-optimal regularization, polynomial decay is a ``hard'' case while exponential decay is an ``easy'' case.
In hard cases, the variance is too large to use an anti-concentration argument:
the error $\|n^{1/2}(\hat{f}-f_{\lambda})\|$, when scaled by the variance, has a degenerate distribution. Despite this challenge, a nonzero incremental factor $\delta=1/\log(n)$ confers valid and sharp inference, as developed in Section~\ref{sec:algo_incremental}.
In easy cases, the variance is small enough to use an anti-concentration argument. In these cases, it is possible obtain not only valid and sharp inference but also nominal coverage by taking $\delta=0$, as developed in Section~\ref{sec:sup-band}.
In practice, we recommend choosing $\delta=1/\log(n)$; the difference in size is often  negligible.
\end{remark}

An important fact underlying many aspects of our analysis is that the $H$-norm is stronger than the $\sup$ norm: for any $f \in H$, $\sup_{x\in S}|f(x)| \le \kappa \|f\|$. This allows us to easily construct uniform confidence bands $\hat C_\alpha(x)$ from $\hat C_\alpha$, which contract at rates in Theorem \ref{thm:near-minimax-band}. Although conservative, they perform well in practice, as documented by extensive simulations in Section~\ref{sec:application}.

\begin{corollary}[Implied uniform confidence bands]\label{cor:implied-band}
The uniform confidence band $\hat f(x) \pm n^{-\nicefrac{1}{2}}\kappa(1+\delta)\hat t_\alpha$ given in Algorithm \ref{algo:incremental} is \textcolor{black}{$O(n^{-\xi})$ valid, for $\xi$ as in Theorem~\ref{thm:near-minimax-band}.}
\end{corollary}

\subsection{Incremental factor handles hard cases}\label{sec:algo_incremental}

To handle the hard cases described above,
we develop inference with the nonzero incremental factor $\delta=1/\log(n)$, building on the technique of \cite{andrews2013inference}. We state our general
result at a high level before filling in the details in Table~\ref{tab:rate-all}. Importantly, our general result allows a wide range of $\lambda$ values.

\begin{assumption}[High level quantities]\label{assumption:light-main}
    We assume there exist functions $(Q,R,L,B)$ so that the following statements hold. We summarize these functions in Table~\ref{tab:rate-all} for leading cases. We present them with generality in Appendix~\ref{sec:explicit}, using Assumptions~\ref{assumption:width} and~\ref{assumption:source}.
    \begin{itemize}
        \item Gaussian approximation. There exists a Gaussian random element $Z$ in $H$ such that with probability at least $1-\eta$,
    $\norm{\sqrt{n}(\hat f - f_\lambda) - Z} \le Q(n,\lambda,\eta).$ This condition will be verified with Theorem~\ref{thm:gauss-main} using local width.
        \item Bootstrap approximation. There exists a random element $Z'$ in $H$ whose conditional distribution given $D$ is almost surely Gaussian with covariance $\Sigma$, and with probability at least $1-\eta$,
    $\bb{P}\left\{\beef \norm{\mathfrak{B} - Z'} \le R(n,\lambda,\eta)\middle|D\right\} \ge 1-\eta.$ This condition will be verified with Theorem~\ref{thm:bs-main} using local width.
        \item Variance lower bound. It holds with probability $1-\eta$ that
    $\snorm{Z} \ge L(\lambda,\eta)$ for some function $L$ which is strictly increasing in $\eta$. This condition will be verified with a lemma, again using local width.
        \item Bias upper bound. It holds that $\sqrt{n}\snorm{f_\lambda - f_0} \le B(\lambda)$. This condition is well known using the source condition:
        $B(\lambda)=n^{1/2}\kappa^{1-r} \lambda^{(r-1)/2}\|f_0\|$ \citep{smale2005shannon}.

    \end{itemize}
 Finally let $\Delta(n,\lambda,\eta)=Q(n,\lambda,\eta) + R(n,\lambda,\eta)$. We abbreviate by suppressing arguments.
\end{assumption}

These high level quantities illustrate the tension between estimation and inference. On the one hand, as $\lambda\downarrow 0$, the bias $B$ shrinks. On the other hand, as $\lambda\downarrow 0$, $\Delta=Q+R$ and $L$ increase because the Gaussian process increases in complexity. Our general result navigates this tension.

\begin{proposition}
[$H$ norm inference via incremental factor]\label{prop:h-band}
For $\a \in (0,1)$, define $\hat{t}_\a$ by $\bb{P}(\snorm{\mathfrak{B}} > \hat{t}_\a|D) = \a$.
Suppose Assumption~\ref{assumption:light-main} holds, and the incremental factor $\d$ satisfies
$$\f 1 2 \ge \d \ge \frac{\D(n,\lambda,\eta) + B(\lambda)}{L(\lambda,1-\a - 2\eta) - \D(n,\lambda,\eta)},$$ where the denominator is positive and $\eta < \min(\alpha, \frac{1-\alpha}{2})$.
Then
$\hat{C}_\a$ in Algorithm~\ref{algo:incremental} is $(3\eta)$-valid and $(2\d,3\eta)$-sharp in $H$ norm.
\end{proposition}

We prove Theorem \ref{thm:near-minimax-band}(a) by combining Proposition~\ref{prop:h-band} with an explicit characterization of $(Q,R,L,B)$, and choosing \textcolor{black}{$\eta = n^{-\xi}$} and $\delta = 1/\log(n)$. Table~\ref{tab:rate-all} provides details. For each regime, we characterize $(Q,R,L,B)$ of Assumption~\ref{assumption:light-main}. We then characterize the restrictions on the regularization parameter $\lambda$ implied by the conditions that $B\ll L$ (undersmoothing) and $(Q+R)\ll L$ (valid approximation) in Proposition~\ref{prop:h-band}. Finally, we demonstrate that these conditions typically allow $\lambda$ to approach the minimax optimal choice $\lambda \asymp n^{-\beta/(r\beta+1)}$ for estimation.

\begin{table}
    \centering
    \caption{\label{tab:rate-all} \textcolor{black}{Theorem \ref{thm:near-minimax-band}(a) under polynomial and exponential spectral decay, suppressing logarithmic factors.$^*$}}

    {\color{black}
    \begin{tabular}{ccc}
       & Poly. : $\nu_s(T) \asymp \omega s^{-\beta}$ & Exp. : $\nu_s(T) \asymp \omega \exp(-\alpha s^\gamma)$ \\
       \cmidrule(lr){2-2}\cmidrule(lr){3-3}
        $B$ & $\lambda^{\f{r-1}{2}}n^{\f 1 2}$ & $\lambda^{\f{r-1}{2}}n^{\f 1 2}$ \\
        $L$ & $\lambda^{- \f 1 2 - \f {1}{2\beta}}$ & $\lambda^{- \f 1 2 }$ \\
        residual & $\lambda^{-1-1/\beta}n^{-\frac{1}{2}}$ & $\lambda^{-1}n^{-\frac{1}{2}}$ \\
        \rule{0pt}{3ex} $Q_{\bullet}$ &
            $\lambda^{-1}n^{\f{(1-\xi)(1-\beta)}{2\beta}}$ &
            $\lambda^{-1} n^{-\frac{1-\xi}{2}}$ \\
        \rule{0pt}{1ex} $R_{\bullet}$ &
            $\left({\lambda^{3+\f{1}{\beta}+\f{2}{\beta-1}}n}\right)^{\f{1-\beta}{4\beta-2}}$ &
            $\lambda^{-\frac{3}{4}}n^{-\frac{1}{4}}$ \\
        \cmidrule(lr){1-3}
        $B \ll L$ & $\lambda \ll n^{-\frac{\beta}{r\beta+1}}$ & $\lambda \ll n^{-\frac{1}{r}}$ \\
        \rule{0pt}{3ex} $Q + R \ll L$ &
            $\lambda \gg n^{-\frac{\beta}{\beta+1}}$ &
            $\lambda \gg n^{-(1-\xi)}$ \\
        minimax? & \checkmark & \checkmark \\
        \cmidrule(lr){1-3}
    \end{tabular}}

\raggedright
     \footnotesize $^*${Here, $Q=Q_{\bullet}+\text{residual}$ and $R=R_{\bullet}+\text{residual}$. The initial two rows present rates for the bias upper bound $B$ and variance lower bound $L$. The third row bounds the KRR residual.  The fourth and fifth rows, when added with the third row, present rates for the Gaussian coupling $Q$ and bootstrap coupling $R$.
 The sixth row presents the restriction on $\lambda$  implied by $B\ll L$.
 The seventh row  presents the restriction on $\lambda$  implied by $Q+R \ll L$. The final row evaluates whether the allowed path for $\lambda$ can approach the minimax optimal choice for estimation.
     \textcolor{black}{Here $\eta=n^{-\xi}$; the polynomial-decay column takes $\xi<1/(\beta+1)$, and the exponential-decay column takes $\xi$ as in Theorem~\ref{thm:near-minimax-band}. Letting $\xi\downarrow0$ compatibly as $\beta\to\infty$, the polynomial-decay rates recover the exponential-decay rates.}}
\end{table}

\subsection{Nominal coverage in easy cases, and extensions}\label{sec:sup-band}

As previewed in Section~\ref{sec:algo_lead}, in easy cases, we can improve our inference guarantee from validity and sharpness to nominal coverage, by building on the techniques of \citet{chernozhukov2014anticoncentration}. As before, we state our general result at a high level then fill in details using Table~\ref{tab:rate-all}.

In such cases, our result allows us to extend inference from $H$ norm bands to additional nonlinear functionals of $n^{1/2}(\hat{f}-f_0)$. For example, it allows us to study variable-width uniform confidence bands for $f_0$, which are useful for empirical researchers. It also allows us to study the functional that corresponds to a test for match effects, which will be the focus of Section~\ref{sec:application}.

We state our results for an arbitrary $\sup$ norm continuous functional $F$. In particular, let $F: H \to \bb{R}$ be a uniformly continuous functional, i.e.
$|F(u) - F(v)| \le \psi(\norm{u-v})$ for some modulus of continuity $\psi: \bb{R} \to \bb{R}$.\footnote{This condition weaker than uniform continuity in $\sup$ norm since $\sup_{x\in S}|u(x)-v(x)|\leq \kappa \|u-v\|$.}
If $F(f) = \|f\|$, then $\psi(x) \le  x$.
If $F(f) = \sup_{x \in S} |f(x)|$, then $\psi(x) \le \kappa x$.
In order to establish nominal coverage for $F$, we use the following additional assumption.

\begin{assumption}[Anti-concentration]\label{assumption:anti-concentration}
{\color{black}
There is a nondecreasing function $\mathfrak{a}:(0,\infty)\to[0,1]$, potentially depending on $n$ and $\lambda$, such that $\mathfrak{a}(r)\downarrow0$ as $r\downarrow0$ and
$
\sup_{t\in\bb R}\bb{P}\{|F(Z)-t|\leq r\}\leq\mathfrak{a}(r)$
 for every $r>0$.
}
\end{assumption}
{\color{black}
Intuitively, $\mathfrak{a}(r)$ controls the probability that $F(Z)$ occupies any window of width $r$. A density bounded by $\zeta$ is one sufficient condition, giving $\mathfrak{a}(r)\leq2\zeta r$. When $F(\cdot)=\|\cdot\|$ is the $H$-norm, as in Theorem~\ref{thm:near-minimax-band}(b), Proposition~\ref{prop:hnorm-anti} verifies that
\(
\mathfrak{a}(r)
\) is of order $r\sqrt{\lambda}$
up to logarithmic factors. Similar bounds have been derived in many other settings \citep{chernozhukov2014anticoncentration}. Since $\mathfrak{a}(r)$ is not always known for nonstandard data and general $F$, we also provide a data-driven bound that can be computed and used as a diagnostic.\footnote{In general, $\mathfrak{a}(r)$ will depend both on the functional $F$ and on the covariates' precise form and distribution.}
}

\begin{proposition}[Uniform inference via anti-concentration]\label{prop:inf-intro}
{\color{black}
Suppose Assumptions~\ref{assumption:light-main} and~\ref{assumption:anti-concentration} hold. Then with probability $1-\eta$,
$$
\sup_{t \in \bb{R}} \left|\bb{P}\left[F\Big\{\sqrt{n}(\hat f - f_0)\Big\} \le t \right] -  \bb{P}\left\{ \beef F(\mathfrak{B}) \le t \middle| D\right\} \right|
\le \mathfrak{a}\{\psi(Q+B)\}+\mathfrak{a}(\psi R)+2\eta.
$$
Moreover, even in the absence of Assumption~\ref{assumption:anti-concentration}, we have the data-driven bound
$$
    \sup_{t \in \bb{R}} \left|\bb{P}\left[F\Big\{\sqrt{n}(\hat f - f_0)\Big\} \le t \right]  -  \bb{P}\left\{ \beef F(\mathfrak{B} ) \le t \middle| D\right\} \right|
    \le 2\sup_{t\in\mathbb{R}} \mathbb{P}\left\{\beef |F(\mathfrak{B}) -t|
     \le 2\psi(\Delta + B) \middle| D \right\} + 4\eta.
$$
}
\end{proposition}

{\color{black}
We prove Theorem~\ref{thm:near-minimax-band}(b) by combining Proposition~\ref{prop:inf-intro} with an explicit characterization of $(Q,R,B)$, then choosing $\eta=n^{-\xi}$ and $F(\cdot)=\|\cdot\|$. Together, Proposition~\ref{prop:hnorm-anti} and Proposition~\ref{prop:inf-intro} imply that the Kolmogorov distance vanishes whenever $\mathfrak{a}(\Delta+B)\to0$. This simplifies to $B\ll\lambda^{-1/2}$ (undersmoothing) and $(Q+R)\ll\lambda^{-1/2}$ (valid approximation). Therefore, replacing $L$ by $\lambda^{-1/2}$ in Table~\ref{tab:rate-all} gives sufficient conditions for nominal coverage of Algorithm~\ref{algo:incremental} with $\delta=0$. Suppressing logarithmic factors, $L\asymp\lambda^{-1/2}$ under exponential decay, so estimation approaches the minimax rate while ensuring nominal coverage via anti-concentration. Under polynomial decay, $L\gg\lambda^{-1/2}$; nominal coverage via anti-concentration requires further undersmoothing.\footnote{Theorem~\ref{thm:near-minimax-band}(a) avoids further undersmoothing by using an incremental factor.}
}

Proposition~\ref{prop:inf-intro} confers nominal coverage not only for $F=\|\cdot\|$, but also for other important choices of $F$. Confidence sets for other choices of $F$ are obtained by lightly extending Algorithm~\ref{algo:incremental}: instead of quantiles of $\|\BS\|$, use quantiles of $F(\BS)$ across bootstrap iterations to determine the critical value. This extension gives variable width uniform confidence bands.

\begin{algorithm}[Variable width confidence band]\label{algo:variable}
Consider the setting of Algorithm~\ref{algo:incremental}.
\begin{enumerate}
    \item For each bootstrap iteration of Algorithm~\ref{algo:incremental}, draw Gaussians and compute $\mathfrak{B}(x) =  K_x \hat\gamma$.
    \item Estimate $\tilde{\s}^2(x)=\E_q\{\mathfrak{B}(x)^2\}$ by averaging across bootstrap iterations.
    \item Across bootstrap iterations, compute the $(1-\alpha)$-quantile, $\hat{t}_{\alpha}$, of $\sup_{x \in S} \left|\tilde{\s}(x)^{-1}\BS(x)\right|$.
    \item Calculate the uniform band $\hat C_\a(x)=\hat f(x) \pm \hat{t}_\a \cdot  n^{-1/2}\tilde{\s}(x)$ for $x\in S$.
\end{enumerate}
\end{algorithm}

Comparing Algorithms~\ref{algo:incremental} and~\ref{algo:variable}, the bootstrap $\mathfrak{B}$ is the same. We modify the critical value to provide less conservative and variable width bands in $\sup$ norm.%

\begin{corollary}[Variable width coverage]\label{cor:variable} Let $w: S \to (0,\infty)$ be a bounded weight function. Then $F_w(f) = \sup_{x \in S}|f(x)w(x)|$ is uniformly continuous with modulus $x \mapsto x \kappa\snorm{w}_\infty$. Moreover, $F_w\{\sqrt{n}(\hat f - f_0)\} \le t$ is precisely the event that for all $x \in S$,
$$ \hat f(x) - t \cdot n^{-1/2}w(x)^{-1} \le f_0(x) \le \hat f(x) + t \cdot n^{-1/2}w(x)^{-1}.$$
When $w(x)^{-1}$ is replaced by the estimated quantity $\tilde{\mathfrak{s}}(x)$, Appendix~\ref{sec:band} \textcolor{black}{discusses conditions under which}
$
\mathbb{P} \{f_0(x)\in \hat{C}_{\a}(x) \text{ for all $x\in S$}\} \rightarrow 1-\a.$
\end{corollary}

{\color{black}
The $\sup$ norm critical value in Algorithm~\ref{algo:variable} has certain practical and theoretical shortcomings compared to the $H$ norm critical value in Algorithm~\ref{algo:incremental}, discussed in Section~\ref{sec:algo_lead}. Future work may characterize the finite-window modulus $\mathfrak{a}$ for Algorithm~\ref{algo:variable}, which may depend on specific details of the covariate space $S$. By contrast, Proposition~\ref{prop:hnorm-anti} characterizes this modulus for Algorithm~\ref{algo:incremental}; when in doubt, we recommend that researchers use the robust uniform bands of Corollary~\ref{cor:implied-band}.
}

\section{Nonasymptotic couplings}\label{sec:partial}

Recall the bias variance decomposition for KRR
$$
n^{1/2}(\hat{f}-f_0)=\underbrace{n^{1/2}\{(\hat{f}-f_{\lambda})-\E_n(U)\}}_{\text{residual}}+\underbrace{n^{1/2}\E_n(U)}_{\text{pre-Gaussian}}+\underbrace{n^{1/2}(f_{\lambda}-f_0)}_{\text{bias}}.
$$
We now prove nonasymptotic Gaussian and bootstrap couplings for $n^{1/2}\E_n(U)$.

We provide results general for i.i.d. centered sequences $(U_1,U_2,...)$ taking values in $H$, with $\bb{E}(U_i)=0$, $\bb{E}\norm{U_i}^2 < \infty$, and
$\Sigma=\bb{E}(U_i \otimes U_i^*)$. We demonstrate that Gaussian and bootstrap couplings in $H$  have excellent dependence on the $L^2$ covering number. Theorem~\ref{thm:gauss-main} constructs a Gaussian coupling $Z$ for $n^{1/2}\E_n(U)$. Theorem~\ref{thm:bs-main} allows us to sample from $Z$ via the bootstrap coupling $Z_{\mathfrak{B}}$.

We apply these general results to the pre-Gaussian term of KRR, where
$$
U_i=(T+\lambda)^{-1}\{(k_{X_i} \otimes k_{X_i}^*-T)(f_0-f_{\lambda}) + \ep_i k_{X_i}\}.
$$
Doing so characterizes $(Q,R)$ in Assumption~\ref{assumption:light-main} and hence underpins Propositions~\ref{prop:h-band} and~\ref{prop:inf-intro}. Within Algorithms~\ref{algo:incremental} and~\ref{algo:variable}, $\mathfrak{B}$ is the empirical counterpart to the bootstrap $Z_{\mathfrak{B}}$ below. %

\subsection{Gaussian and bootstrap approximations for non-standard data}

\textcolor{black}{In addition to our main assumption on the local width, given in Section~\ref{sec:model}, we impose boundedness on the summands $U_i$.}

\begin{definition}[Bounded summands]\label{def:regularity_partial}
\textcolor{black}{$U_i$ is \emph{$a$-bounded} if $\norm{U_i} \le a$ almost surely.}
\end{definition}

\textcolor{black}{For KRR, our maintained bounded-kernel and bounded-residual assumptions imply that we can choose $a=\lambda^{-1}\left(\kappa^2\norm{f_0} + \kappa\bar\sigma\right)$.}

\begin{theorem}[Gaussian coupling]\label{thm:gauss-main}
Suppose the $U_i$ are $a$-bounded. Then for all $m \ge 1$ there exists a Gaussian random variable $Z$ taking values in $H$, with $\bb{E}(Z \otimes Z^*)=\Sigma$, such that
with probability at least $1-\eta$,
$$
\textcolor{black}{\left\|n^{1/2}\E_n(U) - Z\right\| \lesssim
a\sqrt{\frac{m[1+\log(n)]}{n\eta}}
+\sqrt{\log(6/\eta)} \sigma(\Sigma, m).}
$$
\end{theorem}

\textcolor{black}{Theorem~\ref{thm:gauss-main} is a nonasymptotic coupling for an empirical process. Such nonasymptotic couplings are important for approximating KRR, since $U_i$ depends on $\lambda$ and thus varies with the sample size.}

On the right hand side, we obtain a bound for any truncation $m$ of the spectrum.
\textcolor{black}{The former term reflects the quality of the bounded-vector Gaussian coupling for the initial $m$ eigenfunctions, which we isometrically embed into $\R^m$ and analyze using a recent, sharp result of \citet{eldan2020clt}. For KRR, we can choose $a=\lambda^{-1}\left(\kappa^2\norm{f_0} + \kappa\bar\sigma\right)$.} The latter term in the bound is the local width, i.e. the tail sum of eigenvalues after the threshold. For KRR, we verify that
$$
\sigma^2(\Sigma,m)\leq \left(\frac{\kappa\norm{f_0} + \bar\sigma}{\lambda}\right)^2 \sigma^2(T,m).
$$
This finite dimensional projection technique is established \cite{gotze2011estimates}. It seems it has not been applied to the RKHS before, where it is especially powerful, implying approximation in $\sup$ norm with very good dependence on model complexity.

Our proof strategy is classical, yet it appears to use the right tools for the job. Compared to highly general results that apply beyond the RKHS, e.g.~\citet{chernozhukov2014gaussian,chernozhukov2016empirical} and related works, we improve the mode and rate of Gaussian coupling within the RKHS.
In particular, we obtain stronger forms of approximation, and faster rates, by focusing on Hilbert norms.
This strategy appears uniquely suited to the RKHS setting, where we can analyze Gaussian approximation in a Hilbert norm then transfer it to the $\sup$ norm.
See Appendix~\ref{sec:zaitsev} for formal details.

\subsection{Anti-symmetric bootstrap}

In order to perform inference, we need some way to sample from the approximating Gaussian distribution. To do so, we propose the anti-symmetric multiplier bootstrap
$$
Z_{\mathfrak{B}}=\frac{1}{n}\sum_{i=1}^n\sum_{j=1}^n\left(\frac{U_i - U_j}{\sqrt 2}\right)h_{ij}=\frac{1}{n}\sum_{i=1}^n\sum_{j=1}^n \left(\frac{V_i - V_j}{\sqrt 2}\right)h_{ij}
$$
where the $h_{ij}$ are i.i.d.~Gaussian multipliers and
 $V_i = U_i + \mu$, for some unknown, deterministic bias $\mu$.
The anti-symmetric bootstrap is motivated by our application to KRR, where we aim to provide valid uncertainty quantification even when the bias is significant.
It is
based on the observation that $U_i$ has the same covariance as $(U_i-U_j)/\sqrt{2}=(V_i-V_j)/\sqrt{2}$. In Section~\ref{sec:collect} below, we apply it to KRR, matching symbols with Algorithms~\ref{algo:incremental} and~\ref{algo:variable}.

\begin{theorem}[Anti-symmetric bootstrap coupling]\label{thm:bs-main}
Suppose the $U_i$ are $a$-bounded and $n \ge 2$. Then for all $m$, there exists a random variable $Z'$ such that the conditional law of $Z'$ given $U$ is almost surely Gaussian with covariance $\Sigma$, and with probability $1-\eta$,
$$
\bb{P}\left(\beef \norm{Z' - Z_{\mathfrak{B}}} \ge  C \log^{3/2}(C/\eta)  \left[ \left\{\frac{a^2\sigma^2(\Sigma, 0)m}{n}+\frac{a^4m}{n^2}\right\}^{\f 1 4} + \sigma(\Sigma, m) \right] \middle|\, U \right) \le \eta. $$
\end{theorem}

On a high probability event, we can approximately sample from the distribution of $Z'$, and hence from the distribution of the empirical process $n^{1/2}\E_n(U)$, by sampling from the bootstrap process $Z_{\mathfrak{B}}$ conditional on the realized data $U$.  As before, we heavily exploit the geometry of $H$ to derive a stronger form of approximation without incurring slower rates, departing from  \citet{chernozhukov2016empirical}. Within each bound, the former term reflects the quality of bootstrap coupling for the initial $m$ eigenfunctions, and the latter term is the sum of the remaining eigenvalues.  The following technical result underlies our method.

\begin{proposition}[Abstract bootstrap coupling]
For any $m \ge 1$ there exists a $Z'$ that is conditionally Gaussian with covariance $\Sigma$, such that with probability at least $1-3\eta$,
$$
\norm{Z' - Z_{\mathfrak{B}}}\leq \left \{1 + \sqrt{2\log(1/\eta)}\,\right\} \left\{ m^{\f 1 4}\Delta_1^{1/2}+\Delta_2^{1/2} +  2\sigma(\Sigma, m)\right\},
$$
where $\Delta_1=\|\hat{\Sigma}-\Sigma\|_{\HS}$ and $\Delta_2=\max\{\tr\,\{\Pi_m^\perp (\hat \Sigma - \Sigma) \Pi_m^\perp\},0\}$. Here, $\hat{\Sigma}=\bb{E}(Z_{\mathfrak{B}} \otimes Z_{\mathfrak{B}}^*|U)$ is the bootstrap covariance, $\Pi_m^\perp=I-\Pi_m$, and $\Pi_m$ projects onto the initial $m$ eigenfunctions.
\end{proposition}

This technical result only needs bounds on $\hat{\Sigma} - \Sigma$, so it applies to other boostrap procedures. For example, one may consider the traditional Gaussian multiplier bootstrap  $n^{1/2}\E_n( h U)$, subsampled data, or more modern covariance matrix approximations.

\subsection{Application: KRR inference}\label{sec:collect}

We now tie together the main results in Sections~\ref{sec:algo_main} and~\ref{sec:partial}. We demonstrate how our framework translates nonasymptotic Gaussian and bootstrap couplings into sharp and valid inference for KRR via our proposed $\mathfrak{B}$. \textcolor{black}{Inference remains valid for $f_{\lambda}$ under mis-specification, which is important because correct specification is implausible with nonstandard data.}

Within Algorithms~\ref{algo:incremental} and~\ref{algo:variable}, the implicit $H$-valued bootstrap is
$$
\mathfrak{B}=\frac{1}{n}\sum_{i=1}^n\sum_{j=1}^n  \left(\frac{\hat{V}_i-\hat{V}_j}{\sqrt{2}}\right) h_{ij},\quad \hat{V}_i=(\hat{T}+\lambda)^{-1}\hat{\ep}_i k_{X_i},\quad \hat{T}=\E_n(k_X\otimes k_X^*),\quad \hat{\ep}_i=Y_i-\hat{f}(X_i).
$$
The $\hat{V}_i$ are familiar from classical regression, where a bootstrap is often constructed by replacing $Y_i$ with $\hat{\ep}_i$ in the regression formula. Here, KRR is $\hat{f}=(\hat{T}+\lambda)^{-1}\E_n(Yk_X)$.

Bias of the KRR estimator $\hat{f}$, due to regularization or possible mis-specification, translates into bias of $\hat{\ep}_i$ and hence $\hat{V}_i$. Notice that our procedure $\mathfrak{B}$ is the empirical counterpart to
$$
Z_{\mathfrak{B}}
=\frac{1}{n}\sum_{i=1}^n\sum_{j=1}^n \left(\frac{V_i - V_j}{\sqrt 2}\right)h_{ij},\quad  V_i=(T+\lambda)^{-1}\{Y_i-f_{\lambda}(X_i)\} k_{X_i},\quad T=\E(k_X\otimes k_X^*).
$$
The issue is that $V_i$ is biased, with mean $\mu=(T+\lambda)^{-1}T(f_0-f_{\lambda})$. \textcolor{black}{If $f_{\lambda}$ fails to  approximate $f_0$ at a sufficient rate when $\lambda \downarrow 0$, a standard multiplier bootstrap would fail.}

A key insight is that $V_i=U_i+\mu$, where $U_i$ for KRR is stated at the beginning of Section~\ref{sec:partial}. Therefore, our procedure $\mathfrak{B}$ can cancel the bias $\mu$ by taking differences of $\hat{V}_i$ and $\hat{V}_j$. Equivalently, we use anti-symmetric multipliers, as demonstrated in Algorithms~\ref{algo:incremental} and~\ref{algo:variable}. The technique ensures valid inference for \textcolor{black}{$\hat f - f_{\lambda}$ , even when $f_0\not\in H$,} according to the decomposition
$$
n^{1/2}(\hat{f}-f_\lambda)=\underbrace{n^{1/2}\{(\hat{f}-f_{\lambda})-\E_n(U)\}}_{\text{residual}}+\underbrace{n^{1/2}\E_n(U)}_{\text{pre-Gaussian}}.
$$
This insight, for regularized nonparametric regression bands, may be of independent interest.

\textcolor{black}{\begin{corollary}[KRR inference under mis-specification]\label{cor:mis}
Suppose that $\|f_0-f_{\lambda}\|_{\infty}$ is bounded above uniformly in $\lambda$. Then the results of Propositions~\ref{prop:h-band} and~\ref{prop:inf-intro} continue to hold when replacing $f_0$ with $f_{\lambda}$ and replacing $B(\lambda)$ with zero.
\end{corollary}}

\section{Heterogeneous student preferences in Boston schools}\label{sec:application}

Our procedure performs well in nonlinear simulations, both with standard data and with preference data. We also conduct a semi-synthetic exercise using preference data calibrated to Boston Public School students. Our inferential procedure detects match effects well when they are present, whereas an approach based on indicators for preference categories does not.

\subsection{Nominal coverage for standard data and preference data}

We verify nominal coverage in nonlinear regression simulations with standard and preference data. In the standard data design, the true regression function $f_0$ is a weighted average of the initial five eigenfunction of the Gaussian kernel $k(x,x')= \exp\left\{-50(x-x')^2\right\}$.\footnote{The general Gaussian kernel is $k(x,x')= \exp\left\{-\frac{1}{2}\frac{(x-x')^2}{\iota^2}\right\}$. Here, $\iota$ plays the role of the standard deviation, which effectively standardizes the regressors. It has well known heuristics; see Appendix~\ref{sec:simulation_details}.} With Gaussian data, the eigenfunctions of the Gaussian kernel are the Hermite polynomials \citep[Section 4.3]{williams2006gaussian}.
In the preference data design, $f_0$ is a weighted average of the initial five eigenfunction of the kernel based on Kendall's rank correlation in Example~\ref{ex:kendall}.

\begin{figure}
\captionsetup[subfigure]{justification=Centering}
\begin{subfigure}[t]{0.48\textwidth}
         \centering
        \resizebox{\textwidth}{!}{%
       \includegraphics[width=\textwidth]{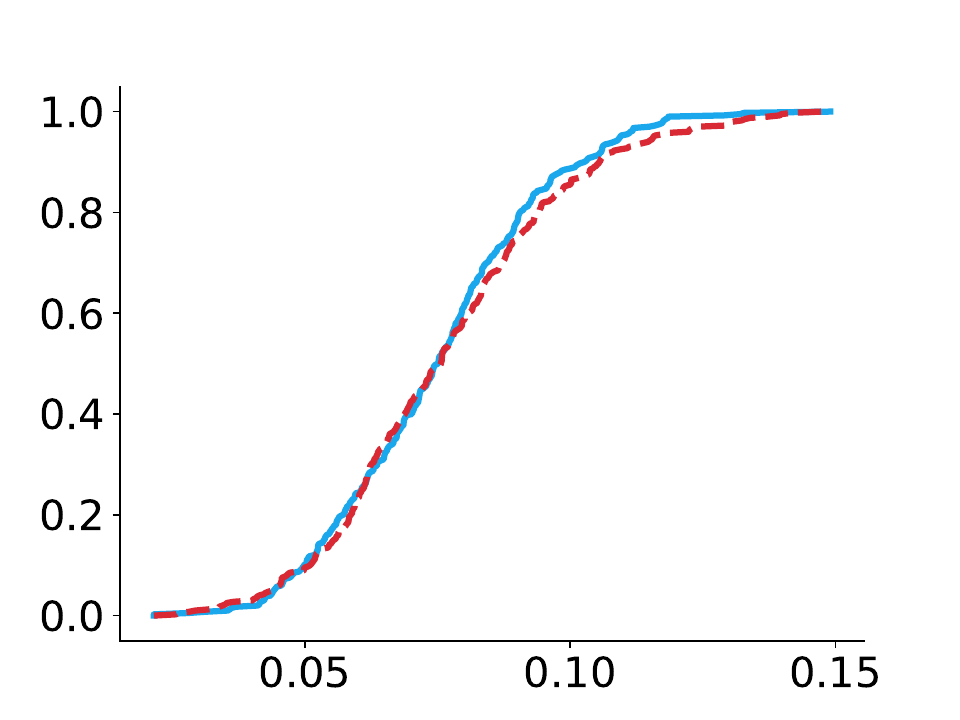}
        }
    \caption{\footnotesize Standard data. Coverage of $f_0$ is $95.0\%$.}\label{fig:cdf}
\end{subfigure}\hspace{\fill} %
\begin{subfigure}[t]{0.48\textwidth}
          \centering
        \resizebox{\textwidth}{!}{%
      \includegraphics[width=\textwidth]{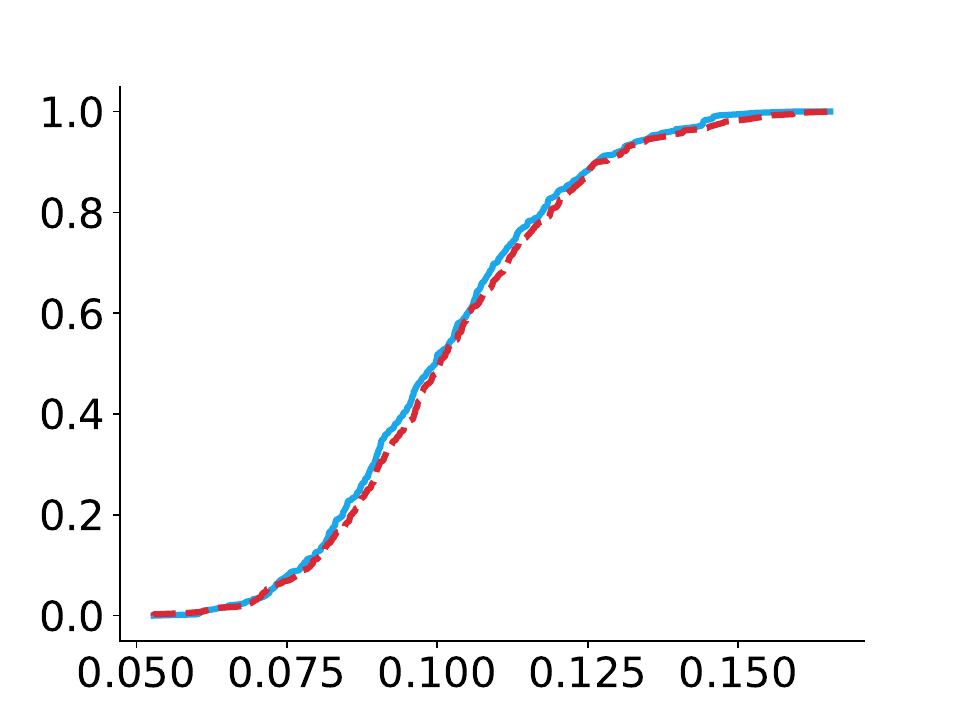}
        }
    \caption{\footnotesize Preference data. Coverage of $f_0$ is $96.8\%$}\label{fig:cdf_ranking}
\end{subfigure}
\caption{We compare the distribution of $n^{1/2}\snorm{\hat f-f_0}_\infty$ across many samples (dashed red), with the distribution of our proposal $\snorm{\mathfrak{B}}_\infty$ across many bootstrap iterations, conditional upon a single sample (solid blue).}\label{fig:cdf_all}
\end{figure}

Figure~\ref{fig:cdf_all} compares distributions that Proposition~\ref{prop:inf-intro} proves to be close: the distribution of the true quantity $\|n^{1/2}(\hat{f}-f_0)\|_{\infty}$ across 500 samples; and the distribution of our proposal $\|\mathfrak{B}\|_{\infty}$ across 500 bootstrap iterations conditional upon the first sample. The distributions closely align, both with standard and nonstandard data, illustrating the fidelity of our results.

We quantify performance in coverage tables. Across 500 draws, Table~\ref{tab:cov_n_all} confirms nominal coverage across sample sizes, with $\lambda=n^{-1/2}$ following Theorem~\ref{thm:near-minimax-band}(b). Different columns record coverage in $\sup$ norm and $H$ norm of $f_0$ and $f_{\lambda}$. Table~\ref{tab:cov_lambda_all} evaluates robustness to alternative tunings of $\lambda$. Across sample sizes and regularization values, the confidence bands attain nominal coverage, unless $\lambda$ strongly deviates from our theoretical guidelines.  Appendix~\ref{sec:simulation_details} examines additional metrics of confidence band performance beyond coverage, and confirms strong performance in Sobolev spaces.

\begin{table}
\captionsetup[subtable]{justification=Centering}
\caption{Coverage is nominal across sample sizes. Across rows, we vary $n$ and set $\lambda=n^{-1/2}$.
}\label{tab:cov_n_all}
\begin{subtable}[t]{0.42\textwidth}
        \centering
        \resizebox{\textwidth}{!}{%
        \begin{tabular}{ccccc}
           \multirow{2}{*}{sample}  & \multicolumn{2}{c}{$\sup$  norm}  & \multicolumn{2}{c}{$H$ norm} \\
           \cmidrule(lr){2-3}\cmidrule(lr){4-5}
             &  true & pseudo & true & pseudo \\
            \hline
  100 &               0.938 &                 0.952 &             0.942 &               0.942 \\
  250 &               0.960 &                 0.958 &             0.970 &               0.968 \\
  500 &               0.968 &                 0.985 &             0.980 &               0.975 \\
 1000 &               0.950 &                 0.955 &             0.965 &               0.965 \\
            \hline
        \end{tabular}
        }
    \caption{\label{tab:cov_n} \footnotesize \textsc{Standard data}}
\end{subtable} \quad\quad\quad %
\begin{subtable}[t]{0.42\textwidth}
          \centering
        \resizebox{\textwidth}{!}{%
      \begin{tabular}{ccccc}
           \multirow{2}{*}{sample}  & \multicolumn{2}{c}{$\sup$ norm}  & \multicolumn{2}{c}{$H$ norm}  \\
           \cmidrule(lr){2-3}\cmidrule(lr){4-5}
             &  true & pseudo & true & pseudo
             \\
            \hline
  100 &               0.982 &                 0.985 &             0.915 &               0.922 \\
  250 &               0.975 &                 0.982 &             0.968 &               0.965 \\
  500 &               0.932 &                 0.960 &             0.965 &               0.972 \\
 1000 &               0.968 &                 0.962 &             0.950 &               0.952 \\
            \hline
        \end{tabular}
        }
    \caption{\label{tab:cov_n_ranking} \footnotesize \textsc{Preference data}}
\end{subtable}

\end{table}

\begin{table}
\captionsetup[subtable]{justification=Centering}
\caption{Coverage is nominal for regularization near $n^{-1/2}$. Across rows, we fix $n=500$ and vary $\lambda$.}\label{tab:cov_lambda_all}
\begin{subtable}[t]{0.42\textwidth}
    \centering
    \resizebox{\textwidth}{!}{%
     \begin{tabular}{ccccc}
         \multirow{2}{*}{reg.} & \multicolumn{2}{c}{$\sup$ norm} &\multicolumn{2}{c}{$H$ norm } \\
         \cmidrule(lr){2-3}\cmidrule(lr){4-5}
        & true & pseudo & true & pseudo \\
        \hline
  0.500 &               0.102 &                 0.963 &             0.897 &               0.970 \\
  0.100 &               0.932 &                 0.958 &             0.953 &               0.953 \\
  0.050 &               0.973 &                 0.970 &             0.965 &               0.968 \\
  0.010 &               0.965 &                 0.965 &             0.968 &               0.968 \\
  0.005 &               0.930 &                 0.930 &             0.965 &               0.965 \\
  0.001 &               0.938 &                 0.938 &             0.963 &               0.963 \\
        \hline
    \end{tabular}
    }
    \caption{\label{tab:cov_lambda} \footnotesize \textsc{Standard data.}}
\end{subtable} \quad\quad\quad
\begin{subtable}[t]{0.42\textwidth}
    \centering
    \resizebox{\textwidth}{!}{%
     \begin{tabular}{ccccc}
         \multirow{2}{*}{reg.} & \multicolumn{2}{c}{$\sup$ norm} &\multicolumn{2}{c}{$H$ norm } \\
         \cmidrule(lr){2-3}\cmidrule(lr){4-5}
        & true & pseudo & true & pseudo \\
        \hline
  0.500 &               0.115 &                 0.983 &             0.833 &               0.973 \\
  0.100 &               0.935 &                 0.980 &             0.953 &               0.958 \\
  0.050 &               0.943 &                 0.945 &             0.973 &               0.973 \\
  0.010 &               0.955 &                 0.955 &             0.960 &               0.960 \\
  0.005 &               0.978 &                 0.980 &             0.970 &               0.970 \\
  0.001 &               0.998 &                 0.998 &             0.115 &               0.115 \\
        \hline
    \end{tabular}
    }
    \caption{\label{tab:cov_lambda_ranking} \footnotesize \textsc{Preference data.}}
\end{subtable}
\end{table}

\textcolor{black}{Finally, we provide simulation results under mis-specification, i.e.~$f_0\not\in H$, showcasing how our inferential results still
tend to cover
$f_{\lambda}$. Figure~\ref{fig:mis_band} visualizes a mis-specified design where the true regression $f_0$ is a step function. Figure~\ref{fig:mis_cdf} compares
the distribution of the pseudo true quantity $\|n^{1/2}(\hat{f}-f_{\lambda})\|_{\infty}$ across 500 samples; and the distribution of our proposal $\|\mathfrak{B}\|_{\infty}$ across 500 bootstrap iterations conditional upon the first sample. The distributions closely align, illustrating the robustness of our theory. Appendix~\ref{sec:simulation_details} gives coverage tables.}

\begin{figure}
\captionsetup[subfigure]{justification=Centering}
\begin{subfigure}[t]{0.48\textwidth}
         \centering
        \resizebox{\textwidth}{!}{%
       \includegraphics[width=\textwidth]{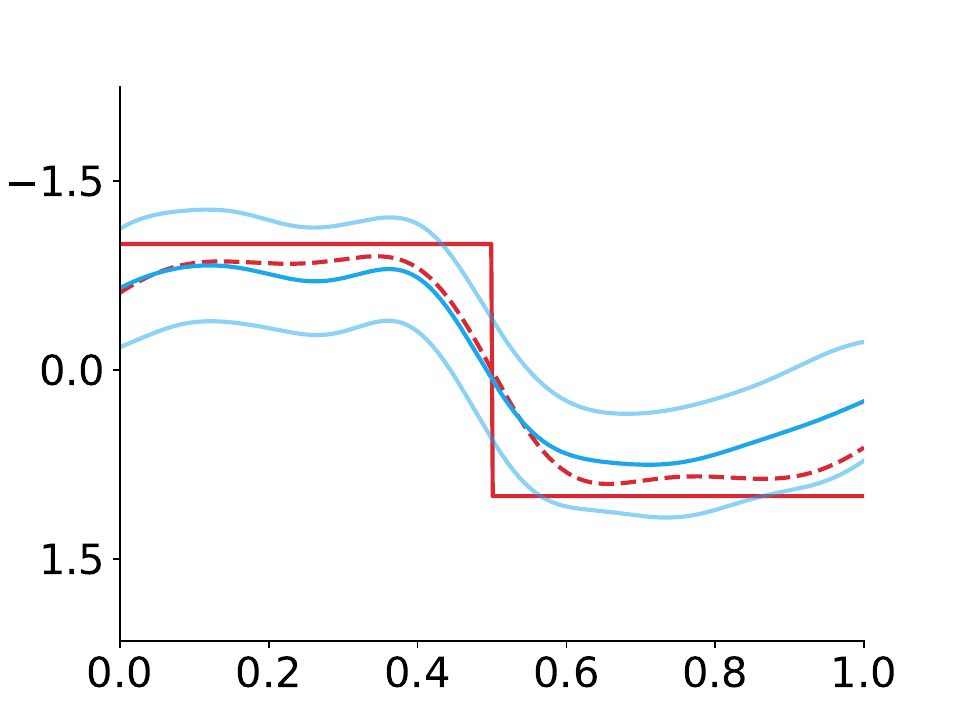}
        }
    \caption{\footnotesize Mis-specification. Here $f_0\not\in H$ yet $f_{\lambda}\in H$.}\label{fig:mis_band}
\end{subfigure}\hspace{\fill} %
\begin{subfigure}[t]{0.48\textwidth}
          \centering
        \resizebox{\textwidth}{!}{%
      \includegraphics[width=\textwidth]{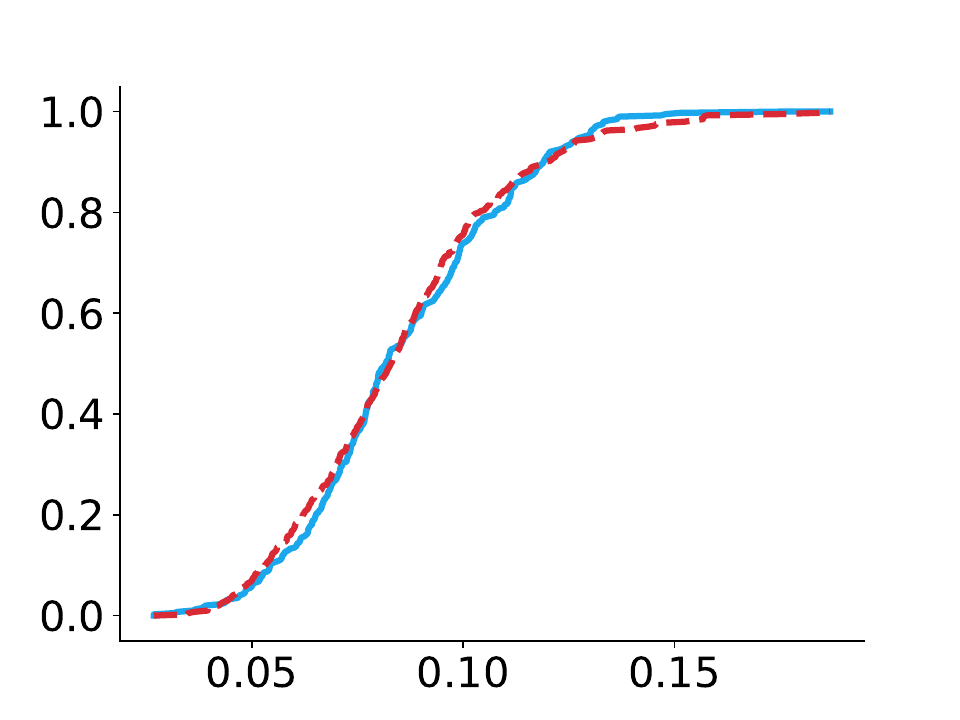}
        }
    \caption{\footnotesize Mis-specification.  Coverage of $f_{\lambda}$ is 95.8\%}\label{fig:mis_cdf}
\end{subfigure}
\caption{
In Figure~\ref{fig:mis_band}, we compare the true regression $f_0$ (red) with the pseudo true parameter $f_{\lambda}$ (dashed red). For one sample, we visualize the KRR estimator (dark blue) and our confidence band (light blue).
In Figure~\ref{fig:mis_cdf}, we compare the distribution of $n^{1/2}\snorm{\hat f-f_\lambda}_\infty$ across many samples (dashed red), with the distribution of our proposal $\snorm{\mathfrak{B}}_\infty$ across many bootstrap iterations, conditional upon a single sample (solid blue).}\label{fig:mis_all}
\end{figure}

\subsection{An improved test for match effects}

Previous work proves that heterogeneous school effects are identified by data from a centralized matching mechanism. Due to random tie breaking in the matching mechanism, conditional upon submitted student preferences, school assignment is as good as random \citep{abdulkadirouglu2017research}. Therefore the heterogeneous effects of attending a certain school, or school sector, are nonparametrically identified as a function of student preferences.

We use KRR to estimate heterogeneous school sector effects. Let $D_i\in\{0,1\}$ indicate whether a student is assigned to a pilot school. Let $X_i$ denote a student's preference over the $25$ schools. The treatment propensity $\pi(X_i)=\mathbb{P}(D_i=1|X_i)$ may be computed to arbitrary precision by running the mechanism many times. We regress the pseudo-outcomes $\tilde{Y}_i=\frac{Y_iD_i}{\pi(X_i)}-\frac{Y_i(1-D_i)}{1-\pi(X_i)}$ on $X_i$ using KRR with the preference kernel of Example~\ref{ex:kendall}, and conduct inference with Algorithm~\ref{algo:variable}. Figure~\ref{fig:eigen} validates our main assumption: there appear to be relatively few types of preferences. These semi-synthetic preferences
have a low effective dimension. %
Appendix~\ref{sec:application_details} visualizes and interprets the eigenfunctions of the preference kernel, i.e. its implicit nonlinear basis.

The %
regression function over all $25!$ preferences is challenging to visualize, so we calculate $25$ group averages. Each group $S_{\rho}$ consists of students whose highest rank given to a pilot school is $\rho$. Proposition~\ref{prop:inf-intro} gives simultaneous inference on the $25$ group effects via the functional $F(f)=\max_{1\leq \rho\leq 25}|\E_n\{f(X_i)|X_i\in S_{\rho}\}|$, which is continuous in $\sup$ norm.\footnote{Figure~\ref{fig:sub_all} reports effects for $(S_1,...,S_{10})$, since $(S_{11},...S_{25})$ are never assigned to pilot schools in practice.}

\begin{figure}
\captionsetup[subfigure]{justification=Centering}
\begin{subfigure}[t]{0.48\textwidth}
         \centering
        \resizebox{\textwidth}{!}{%
       \includegraphics[width=\textwidth]{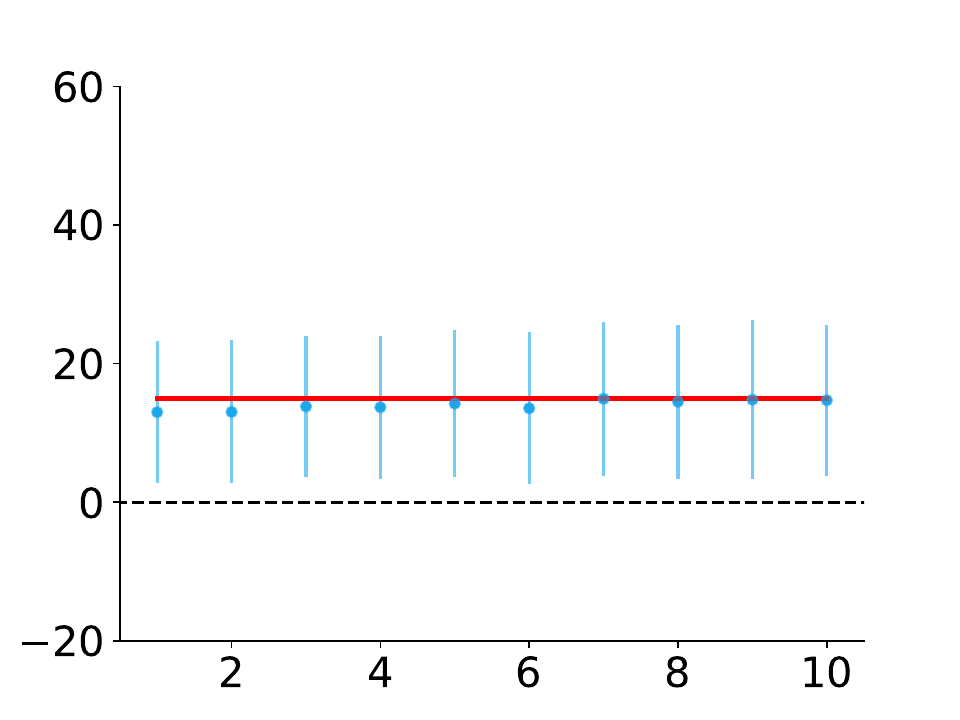}
        }
    \caption{\footnotesize No match effects.}\label{fig:sub1}
\end{subfigure}\hspace{\fill} %
\begin{subfigure}[t]{0.48\textwidth}
          \centering
        \resizebox{\textwidth}{!}{%
        \includegraphics[width=\textwidth]{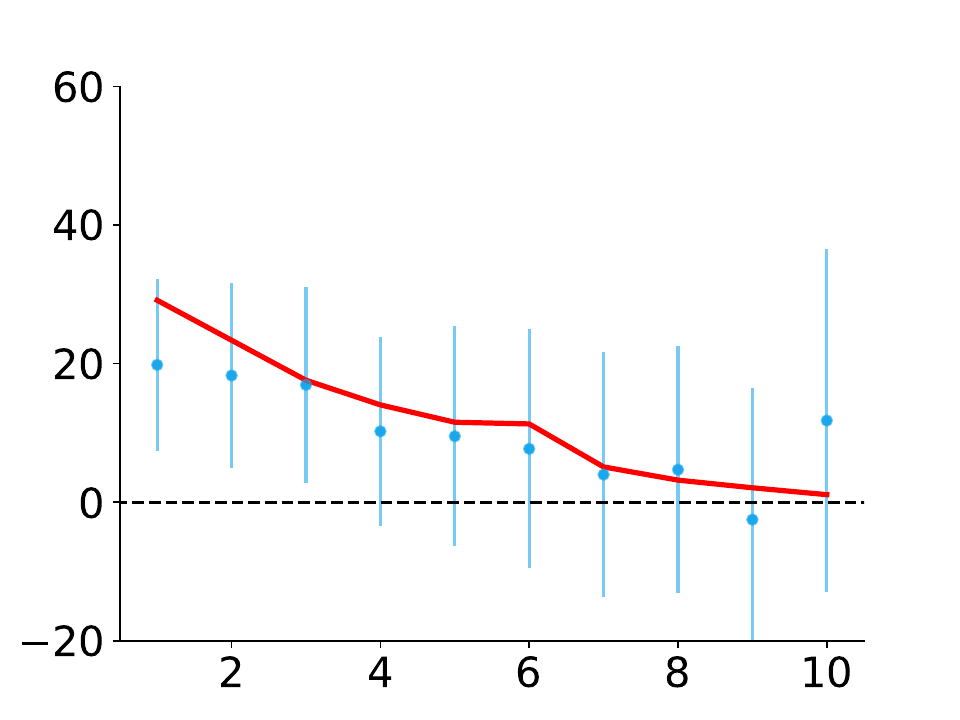}
        }
    \caption{\footnotesize Match effects.}\label{fig:sub2}
\end{subfigure}
\caption{KRR inference detects match effects. The true group average treatment effects are in red. The dark blue estimates are from KRR. The light blue intervals are from our KRR inference procedure.}\label{fig:sub_all}
\end{figure}

\begin{figure}
\centering
\begin{subfigure}{.32\textwidth}
  \centering
  \includegraphics[width=\linewidth]{img_semi/ranking_no_match_fx.pdf}
  \caption{KRR (uniform bands)}
  \label{fig:sub1_appendix}
\end{subfigure}
\begin{subfigure}{.32\textwidth}
  \centering
  \includegraphics[width=\linewidth]{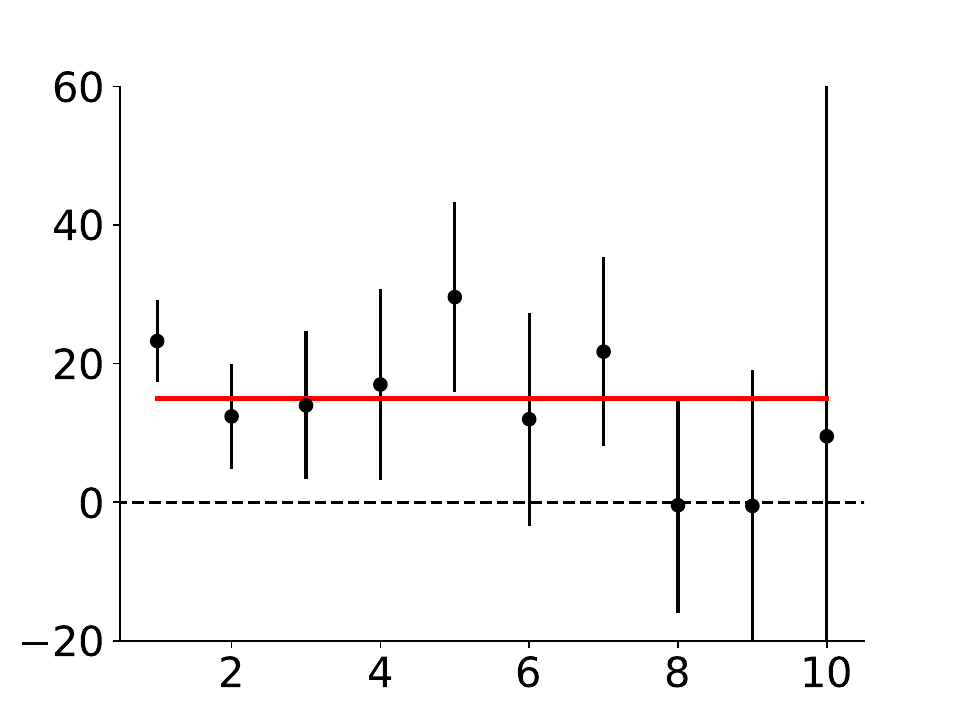}
  \caption{Indicators (pointwise bands)}
  \label{fig:sub2_appendix}
\end{subfigure}
\begin{subfigure}{.32\textwidth}
  \centering
  \includegraphics[width=\linewidth]{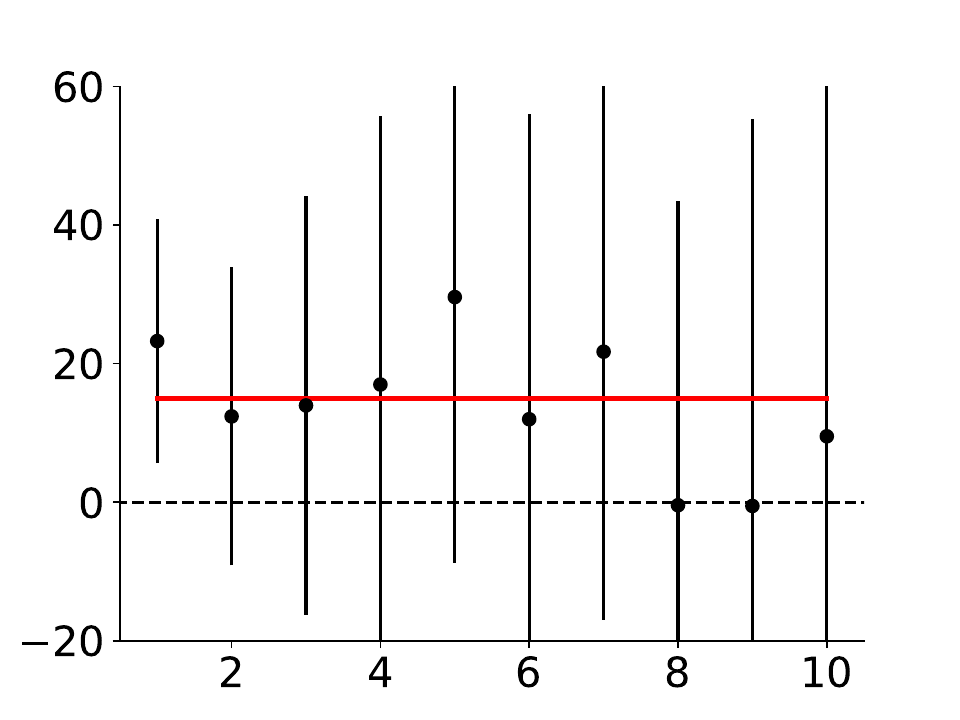}
  \caption{Indicators (uniform bands)}
  \label{fig:sub3_appendix}
\end{subfigure}
\caption{With no match effects, KRR inference improves power compared to an indicator approach. The true group average treatment effects are in red. Our KRR approach is in blue. The indicator approach is in black.
}\label{fig:no_match_appendix}
\end{figure}

\begin{figure}
 \centering
\begin{subfigure}{.32\textwidth}
  \centering
  \includegraphics[width=\linewidth]{img_semi/ranking_match_fx.pdf}
  \caption{KRR (uniform bands)}
  \label{fig:sub4_appendix}
\end{subfigure}
\begin{subfigure}{.32\textwidth}
  \centering
  \includegraphics[width=\linewidth]{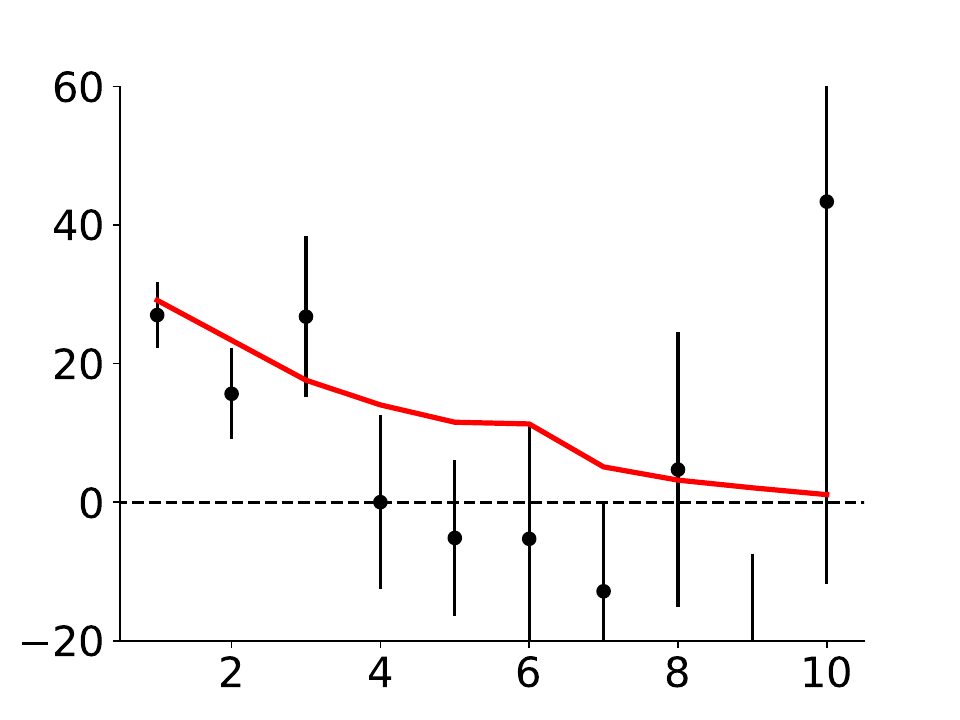}
  \caption{Indicators (pointwise bands)}
  \label{fig:sub5_appendix}
\end{subfigure}
\begin{subfigure}{.32\textwidth}
  \centering
  \includegraphics[width=\linewidth]{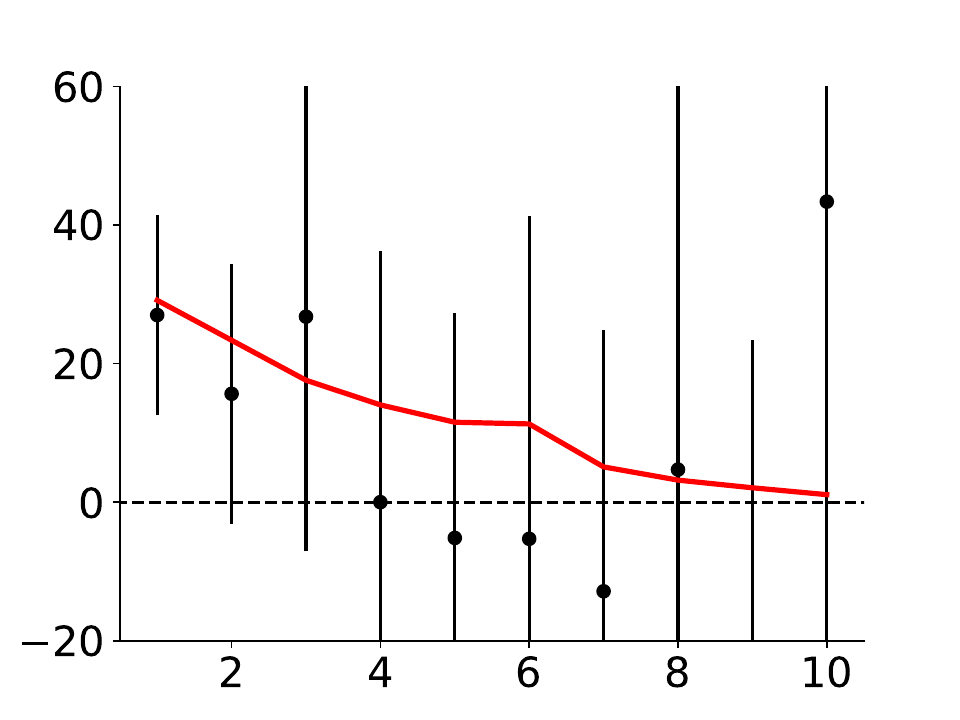}
  \caption{Indicators (uniform bands)}
  \label{fig:sub6_appendix}
\end{subfigure}
\caption{With match effects, KRR inference improves power compared to an indicator approach. The true group average treatment effects are in red. Our KRR approach is in blue. The indicator approach is in black.
}\label{fig:match_appendix}
\end{figure}

With this methodology, we demonstrate that KRR inference can distinguish between two counterfactual scenarios: no match effects (Figure~\ref{fig:sub1}) versus match effects (Figure~\ref{fig:sub2}). In the no match effect design, we generate 4000 student preferences calibrated to real preferences of students at Boston Public Schools \citep{pathak2021well}. We impose no match effects, hence flat group average treatment effects, visualized in red. In the match effect design, we generate 4000 student preferences calibrated as before but now modified so that students who rank pilot schools highly also benefit more from them. This leads to decreasing group average treatment effects, again visualized in red. Our KRR inference procedure, in blue, covers the truth in both scenarios. In the match effect design, it rejects the null hypothesis of no effect for students who highly rank pilot schools.

Our approach, based on KRR, asserts that some preferences are closer than others. It improves power compared to an ``indicator'' approach based on \cite{abdulkadirouglu2017research}, which does not consider similarity between preferences. The indicator approach averages the pseudo-outcomes $\tilde{Y}_i$ within the groups $S_{\rho}$, recovering a standard inverse propensity weighted estimator of the group average treatment effect. In Figures~\ref{fig:no_match_appendix} and~\ref{fig:match_appendix}, it
often gives the incorrect sign, with pointwise significance. Its uniform bands are less informative than the KRR inference bands; the test for match effects based on KRR is more powerful.

\section{Discussion: Kernel methods in economics}\label{sec:conclusion}

 Our main theoretical contribution is to prove uniform inference guarantees for KRR.
 Our uniform confidence bands have simple closed form expressions, strong statistical guarantees, and robust empirical performance with nonstandard data.
 While Section~\ref{sec:algo_main} focuses on KRR, Section~\ref{sec:partial} builds a framework to use Gaussian and bootstrap couplings for more general kernel methods. Future research may apply our results to nonparametric estimands beyond regression.

Our main practical contribution is to design a new inferential procedure for preference data.
Preferences,  and other nonstandard data, are increasingly used in empirical economics.
While Section~\ref{sec:application} focuses on a test for match effects in preference data, future work may apply our results to test shape restrictions on preferences, for various choice models. Future work may also analyze nonstandard economic data via kernels for sequences, networks, images, and text.

\bibliographystyle{hapalike_mod}
\DeclareRobustCommand{\VAN}[2]{#2}  %
\spacingset{1}
{
\bibliography{0_refs}}
\spacingset{1.5}
\DeclareRobustCommand{\VAN}[2]{#1}  %

\clearpage
\appendix

\begin{center}
{\large\bf PRIMARY APPENDIX}
\end{center}

Appendix~\ref{sec:zaitsev} proves Theorem~\ref{thm:gauss-main}.
Appendix~\ref{sec:cov} proves Theorem~\ref{thm:bs-main}.
Appendix~\ref{sec:symbols} links Theorems~\ref{thm:gauss-main} and~\ref{thm:bs-main} with Propositions~\ref{prop:h-band} and~\ref{prop:inf-intro}, by matching symbols between the abstract empirical process and the pre-Gaussian term of KRR.

\section{Gaussian coupling}\label{sec:zaitsev}

We prove a high-probability bound on
$
\norm{n^{1/2}\E_n(U) - Z},
$
 where $Z$ is a Gaussian random variable in $H$ with covariance $\Sigma = \bb{E}(U_i \otimes U_i^*)$. This implies nonasymptotic Gaussian couplings for RKHS-valued partial sums in $L^\infty$ since, by the Cauchy-Schwarz inequality,
$$
\sup_{x \in S} \left| n^{1/2}\E_n\{U(x)\} - Z(x)\right|
=
\sup_{x \in S} \left|  \langle n^{1/2}\E_n(U) - Z, k_x\rangle \right| \leq \|n^{1/2}\E_n(U) - Z\| \cdot \sup_{x\in S}\|k_x\|
$$
and we assume $\sup_{x\in S} \|k_x\|\leq \kappa$.

Nonasymptotic bounds are important in inverse problems such as KRR, where the complexity of $\Sigma$ increases with $n$, and thus the limiting process is non-Donsker.
The bounds we derive only depend on $n$, the spectrum of $\Sigma$, and the kernel bound $\kappa$.
They apply across a broad range of settings in which RKHS methods find use, unlike couplings based upon the Hungarian construction. Appendix~\ref{sec:spectrum} specializes these bounds for leading cases.

\paragraph{\textcolor{black}{Notation.}} Let $U = (U_1, U_2, \ldots )$ be an i.i.d.~sequence of centered random elements in $\H$. Let $\Sigma: \H \to \H$ denote the covariance operator of the summands: for all $u,v \in \H$,
$\bb{E}(\bk{u}{U}\bk{v}{U}) = \bk{u}{\Sigma v}.$ Suppose that  $\bb{E}\norm{U_i}^2 < \infty$. Then $\Sigma$ is trace-class and self-adjoint, so that we may choose an $\H$-orthonormal basis of $\Sigma$-eigenvectors $(e_1, e_2, \ldots)$, with corresponding eigenvalues $(\nu_1, \nu_2, \ldots)$. We define
$\sigma^2(m)=\sigma^2(\Sigma,m) = \sum_{k > m} \nu_k,$ which plays a role similar to the metric entropy, quantifying the compactness of $\Sigma$. Let
$\Pi_m = \sum_{i=1}^m e_i \otimes e_i^* $
be the self-adjoint and idempotent projection onto the span of $(e_1, e_2, \ldots, e_m)$.

To construct couplings, we require that the background probability space $(\bb{P},\c F, \Omega)$ is rich enough. It suffices that there exists a countable sequence of standard normal random variable~independent of $U$, which is a minor technicality. %

\paragraph{\textcolor{black}{Overview.}} To construct the Gaussian approximation $Z$, we project the summands $U_i$ onto an appropriate $m$-dimensional subspace. We then appeal to known, finite-dimensional Euclidean coupling results in $\mathbb{R}^m$. Finally, we apply high-probability bounds on the distance between $U_i$ and its projection. This argument exploits Hilbert space structure to give sharp bounds.

\paragraph{\textcolor{black}{Finite-dimensional coupling.}} \textcolor{black}{The following Euclidean result is the key input to our construction.}
\begin{lemma}[\textcolor{black}{Theorem 1 of \citealp{eldan2020clt}}]\label{lemma:fd-coupling}
\textcolor{black}{Let $\xi_1,\ldots,\xi_n$ be i.i.d.~centered random vectors in $\bb{R}^m$ satisfying $\norm{\xi_i}_{\mathbb{R}^m}\leq a$ almost surely. There exists a Gaussian random vector $\zeta\sim N\{0,\mathbb{E}[\xi_1\xi_1^\top]\}$ coupled with $S_n=n^{-1/2}\sum_{i=1}^n\xi_i$ such that, for every $\eta\in(0,1)$,
\[
\bb{P}\left\{\norm{S_n-\zeta}_{\mathbb{R}^m}>
a\sqrt{\frac{m\{32+2\log_2(n)\}}{n\eta}}\right\}\leq\eta.
\]}
\end{lemma}

\begin{proof}
\color{black}
The cited theorem gives
\[
\mathcal{W}_2^2\left(S_n,N\{0,\mathbb{E}[\xi_1\xi_1^\top]\}\right)
\leq \frac{a^2 m[32+2\log_2(n)]}{n}.
\]
By definition of the Wasserstein distance, we can construct a coupling of $(S_n,\zeta)$ on a suitable probability space where $\mathbb{E}\|S_n - \zeta\|_{\mathbb{R}^m}^2$ satisfies the same bound. Chebyshev's inequality then gives
\[
\bb{P}\left\{\norm{S_n-\zeta}_{\mathbb{R}^m}>
a\sqrt{\frac{m\{32+2\log_2(n)\}}{n\eta}}\right\}
\leq \frac{\bb{E}\norm{S_n-\zeta}_{\mathbb{R}^m}^2}
{a^2m\{32+2\log_2(n)\}/(n\eta)}\leq\eta.
 \qedhere \]
\end{proof}

\paragraph{\textcolor{black}{High probability bounds.}} Having coupled $\Pi_m Z$ and $\Pi_m U_i$, we use the following vector-valued concentration inequalities to bound $\|Z - \Pi_m Z\|$ and $\|\mathbb{E}_n(U_i - \Pi_m U_i)\|$.

\begin{lemma}[Bounded tail]\label{lemma:tail} If $\norm{U_i} \le a$ almost surely then w.p. $1-\eta$,
$$
\norm{\frac{1}{\sqrt n}\sum_{i=1}^n (I - \Pi_m)U_i} \le 2\sigma(m)\sqrt{\log(2/\eta)} \vee \frac{4a\log(2/\eta)}{\sqrt n}.
$$
\end{lemma}

\begin{proof}
  Let  $\xi_i= (I-\Pi_m)U_i.$ Then $\mathbb{E}(\xi_i)=\mathbb{E}\left\{(I-\Pi_m)U_i\right\}=(I-\Pi_m)\mathbb{E}(U_i)=0$. Moreover,
  $
    \left\|(I-\Pi_m)U_i \right\| \leq \left\|U_i \right\| \leq a
    $
    and
     $$
    \mathbb{E}\left\|(I-\Pi_m)U_i \right\|^2
    =
    \mathbb{E} \langle U_i,(I-\Pi_m)U_i\rangle
    =
    \mathbb{E}  \tr \{(I-\Pi_m)U_i\otimes U_i^*\}
    =
    \tr \{(I-\Pi_m)\Sigma\}
    =
    \sigma^2(m).
    $$
    Combining these, we can bound
    $\bb{E}\norm{(I-\Pi_m)U_i}^\ell \le \sigma^2(m)a^{\ell-2}
    \le \frac{\ell!}{2}\sigma^2(m)a^{\ell-2}$ for $\ell \ge 2$.
    Finally, apply Bernstein's inequality (Lemma \ref{lemma:c_dv}) with $A = 2a$ and $B = \sigma(m)$.%
\end{proof}

\begin{lemma}[Gaussian tail]\label{lemma:tail_gaussian} For each $i$, let $Z_i$ be an independent Gaussian element with covariance $\Sigma$. Then w.p. $1-\eta$,
$
\norm{\frac{1}{\sqrt n} \sum_{i=1}^n (1 - \Pi_m)Z_i} \le \left\{1 + \sqrt{2\log(1/\eta)} \right\} \sigma(m).
$
\end{lemma}

\begin{proof}
We compute
\begin{align*}
\bb{E}\norm{\frac{1}{\sqrt n} \sum_{i=1}^n (1 - \Pi_m)Z_i}^2 &= \frac{1}{n}\sum_{i=1}^n\sum_{j=1}^n \bb{E}\bk{(1 - \Pi_m)Z_i}{(1 - \Pi_m)Z_j}= \frac{1}{n}\sum_{i=1}^n\ \bb{E}\norm{(1 - \Pi_m)Z_i}^2
\end{align*}
which equals $\sigma^2(m)$ by the same argument as Lemma~\ref{lemma:tail}, since $Z_i$ has covariance $\Sigma$. The result then follows from Borell's inequality (Lemma \ref{lemma:gaussian-concentration-trace}).
\end{proof}

\paragraph{\textcolor{black}{Main results.}} We tie together the results above.

\begin{theorem}[Bounded coupling]\label{thm:technical}
\textcolor{black}{Suppose $\|U_i\|\leq a$ almost surely. Suppose the probability space is rich enough to support the coupling below. Then, for all $m$, there exists a Gaussian random element $Z$ with covariance $\Sigma$ such that w.p.~$1-\eta$,
\[
\left\|\frac{1}{\sqrt{n}} \sum_{i=1}^n U_i-Z\right\|
\lesssim \sqrt{\log(6/\eta)}\,\sigma(m)
+a\sqrt{\frac{m\{32+2\log_2(n)\}}{n\eta}}.
\]}
\end{theorem}

\begin{proof}
\color{black}
We identify the range of $\Pi_m$, spanned by the top $m$
 eigenvectors of $\Sigma$, with $\bb{R}^m$.
 Let $A$ be the orthogonal projection of $\H$ onto $\bb{R}^m$ defined by
$A: f \mapsto (\bk{f}{e_i})_{1 \le i \le m}.$
Its adjoint $A^*$ isometrically embeds $\bb{R}^m$ into the range of $\Pi_m$, and $\Pi_m = A^*A$.

 Using Lemma~\ref{lemma:fd-coupling}, we couple $n^{-1/2}\sum_{i=1}^n A U_i$ with a Gaussian vector $\zeta\in\bb R^m$ having covariance $A\Sigma A^*$. Since $A^*$ is an isometry and $A^*A=\Pi_m$, w.p.~$1-\eta$,
\[
\norm{\frac{1}{\sqrt n}\sum_{i=1}^n\Pi_mU_i-A^*\zeta}
\leq a\sqrt{\frac{m\{32+2\log_2(n)\}}{n\eta}}.
\]
Let $(h_s)_{s>m}$ be independent standard Gaussian random variables, independent of the coupling above, and set
\[
Z=A^*\zeta+\sum_{s=m+1}^\infty h_s\sqrt{\nu_s}e_s.
\]
Then $Z$ is Gaussian with covariance $\Sigma$. By the triangle inequality,
\begin{equation*}
    \begin{split}
    \norm{\frac{1}{\sqrt n}\sum_{i=1}^n U_i-Z}
    &\leq \norm{\frac{1}{\sqrt n}\sum_{i=1}^n\Pi_mU_i-\Pi_mZ} +
        \norm{\frac{1}{\sqrt n}\sum_{i=1}^n (1 - \Pi_m)U_i}
        + \norm{(1 - \Pi_m)Z}.
    \end{split} \label{orthogonal-decomp}
\end{equation*}
The first term is controlled by our construction above. The second and third terms are controlled by Lemmas~\ref{lemma:tail} and~\ref{lemma:tail_gaussian} (with $n=1$), respectively. We apply a union bound over these three events, replace $\eta$ with $\eta/3$, and suppress universal constants. The term $a\log(6/\eta)/\sqrt n$ from Lemma~\ref{lemma:tail} is absorbed by the finite-dimensional coupling term.
\end{proof}

\paragraph{\textcolor{black}{Comparisons.}} Finally, we extend our comparisons between Theorem~\ref{thm:gauss-main} and related work.
Theorem \ref{thm:gauss-main} is a strong form of approximation, describing the distribution of any continuous functional $F: H \to \bb{R}$.
Unlike \citet{chernozhukov2014gaussian} and related work, it is not limited to suprema of linear functionals. By using Hilbert space geometry, we improve the mode and rate of convergence.

To roughly compare rates,
\citet{berthet2006revisiting} and \citet{chernozhukov2014gaussian,chernozhukov2016empirical} cover the index set with $p = 2^m$ points at resolution $\d(p)$ and incur the approximation error $\sigma(\Sigma, \log p)$, which is the local Gaussian complexity at the scale $\d(p)$.\footnote{Appendix~\ref{sec:spectrum} characterizes the relationship between the local width and the local Gaussian complexity.}
Applying the recent result of \citet[Theorem~2.1]{chernozhuokov2022improved} in the above argument gives the comparable rate $\{\log^5(np)/n\}^{1/4} + \sigma(\Sigma, \log p)$, with a weaker form of approximation. \textcolor{black}{By contrast, for bounded data we obtain the rate $\{\log(p)\log(n)/n\}^{1/2}+\sigma(\Sigma,\log p)$ for coupling in $H$ norm, hence also in $\sup$ norm.}
Such improvements are specific to the RKHS setting, and are important for constructing confidence sets that shrink at nearly the minimax rate.

\section{Bootstrap coupling}\label{sec:cov}

We prove the distribution of $Z_{\mathfrak{B}} =   \frac{1}{n}\sum_{i=1}^n\sum_{j=1}^n h_{ij} \frac{V_i-V_j}{\sqrt{2}}$ conditional upon the data $D$ is approximately Gaussian with covariance $\Sigma = \bb{E}(U_i \otimes U_i^*)$.

\paragraph{\textcolor{black}{Overview.}} First, we show that the conditional distribution of $Z_{\mathfrak{B}}$ is Gaussian with covariance $\hat \Sigma$, where
$
\hat{\Sigma}=\mathbb{E}_n(U_i\otimes U_i^*)-\mathbb{E}_n(U_i)\otimes \{\mathbb{E}_n(U_i)\}^*.
$ We use the empirically centered $\hat{\Sigma}$ because it corresponds to the anti-symmetric bootstrap $Z_{\mathfrak B}$.
Second, we prove a bound on $\snorm{\hat \Sigma^{\f 1 2} - \Sigma^{\f 1 2}}$ and show that this implies a bound on the distance between Gaussian distributions.
Third, we construct $Z'$ that is conditionally Gaussian with covariance $\Sigma$, such that $\norm{Z'-Z_{\mathfrak{B}}}$ is small w.h.p.

\paragraph{\textcolor{black}{Notation.}} As before, let $U = (U_1, U_2, \ldots)$ denote an i.i.d.~sequence of centered random elements in $H$ with $\Sigma = \bb{E}(U_i \otimes U_i^*)$. Let $V_i=U_i+\mu$ for some arbitrary deterministic $\mu \in H$. We require that $U$ is $\sigma(D)$ measurable.

Given an independent sequence of standard normals $(g_1, g_2, \ldots)$ we write
$g = \sum_{i=1}^\infty g_ie_i$ and $\Sigma^{\f 1 2}g = \sum_{i=1}^\infty \sqrt{\nu_i}g_ie_i$, which belong almost surely to $L^2$ and $H$, respectively.

Let $X \overset{D}{\sim} Y$ mean that the $\sigma(D)$-conditional distributions of the random variables $X$ and $Y$ are equal, i.e. for any Borel set $A$,
$\bb{P}(X \in A|D) = \bb{P}(Y \in A|D)$ holds
$D$-almost surely.

\paragraph{\textcolor{black}{Characterizing the anti-symmetric bootstrap process.}}

\begin{lemma}[Anti-symmetric bootstrap covariance]\label{lemma:sym_cov}
$\hat{\Sigma}^{1/2}g \overset{D}{\sim} \frac{1}{n} \sum_{i=1}^{n} \sum_{j=1}^n h_{ij} \left(\frac{V_i-V_j}{\sqrt 2}\right).$
\end{lemma}

\begin{proof}Since
$\frac{1}{n} \sum_{i=1}^{n} \sum_{j=1}^n h_{ij} \left(\frac{V_i-V_j}{\sqrt 2}\right)$
is jointly Gaussian conditional upon $V = (V_1, V_2, \ldots, V_n)$, it suffices to compute
\begin{align*}
    & \bb{E}_h \left[ \left\{ \frac{1}{n} \sum_{i=1}^{n} \sum_{j=1}^n h_{ij} \left(\frac{V_i-V_j}{\sqrt 2}\right) \right\} \otimes \left\{ \frac{1}{n} \sum_{k=1}^{n} \sum_{\ell=1}^n h_{k\ell} \left(\frac{V_k-V_{\ell}}{\sqrt 2}\right) \right\}^{*} \right]
    \\
    &= \bb{E}_h \left[ \left\{ \frac{1}{n} \sum_{i=1}^{n} \sum_{j=1}^n h_{ij} \left(\frac{U_i-U_j}{\sqrt 2}\right) \right\} \otimes \left\{ \frac{1}{n} \sum_{k=1}^{n} \sum_{\ell=1}^n h_{k\ell} \left(\frac{U_k-U_\ell}{\sqrt 2}\right) \right\}^{*} \right]
    \\
    &= \frac{1}{2n^2}\sum_{i=1}^n\sum_{j=1}^n \{(U_i- U_j) \otimes (U_i - U_j)^*\} \\
    &= \frac{1}{n}\sum_{i=1}^n (U_i\otimes U_i^*) - \left(\frac{1}{n}\sum_{i=1}^n U_i \otimes \frac{1}{n}\sum_{j=1}^n U_j^*\right)
    = \hat{\Sigma}.
\end{align*}
The second equality uses independence of standard Gaussian multipliers $h_{ij}$. The fourth equality uses symmetry under transposition of $i$ and $j$.
Since $\hat{\Sigma}^{\f 1 2} g$ given $D$ is also jointly Gaussian with the same covariance, the two are equal in conditional distribution.
\end{proof}

\begin{lemma}[Anti-symmetric bootstrap as single sum]\label{lemma:single}
$$
\frac{1}{n} \sum_{i=1}^{n} \sum_{j=1}^n h_{ij} \left(\frac{V_i-V_j}{\sqrt 2}\right)
=
\sqrt{n}\mathbb{E}_n(q_iU_i),\quad q=\frac{1}{\sqrt{2n}} (h-h^{\top}) \boldsymbol{1},\quad var(q_i)<1
$$
where $(q_i)$ are non-independent Gaussians defined in terms of $(h_{ij})$.
\end{lemma}

\begin{proof}
Write
    \begin{align*}
     &\frac{1}{n} \sum_{i=1}^n \sum_{j=1}^n  h_{ij} \left(\frac{V_i-V_j}{\sqrt 2}\right)
     = \frac{1}{n} \sum_{i=1}^n \sum_{j=1}^n  h_{ij} \left(\frac{U_i-U_j}{\sqrt 2}\right) \\
     &=\frac{1}{\sqrt{n}}\sum_{i=1}^n U_i  \left(\frac{1}{\sqrt{2n}}\sum_{j=1}^n  h_{ij} \right)
     - \frac{1}{\sqrt{n}} \sum_{j=1}^n U_j \left(\frac{1}{\sqrt{2n}} \sum_{i=1}^n  h_{ij} \right)
     =\frac{1}{\sqrt{n}}\sum_{i=1}^n U_i q_i
\end{align*}
where
$
q_i=\frac{1}{\sqrt{2n}} \sum_{j=1}^n\left( h_{ij}-h_{ji} \right)$ and hence $q=\frac{1}{\sqrt{2n}} (h-h^{\top}) \boldsymbol{1}.$
The variance of $q_i$ is strictly less than one since $h_{ij}-h_{ji}=0$ when $i=j$.
\end{proof}

\paragraph{\textcolor{black}{Validity via covariance estimation.}}
Next, we bound
$
\|\Sigma^{1/2}g - \hat{\Sigma}^{1/2} g\|
$ in terms of $\Sigma$, $\hat\Sigma$.
As before, we will proceed using finite dimensional approximation; our use of orthogonal projections exploits Hilbert space structure, leading to simple and sharp bounds.

We write this subsection with some additional generality to accommodate alternative bootstraps. Throughout this subsection, $\hat{\Sigma}$ denotes some feasible covariance estimator, with the property that
$\hat{\Sigma}^{1/2}g \overset{D}{\sim} \sqrt{n}\mathbb{E}_n(q_iU_i)$ for some jointly Gaussian random variables $q_i$ that have variance at most one and that may be correlated. As before, $\Pi_m$ denotes the projection onto the top $m$
eigenvectors of $\Sigma$; we also put $\Pi_m^\perp = I - \Pi_m$.

The structure of the argument is as follows: (i) we state and prove necessary lemmas; (ii) we give an abstract bound on $\|\Sigma^{1/2}g - \hat{\Sigma}^{1/2} g\|$ that is agnostic to the covariance estimator $\hat \Sigma$; (iii) we specialize the bound to the $\hat\Sigma$ that corresponds to the anti-symmetric bootstrap.

 We first prove infinite-dimensional variants of familiar results: orthogonal decomposition symmetric matrices and the ``sandwich'' covariance formula.
 \begin{lemma}[Covariance splitting]\label{lemma:limit-tail-split-bs}
We have
$\Sigma^{\f 1 2} = (\Pi_m \Sigma \Pi_m)^{\f 1 2} + (\Pi_m^\perp \Sigma \Pi_m^\perp)^{\f 1 2}.$
\end{lemma}

\begin{proof}
Write $\Sigma^{\f 1 2} = \Sigma^{\f 1 2}(\Pi_m + \Pi_m^\perp)$.
Since
$(e_i \otimes e_i^*)(e_j \otimes e_j^*) = \bk{e_i}{e_j}(e_i \otimes e_j^*) = \d_{ij}(e_i \otimes e_j^*),$
it follows by definition of $\Sigma^{1/2}$ and $\Pi_m$ that
$$
    \Sigma^{\f 1 2} \Pi_m
    =\left(\sum_{i=1}^{\infty}\sqrt{\nu_i}e_i\otimes e_i^*\right)\left(\sum_{j=1}^{m}e_j\otimes e_j^*\right)
    =\sum_{j=1}^{m}\sqrt{\nu_j}(e_j \otimes e_j^*)
$$
which is self-adjoint. Therefore
$\Sigma^{\f 1 2}\Pi_m = (\Sigma^{\f 1 2}\Pi_m)^* = \Pi_m^*(\Sigma^{\f 1 2})^* = \Pi_m \Sigma^{\f 1 2}$ and hence
$$
(\Pi_m \Sigma \Pi_m)^{\f 1 2}=\Pi_m\Sigma^{\f 1 2}\Pi_m=\Sigma^{\f 1 2}\Pi_m^2=\Sigma^{\f 1 2}\Pi_m.
$$
The same holds for $\Pi_m^\perp$, replacing the indexing from $j\leq m$ to $j>m$.
\end{proof}

\begin{lemma}[Sandwich]\label{lemma:joint-gaussian-conditional-covariance}
    If $(q_i)$ is a sequence of jointly Gaussian random variables, $A: H \to H$ is a self-adjoint operator, and
    $\frac{1}{\sqrt n} \sum_{i=1}^n q_i U_i \overset{D}{\sim} \hat{\Sigma}^{\f 1 2} g$
    then
    $\frac{1}{\sqrt n} \sum_{i=1}^n q_i AU_i \overset{D}{\sim} (A\hat{\Sigma} A)^{\f 1 2}g.$
\end{lemma}
\begin{proof}
    Since
    $\frac{1}{\sqrt n} \sum_{i=1}^n q_i U_i \overset{D}{\sim} \hat{\Sigma}^{\f 1 2} g,$ both vectors must have covariance operator $\hat{\Sigma}$.
    Now, conditional upon $U$, both vectors
    are jointly Gaussian in $\vspan(U_1, AU_1, \ldots, U_n, AU_n) \subset H$. Therefore it suffices to compute the covariance
    \begin{align*}
    &\bb{E}\left(\bk{\frac{1}{\sqrt n} \sum_{i=1}^n q_i AU_i}{u}\bk{ \frac{1}{\sqrt n} \sum_{i=1}^n q_i AU_i}{v}\right)
        = \bb{E}\left(\bk{\frac{1}{\sqrt n} \sum_{i=1}^n q_i U_i}{Au}\bk{ \frac{1}{\sqrt n} \sum_{i=1}^n q_i U_i}{Av}\right) \\
        &=
    \bb{E}\left(\bk{\hat{\Sigma}^{\f 1 2} g}{Au}\bk{\hat{\Sigma}^{\f 1 2} g}{Av}\right)
    =
        \bk{Au}{\hat{\Sigma} Av}
        = \bk{u}{A\hat{\Sigma} A v}
    \end{align*}
   by repeatedly using self-adjointness of $A$ and the definition of the covariance.
\end{proof}
A challenge of finite-dimensional projection is that $\hat \Sigma^{\f 1 2} \ne (\Pi_m \hat \Sigma \Pi_m)^{\f 1 2} + (\Pi_m^\perp \hat \Sigma \Pi_m^\perp)^{\f 1 2},$ in contrast to Lemma \ref{lemma:limit-tail-split-bs}, since the principal components of $\Sigma$ and $\hat \Sigma$ are not aligned. We need the following technical lemma, which says $\hat \Sigma^{\f 1 2}g$ and $(\Pi_m \hat \Sigma \Pi_m)^{\f 1 2}g$ are close in distribution whenever $(\Pi_m^\perp \hat \Sigma \Pi_m^\perp)^{\f 1 2}g$ is small.
\begin{lemma}[Covariance alignment]\label{lemma:plug-in-tail-split-bs}
Suppose $\bb{P}\{\|(\Pi_m^\perp \hat{\Sigma} \Pi_m^\perp)^{\f 1 2}g\| > \a \} \le \beta$. Then, there exists a random variable $G \overset{D}{\sim} \hat{\Sigma}^{\f 1 2}g$ such that w.p. $1-\beta$,
$\norm{(\Pi_m \hat{\Sigma} \Pi_m)^{\f 1 2}g-G} \le \a.$
\end{lemma}

\begin{proof}
By Lemma~\ref{lemma:joint-gaussian-conditional-covariance},
$$
 \frac{1}{\sqrt n} \sum_{i=1}^n q_iU_i  \overset{D}{\sim} \hat{\Sigma}^{\f 1 2}g, \quad \frac{1}{\sqrt n} \sum_{i=1}^n q_i\Pi_mU_i \overset{D}{\sim} (\Pi_m \hat{\Sigma} \Pi_m)^{\f 1 2}g, \quad \frac{1}{\sqrt n} \sum_{i=1}^n q_i\Pi_m^\perp U_i \overset{D}{\sim} (\Pi_m^\perp \hat{\Sigma} \Pi_m^\perp)^{\f 1 2}g.$$
Moreover,
$\frac{1}{\sqrt n} \sum_{i=1}^n q_iU_i = \frac{1}{\sqrt n} \sum_{i=1}^n q_i\Pi_mU_i + \frac{1}{\sqrt n} \sum_{i=1}^n q_i\Pi_m^\perp U_i.$
Thus we may apply Strassen's lemma (Lemma~\ref{lemma:pushforward-coupling}) with $\c F = \sigma(D)$, choosing the random variables to be
$X=(\Pi_m \hat{\Sigma} \Pi_m)^{\f 1 2}g$, $X'=\frac{1}{\sqrt n} \sum_{i=1}^n q_i\Pi_mU_i$, and $Y'= \frac{1}{\sqrt n} \sum_{i=1}^n q_iU_i.$
In particular,
$$\bb{P}\left(\beef \norm{X'-Y'} > \a\right) = \bb{P}\left( \norm{\frac{1}{\sqrt n} \sum_{i=1}^n q_i\Pi_m^\perp U_i} > \a \right) = \bb{P}\left(\beef\|(\Pi_m^\perp \hat{\Sigma} \Pi_m^\perp)^{\f 1 2}g\| > \a \right) \le \beta,$$
by construction of $X'$ and $Y'$ for the first step and equality in conditional distribution for the second. Thus Strassen's lemma (Lemma~\ref{lemma:pushforward-coupling}) guarantees the random variable $Y=G$ exists
with the same conditional distribution as $Y'$, and hence $\hat \Sigma^{\f 1 2}g$, such that
$\bb{P}\left(\beef \norm{X-Y} > \a\right) \le \b.$
\end{proof}

Next, we bound $
\snorm{\Sigma^{1/2}g - \hat{\Sigma}^{1/2} g}$ with high probability conditional upon $D$. We prove the gap generally depends on
$
    \Delta_1=\|\hat{\Sigma}-\Sigma\|_{\HS}$ and $
    \Delta_2= \tr\,\{\Pi_m^\perp (\hat \Sigma - \Sigma) \Pi_m^\perp\} \vee 0.
$
We will bound these quantities later for a particular $\hat{\Sigma}$ using the randomness in $U$.

\begin{proposition}[Abstract bound]\label{prop:bs-coupling}
There exists a random variable $G \overset{D}{\sim} \hat{\Sigma}^{\f 1 2}g$ such that w.p. $1-3\eta$, conditional on $D$,
$
\norm{\Sigma^{\f 1 2} g-G}\leq \left \{1 + \sqrt{2\log(1/\eta)}\,\right\} \left\{ m^{\f 1 4}\Delta_1^{1/2}+\Delta_2^{1/2} +  2\sigma(m)\right\}.
$
\end{proposition}

\begin{proof}
We proceed in steps, first decomposing $\norm{\Sigma^{\f 1 2} g-G}$ into projection and remainder terms, then bounding each.
\begin{enumerate}
    \item We show that if $\bb{P}\{\|(\Pi_m^\perp \hat{\Sigma} \Pi_m^\perp)^{\f 1 2}g\| > \a \} \le \beta$, then w.p. $1 - \beta$ it holds that
   $$
   \norm{\Sigma^{\f 1 2} g-G}  \le \norm{(\Pi_m\Sigma\Pi_m)^{\f 1 2} g - (\Pi_m\hat{\Sigma}\Pi_m)^{\f 1 2}g} + \norm{(\Pi_m^\perp \Sigma\Pi_m^\perp )^{\f 1 2} g}+\a.
$$
Let $G$ be constructed as in Lemma~\ref{lemma:plug-in-tail-split-bs}. By Lemma~\ref{lemma:limit-tail-split-bs},
$\Sigma^{\f 1 2}g = (\Pi_m \Sigma \Pi_m)^{\f 1 2}g + (\Pi_m^\perp \Sigma \Pi_m^\perp)^{\f 1 2}g.$
Thus, by adding and subtracting, we have the decomposition
\begin{equation*}
    \Sigma^{\f 1 2} g - G = \{(\Pi_m \Sigma \Pi_m)^{\f 1 2}g - (\Pi_m \hat{\Sigma} \Pi_m)^{\f 1 2}g\}
     + (\Pi_m^\perp \Sigma \Pi_m^\perp)^{\f 1 2}g
  + \{(\Pi_m \hat{\Sigma} \Pi_m)^{\f 1 2}g - G\}.
\end{equation*}
We use the triangle inequality. By Lemma~\ref{lemma:plug-in-tail-split-bs}, the last term has norm at most $\a$ w.p. $1-\beta$.

    \item We show that if $\bb{P}\{\|(\Pi_m^\perp \hat{\Sigma} \Pi_m^\perp)^{\f 1 2}g\| > \a \} \le \beta$, then w.p. $1 - \beta-2\eta $
   $$
   \norm{\Sigma^{\f 1 2} g-G}
   \le \left\{ m^{\f 1 4}\snorm{\Sigma - \hat{\Sigma}}_{\HS}^{\f 1 2} +  \sigma(m)\right\}\left \{1 + \sqrt{2\log(1/\eta)}\,\right\}+\a.
   $$
We apply Borell's inequality (Lemma~\ref{lemma:gaussian-concentration-trace}) to the first and second terms in the above decomposition, conditional upon the data. For the first term,  w.p. $1-\eta$,
\begin{align*}
    &\snorm{(\Pi_m \Sigma \Pi_m)^{\f 1 2}g- (\Pi_m \hat{\Sigma} \Pi_m)^{\f 1 2}g}
    \le \snorm{(\Pi_m \Sigma \Pi_m)^{\f 1 2}- (\Pi_m \hat{\Sigma} \Pi_m)^{\f 1 2}}_{\HS}\{1 + \sqrt{2\log(1/\eta)}\} \\
&\le  m^{\f 1 4}\snorm{(\Pi_m \Sigma \Pi_m)- (\Pi_m \hat{\Sigma} \Pi_m)}_{\HS}^{\f 1 2}\{1 + \sqrt{2\log(1/\eta)}\} \\
&\le  m^{\f 1 4}\snorm{\Pi_m}_{\op}^{\f 1 2}\snorm{\Sigma- \hat{\Sigma} }_{\HS}^{\f 1 2}\snorm{\Pi_m}_{\op}^{\f 1 2}\{1 + \sqrt{2\log(1/\eta)}\}
\le m^{\f 1 4}\snorm{\Sigma- \hat{\Sigma} }_{\HS}^{\f 1 2}\{1 + \sqrt{2\log(1/\eta)}\}
\end{align*}
using the Powers-Stormer inequality (Lemma~\ref{lemma:holder-frob}) and $\norm{\Pi_m}_{\op}=1$. For the other term we apply Borell's inequality (Lemma~\ref{lemma:gaussian-concentration-trace}) similarly. Since $\Pi_m^\perp \Sigma^{\f 1 2}\Pi_m^\perp$ is self-adjoint we have
$
    \snorm{(\Pi_m^\perp \Sigma \Pi_m^\perp)^{\f 1 2}}_{\HS}= \sqrt{\tr \Pi_m^\perp \Sigma \Pi_m^\perp}
= \sqrt{\sigma^2(m)}.
$
After a union bound over events w.p. $\eta$, $\eta$, and $\beta$, the proof is complete.

    \item Finally, we determine $\a$ and $\beta$. By Borell's inequality (Lemma~\ref{lemma:gaussian-concentration-trace}), w.p. $1-\eta$,
  $
        \beef\|(\Pi_m^\perp \hat \Sigma \Pi_m^\perp)^{\f 1 2}g\|
        \leq  \beef\|(\Pi_m^\perp \hat \Sigma \Pi_m^\perp)^{\f 1 2}\|_{\HS} \{1 + \sqrt{2\log(1/\eta)}\}.
  $
   Moreover, with $|x|_+ = x \vee 0$,
   \begin{align*}
       &\beef\|(\Pi_m^\perp \hat{\Sigma} \Pi_m^\perp)^{\f 1 2}\|_{\HS}
       =\{\tr (\Pi_m^\perp \hat{\Sigma} \Pi_m^\perp)\}^{1/2}
       =[\tr \{\Pi_m^\perp (\hat{\Sigma}-\Sigma) \Pi_m^\perp\}+\tr (\Pi_m^\perp \Sigma \Pi_m^\perp)]^{1/2} \\
       &\leq |\tr \{\Pi_m^\perp  (\hat{\Sigma}-\Sigma) \Pi_m^\perp\}|^{1/2}_+ +\{\tr (\Pi_m^\perp \Sigma \Pi_m^\perp) \}^{1/2}
       =|\tr \{\Pi_m^\perp (\hat{\Sigma}-\Sigma) \Pi_m^\perp\}|^{1/2}_+ +\sigma(m).
   \end{align*}
    In summary, w.p. $1-\eta$
    $$
    \beef\|(\Pi_m^\perp \hat \Sigma \Pi_m^\perp)^{\f 1 2}g\|
        \leq \left\{1 + \sqrt{2\log(1/\eta)}\,\right\} \left(|\tr \{\Pi_m^\perp (\hat{\Sigma}-\Sigma) \Pi_m^\perp\}|_+^{1/2}+\sigma(m)\right)=\alpha.
    $$
    Therefore by the results above, w.p. $1-3\eta$, $\norm{\Sigma^{\f 1 2} g-G}$ is bounded by
    $$
   \left \{1 + \sqrt{2\log(1/\eta)}\,\right\} \left\{ m^{\f 1 4}\snorm{\Sigma - \hat \Sigma}_{\HS}^{\f 1 2}+|\tr \{\Pi_m^\perp (\hat{\Sigma}-\Sigma) \Pi_m^\perp\}|^{1/2}_+ +  2\sigma(m)\right\}. \qedhere
$$
\end{enumerate}
\end{proof}

\paragraph{\textcolor{black}{Covariance error bound.}} We bound $\Delta_1$ and $\Delta_2$ from Proposition \ref{prop:bs-coupling} for the anti-symmetric bootstrap, where  $
\hat{\Sigma}=\mathbb{E}_n(U_i\otimes U_i^*)-\mathbb{E}_n(U_i)\otimes \{\mathbb{E}_n(U_i)\}^*.
$ Appendix~\ref{sec:cov_appendix} gives proofs.

\begin{lemma}[Covariance estimation]\label{lemma:all}
Under $a$-boundedness, w.p. $1-3\eta$, both %
\begin{align*}
    \Delta_1&\leq 2\log(2/\eta)^2\left\{\sqrt{\frac{a^2\sigma^2(0)}{n}} \vee \frac{4a^2}{n} \vee \frac{8a^2}{n^2}\right\},\quad
    \Delta_2\leq 2\log(2/\eta)\left\{\sqrt{\frac{a^2\sigma^2(m)}{n}} \vee \frac{2a^2}{n}\right\}.
\end{align*}
\end{lemma}

\paragraph{\textcolor{black}{Main results.}} The main results consider randomness of the multipliers and of the data.

\begin{theorem}[Bootstrap coupling]\label{thm:bootstrap}
Suppose $a$-boundedness holds and $n \ge 2$. Then there exists a random element $G$ with the same conditional distribution as the multiplier bootstrap process $\hat \Sigma^{\f 1 2}g$ such that with total probability at least $1 - 6\eta$
$$
\norm{\Sigma^{\f 1 2}g - G} \le  C\log(2/\eta)^{3/2} R_{\mathrm{bd}}(n),\quad R_{\mathrm{bd}}(n)=\inf_{m\geq 1} \left[m^{\f 1 4}\left\{\frac{a^2\sigma^2(0)}{n} + \frac{a^4}{n^2}\right\}^{\f 1 4} + \sigma(m)\right].$$
\end{theorem}

\begin{proof}
\textcolor{black}{We substitute the bounds from Lemma~\ref{lemma:all} into Proposition~\ref{prop:bs-coupling}, using a union bound. Here the bound on $\Delta_1$ dominates $\Delta_2$.}
\textcolor{black}{With probability $1-6\eta$,}
\begin{align*}
   \norm{\Sigma^{\f 1 2} g-G}
   &\le C \log(2/\eta)^{3/2} \left[ m^{\f 1 4}\left\{\sqrt{\frac{a^2\sigma^2(0)}{n}} \vee \frac{4a^2}{n}\vee \frac{8a^2}{n^2}\right\}^{\f 1 2} + \sigma(m)\right].
\end{align*}
If $n \ge 2$ then $8a^2/n^2 \le 4a^2/n$.
\end{proof}

The following corollary is useful in case we cannot sample from the multiplier bootstrap process $\hat \Sigma^{\f 1 2}g$ directly, but can sample from a proxy for it, namely $\mathfrak{B}$.  %

\begin{corollary}[Approximate bootstrap]\label{cor:bootstrap-apx}
{\color{black}
    Fix $\eta\in(0,1)$. Suppose $W,W'$ are random variables in $H$ with $\bb{P}(\norm{W'-W} \ge \d) \le \eta^2/2$ and $W \overset{D}{\sim} \hat \Sigma^{\f 1 2} g$. Under the conditions of Theorem \ref{thm:bootstrap}, there exists a Gaussian $Z\overset{D}{\sim}\Sigma^{\f 1 2}g$ such that w.p. $1-\eta$,
$
\bb{P}\{\beef \norm{Z - W'} \ge  C' \log(6/\eta)^{3/2} R_{\mathrm{bd}}(n) + \d | D \} \le \eta
$.
If $\eta=n^{-\xi}$, the premise uses failure probability $n^{-2\xi}/2$, affecting only logarithmic factors.}
\end{corollary}

\begin{proof}[Proof of Corollary~\ref{cor:bootstrap-apx}]
{\color{black}
We proceed in steps.

\begin{enumerate}
\item Apply Theorem \ref{thm:bootstrap} with failure parameter $\eta^2/12$. Let $G \overset U \sim \hat \Sigma^{\f 1 2}g$ be the resulting random element. Then
$\bb{P}\left\{\snorm{\Sigma^{\f 1 2}g - G} \ge C\log(24/\eta^2)^{3/2} R_{\mathrm{bd}}(n)\right\} \le \eta^2/2.$
By Strassen's lemma (Lemma~\ref{lemma:pushforward-coupling}) with $X=W$, $X'=G$, $Y'=\Sigma^{\f 1 2}g$, and $Y=Z$,  there exists a $Z \overset D \sim \Sigma^{\f 1 2}g$  such that
    $$\bb{P}\{\norm{Z - W} \ge C\log(24/\eta^2)^{3/2} R_{\mathrm{bd}}(n)\} \le \eta^2/2.$$

\item By the triangle inequality and the hypothesis on $W'-W$,
\begin{align*}
\bb{P}\left\{\norm{Z-W'} \ge C\log(24/\eta^2)^{3/2} R_{\mathrm{bd}}(n)+\d\right\}
\le \eta^2.
\end{align*}
Consequently, for
$$A(\eta) = \bb{P}\left\{\beef \norm{Z-W'} \ge C\log(24/\eta^2)^{3/2} R_{\mathrm{bd}}(n)+\d\middle|D \right\},$$
we have $\bb{E}_D\{A(\eta)\}\leq\eta^2$.

\item Markov's inequality gives $\mathbb{P}_U\{A(\eta)>\eta\}\leq\eta$. Hence, w.p.~$1-\eta$,
\[
 \bb{P}\left\{\beef  \norm{Z - W'} \ge C\log(24/\eta^2)^{3/2} R_{\mathrm{bd}}(n)+\d\middle|D \right\}\leq \eta.
\]
Finally, note that \(
C\{\log(24/\eta^2)\}^{\f 3 2}
\le 2^{\f 3 2}C\{\log(6/\eta)\}^{3/2}
=C'\{\log(6/\eta)\}^{3/2}.\) \qedhere

\end{enumerate}}
\end{proof}

\section{Applying general results to KRR}\label{sec:symbols}

Our abstract results in Section~\ref{sec:partial} are for partial sums of $U_i$ with covariance $\Sigma=\mathbb{E}(U_i \otimes U_i^*)$. Our concrete results in Section~\ref{sec:algo_main} are for KRR. We now relate our abstract assumptions on $U_i$ and $\Sigma$ to KRR, where we take $U_i = T_{\lambda}^{-1}\{(k_{X_i} \otimes k_{X_i}^* - T)(f_0-f_{\lambda}) + \ep_i k_{X_i}\}$ with $T_{\lambda}=T+\lambda$.

\paragraph{\textcolor{black}{Local width.}}  We translate our assumptions for KRR into bounds on the local width of $\Sigma$.

\begin{lemma}[Upper bounding the covariance]\label{lemma:krr-bahadur-covariance}
We have
$ 0 \preceq \Sigma \preceq (\kappa^2\norm{f_0}^2 + \bar\sigma^2)T_{\lambda}^{-2}T.$
\end{lemma}

\begin{proof}
By bilinearity of the tensor product and  $\bb{E}(\eps_i|X_i)=0$,  $\Sigma=\Sigma_1+\Sigma_2$ where
\begin{align*}
0 &\preceq \Sigma_1
 = \bb{E}\left[\left\{\beef T_{\lambda}^{-1}(k_{X_i} \otimes k_{X_i}^* - T)(f_0-f_\lambda) \right\} \otimes \left\{\beef T_{\lambda}^{-1}(k_{X_i} \otimes k_{X_i}^* - T)(f_0-f_\lambda) \right\}^*\right] \\
0 &\preceq \Sigma_2=\bb{E}\left\{\left(\beef\ep_iT_{\lambda}^{-1}k_{X_i} \right) \otimes \left(\beef\ep_iT_{\lambda}^{-1}k_{X_i} \right)^*\right\}.
\end{align*}
For the first term, since $\bb{E}\{(f - \bb{E}f) \otimes (f - \bb{E}f)^*\}
\preceq \bb{E}(f \otimes f^*),$
we have
\begin{align*}
\Sigma_1
&\preceq  \bb{E}\left[\left\{\beef T_{\lambda}^{-1}(k_{X_i} \otimes k_{X_i}^*)(f_0-f_\lambda) \right\} \otimes \left\{\beef T_{\lambda}^{-1}(k_{X_i} \otimes k_{X_i}^*)(f_0-f_\lambda) \right\}^*\right] \\
&= \bb{E}\left[ (T_{\lambda}^{-1}k_{X_i}) \otimes \left(T_{\lambda}^{-1}k_{X_i}\right)^* \{f_0(X_i) - f_\lambda(X_i)\}^2 \right].
\end{align*}
Combining this bound with the second term,
\[0 \preceq \Sigma \preceq \bb{E}\left(\left\{\beef (T_{\lambda}^{-1}k_{X_i}) \otimes  (T_{\lambda}^{-1}k_{X_i})^*\right\}\left[\beef\{f_0(X_i) - f_\lambda(X_i)\}^2 + \ep_i^2\right]\right).\]
We bound the scalars as $|\ep_i| \le \bar\sigma$ and
$\left|f_0(X_i) - f_\lambda(X_i)\right| =\left|\bk{f_0-f_{\lambda}}{k_{X_i}}\right| \le \kappa\norm{f_0}$. The latter follows from Cauchy-Schwarz; $f_0-f_{\lambda}=(I-T_{\lambda}^{-1}T)f_0$ implies $\|f_0-f_{\lambda}\|\leq \|I- T_{\lambda}^{-1}T)\|_{\op}\|f_0\| \le \|f_0\|$, as $I- T_{\lambda}^{-1}T$ is self-adjoint with eigenvalues $0\le \lambda/(\lambda + \nu_s) \le 1$.
Lastly,
\[\bb{E}\left\{(T_{\lambda}^{-1}k_{X_i}) \otimes \left(T_{\lambda}^{-1}k_{X_i}\right)^*\right\} = \bb{E}\{T_{\lambda}^{-1}(k_{X_i} \otimes k_{X_i}^*)T_{\lambda}^{-1}\} = T_{\lambda}^{-1}TT_{\lambda}^{-1}=T_{\lambda}^{-2}T. \qedhere \]
\end{proof}

\begin{lemma}[Lower bounding the covariance]\label{lemma:krr-bahadur-covariance2}
If $\mathbb{E}(\ep_i^2|X_i)\geq \underline{\sigma}^2$ then $\Sigma \succeq \underline{\sigma}^2 T_{\lambda}^{-2}T $.
\end{lemma}

\begin{proof}
As argued in Lemma~\ref{lemma:krr-bahadur-covariance},
$
     \Sigma \succeq \Sigma_2
     \succeq \underline{\sigma}^2 \bb{E}\left\{\left(\beef T_{\lambda}^{-1}k_{X_i} \right) \otimes \left(\beef T_{\lambda}^{-1}k_{X_i} \right)^*\right\}
     .%
$
\end{proof}

\begin{lemma}[Local width comparison]\label{lemma:ordering}
If $A \preceq B$ for trace-class, self-adjoint operators $A$ and $B$, then
$
\sigma^2(A,m)  \leq  \sigma^2(B,m).
$
In particular, taking $m=0$, we recover $\tr(A)\leq \tr(B)$.
\end{lemma}

\begin{proof}
Recall the definition
$\sigma^2(A,m) = \sum_{s=m+1}^\infty \nu_s(A) = \sum_{s=m+1}^\infty\bk{e_s(A)}{Ae_s(A)},$ where $\nu_s(A)$ are eigenvalues and $e_s(A)$ are corresponding eigenfunctions of $A$. Let $(f_1, f_2, \ldots)$ denote any other orthonormal basis of $H$. By the variational representation of the top $m$ eigenfunctions,
$\sum_{s=1}^m \nu_s(A) \ge \sum_{s=1}^m \bk{f_s}{Af_s}.$
Taking $f_s=e_s(B)$, and noting that the trace is independent of the chosen orthonormal basis,
\begin{align*}
\sigma^2(A,m)
&= \sum_{s=m+1}^\infty \nu_s(A)
= \tr(A) - \sum_{s=1}^m \nu_s(A)
\le \tr(A) - \sum_{s=1}^m \bk{e_s(B)}{Ae_s(B)} \\
&= \sum_{s=m+1}^\infty \bk{e_s(B)}{Ae_s(B)}
\le \sum_{s=m+1}^\infty \bk{e_s(B)}{Be_s(B)}
= \sigma^2(B,m). & \qedhere
\end{align*}
\end{proof}

\begin{lemma}[Local width bounds]\label{lemma:inference-rv-bounds}
In our setting,
$
\sigma(\Sigma,m) \le \left(\frac{\kappa\norm{f_0} + \bar\sigma}{\lambda}\right) \sigma(T,m)$ and $\sigma(\Sigma,0) \le (\kappa\norm{f_0} + \bar\sigma)\sqrt{\N(\lambda)}.
$
\end{lemma}

\begin{proof}
By Lemmas~\ref{lemma:krr-bahadur-covariance},
$0 \preceq \Sigma \preceq (\kappa\norm{f_0} + \bar\sigma)^2T_{\lambda}^{-2}T \preceq \left(\frac{\kappa\norm{f_0} + \bar\sigma}{\lambda}\right)^2T.$
Therefore by Lemma~\ref{lemma:ordering},
$
    \sigma^2(\Sigma,0)=\tr(\Sigma) \le (\kappa\norm{f_0} + \bar\sigma)^2\tr(T_{\lambda}^{-2}T)
 $ and $
    \sigma^2(\Sigma,m) \le \left(\frac{\kappa\norm{f_0} + \bar\sigma}{\lambda}\right)^2 \sigma^2(T,m).
$
Finally recall the definition of $\N(\lambda)$.
\end{proof}

\paragraph{\textcolor{black}{Summands.}} \textcolor{black}{We establish boundedness of $U_i$ from the KRR assumptions.}

\begin{lemma}[Bounded summands]\label{lemma:inference-rv-bounds-bounded}
We have
$\norm{U_i} \le a=\left(\frac{\kappa^2\norm{f_0} + \kappa\bar\sigma}{\lambda}\right).$
\end{lemma}

\begin{proof}
Write $\snorm{U_i}\leq \snorm{T_{\lambda}^{-1}}_{\op}\left\{ \snorm{k_{X_i} \otimes k_{X_i}^* - T}_{\op}\snorm{f_0-f_{\lambda}} + \snorm{\ep_i k_{X_i}} \right\}$. Clearly $\snorm{T_{\lambda}^{-1}}_{\op}\leq\lambda^{-1}$ and $\snorm{\ep_i k_{X_i}} \leq \bar{\sigma}\kappa$. Since $k_{X_i} \otimes k_{X_i}^* - T$ is a difference of two positive definite operators, $\snorm{k_{X_i} \otimes k_{X_i}^* - T}_{\op} \le \snorm{k_{X_i} \otimes k_{X_i}^*}_{\op} \vee \snorm{T}_{\op}\leq \kappa^2$. Finally,  we bound $\|f_0-f_{\lambda}\|\leq \|f_0\|$ as in the proof of Lemma~\ref{lemma:krr-bahadur-covariance}.
\end{proof}

\begin{center}
{\large\bf SECONDARY APPENDIX}
\end{center}

Appendix~\ref{sec:algo} derives Algorithms~\ref{algo:incremental} and~\ref{algo:variable}.
Appendix~\ref{sec:technical} lists technical lemmas from probability theory.
Appendix~\ref{sec:bahadur} proves that the KRR residual term vanishes.
Appendix~\ref{sec:bahadur2} proves an analogous result for our proposed bootstrap.
Appendix~\ref{sec:anti} proves Propositions~\ref{prop:h-band} and~\ref{prop:inf-intro}.
Appendix~\ref{sec:simulation_details} provides simulation details.
Appendix~\ref{sec:application_details} provides application details.
\section{Closed form inference}\label{sec:algo}

 We derive closed form expressions for $\mathfrak{B}(x)$ and $\mathfrak{s}^2(x)$, justifying Algorithms \ref{algo:incremental} and \ref{algo:variable}.%

\paragraph{\textcolor{black}{Notation.}}
Let $h \in\mathbb{R}^{n\times n}$ be a matrix  of independent, standard Gaussians $h_{ij}$. Define the kernel matrix $K\in\mathbb{R}^{n\times n}$ with entries $K_{ij}=k(X_i,X_j)$. Define the vectors $K_x,\,\hat\varepsilon,\,\boldsymbol{1}\in \mathbb{R}^n$ with entries
$(K_x)_i=k(x,X_i)$, $\hat\ep_i=Y_i - \hat f(X_i)$, and $\boldsymbol{1}_i=1$, respectively. Finally, define the feature operator $\Phi : H \to \bb{R}^n$ which satisfies $\Phi: f \mapsto (\bk{k_{X_i}} {f})_{i=1}^n$. Intuitively, each ``row'' is $k_{X_i}^*$, generalizing the matrix of covariates. The covariance operator is $\hat T = n^{-1}\Phi^*\Phi$, the Gram matrix is $K=\Phi\Phi^*$, and the evaluation vector is $K_x=k_x^*\Phi^*$.

\paragraph{\textcolor{black}{Closed form expressions.}}
We first express $\mathfrak{B}(x)$ in closed form, then $\tilde{\mathfrak{s}}(x)$.

\begin{proposition}\label{prop:b}
$
\mathfrak{B}(x)=K_x(K+n\lambda)^{-1}\beta
$
 where $\beta=\diag(\hat\ep)\frac{1}{\sqrt{2}}(h - h^{\top})\boldsymbol{1}$.
\end{proposition}

\begin{proof}
    To begin, write
    \begin{align*}
        \mathfrak{B}&=\frac{1}{n}\sum_{i=1}^n\sum_{j=1}^n h_{ij} \frac{\hat{V}_i-\hat{V}_j}{\sqrt{2}}
        = \hat{T}_{\lambda}^{-1}\left\{\frac 1 n \sum_{i=1}^n \sum_{j=1}^n \left(\frac{k_{X_i}\hat\ep_i - k_{X_j}\hat\ep_j}{\sqrt{2}}\right) h_{ij} \right\}.
    \end{align*}
    Focusing on the inner expression,
    \begin{align*}
    &\sum_{i=1}^n \sum_{j=1}^n \left(\frac{k_{X_i}\hat\ep_i - k_{X_j}\hat\ep_j}{\sqrt{2}}\right) h_{ij}
    =  \frac{1}{\sqrt 2} \left(\sum_{i=1}^n k_{X_i}\hat\ep_i \sum_{j=1}^n h_{ij} - \sum_{j=1}^n  k_{X_j}\hat\ep_j \sum_{i=1}^n h_{ij}   \right) \\
    &=  \frac{1}{\sqrt 2}\sum_{i=1}^n k_{X_i}\hat\ep_i \left\{\sum_{j=1}^n (h_{ij} - h_{ji})\right\}
   = \sum_{i=1}^n \beta_i k_{X_i}
   =\Phi^*\beta.
\end{align*}
Substituting this into the full expression,
\begin{align*}
   \mathfrak{B}=\left(\frac{1}{n}\Phi^*\Phi+\lambda\right)^{-1} \left(\frac{1}{n}\Phi^*\beta\right)
    =\left(\Phi^*\Phi+n\lambda\right)^{-1} \Phi^*\beta
    =\Phi^*(\Phi\Phi^*+n\lambda)^{-1}\beta
\end{align*}
and hence
$
\mathfrak{B}(x)=K_x(K+n\lambda)^{-1}\beta.
$
\end{proof}

\begin{proposition}\label{prop:s_tilde}
Let $\mathbb{E}_h$ denote the expectation over $h$, conditional upon data. Then
   $$
   \tilde{\s}^2(x)=\mathbb{E}_h\{\mathfrak{B}(x)^2\}=n K_x(K+n\lambda)^{-1}\diag(\hat\ep) (I-\boldsymbol{1}\boldsymbol{1}^{\top}/n) \diag(\hat\ep)(K+n\lambda)^{-1}K_x^{\top}.
    $$
\end{proposition}

\begin{proof}
By Proposition~\ref{prop:b},
    $
    \mathbb{E}_h\{\mathfrak{B}(x)^2\}
        =K_x(K+n\lambda)^{-1}\mathbb{E}_h(\beta \beta^{\top})(K+n\lambda)^{-1}K_x^{\top}.
    $
Observe that $\beta=\diag(\hat\ep) \sqrt{n}q$ where $q=\frac{1}{\sqrt{2n}} (h-h^{\top}) \boldsymbol{1}$, so
    $\mathbb{E}_h(\beta \beta^{\top})=n\diag(\hat\ep)  \mathbb{E}_h(q q^{\top})\diag(\hat\ep)$. What remains is $\mathbb{E}(q q^{\top})$. We will show that $\mathbb{E}_h(qq^{\top})=I-\boldsymbol{1}\boldsymbol{1}^{\top}/n$, proving the result. Towards this end, write $
q_i=\frac{1}{\sqrt{2n}} \sum_{j=1}^n\left( h_{ij}-h_{ji} \right)$.

\begin{enumerate}
\item For diagonal terms, fix $i$ and write $\mathbb{E}_h(q_i^2)$ as
    \begin{align*}
        &\frac{1}{2n} \sum_{j=1}^n\sum_{k=1}^n\mathbb{E}_h\left\{\left( h_{ij}-h_{ji} \right)\left( h_{ik}-h_{ki} \right)\right\}
        =\frac{1}{2n} \sum_{j=1}^n\sum_{k=1}^n\mathbb{E}_h\left( h_{ij}h_{ik}-h_{ij}h_{ki}-h_{ji}h_{ik}+h_{ji}h_{ki} \right)\\
        &=\frac{1}{2n} \sum_{j=1}^n\sum_{k=1}^n\left( 1_{j=k}-1_{i=j=k}-1_{i=j=k}+1_{j=k} \right)
        =1-1/n.
    \end{align*}
    \item For off diagonal terms, fix $i\neq \ell$ and write $\mathbb{E}_h(q_iq_{\ell})$ as
    \begin{align*}
    &\frac{1}{2n} \sum_{j=1}^n\sum_{k=1}^n\mathbb{E}_h\left\{\left( h_{ij}-h_{ji} \right)\left( h_{\ell k}-h_{k\ell } \right)\right\}
        =\frac{1}{2n} \sum_{j=1}^n\sum_{k=1}^n\mathbb{E}_h\left( h_{ij}h_{\ell k}-h_{ij}h_{k\ell}-h_{ji}h_{\ell k}+h_{ji}h_{k\ell} \right)\\
          &=\frac{1}{2n} \sum_{j=1}^n\sum_{k=1}^n\left( 1_{i=\ell,j=k}-1_{i=k,j=\ell}-1_{j=\ell,i=k}+1_{j=k,i=\ell} \right)
          =-1/n. & \qedhere
    \end{align*}
\end{enumerate}
\end{proof}

\section{Technical lemmas}\label{sec:technical}

\paragraph{\textcolor{black}{Bernstein's inequality.}}
We use a concentration inequality for i.i.d.~sums in Hilbert space.

\begin{lemma}[Proposition 2 of \cite{caponnetto2007optimal}]\label{lemma:c_dv}
Suppose that $\xi_i$ are i.i.d.~random elements of a Hilbert space, which satisfy, for all $\ell \ge 2$
$\bb{E}\norm{\xi_i-\bb{E}\xi_i}^\ell \le \frac 1 2 \ell! B^2 (A/2)^{\ell-2}.$
Then
 for any $0<\eta < 1$ it holds w.p. $1-\eta$ that
$$
\norm{\frac{1}{n}\sum_{i=1}^n \xi_i-\mathbb{E}(\xi_i)}
\le 2 \left\{\sqrt{\frac{B^2\log(2/\eta)}{n}} \vee \frac{A\log(2/\eta)}{n}\right\}
\leq 2 \log (2/\eta)\left(\frac{A}{n} \vee \sqrt{\frac{B^2}{n}}\right).
$$
In particular, this holds if
$\bb{E}(\norm{\xi_i}^2) \le B^2$ and $\norm{\xi_i} \le A/2$ almost surely.
\end{lemma}

Lemma~\ref{lemma:c_dv} for bounded random vectors may be recovered from Talagrand's concentration inequality for empirical processes (see \citealp{massart2000constants}), by considering the special case where the sample paths are linear and the parameter space is an ellipsoid.

\paragraph{\textcolor{black}{Gaussian concentration.}}
We quote a result of Borell, Sudakov, Ibragamov, and Tsirelson.
\begin{lemma}[Borell's inequality, Theorem 2.5.8 of {\citet{gine2021mathematical}}] \label{lemma:borell}
Let $G_t$ be a centered Gaussian process, a.s.~bounded on $T$. Then for $u > 0$,
$$\mathbb{P}\left( \sup_{t\in T} G_t-\mathbb{E}\sup_{t\in T} G_t >u \right) \vee  \mathbb{P}\left( \sup_{t\in T} G_t-\mathbb{E}\sup_{t\in T} G_t < -u \right) \leq \exp\left(\frac{-u^2}{2\sigma^2_T}\right)
$$
where
$\sigma^2_T=\sup_{t\in T} \mathbb{E}G_t^2$.
\end{lemma}

The following corollary bounds the norm of a Gaussian vector.

\begin{lemma}[Gaussian norm bound]\label{lemma:gaussian-concentration-trace} Let $Z$ be a Gaussian random element in a Hilbert space $H$ such that $\bb{E}\norm{Z}^2 < \infty$. Then, w.p. $1-\eta$,
$$\norm{Z} \le
\left\{1 + \sqrt{2\log(1/\eta)}\right\}\sqrt{\bb{E}\norm{Z}^2}.$$
In particular, if $A: H \to H$ is a Hilbert-Schmidt operator, then w.p. $1-\eta$ with respect to $g$,
$$ \snorm{Ag} \le \left\{1 + \sqrt{2\log(1/\eta)}\,\right\}\snorm{A}_{\HS}.
$$
\end{lemma}

\begin{proof}
We proceed in steps.
\begin{enumerate}
    \item For the first claim, we express $\snorm{Z}$ as the supremum of a separable Gaussian process, in particular
$\snorm{Z} = \sup_{t \in B_H}\bk{t}{Z}=\sup_{t\in T} G_t$ where $B_H$ is the unit ball in $H$.
The result follows from Borell's inequality (Lemma \ref{lemma:borell}) provided we can estimate
\[\mathbb{E}\sup_{t\in T} G_t=\bb{E}\sup_{t \in B_H}\bk{t}{Z} = \bb{E}\snorm{Z}, \quad \sigma^2_T=\sup_{t \in B_H}\bb{E}\bk{t}{Z}^2,\] and show that the process is a.s.~bounded.

By Jensen's inequality we have
$
    (\bb{E}\snorm{Z})^2 \le \bb{E}\snorm{Z}^2 < \infty,
$
so we may deduce from Markov's inequality that
$\bk{t}{Z}$ is a.s.~bounded.
Also by Jensen's inequality
\[\sup_{t \in B_H}\bb{E}\bk{t}{Z}^2 \le \bb{E} \sup_{t \in B_H}\bk{t}{Z}^2 = \mathbb{E}\norm{Z}^2.\]
Plugging these estimates into Borell's inequality (Lemma~\ref{lemma:borell}) then gives
\[\bb{P}\left(\norm{Z} \ge \sqrt{\bb{E}\norm{Z}^2} + u\right) \le \exp\left(\frac{-u^2}{2\bb{E}\norm{Z}^2}\right).\] Choosing $u = \sqrt{2\log(1/\eta)\bb{E}\norm{Z}^2}$ gives the desired result.

    \item For the second claim, take $Z = Ag$. We need to check that $\bb{E}\norm{Ag}^2 = \norm{A}_{\HS}^2$.
Indeed,
$$
\bb{E}\|Ag\|^2 = \sum_{s,t} \mathbb{E}[g_sg_t]\langle Ae_s,Ae_t \rangle=\sum_s \|Ae_s\|^2= \|A\|^2_{\HS}. \qedhere
$$
\end{enumerate}
\end{proof}

\paragraph{\textcolor{black}{Powers-Stormer inequality.}} We cite the following matrix norm inequality.

\begin{lemma}[cf.~{\citet[Theorem 1.1]{wihler2009holder}}]\label{lemma:holder-frob}
Let $A$ and $B$ be $m \times m$ real, positive semidefinite, symmetric matrices. Then
$\snorm{A^{\f 1 2} - B^{\f 1 2}}_{\HS} \le m^{\f 1 4}\snorm{A - B}_{\HS}^{\f 1 2}.$
\end{lemma}
\begin{proof} According to \citet[Theorem 1.1]{wihler2009holder}, we have
$\snorm{f(A) - f(B)}_{\HS} \le (f)_{0,\f 1 2}m^{\f 1 4}\snorm{A - B}_{\HS}^{\f 1 2}$ where $(f)_{0,\f 1 2}$ is the H\"older constant
$(f)_{0,\f 1 2} = \sup_{x \ne y} \frac{|f(x)-f(y)|}{|x-y|^{\f 1 2}}.$ We show that for $f(x) = \sqrt{x}$ this constant is at most $1$. Without loss of generality, $x > y$. In this case we have
$(\sqrt x - \sqrt y)^2 = x(1-\sqrt{y/x})^2 \le x(1-y/x) = x-y,$
since $|1-a| \le \sqrt{|1-a^2|}$ for $a = \sqrt{y/x}$, as the graph of the semicircle is concave over $[-1,1]$. %
\end{proof}

\paragraph{\textcolor{black}{Strassen's lemma.}} The following result is useful for coupling probability distributions.

\begin{lemma}[Conditional Strassen's lemma; {\citealt[Theorem 4]{monrad1991nearby}}]\label{lemma:conditional-strassen}
Let $X$ be a random variable on a probability space
$(\Omega, \c S, \bb{P})$, and suppose that $X$ takes values in a complete metric space $(S,d)$. Let $\c F \subset \c S$ be countably generated as a $\sigma$-algebra, and assume that there exists a random variable $R$ on $(\Omega, \c S, \bb{P})$ that
is independent of $\c F \vee \sigma(X)$. Let $G(-| \c F)$ be a regular conditional distribution on the Borel sets $\c B$ of $(S,d)$ and suppose that for some non-negative numbers $\a$ and $\beta$
\[ \bb{E}\sup_{A \in \c B} \left[ \beef \bb{P}(X\in A|\c F) - G\{\mr{cl}(A^{\a})|\c F\} \right] \le \beta\]
where $A^{\a}$ is the $\a$-extension of $A$ and the randomness in the expectation is over $\c F$.

Then there exists a random variable $Y$ with values in $S$, defined on $(\Omega, \c S, \bb{P})$
with conditional distribution $G$ satisfying
$\bb{P}\left\{ \beef d(X, Y) > \a \right\} \le \beta.$
\end{lemma}

\begin{lemma}[Simplified Strassen's lemma]\label{lemma:pushforward-coupling}
Under the same conditions as Lemma~\ref{lemma:conditional-strassen}, suppose there exist random variables $X'$ and $Y'$ such that (i) $X$ and $X'$ have the same distribution conditional upon $\c F$, and (ii) $\bb{P}\{d(X',Y') > \alpha\} \le \beta$. Then there exists some $Y$ with the same conditional distribution as $Y'$ such that $\bb{P}\{d(X,Y) > \a\} \le \beta$.
\end{lemma}
\begin{proof}
According to Lemma \ref{lemma:conditional-strassen}, it suffices to show
\[ \bb{E}\sup_{A \in \c B} \left[ \beef \bb{P}(X\in A|\c F) - \bb{P}\{Y' \in \mr{cl}(A^{\a})|\c F\} \right] = \bb{E}\sup_{A \in \c B} \left[ \beef \bb{P}(X'\in A|\c F) - \bb{P}\{Y' \in \mr{cl}(A^{\a})|\c F\} \right]\le \beta,\] where we have first used the fact that $X'$ and $X$ are equal in conditional distribution.

Now, consider the event
$E = \set{\beef d(X',Y') \le \a}.$ By hypothesis, $\bb{P}(E) \ge 1-\beta$.
Also, by construction, for any Borel set $A$, we have
$\set{\beef X' \in A} \cap E \subseteq \set{\beef Y' \in \mr{cl}(A^{\a})}.$
It follows that on the event $E$,
$\sup_{A \in \c S} \mathbbm{1}\set{\beef X' \in A} - \mathbbm{1}\set{\beef Y' \in \mr{cl}(A^{\a})} \le 0$, where $\mathbbm{1}(\cdot)$ is the indicator function for an event.
Moreover, since the expression inside the supremum is a difference of two probabilities, it is at most $1$ everywhere. In particular, this crude bound holds on the complement of $E$. Thus, by the conditional version of Jensen's inequality,
\begin{align*}
&\bb{E}\sup_{A \in \c S} \left[ \beef \bb{P}(X'\in A|\c F) - \bb{P}\{Y' \in \mr{cl}(A^{\a})|\c F\} \right]
= \bb{E} \left(\sup_{A \in \c S} \bb{E}\left[ \mathbbm{1}\set{\beef X' \in A} - \mathbbm{1}\set{\beef Y' \in \mr{cl}(A^{\a})} \middle| \c F \right] \right)  \\
&\le \bb{E} \left[\sup_{A \in \c S}  \mathbbm{1}\set{\beef X' \in A} - \mathbbm{1}\set{\beef Y' \in \mr{cl}(A^{\a})}\right]
\le 0 \cdot \bb{P}(E) + 1 \cdot\{1 - \bb{P}(E)\} \le \beta. & \qedhere
\end{align*}
\end{proof}

\paragraph{\textcolor{black}{Iterated resolvent identity.}} We give a convenient series expansion of $A^{-1}-B^{-1}$.

\begin{lemma}[Higher-order resolvent]\label{lemma:powers}
Let $V$ be a vector space and $A, B: V \to V$ be invertible linear operators. Then, for all $\ell \ge 1$, it holds
$$\inv A - \inv B = A^{-1}\{(B-A)B^{-1}\}^\ell+ \sum_{r=1}^{\ell-1}B^{-1}\{(B-A)B^{-1}\}^{r}.$$
\end{lemma}

\begin{proof}
When $\ell=1$, this reduces to the familiar ``resolvent identity''
\[ A^{-1}-B^{-1}=A^{-1}(B-A)B^{-1} \iff A^{-1} = A^{-1}(B-A)B^{-1} +  B^{-1}.\]
We proceed by induction. Suppose the inequality holds for $\ell-1$. Plugging in the resolvent identity for the left-most appearance of $A^{-1}$ gives
\begin{align*}
\quad \inv A - \inv B
&= A^{-1}\{(B-A)B^{-1}\}^{\ell-1}+ \sum_{r=1}^{\ell-2}B^{-1}\{(B-A)B^{-1}\}^{r} \\
&= \{A^{-1}(B-A)B^{-1} + B^{-1}\}\{(B-A)B^{-1}\}^{\ell-1}+ \sum_{r=1}^{\ell-2}B^{-1}\{(B-A)B^{-1}\}^{r} \\
&= A^{-1}\{(B-A)B^{-1}\}^{\ell} + B^{-1}\{(B-A)B^{-1}\}^{\ell-1} + \sum_{r=1}^{\ell-2}B^{-1}\{(B-A)B^{-1}\}^{r} \\
&= A^{-1}\{(B-A)B^{-1}\}^{\ell} + \sum_{r=1}^{\ell-1}B^{-1}\{(B-A)B^{-1}\}^{r}. & \qedhere
\end{align*}
\end{proof}

\section{Bahadur representation}\label{sec:bahadur}

We prove
$
    \sqrt{n}(\hat{f}-f_{\lambda})
$
is well approximated by
$
\sqrt{n} T_{\lambda}^{-1}\{(\hat{T}-T)(f_0-f_{\lambda})+\mathbb{E}_n(k_{X_i}\ep_i)\}=\sqrt{n}\mathbb{E}_n(U_i).
$
To lighten notation, we write $T_{\lambda}=T+\lambda$, $T_i=k_{X_i}\otimes k_{X_i}^*$, and $\N(\lambda)=\tr(T_{\lambda}^{-2}T)$.

\paragraph{\textcolor{black}{Overview.}} We begin by establishing two high-probability bounds using randomness in the data. We then show that these bounds imply a bound on the gap between the quantities above.

\paragraph{\textcolor{black}{High probability events.}}

\begin{lemma}\label{lemma:numerator-concentration}
$
    \norm{\frac 1 n \sum_{i=1}^n T_{\lambda}^{-1}\ep_i k_{X_i}} \le 2\bar{\sigma} \ln(2/\eta)\left\{\sqrt{\frac{\N(\lambda)}{n}} \vee \frac{2\kappa}{n\lambda}\right\}
$ w.p. $1-\eta$.
\end{lemma}

\begin{proof}
Note that $\bb{E}(T_{\lambda}^{-1}\ep_i k_{X_i}) = \bb{E}\{T_{\lambda}^{-1} k_{X_i}\bb{E}(\ep_i|X_i)\} = 0$. It therefore suffices to show $\norm{T_{\lambda}^{-1}\ep_i k_{X_i}} \le \norm{T_{\lambda}^{-1}}_{\op}\norm{\ep_i k_{X_i}} \le \frac{\kappa \bar{\sigma}}{\lambda}$ and
\begin{align*}
    \bb{E}\norm{T_{\lambda}^{-1}\ep_i k_{X_i}}^2 &\le \bar{\sigma}^2 \bb{E}\norm{T_{\lambda}^{-1} k_{X_i}}^2
    = \bar{\sigma}^2\sum_{s=1}^\infty \frac{\bb{E}\bk{k_{X_i}}{e_s}^2}{(\nu_s + \lambda)^2}
    = \bar{\sigma}^2 \sum_{s=1}^\infty \frac{\nu_s}{(\nu_s + \lambda)^2}
    = \bar{\sigma}^2\N(\lambda).
    \end{align*}
Plugging these into Bernstein inequality (Lemma~\ref{lemma:c_dv}) gives the result.
\end{proof}

\begin{lemma}\label{lemma:denominator-concentration}
$
    \norm{\frac 1 n \sum_{i=1}^n T_{\lambda}^{-1}(T_i - T)}_{\HS} \le 2\kappa \ln(2/\eta)\left\{\sqrt{\frac{\N(\lambda)}{n}} \vee \frac{4\kappa}{n\lambda}\right\}
$ w.p. $1-\eta$.
\end{lemma}

\begin{proof}
Note that $\bb{E}\{T_{\lambda}^{-1}(T_i - T)\} = T_{\lambda}^{-1} \bb{E}(T_i - T) = 0$. Also, $\norm{T_{\lambda}^{-1}(T_i - T)}_{\HS} \le \norm{T_{\lambda}^{-1}}_{\op}\norm{T_i - T}_{\HS} \le \frac{2\kappa^2}{\lambda}.$
Since $T$, $T_i$ and $T_{\lambda}^{-1}$ are self-adjoint,
\begin{align*}
    &\bb{E}\norm{T_{\lambda}^{-1}(T_i - T)}^2_{\HS}
    = \bb{E} \tr\{T_{\lambda}^{-1}(T_i - T)^2T_{\lambda}^{-1}\}
    = \bb{E} \tr(T_{\lambda}^{-1}T_i^2T_{\lambda}^{-1}-2T_{\lambda}^{-1}T_iTT_{\lambda}^{-1}+T_{\lambda}^{-1}T^2T_{\lambda}^{-1}) \\
     &= \bb{E} \tr(T_{\lambda}^{-1}T_i^2T_{\lambda}^{-1}-T_{\lambda}^{-1}T^2T_{\lambda}^{-1})
     \leq \bb{E} \tr(T_{\lambda}^{-1}T_i^2T_{\lambda}^{-1})
     = \bb{E} \tr(T_i^2T_{\lambda}^{-2})
   \leq \kappa^2\tr(TT_{\lambda}^{-2})
   =\kappa^2\N(\lambda).
    \end{align*}
Plugging these into Bernstein inequality (Lemma~\ref{lemma:c_dv}) gives the result.
\end{proof}

\paragraph{\textcolor{black}{Main results.}}

\begin{lemma}[Linearization]\label{lemma:first-order-denom}Suppose $\|T_{\lambda}^{-1}(\hat T - T)\|_{\HS} \le \d < 1$.
Then for all $k\geq 1$
$$
(\hat{T}_{\lambda}^{-1}-T_{\lambda}^{-1})u=A_1u+A_2 T_{\lambda}^{-1}u+A_3 T^{-1}_{\lambda} u,\quad
\|A_1\|_{\HS}\leq \frac{\delta^k}{\lambda},\quad \|A_2\|_{\HS}\leq \delta,\quad \|A_3\|_{\HS}\leq  \frac{\delta^2}{1-\delta}.
$$
If in addition $\delta\leq1/2$ then
$
\|(\hat{T}^{-1}_{\lambda}-T^{-1}_{\lambda})u\|\leq 2\delta \|T^{-1}_{\lambda}u\|.
$
\end{lemma}
\begin{proof}
By the iterated resolvent identity (Lemma~\ref{lemma:powers}) with $A =\hat{T}_{\lambda}$ and $B=T_{\lambda}$, we obtain
\begin{align*}
    \hat{T}_{\lambda}^{-1}-T_{\lambda}^{-1}
    &=\hat{T}_{\lambda}^{-1}\{(T-\hat{T})T_{\lambda}^{-1}\}^k+ \sum_{r=1}^{k-1}T_{\lambda}^{-1}\{(T-\hat{T})T_{\lambda}^{-1}\}^{r}\\
    &=\hat{T}_{\lambda}^{-1}\{(T-\hat{T})T_{\lambda}^{-1}\}^k
    +
    T_{\lambda}^{-1}(T-\hat{T})T_{\lambda}^{-1}
    + \sum_{r=2}^{k-1}T_{\lambda}^{-1}\{(T-\hat{T})T_{\lambda}^{-1}\}^{r} \\
    &=\hat{T}_{\lambda}^{-1}\{(T-\hat{T})T_{\lambda}^{-1}\}^k
    +
    T_{\lambda}^{-1}(T-\hat{T})T_{\lambda}^{-1}
    + \sum_{r=2}^{k-1} \left\{T_{\lambda}^{-1}(T-\hat{T})\right\}^{r}T_{\lambda}^{-1}\\
    &= A_1+A_2 T_{\lambda}^{-1} +A_3T_{\lambda}^{-1}.
\end{align*}
where in the third equality, we use $A(BA)^r=(AB)^rA$.
Since $\|(T-\hat{T})T_{\lambda}^{-1}\|_{\HS}\leq\delta<1$,
$
\|A_1\|_{\HS}=\|\hat{T}_{\lambda}^{-1}\{(T-\hat{T})T_{\lambda}^{-1}\}^k\|_{\HS}\leq \frac{\delta^k}{\lambda}$ and $\|A_2\|_{\HS}=\|T_{\lambda}^{-1}(T-\hat{T})\|_{\HS}\leq \delta.
$
Moreover, by the triangle inequality
\begin{align*}
    \|A_3\|_{\HS}
    &\leq \sum_{r=2}^{k-1}\left\|T_{\lambda}^{-1}(T-\hat{T})\right\|_{\HS}^{r}
    \leq \sum_{r=2}^{k-1} \d^r \le \sum_{r=2}^{\infty} \d^r =\frac{\delta^2}{1-\delta}.
\end{align*}
We use the first result to prove the second. By the triangle inequality,
\begin{align*}
    &\|(\hat{T}^{-1}_{\lambda}-T^{-1}_{\lambda})u\|
    \leq \|A_1\|_{\HS}\cdot \|u\|+\|A_2\|_{\HS}\cdot \|T_{\lambda}^{-1}u\|+\|A_3\|_{\HS}\cdot \|T_{\lambda}^{-1}u\| \\
    &\leq \frac{\delta^k}{\lambda}\|u\|+\left(\delta+\frac{\delta^2}{1-\delta}\right) \|T_{\lambda}^{-1}u\|
    =\frac{\delta^k}{\lambda}\|u\|+\frac{\delta}{1-\delta} \|T_{\lambda}^{-1}u\|.
\end{align*}
The first term vanishes as $k \uparrow \infty$. When $\delta\leq 1/2$, $\delta(1-\delta)^{-1}\leq 2\delta$.
\end{proof}

\begin{proposition}[Abstract Bahadur representation]\label{prop:bahadur-expansion}
Suppose $\|T_{\lambda}^{-1}(\hat T - T)\|_{\HS} \le \d \le \f 1 2$, and $\norm{T_{\lambda}^{-1}\bb{E}_n(k_{X_i}\ep_i)} \le \gamma$. Then
$\hat f - f_\lambda = T_{\lambda}^{-1}\bb{E}_n(k_{X_i}\ep_i) + T_{\lambda}^{-1}(\hat T - T)(f_0 - f_\lambda) + u$
for some $u$ with
$
\norm{u} \leq  2\delta  \left(\gamma+\delta\|f_0-f_{\lambda}\|\right).
$
\end{proposition}

\begin{proof}
We proceed in steps.
\begin{enumerate}
    \item Decomposition. Write
\begin{align*}
&\hat{f}-f_{\lambda}
    =\hat{T}^{-1}_{\lambda} \mathbb{E}_n(k_{X_i}Y_i)-T^{-1}_{\lambda}Tf_0 =\hat{T}^{-1}_{\lambda} \hat{T}f_0+\hat{T}^{-1}_{\lambda} \mathbb{E}_n(k_{X_i}\ep_i)-T^{-1}_{\lambda}Tf_0 \\
    &=(\hat{T}^{-1}_{\lambda}-T^{-1}_{\lambda})\mathbb{E}_n(k_{X_i}\ep_i)
    +T^{-1}_{\lambda} \mathbb{E}_n(k_{X_i}\ep_i)
     +(\hat{T}^{-1}_{\lambda}-T^{-1}_{\lambda})\hat{T}f_0
    +T_{\lambda}^{-1}\hat{T}f_0
    -T^{-1}_{\lambda}Tf_0 \\
    &=(\hat{T}^{-1}_{\lambda}-T^{-1}_{\lambda})\mathbb{E}_n(k_{X_i}\ep_i)
    +T^{-1}_{\lambda} \mathbb{E}_n(k_{X_i}\ep_i)
    \\ & \qquad +(\hat{T}^{-1}_{\lambda}-T^{-1}_{\lambda})(\hat{T}-T)f_0
    +(\hat{T}^{-1}_{\lambda}-T^{-1}_{\lambda})Tf_0
    +T_{\lambda}^{-1}\hat{T}f_0
    -T^{-1}_{\lambda}Tf_0 \\
    &=(\hat{T}^{-1}_{\lambda}-T^{-1}_{\lambda})\{\mathbb{E}_n(k_{X_i}\ep_i)
     +(\hat{T}-T)f_0\}+T^{-1}_{\lambda} \mathbb{E}_n(k_{X_i}\ep_i)
      -\hat{T}^{-1}_{\lambda}(\hat{T}-T)f_{\lambda}+T_{\lambda}^{-1}(\hat{T}-T)f_0,
\end{align*}
where in the last line we use the resolvent identity (base case of Lemma~\ref{lemma:powers}) to write
\begin{align*}
    (\hat{T}^{-1}_{\lambda}-T^{-1}_{\lambda})Tf_0
    =\hat{T}^{-1}_{\lambda}(T-\hat{T})T_{\lambda}^{-1}Tf_0
    =\hat{T}^{-1}_{\lambda}(T-\hat{T})f_{\lambda}
    =-\hat{T}^{-1}_{\lambda}(\hat{T}-T)f_{\lambda}.
\end{align*}
Focusing on the final two terms
\begin{align*}
    &-\hat{T}^{-1}_{\lambda}(\hat{T}-T)f_{\lambda}+T_{\lambda}^{-1}(\hat{T}-T)f_0\pm T^{-1}_{\lambda}(\hat{T}-T)f_{\lambda}  \\
&=(T^{-1}_{\lambda}-\hat{T}^{-1}_{\lambda})(\hat{T}-T)f_{\lambda}+T_{\lambda}^{-1}(\hat{T}-T)(f_0-f_{\lambda}) \\
&=(\hat{T}^{-1}_{\lambda}-T^{-1}_{\lambda})(\hat{T}-T)(-f_{\lambda})+T_{\lambda}^{-1}(\hat{T}-T)(f_0-f_{\lambda}).
\end{align*}
Therefore $\hat{f}-f_{\lambda}=(\hat{T}^{-1}_{\lambda}-T^{-1}_{\lambda})S_n +T^{-1}_{\lambda}S_n$ for $S_n =\mathbb{E}_n(k_{X_i}\ep_i)
    +(\hat{T}-T)(f_0-f_{\lambda})$.

    \item What remains is to control the first term:
$
(\hat{T}^{-1}_{\lambda}-T^{-1}_{\lambda})S_n.
$
    By Lemma~\ref{lemma:first-order-denom},
$$
    \|(\hat{T}^{-1}_{\lambda}-T^{-1}_{\lambda})S_n\|
    \leq 2\delta \|T_{\lambda}^{-1}S_n\|
    \leq  2\delta \left(\gamma+\delta\|f_0-f_{\lambda}\|\right). \qedhere
$$
\end{enumerate}
\end{proof}

\begin{theorem}[Bahadur representation]\label{thm:bahadur}
If $n \ge 16\kappa^2\ln(4/\eta)^2\{\N(\lambda) \vee \lambda^{-1}\}$ then w.p. $1-\eta$,
$\hat f - f_\lambda = \bb{E}_n(T_{\lambda}^{-1}k_{X_i}\ep_i) + T_{\lambda}^{-1}(\hat T - T)(f_0 - f_\lambda) + u$ for some $u$ with
$\norm{u} \le 8(\kappa^2\norm{f_0-f_{\lambda}} + \bar{\sigma}\kappa)\ln(4/\eta)^2\left\{\sqrt{\frac{\N(\lambda)}{n}} \vee \frac{4\kappa}{n\lambda}\right\}^2.$
\end{theorem}

\begin{proof}
By combining Lemmas \ref{lemma:numerator-concentration} and \ref{lemma:denominator-concentration} with a union bound, w.p. $1-2\eta$, we have $\norm{T_{\lambda}^{-1}\bb{E}_n(k_{X_i}\ep_i)} \le \gamma$ and $\|T_{\lambda}^{-1}(\hat T - T)\|_{\HS} \le \delta$, with
$$
\gamma=2\bar{\sigma} \ln(2/\eta)\left\{\sqrt{\frac{\N(\lambda)}{n}} \vee \frac{2\kappa}{n\lambda}\right\},
\quad \delta=2\kappa \ln(2/\eta)\left\{\sqrt{\frac{\N(\lambda)}{n}} \vee \frac{4\kappa}{n\lambda}\right\}.
$$
On the event that these inequalities hold, when $\d < \f 1 2$, we have from Proposition \ref{prop:bahadur-expansion} that
\[\norm{u} \le 2\delta  \left(\gamma+\delta\|f_0-f_{\lambda}\|\right) \leq 8(\bar{\sigma}\kappa+\kappa^2\norm{f_{\lambda}-f_0})\ln(2/\eta)^2\left\{\sqrt{\frac{\N(\lambda)}{n}} \vee \frac{4\kappa}{n\lambda}\right\}^2.\]
Moreover, the condition $\d \le \f 1 2$ can be seen to hold whenever
$
2\kappa \ln(2/\eta)\sqrt{\frac{\N(\lambda)}{n}} <\frac{1}{2} \iff 16\kappa^2  \ln(2/\eta)^2 \N(\lambda) <n
$
and
$
2\kappa \ln(2/\eta)\frac{2\kappa}{n\lambda} <\frac{1}{2} \iff 16\kappa^2 \ln(2/\eta)\frac{1}{\lambda} < n.
$
\end{proof}

\section{Feasible bootstrap}\label{sec:bahadur2}

We bound $\snorm{Z_{\mathfrak{B}}-\mathfrak{B}}=
\snorm{\frac{1}{n}\sum_{i=1}^n\sum_{j=1}^n h_{ij} \frac{V_i-V_j}{\sqrt{2}} - \frac{1}{n}\sum_{i=1}^n\sum_{j=1}^n h_{ij} \frac{\hat{V}_i-\hat{V}_j}{\sqrt{2}}}$. To lighten notation, we write $T_{\lambda}=T+\lambda$, $T_i=k_{X_i}\otimes k_{X_i}^*$, and $\N(\lambda)=\tr(T_{\lambda}^{-2}T)$.

\paragraph{\textcolor{black}{Overview.}} We again begin by showing two high-probability bounds. We then decompose $Z_{\mathfrak{B}}-\mathfrak{B}$ as $\Delta_1 + \Delta_2$. Finally, we show our high-probability bounds control $\Delta_1$ and $\Delta_2$.

\paragraph{\textcolor{black}{High probability events.}} %

\begin{lemma}\label{lemma:high1}
    $
\mathbb{E}_n\|  T_{\lambda}^{-1} \ep_i k_{X_i}\|^2 \leq \bar{\sigma}^2 \N(\lambda)+4\bar{\sigma}^2\kappa^2\ln(2/\eta)\left\{\frac{1}{n\lambda^2}+\sqrt{\frac{\N(\lambda)}{n\lambda^2}}\right\}
$ w.p. $1-\eta$.
\end{lemma}

\begin{proof}
    Write $\xi_i=\|  T_{\lambda}^{-1} \ep_i k_{X_i}\|^2\geq 0$. Clearly $\xi_i\leq \frac{\bar{\sigma}^2\kappa^2}{\lambda^2}$. Moreover
    \begin{align*}
   \mathbb{E}\xi_i
   &=\mathbb{E}\|  T_{\lambda}^{-1} \ep_i k_{X_i}\|^2
   = \int
    \ep_i^2 \langle T_{\lambda}^{-1} k_{X_i},T_{\lambda}^{-1} k_{X_i} \rangle
    \mathrm{d}\mathbb{P}
    \leq \bar{\sigma}^2 \int
    \langle k_{X_i},T_{\lambda}^{-2} k_{X_i} \rangle  \mathrm{d}\mathbb{P}  \\
    &=\bar{\sigma}^2 \int
    \tr (T_{\lambda}^{-2} T_i)  \mathrm{d}\mathbb{P}
    = \bar{\sigma}^2\tr (T_{\lambda}^{-2} T)
    =\bar{\sigma}^2 \N(\lambda).
\end{align*}
These imply $\mathbb{E}\xi_i^2\leq \frac{\bar{\sigma}^4\kappa^2}{\lambda^2} \N(\lambda)$. Therefore by Bernstein inequality (Lemma~\ref{lemma:c_dv}), w.p. $1-\eta$
$$
\mathbb{E}_n\xi_i \leq \mathbb{E}\xi_i+2\ln(2/\eta)\left\{\frac{2\bar{\sigma}^2\kappa^2}{n\lambda^2}+\sqrt{\frac{\bar{\sigma}^4\kappa^2\N(\lambda)}{n\lambda^2}}\right\}.\qedhere
$$
\end{proof}

\begin{lemma}\label{lemma:high2}
    $
    \mathbb{E}_n\|  T_{\lambda}^{-1} T_i\|_{\HS}^2 \leq \kappa^2 \N(\lambda)+4\kappa^4\ln(2/\eta)\left\{\frac{1}{n\lambda^2}+\sqrt{\frac{\N(\lambda)}{n\lambda^2}}\right\}
    $ w.p. $1-\eta$.
\end{lemma}

\begin{proof}
    Write $\xi_i=\|  T_{\lambda}^{-1} T_i\|_{\HS}^2\geq 0$. Clearly $\xi_i\leq \frac{\kappa^4}{\lambda^2}$. Moreover
    \begin{align*}
   \mathbb{E}\xi_i
   &=\mathbb{E}\|  T_{\lambda}^{-1} T_i\|_{\HS}^2
   = \int
    \tr (T_i T_{\lambda}^{-2} T_i)
    \mathrm{d}\mathbb{P}
    \leq \kappa^2 \int
    \tr (T_{\lambda}^{-2} T_i)  \mathrm{d}\mathbb{P}
    = \kappa^2\tr (T_{\lambda}^{-2} T)
    =\kappa^2\N(\lambda).
\end{align*}
These imply $\mathbb{E}\xi_i^2\leq \frac{\kappa^6}{\lambda^2} \N(\lambda)$. Therefore by Bernstein inequality (Lemma~\ref{lemma:c_dv}), w.p. $1-\eta$
$$
\mathbb{E}_n\xi_i \leq \mathbb{E}\xi_i+2\ln(2/\eta)\left\{\frac{2\kappa^4}{n\lambda^2}+\sqrt{\frac{\kappa^6\N(\lambda)}{n\lambda^2}}\right\}.\qedhere
$$
\end{proof}
\paragraph{\textcolor{black}{Decomposition.}}
Define $\ep^{\lambda}_i = Y_i-f_{\lambda}(X_i)$ and recall that
\begin{align*}
    V_i&=T_{\lambda}^{-1}\{(k_{X_i} \otimes k_{X_i}^*)(f_0-f_{\lambda}) + \ep_i k_{X_i}\}=T_{\lambda}^{-1}\{Y_i-f_{\lambda}(X_i)\} k_{X_i}=T_{\lambda}^{-1}\ep^{\lambda}_i k_{X_i}, \\
    \hat{V_i}&=\hat{T}_{\lambda}^{-1}\{(k_{X_i} \otimes k_{X_i}^*)(f_0-\hat{f}) + \ep_ik_{X_i}\}=\hat{T}_{\lambda}^{-1}\{Y_i-\hat{f}(X_i)\} k_{X_i}=\hat{T}_{\lambda}^{-1}\hat{\ep}_i k_{X_i}.
\end{align*}
To lighten notation, let
$
w_{ij}=\frac{1}{\sqrt{2}}(\ep^{\lambda}_i k_{X_i}-\ep^{\lambda}_j k_{X_j})$ and $
\hat{w}_{ij}=\frac{1}{\sqrt{2}}(\hat{\ep}_i k_{X_i}-\hat{\ep}_j k_{X_j}).
$

\begin{lemma}\label{lemma:decomp_bahadur2}
$\mathfrak{B}-Z_{\mathfrak{B}}=\Delta_1+\Delta_2$ where
$
    \Delta_1=(\hat{T}_{\lambda}^{-1}-T_{\lambda}^{-1})\left(\frac{1}{n} \sum_{i=1}^n \sum_{j=1}^n \hat{w}_{ij} h_{ij} \right)$ and $
    \Delta_2=T_{\lambda}^{-1}\left\{\frac{1}{n} \sum_{i=1}^n \sum_{j=1}^n (\hat{w}_{ij}-w_{ij}) h_{ij} \right\}.$
\end{lemma}

\begin{proof}
Write $\mathfrak{B}-Z_{\mathfrak{B}}$ as
\begin{align*}
    &\frac{1}{n}\sum_{i=1}^n\sum_{j=1}^n h_{ij} \frac{\hat{V}_i-\hat{V}_j}{\sqrt{2}}-\frac{1}{n}\sum_{i=1}^n\sum_{j=1}^n h_{ij} \frac{V_i-V_j}{\sqrt{2}}
    =\frac{1}{n}\sum_{i=1}^n\sum_{j=1}^n h_{ij} \hat{T}_{\lambda}^{-1}\hat{w}_{ij}-\frac{1}{n}\sum_{i=1}^n\sum_{j=1}^n h_{ij} T_{\lambda}^{-1}w_{ij} \\
    &=\frac{1}{n}\sum_{i=1}^n\sum_{j=1}^n h_{ij}(\hat{T}_{\lambda}^{-1}\hat{w}_{ij}-T_{\lambda}^{-1}w_{ij}\pm T_{\lambda}^{-1}\hat{w}_{ij} ). &\qedhere
\end{align*}
\end{proof}

\paragraph{\textcolor{black}{First term.}} We initially focus on
$
\Delta_1=(\hat{T}_{\lambda}^{-1}-T_{\lambda}^{-1})u$ where $u=\frac{1}{n} \sum_{i=1}^n \sum_{j=1}^n \hat{w}_{ij} h_{ij}.
$

\begin{lemma}[First term]\label{lemma:delta1}
If $\mathbb{E}_n\|  T_{\lambda}^{-1} T_i\|_{\HS}^2  \leq \delta'$ and $\mathbb{E}_n\|  T_{\lambda}^{-1} \ep_i k_{X_i}\|^2 \leq \gamma'$, then conditional on data,  w.p. $1-\eta$,
$
\norm{T_{\lambda}^{-1} u}\leq \{1 + \sqrt{2\log(1/\eta)}\}\sqrt{4(\gamma'+\|\hat{f}-f_0\|^2\delta')}.
$
If additionally $\|T_{\lambda}^{-1}(\hat T - T)\|_{\HS} \le \d \le \f 1 2$ then also $\|\Delta_1\|\leq \{1 + \sqrt{2\log(1/\eta)}\}\cdot 2\delta \cdot \sqrt{4(\gamma'+\|\hat{f}-f_0\|^2\delta')}$.
\end{lemma}

\begin{proof}
Since $T_{\lambda}^{-1} u$ is Gaussian
conditional upon data, by Borell's inequality (Lemma~\ref{lemma:gaussian-concentration-trace})
w.p. $1-\eta$,
$\norm{T_{\lambda}^{-1}u} \le
\{1 + \sqrt{2\log(1/\eta)}\}\sqrt{\bb{E}\norm{T_{\lambda}^{-1} u}^2}.$
Note that  $\bb{E}_h\norm{T_{\lambda}^{-1} u}^2=\frac{1}{n^2}\sum_{i=1}^n\sum_{j=1}^n  \| T_{\lambda}^{-1} \hat{w}_{ij}\|^2$. Within each term,
\begin{align*}
\| T_{\lambda}^{-1} \hat{w}_{ij}\|^2
    &=\frac{1}{2} \left\| T_{\lambda}^{-1} (\hat{\ep}_i k_{X_i}-\hat{\ep}_j k_{X_j})\right\|^2
    \leq \| T_{\lambda}^{-1} \hat{\ep}_i k_{X_i}\|^2+\| T_{\lambda}^{-1} \hat{\ep}_j k_{X_j}\|^2  \\
    \hat{\ep}_i k_{X_i}
    &=\{Y_i-\hat{f}(X_i)\}k_{X_i}
    =\{\ep_i+f_0(X_i)-\hat{f}(X_i)\}k_{X_i}
    =\ep_i k_{X_i}+T_i(f_0-\hat{f}).
\end{align*}
Therefore by the triangle inequality and parallelogram law
\begin{align*}
    &\frac{1}{n^2}\sum_{i=1}^n\sum_{j=1}^n  \| T_{\lambda}^{-1} \hat{w}_{ij}\|^2
\leq \frac{2}{n} \sum_{i=1}^n\|  T_{\lambda}^{-1} \hat{\ep}_i k_{X_i}\|^2
\leq \frac{4}{n} \sum_{i=1}^n\|  T_{\lambda}^{-1} \ep_i k_{X_i}\|^2+\frac{4}{n} \sum_{i=1}^n\|  T_{\lambda}^{-1} T_i(f_0-\hat{f})\|^2 \\
&\leq 4\left(\mathbb{E}_n\|  T_{\lambda}^{-1} \ep_i k_{X_i}\|^2+\|f_0-\hat{f}\|^2 \mathbb{E}_n\|  T_{\lambda}^{-1} T_i\|_{\HS}^2  \right)
\leq 4(\gamma'+\|\hat{f}-f_0\|^2\delta').
\end{align*}
In summary, $\norm{T_{\lambda}^{-1}u}\leq \{1 + \sqrt{2\log(1/\eta)}\}\sqrt{4(\gamma'+\|\hat{f}-f_0\|^2\delta')}$.
The second result follows from Lemma~\ref{lemma:first-order-denom}:
  $\|\Delta_1\| = \|(\hat{T}_{\lambda}^{-1}-T_{\lambda}^{-1})u\|\leq 2\delta \norm{T_{\lambda}^{-1}u}$.
\end{proof}

\paragraph{\textcolor{black}{Second term.}} Next, we turn to
$
\Delta_2=T_{\lambda}^{-1}\left\{\frac{1}{n} \sum_{i=1}^n \sum_{j=1}^n (\hat{w}_{ij}-w_{ij}) h_{ij} \right\}.
$

\begin{lemma}[Second term]\label{lemma:delta2}
If $\mathbb{E}_n\|  T_{\lambda}^{-1} T_i\|_{\HS}^2  \leq \delta'$, then conditional on data,  w.p. $1-\eta$,
$\|\Delta_2\|\leq \{1 + \sqrt{2\log(1/\eta)}\}\cdot \|\hat{f}-f_{\lambda}\|\cdot \sqrt{2\delta'}$
\end{lemma}

\begin{proof}
To begin, we show $ \|\Delta_2\|\leq \left\|T_{\lambda}^{-1}u\right\|_{\HS}\cdot \|\hat{f}-f_{\lambda}\|$ where $u=\frac{1}{n} \sum_{i=1}^n \sum_{j=1}^n \frac{1}{\sqrt{2}}(T_i-T_j) h_{ij}$. Observe that $\hat{w}_{ij}-w_{ij}$ equals
\begin{align*}
 &\frac{1}{\sqrt{2}}\left[\{Y_i-\hat{f}(X_i)\}k_{X_i}-\{Y_j-\hat{f}(X_j)\}k_{X_j}\right] -
    \frac{1}{\sqrt{2}}\left[\{Y_i-f_{\lambda}(X_i)\}k_{X_i}-\{Y_j-f_{\lambda}(X_j)\}k_{X_j}\right] \\
    &=\frac{1}{\sqrt{2}}\left[\{f_{\lambda}(X_i)-\hat{f}(X_i)\}k_{X_i}-\{f_{\lambda}(X_j)-\hat{f}(X_j)\}k_{X_j}\right]
    =\frac{1}{\sqrt{2}}\left[T_i\{f_{\lambda}-\hat{f}\}-T_j\{f_{\lambda}-\hat{f}\}\right].
\end{align*}
Therefore $\Delta_2=T_{\lambda}^{-1} \left\{\frac{1}{n} \sum_{i=1}^n \sum_{j=1}^n \frac{1}{\sqrt{2}}(T_i-T_j) (f_{\lambda}-\hat{f}) h_{ij} \right\}$.

    Since $T_{\lambda}^{-1}u$ is Gaussian conditional upon data, by Borell's inequality (Lemma~\ref{lemma:gaussian-concentration-trace}), w.p. $1-\eta$,
    $
    \|T_{\lambda}^{-1} u\|_{\HS} \leq \{1 + \sqrt{2\log(1/\eta)}\}\sqrt{\bb{E}\norm{T_{\lambda}^{-1} u}_{\HS}^2}.
    $ Note that $\bb{E}_h\norm{T_{\lambda}^{-1} u}_{\HS}^2=\frac{1}{n^2}\sum_{i=1}^n\sum_{j=1}^n  \| T_{\lambda}^{-1} \frac{1}{\sqrt{2}}(T_i-T_j)\|_{\HS}^2$. Within each term,
    $$
        \| T_{\lambda}^{-1} 2^{-1/2}(T_i-T_j)\|_{\HS}^2=\frac{1}{2}\| T_{\lambda}^{-1}(T_i-T_j)\|_{\HS}^2\leq \| T_{\lambda}^{-1}T_i\|^2+\| T_{\lambda}^{-1}T_j\|_{\HS}^2.
 $$
 Therefore by the triangle inequality and the parallelogram law
 \begin{align*}
     &\frac{1}{n^2}\sum_{i=1}^n\sum_{j=1}^n  \| T_{\lambda}^{-1} 2^{-1/2}(T_i-T_j)\|_{\HS}^2
     \leq \frac{2}{n}\sum_{i=1}^n  \| T_{\lambda}^{-1}T_i\|_{\HS}^2=2\mathbb{E}_n\| T_{\lambda}^{-1}T_i\|_{\HS}^2\leq 2\delta'.
 \end{align*}
 In summary, $  \|T_{\lambda}^{-1} u\|_{\HS} \leq \{1 + \sqrt{2\log(1/\eta)}\}\sqrt{2\delta'}.$
\end{proof}

\paragraph{\textcolor{black}{Main results.}} Before proving the main results, we show for completeness that Proposition~\ref{prop:bahadur-expansion}  implies consistency of KRR. This result mirrors ~\citet[Theorem 16]{fischer2020sobolev}.

\begin{lemma}[Error bound for KRR]\label{lemma:delta2_latter}
Suppose $\|T_{\lambda}^{-1}(\hat T - T)\|_{\HS} \le \d \le \f 1 2$ and $\norm{T_{\lambda}^{-1}\bb{E}_n(k_{X_i}\ep_i)} \le \gamma$. Then
$\snorm{\hat f - f_\lambda} \le 2(\gamma + \d\snorm{f_0})$ and
$\snorm{\hat f - f_0} \le 2(\gamma +\snorm{f_0})$.
\end{lemma}

\begin{proof}
 By Proposition \ref{prop:bahadur-expansion},
$
\hat f - f_\lambda = T_{\lambda}^{-1}\bb{E}_n(k_{X_i}\ep_i) + T_{\lambda}^{-1}(\hat T - T)(f_0 - f_\lambda) + u
$
for some $u$ with
$
\norm{u} \leq 2\delta  \left(\gamma+\delta\|f_0-f_{\lambda}\|\right).
$
By the triangle inequality,
\[
 \snorm{\hat f - f_\lambda} \le \gamma + \d\snorm{f_0 - f_\lambda} + 2\d(\gamma + \d\snorm{f_0 - f_\lambda}).
\]
Using $f_0 - f_\lambda = (I - T_{\lambda}^{-1}T)f_0$,
$0 \preceq (I - T_{\lambda}^{-1}T) \preceq I,$ and $\d \le \f 1 2$,
\[
 \snorm{\hat f - f_\lambda} \le \gamma + \d\snorm{f_0} + 2\d(\gamma + \d\snorm{f_0}) \le 2(\gamma + \d\snorm{f_0}).
\]
We use the first result to prove the second. By the triangle inequality, we bound $\|\hat{f}-f_0\|$  by
$$
 \|\hat{f}-f_{\lambda}\|+\|f_{\lambda}-f_0\|
\leq \|\hat{f}-f_{\lambda}\|+\|f_0\|
\leq 2(\gamma +\delta \snorm{f_0})+\|f_0\|
\leq 2(\gamma+\|f_0\|). \qedhere
$$
\end{proof}

\begin{theorem}[Feasible bootstrap]\label{thm:bahadur2}
 If $n \ge 16\kappa^2\ln(12/\eta)^2\{\N(\lambda) \vee \lambda^{-1}\}$ then w.p. $1-\eta$
\begin{align*}
    &\left\|\mathfrak{B}-Z_{\mathfrak{B}} \right\|
    \leq 180\kappa^2(\bar{\sigma}+\kappa\|f_0\|) \ln(12/\eta)^2\left\{\sqrt{\frac{\N(\lambda)}{n}} \vee \frac{4\kappa}{n\lambda}\right\} \cdot \left[ \sqrt{\N(\lambda)}+\left\{\frac{1}{\sqrt{n}\lambda} \vee \frac{ \N(\lambda)^{1/4}}{n^{1/4}\lambda^{1/2}}\right\}  \right].
\end{align*}
\end{theorem}

\begin{proof}
We proceed in steps. First we collect high-probability bounds proved in Appendices \ref{sec:bahadur} and \ref{sec:bahadur2}. We then show that when combined, these imply bounds on $\Delta_1$ and $\Delta_2$. We conclude by applying a union bound and simplifying.
\begin{enumerate}
    \item As argued in Theorem~\ref{thm:bahadur}, w.p. $1-2\eta$, we have $\norm{T_{\lambda}^{-1}\bb{E}_n(k_{X_i}\ep_i)} \le \gamma$ and $\|T_{\lambda}^{-1}(\hat T - T)\|_{\HS} \le \delta$, with
$$
\gamma=2\bar{\sigma} \ln(2/\eta)\left\{\sqrt{\frac{\N(\lambda)}{n}} \vee \frac{2\kappa}{n\lambda}\right\},
\quad \delta=2\kappa \ln(2/\eta)\left\{\sqrt{\frac{\N(\lambda)}{n}} \vee \frac{4\kappa}{n\lambda}\right\},
$$
and the condition on $n$ suffices for $\delta\leq 1/2$. By Lemmas~\ref{lemma:high1} and~\ref{lemma:high2} and a union bound, w.p. $1-2\eta$, we have $\mathbb{E}_n\|  T_{\lambda}^{-1} \ep_i k_{X_i}\|^2 \leq \gamma'$ and $\mathbb{E}_n\|  T_{\lambda}^{-1} T_i\|_{\HS}^2  \leq \delta'$, with
$$
\gamma'=\bar{\sigma}^2 \N(\lambda)+4\bar{\sigma}^2\kappa^2\ln(2/\eta)\left\{\frac{1}{n\lambda^2}+\sqrt{\frac{\N(\lambda)}{n\lambda^2}}\right\},\quad
\delta'=\kappa^2 \N(\lambda)+4\kappa^4\ln(2/\eta)\left\{\frac{1}{n\lambda^2}+\sqrt{\frac{\N(\lambda)}{n\lambda^2}}\right\}.
$$
    \item  By Lemmas~\ref{lemma:decomp_bahadur2},~\ref{lemma:delta1}, and~\ref{lemma:delta2}, w.p. $1-2\eta$ conditional upon data,
    \begin{align*}
        \|\mathfrak{B}-Z_{\mathfrak{B}}\|
        &\leq \|\Delta_1\|+\|\Delta_2\|
        \leq \{1 + \sqrt{2\log(1/\eta)}\}\cdot
        \{ 2\delta \cdot \sqrt{4(\gamma'+\|\hat{f}-f_0\|^2\delta')}+\|\hat{f}-f_{\lambda}\|\cdot \sqrt{2\delta'}\}.
    \end{align*}
   Focusing on the latter factor's first term, by Lemma~\ref{lemma:delta2_latter},
\begin{align*}
    &2\delta \cdot \sqrt{4(\gamma'+\|\hat{f}-f_0\|^2\delta')}
\leq 4\delta (\sqrt{\gamma'}+\|\hat{f}-f_0\|\sqrt{\delta'})
\leq 4\delta \{\sqrt{\gamma'}+2(\gamma +\snorm{f_0})\sqrt{\delta'}\}\\
&\leq 4\delta \sqrt{\gamma'} +4\gamma\sqrt{\delta'}+8\delta \|f_0\|\sqrt{\delta'}.
\end{align*}
   Focusing on the latter factor's second term, by Lemma~\ref{lemma:delta2_latter},
   $$
   \|\hat{f}-f_{\lambda}\|\cdot \sqrt{2\delta'}
   \leq 2(\gamma + \d\snorm{f_0}) \cdot \sqrt{2\delta'}
   \leq 3\gamma \sqrt{\delta'}+3\d\snorm{f_0}\sqrt{\delta'}.
   $$
   Therefore the latter factor is bounded by
   $
   4\delta \sqrt{\gamma'}+7\gamma\sqrt{\delta'}+11\delta \|f_0\|\sqrt{\delta'}.
   $
    \item Combining the randomness from sampling and from the multipliers, w.p. $1-6\eta$, the intersection of the various events holds. Replacing $\eta$ with $\eta/6$, the former factor satisfies $\{1+\sqrt{2\ln(6/\eta)}\}\leq 3\ln(6/\eta)^{1/2}$. The terms in the latter factor satisfy
    \begin{align*}
        &4\delta \sqrt{\gamma'}
        \leq
        4\cdot 2\kappa \ln(12/\eta)\left\{\sqrt{\frac{\N(\lambda)}{n}} \vee \frac{4\kappa}{n\lambda}\right\}
        \left[
        \bar{\sigma} \N(\lambda)^{1/2}+2\bar{\sigma}\kappa\ln(12/\eta)^{1/2}\left\{\frac{1}{n^{1/2}\lambda}+\frac{\N(\lambda)^{1/4}}{n^{1/4}\lambda^{1/2}}\right\}
        \right] \\
       &(7\gamma+11\delta \|f_0\|)\sqrt{\delta'}  \\
       &\leq
       11\cdot
       2(\bar{\sigma}+\kappa\|f_0\|) \ln(12/\eta)\left\{\sqrt{\frac{\N(\lambda)}{n}} \vee \frac{4\kappa}{n\lambda}\right\}
       \left[
       \kappa \N(\lambda)^{1/2}+2\kappa^2\ln(12/\eta)^{1/2}\left\{\frac{1}{n^{1/2}\lambda}+\frac{\N(\lambda)^{1/4}}{n^{1/4}\lambda^{1/2}}\right\}
       \right].
    \end{align*}
    From these, we pull out the constants $\kappa^2(\bar{\sigma}+\kappa\|f_0\|)$, $4\cdot2\cdot2$, and $11\cdot2\cdot2$. \qedhere
\end{enumerate}
\end{proof}

\section{Uniform confidence band}\label{sec:anti}

\paragraph{\textcolor{black}{Notation and background.}} In this section we prove Propositions \ref{prop:h-band} and \ref{prop:inf-intro} from Section \ref{sec:algo_main}. Following Assumption \ref{assumption:light-main}, we define $(Q,R,L,B)$ so that the following hold.
\begin{itemize}
     \item There exists a Gaussian $Z$ in $H$ with covariance $\Sigma$ such that w.p. $1-\eta$,
    $\snorm{\sqrt{n}(\hat f - f_\lambda) - Z} \le Q(n,\lambda,\eta).$
    \item There exists $Z'$ in $H$ that, conditional on $D$, is almost surely Gaussian with covariance $\Sigma$. Upon an event $\mathcal{E} \in \sigma(D)$ that holds w.p. $1-\eta$,
    $\bb{P}\left\{\beef \norm{\mathfrak{B} - Z'} \le R(n,\lambda,\eta)|D\right\} \ge 1-\eta.$
    \item It holds w.p. at least $1-\eta$ that
    $\snorm{Z} \ge L(\lambda,\eta)$.

    \item The bias satisfies $\sqrt{n}\snorm{f_\lambda - f_0} \le B(\lambda)$.
\end{itemize}
Let $\Delta(n,\lambda,\eta) = Q(n,\lambda,\eta) + R(n,\lambda,\eta)$. We abbreviate by suppressing arguments.

We recall that $\hat t_\a = \hat t_\a(D)$ is the $1-\a$ quantile of the conditional distribution of $\|\BS\|$ given the data, i.e. $\mathbb{P}(\|\BS\| > \hat t_\a|D) = \a$. Quantiles are unique since a.s.~finite Gaussian suprema are continuously distributed. We similarly define $t_\a$ as the unique $1-\a$ quantile of $\|Z\|$. By our assumption on $L$ above, it follows that $L(\lambda,\eta) \le t_{1-\eta}$.
By our coupling construction, $Z'$ has the same conditional distribution given any realization of $D$ as the marginal distribution of $Z$. The quantity $t_{\a}$ thus also uniquely satisfies $\mathbb{P}(\|Z'\| > t_{\a}|D) = \a$, so for any $\sigma(D)$-measurable $\hat t$, we have
$\{\mathbb{P}(\|Z'\| > \hat t |D) \le \a \}  \subset \{ \hat t \ge t_\a\}$ as $\sigma(D)$-measurable events.

\paragraph{\textcolor{black}{Overview.}} We first prove validity and sharpness of $H$-norm confidence sets $\hat C_\a$ using an incremental factor approach, then establish exact inference for arbitrary functionals $F: H\to\mathbb{R}$ under continuity and anti-concentration conditions, allowing sup-norm confidence bands. We use our Gaussian and bootstrap couplings to show that $\sqrt n (\hat f - f_0)$ and its bootstrap counterpart $\BS$ are  close in distribution to the finite-sample Gaussian approximation $Z$.

\paragraph{\textcolor{black}{Incremental factor.}} We show that coupling $V$ and $W$ implies similarity of their CDFs.

\begin{lemma}[One-sided error bound]\label{lemma:inf-factor}
    Let $V, W$ be random variables such that $\mathbb{P}(|V-W| > r_1| \c A) \le r_2$ for some $r_1, r_2 > 0$, where $\c A$ is $\sigma$-subalgebra of $\mathbb{P}$. Then, for any random variable $\hat t \in \bb{R}$, it holds that
    $\bb{P}(V > \hat t| \c A) \le \bb{P}(W > \hat t - r_1| \c A) + r_2.$
\end{lemma}

\begin{proof}
Note that if $V > \hat t$, then either $W > \hat t - r_1$ or $|W-V| \ge r_1$.
Thus, for any $A \in \c A$, $\mathbbm{1}(V > \hat t)\mathbbm{1}_A \le \{\mathbbm{1}(W > \hat t - r_1)+\mathbbm{1}(|W-V| \ge r_1)\}\mathbbm{1}_A$. The result follows by taking expectations and using the definition of $\mathbb{E}[-|\mathcal{A}]$.%
\end{proof}

We apply this to show that the CDFs of $\|\BS\|$ and $\sqrt{n}\|\hat f - f_0\|$ are similar to that of $\|Z\|$.
\begin{lemma}[High probability events]\label{lemma:inference_f0}
For any $t \in \mathbb{R}$, we have
$$
\bb{P}(\sqrt{n}\snorm{\hat f - f_0}
>  t)\le \bb{P}(\snorm{Z} > t -Q-B)
 +  \eta,\;
\bb{P}( \snorm{Z}
> t) \le \bb{P}\{\sqrt{n}\snorm{\hat f - f_0} > t -Q-B\}
 +  \eta.
$$
Moreover, on the event $\c E \in \sigma(D)$ with $\mathbb{P}(\mathcal{E}) \ge 1 - \eta$, for any random variable $\hat t$,
$$
\bb{P}(\|Z'\|
> \hat t |D ) \le \bb{P}\left( \|\mathfrak{B}\| > \hat t -R |D\right)
 +  \eta,\quad
\bb{P}(\|\mathfrak{B}\|
> \hat t |D ) \le \bb{P}( \|Z'\| > \hat t -R |D)
 +  \eta.
$$
\end{lemma}
\begin{proof}
We proceed in steps.
\begin{enumerate}
    \item By two applications of the triangle inequality and the definitions of $(B,Q)$, w.p. $1-\eta$,
$$\mathclap{
\left|\sqrt{n}\norm{\hat f - f_0}-\norm{Z}\right|\leq \norm{\sqrt{n}(\hat f - f_0) - Z} \le \norm{\sqrt{n}(f_\lambda - f_0)} + \norm{\sqrt{n}(\hat f - f_\lambda) - Z}
\le B + Q.}
$$
By Lemma~\ref{lemma:inf-factor} with $\c A$ chosen to be the trivial $\sigma$-algebra, $V=\sqrt{n}\snorm{\hat f - f_0}$, $W=\snorm{Z}$, $r_1= Q + B$, and $r_2=\eta$,
$$
\bb{P}(\sqrt{n}\snorm{\hat f - f_0}
> t)\le \bb{P}\{ \snorm{Z} > t -(Q+B)\}
 +  \eta.
$$
If we reverse the roles of $V$ and $W$ in our application of Lemma \ref{lemma:inf-factor}, we obtain
$$
\bb{P}( \snorm{Z}
> t) \le \bb{P}\{\sqrt{n}\snorm{\hat f - f_0} > t -(Q+B)\}
 +  \eta.
$$

    \item By Assumption \ref{assumption:light-main}, on an event $\mathcal{E}$ w.p. $1-\eta$,
$\mathbb{P}(\beef\norm{\mathfrak{B}-Z'}
> R | D) \le \eta$. Therefore by the triangle inequality, upon $\mathcal{E}$,
$\mathbb{P}(\beef|\norm{\mathfrak{B}}-\norm{Z'}|
> R | D) \le \eta$. By Lemma~\ref{lemma:inf-factor}, with $\c A = \sigma(D)$,
with  $V=\|Z'\|$, $W=\|\mathfrak{B}\|$, $r_1=R$, and $r_2=\eta$, it follows that on $\mathcal{E}$,
$$
\bb{P}(\|Z'\|
> \hat t |D ) \le \bb{P}\left( \|\mathfrak{B}\| > \hat t -R |D\right)
 +  \eta.
$$
Reversing $V$ and $W$ gives us that on $\mathcal{E}$,
$
\bb{P}(\|\mathfrak{B}\|
> \hat t |D ) \le \bb{P}( \|Z'\| > \hat t -R |D)
 +  \eta.
$
\qedhere
\end{enumerate}
\end{proof}

\begin{proof}[Proof of Proposition~\ref{prop:h-band}(a)]
Recall that $\hat t_\alpha$ has been defined so that $\mathbb{P}(\|\BS\|>\hat t_\alpha|D) = \alpha$. By the construction of $\BS$, $\hat t_\alpha$ is $\sigma(D)$-measurable. Let the event $\mathcal{E}$ be  defined as in Lemma \ref{lemma:inference_f0} above.

Our first aim is to show that $\mathbb{P}(\sqrt{n}\|\hat f - f_0\| > (1+\delta) \hat t_\alpha ) \le \alpha + 3\eta$, which corresponds to validity. We show that on the high-probability event $\mathcal{E}$, the bootstrap critical value $\hat t_\alpha$ resembles that of the approximating Gaussian $\|Z\|$, so it is typically large enough to cover $f_0$.
We proceed in steps.
\begin{enumerate}
\item
On the event $\mathcal E$, Lemma \ref{lemma:inference_f0} implies
$$\mathbb{P}(\|Z'\| \ge \hat t_{\alpha} + R|D) \le \mathbb{P}(\|\BS\| > \hat t_\alpha|D) + \eta= \alpha + \eta = \mathbb{P}(\|Z'\| \ge  t_{\alpha+\eta}|D).$$ Thus, $\hat t_\alpha + R \ge t_{\alpha + \eta}$ on $\c E$, since $t_{\alpha}$ is the unique $1-\alpha$ quantile of $\|Z'\|$ given $D$.

\item Also by Lemma \ref{lemma:inference_f0},
$\mathbb{P}(\sqrt{n}\|\hat f - f_0\| > t_{\alpha+\eta} + Q + B) \le \mathbb{P}(\|Z\| > t_{\alpha + \eta}) + \eta = \alpha + 2\eta$.

\item To show coverage unconditionally, we have
\[\mathbb{P}(\sqrt{n}\|\hat f - f_0\| > \hat t_\alpha + Q + R + B) = \mathbb{E}[\mathbbm{1}\{\sqrt{n}\|\hat f - f_0\| >
\hat t_\alpha + Q + R + B\}
(\mathbbm{1}_{\mathcal{E}} + \mathbbm{1}_{\mathcal{E}^c})]
\]
For the first summand, we are on the good event $\mathcal E$, so by  steps 1 and 2 above,
\begin{align*}\mathbb{E}[\mathbbm{1}\{\sqrt{n}\|\hat f - f_0\| >
\hat t_\alpha + Q + R + B\}
\mathbbm{1}_{\mathcal{E}}] &\le \mathbb{E}[\mathbbm{1}\{\sqrt{n}\|\hat f - f_0\| >
t_{\alpha+\eta} + Q + B\}
\mathbbm{1}_{\mathcal{E}}] \\
&\le \mathbb{P}(\sqrt{n}\|\hat f - f_0\| >
t_{\alpha+\eta} + Q + B) \le \alpha + 2\eta.
\end{align*}
 For the second summand, $\mathbb{E}[\mathbbm{1}\{\sqrt{n}\|\hat f - f_0\| >
\hat t_\alpha + Q + R + B\}
\mathbbm{1}_{\mathcal{E}^c}] \le \mathbb{P}(\mathcal{E}^c) \le \eta$. Combining these, with $\D = Q +R$, we find
$\mathbb{P}(\sqrt{n}\|\hat f - f_0\| > \hat t_\alpha + \D + B) \le \alpha + 3\eta$.

\item We have finished showing validity if $\hat t_\alpha + \D + B \le (1 + \delta) \hat t_{\alpha}$.
Write $\tilde L = L(\lambda,1-\alpha-2\eta)$.
By definition, $\tilde L \le t_{\alpha+2\eta} \le t_{\alpha+\eta} $.
Combining with step 1 gives $\hat t_\alpha + \D \ge \hat t_\alpha + R \ge t_{\alpha + \eta} \ge \tilde L$, so $\hat t_\a \ge \tilde L - \D$.
For any $\delta$ that satisfies
    $
    \f 1 2 \ge \d \ge \frac{\D + B}{\tilde{L} - \D},$ we then have
    $
    \d \hat t_\a \ge \frac{\D + B}{\tilde{L} - \D}  (\tilde L - \D) = \D + B.
    $
    It follows that $(1+\d)\hat t_\a \ge \hat t_\a + B + \Delta$, as needed.
\end{enumerate}
Next we must show that $\hat C_\alpha$ is $(2\delta, 3\eta)$-sharp at level $\a$, namely that
$\bb{P}\{f_0 \in 2\d \hat f + (1-2\delta) \hat{C}_{\alpha} \} \le 1 - \a + 3\eta.$ For this, it suffices to show that with probability at least $\a - 3\eta$,
\[
\sqrt n \|\hat{f}-f_0\| > (1-\d)\hat{t}_\a
 \ge(1-2\delta) (1+\d)\hat{t}_\a  ,\]
which follows by the definition of $\hat C_\a$ and since $(1-2\d)(1+\d) \le 1-\d$.
\begin{enumerate}
\item By Lemma \ref{lemma:inference_f0}, on $\c E$, $\a = \mathbb{P}(\|\BS\| > \hat t_\alpha|D)
 \le \mathbb{P}(\|Z'\| \ge \hat t_{\alpha} - R|D) + \eta$. It follows that $\hat t_\alpha - R \le t_{\alpha - \eta}$ on $\c E$, hence with probability $\mathbb{P}(\c E) \ge 1-\eta$.
 \item Also by Lemma \ref{lemma:inference_f0}
 $
 \alpha - \eta = \bb{P}( \snorm{Z}
> t_{\alpha-\eta}) \le \bb{P}\{\sqrt{n}\snorm{\hat f - f_0} > t_{\alpha-\eta} -(Q+B)\}
 +  \eta.
 $
 \item By step 1, the event $\hat t_\a - R \le t_{\a - \eta}$ holds with probability at least $1-\eta$. By step 2, the event $\{\sqrt{n}\snorm{\hat f - f_0} > t_{\alpha-\eta} -(Q+B)\}$ holds with probability at least $1-\eta' = \a - 2\eta$ so $\eta' = 1 - \a + 2\eta$.
 By a union bound, these simultaneously hold with probability at least $1 - \eta - \eta' = 1 - \eta - (1 - \a + 2\eta) = \a-3\eta$, hence
 \[\mathbb{P}\{\sqrt{n}\snorm{\hat f - f_0} > \hat t_{\alpha} -(R+Q+B)\} \ge \a - 3 \eta.\]
 \item Thus, noting $\D = Q+R$, we are done if we can show that $\hat t_\a - \D - B \ge (1-\d)\hat t_\a$. This follows from $\d\hat t_\a \ge \D + B$, which was derived when proving validity.\qedhere
 \end{enumerate}
\end{proof}

\paragraph{\textcolor{black}{Anti-concentration.}}
We obtain valid inference for any functional $F: \H \mapsto \mathbb{R}$ under uniform continuity and anti-concentration conditions. To begin, we bound the Kolmogorov distance between coupled random variables.

\begin{lemma}[cf. {\citet[Lemma 2.1]{chernozhukov2016empirical}}]\label{lemma:chern-anti-lemma}
Let $V, W$ be real-valued random variables such that $\mathbb{P}(|V - W| >
r_1| \c A) \le r_2$ for some constants $r_1, r_2 > 0$ where $\c A$ is a $\sigma$-algebra comprised of Borel sets. Then we have
\[\sup_{t \in \bb R}\left|\beef\bb P(V \le t|\c A) - \bb P(W \le t| \c A) \right| \le \sup_{t \in \bb R}
\bb P\left(\beef|W - t| \le r_1 \middle| \c A \right) + r_2.\]
\end{lemma}
\begin{proof}
We proceed in steps.
\begin{enumerate}
    \item To begin, we show that for fixed $(V,W,t,z)$, where $z>0$,
    $$
    \mathbbm{1}(V \le t) - \mathbbm{1}(W \le t) \le \mathbbm{1}(|W-t| \le z) \vee \mathbbm{1}(|V-W| > z).
    $$
    If the left hand side is one, then $V \le t < W$, so $|W-t| < |W-V|$. Thus, it is not possible that $|W-t| > z \ge |W-V|$, so the right hand side is also one. Otherwise the left hand side~is at most zero and the right hand side~is at least zero.
    \item Taking $z=r_1$, bounding the max by the sum, and multiplying by indicators,
    $$\mathbbm{1}_A\mathbbm{1}(V \le t) - \mathbbm{1}_A\mathbbm{1}(W \le t)
   \le \mathbbm{1}_A\mathbbm{1}(|W-t| \le r_1) + \mathbbm{1}_A\mathbbm{1}(|V-W| > r_1).$$
Taking expectations, the result without the absolute value follows by definition of $\mathbb{E}[-|\c A]$.
    \item The same bound holds reversing the roles of $V$ and $W$, so that we may replace the left hand side by its absolute value. Finally, we take the supremum over $t \in \bb{R}$. \qedhere
\end{enumerate}
\end{proof}

We then apply this to our Gaussian and bootstrap couplings.

\begin{lemma}[High probability events]\label{lemma:inference_sup}
        If $F$ is uniformly continuous then on the event $\mathcal{E}$, with $\mathbb{P}(\c E) \ge 1-\eta$,
        \begin{align*}
            &\sup_{t \in \bb{R}} \left|\bb{P}\left[F\Big\{\sqrt{n}(\hat f - f_0)\Big\}
\le t \right] -  \bb{P}\left\{ \beef F(Z) \le t \right\} \right|
 \leq \sup_{t \in \bb R}
\bb P\left\{\beef|F(Z) - t| \le \psi (Q+B) \right\} + \eta, \\
&\sup_{t \in \bb{R}} \bigg|\bb{P}\left\{F(\mathfrak{B}) \le t|D \right\} -  \bb{P}\left\{F(Z')
\le t |D\right\} \bigg|
\le \sup_{t \in \bb{R}} \mathbb{P}\left\{\beef |F(Z') -t| \le \psi R \middle| D \right\} + \eta.
        \end{align*}
\end{lemma}

\begin{proof}
By continuity of $F$, w.p. $1-\eta$, $|F\{\sqrt{n}(\hat f - f_0)\} - F(Z)| \le \psi (Q+B)$. The first result then follows by Lemma~\ref{lemma:chern-anti-lemma} with $\c A$ as the trivial $\sigma$-algebra, $V=F\{\sqrt{n}(\hat f - f_0)\}$, $W=F(Z)$, $r_1=\psi (Q+B)$, and $r_2=\eta$.

Upon $\mathcal{E}$, by continuity of $F$, $\mathbb{P}\left\{\beef|F(\mathfrak{B}) - F(Z')|
\leq \psi R | D \right\} \geq 1 - \eta$, which can be written $\mathbb{P}\left\{\beef|F(\mathfrak{B}) - F(Z')|
> \psi R | D \right\} < \eta$. The second result then follows by Lemma~\ref{lemma:chern-anti-lemma}, with $\c A = \sigma(D)$, $V=F(\mathfrak{B})$, $W=F(Z')$, $r_1=\psi R$, and $r_2=\eta$.
\end{proof}

\begin{proof}[Proof of Proposition~\ref{prop:inf-intro}]
We reduce the result to the bounds in Lemma~\ref{lemma:inference_sup}, then simplify.
\begin{enumerate}
\item By the triangle inequality, Lemma \ref{lemma:chern-anti-lemma}, and $\mathbb{P}(Z \in - ) = \mathbb{P}(Z' \in - | D)$,
\begin{align*}
    &\sup_{t \in \bb{R}} \left|\bb{P}\left[F\Big\{\sqrt{n}(\hat f - f_0)\Big\} \le t \right] -  \bb{P}\left\{ \beef F(\mathfrak{B}) \le t | D\right\} \right|  \\
    &\leq
    \sup_{t \in \bb{R}} \left|\bb{P}\left[F\Big\{\sqrt{n}(\hat f - f_0)\Big\}
\le t \right] -  \bb{P}\left\{ \beef F(Z) \le t \right\} \right|
+
\sup_{t \in \bb{R}} \bigg|\bb{P}\left\{F(\mathfrak{B}) \le t|D \right\} -  \bb{P}\left\{F(Z')
\le t |D\right\} \bigg| \\
&\leq
\sup_{t \in \bb R}
\bb P\left\{\beef|F(Z) - t| \le \psi (Q+B) \right\}
+
\sup_{t \in \bb{R}} \mathbb{P}\left\{\beef |F(Z') -t| \le \psi R \middle| D \right\} +2 \eta. \\
&=
\sup_{t \in \bb R}
\bb P\left\{\beef|F(Z) - t| \le \psi (Q+B) \right\}
+
\sup_{t \in \bb{R}} \mathbb{P}\left\{\beef |F(Z) -t| \le \psi R \right\} +2 \eta.
\end{align*}
    \item {\color{black}
    By anti-concentration (Assumption~\ref{assumption:anti-concentration}),
\begin{align*}
  &\sup_{t \in \bb R}
\bb P\left\{\beef|F(Z) - t| \le \psi (Q+B) \right\}
+
\sup_{t \in \bb{R}} \mathbb{P}\left\{\beef |F(Z) -t| \le \psi R\right\} \\
&\qquad\leq
\mathfrak{a}\{\psi(Q+B)\}+\mathfrak{a}(\psi R),
\end{align*}
which proves the first statement.
}

    \item For the second statement, given $\d > 0$, we define
    $$
    \mathcal{E}_1=\{|F(Z')-t|\leq\delta\}, \quad\mathcal{E}_2 = \{|F(\mathfrak{B})-t|\leq 2\delta\} \cup \{|F(\mathfrak{B})-F(Z')|>\delta \}
    $$
    We prove that $\mathcal{E}_2^c\subset\mathcal{E}_1^c$, which is equivalent to $\c{E}_1 \subset \c{E}_2$. To begin, we observe that
    $
    \mathcal{E}_2^c
        =\{|F(\mathfrak{B})-t|> 2\delta \} \cap \{|F(\mathfrak{B})-F(Z')|\leq \delta\}.
        $
    Therefore, on $\mathcal{E}_2^c$,
    $$
    2\delta< |F(\mathfrak{B})-t| \leq |F(\mathfrak{B})-F(Z')|+|F(Z')-t| \leq \delta+|F(Z')-t|
    $$
    which implies $\delta<|F(Z')-t|$, i.e. $\mathcal{E}_1^c$ holds.
    Since $\c{E}_1 \subset \c{E}_2$, it follows that $\c{E}_1 \cap A \subset \c{E}_2 \cap A$ for any $A \in \sigma(D)$, so $\mathbb{P}(\c E_1|D) \le \mathbb{P}(\c E_2 | D)$. Therefore, since $\mathbb{P}(Z \in - ) = \mathbb{P}(Z' \in - | D)$,
    \begin{align*}
        \mathbb{P}\left\{\beef |F(Z) -t| \le \delta \right\}
        &= \mathbb{P}\left\{\beef |F(Z') -t| \le \delta \middle| D \right\}
        =\mathbb{P}(\mathcal{E}_1|D)
        \leq \mathbb{P}(\mathcal{E}_2|D) \\
        &\leq \mathbb{P}\left\{\beef |F(\mathfrak{B}) -t| \le 2\delta \middle| D \right\} + \mathbb{P}\left\{\beef|F(Z') - F(\mathfrak{B})|
> \delta \middle| D \right\}.
    \end{align*}

We showed $\mathbb{P}\{|F(Z') - F(\mathfrak{B})|
> \d | D \} <\eta$ on $\c E$ in the proof of Lemma~\ref{lemma:inference_sup}, for $\d = \psi R$, hence also for $\delta=\psi(\Delta+B) > \psi R$. Taking the supremum over $t \in \mathbb{R}$ then yields that on $\c E$,
    \begin{align*}
        &\sup_{t\in\mathbb{R}}\mathbb{P}\left\{\beef |F(Z) -t| \le \psi(\Delta+B) \right\}
        = \sup_{t\in\mathbb{R}} \mathbb{P}\left\{\beef |F(Z') -t| \le \psi(\Delta+B) \middle| D \right\} \\
        &\leq \sup_{t\in\mathbb{R}} \mathbb{P}\left\{\beef |F(\mathfrak{B}) -t| \le 2\psi(\Delta+B) \middle| D \right\} +\eta.
    \end{align*}
   We conclude that on $\c E$, with $\mathbb{P}(\c E) \ge 1-\eta$,
   \begin{align*}
       &\sup_{t \in \bb R}
\bb P\left\{\beef|F(Z) - t| \le \psi (Q+B) \right\}
+
\sup_{t \in \bb{R}} \mathbb{P}\left\{\beef |F(Z') -t| \le \psi R \middle| D \right\} \\
&\leq
2\sup_{t\in\mathbb{R}} \mathbb{P}\left\{\beef |F(\mathfrak{B}) -t| \le 2\psi(\Delta+B) \middle| D \right\} + 4\eta. & \qedhere
  \end{align*}
\end{enumerate}
\end{proof}

\begin{proof}[Proof of Corollary~\ref{cor:mis}]
    Observe that $U_i\in H$ even when $f_0\not \in H$ since we can express
    $$U_i=T_{\lambda}^{-1}[\{Y-f_\lambda(X)\}k_{X_i}-\lambda f_{\lambda}].$$
    Then $\E(U_i)=0$ by the population first order condition and Definition~\ref{def:regularity_partial} is satisfied with
    $$
    \|U_i\|\leq \frac{2\kappa}{\lambda}(\bar{\sigma}+\bar{\rho}),\quad \|f_0-f_{\lambda}\|_{\infty}\leq \bar{\rho}.
    $$
    Similarly, the remaining results in Appendices~\ref{sec:symbols}--\ref{sec:anti} generalize with a modified constant involving $\bar{\rho}$ rather than $\kappa\|f_0\|_H$.
\end{proof}
\section{Simulation details}\label{sec:simulation_details}

\paragraph{\textcolor{black}{Simulation design and preference kernel universality.}} Each observation in the standard data design is generated as follows. Draw $X_i \sim \mathrm{Unif}(0,1)$ and $\varepsilon_i \sim \mathrm{Unif}(-2,2)$ independently. Then set $Y_i = f_0(X_i)+\varepsilon_i$ where $f_0$ is a weighted average of the initial five eigenfunctions of the covariance operator $T$. Specifically,
$
f_0=\tilde{f_0}/(10\|\tilde{f_0}\|)$ where $\tilde{f}_0=\sum_{s=1}^5 g_s e_s(T)$, and
$g_s$ are standard normal random variables drawn once to serve as coefficients; $f_0$ is fixed across experiments.
We use the Gaussian kernel $k(x,x') = \exp\{-\snorm{x-x'}^2/(2\iota^2)\}$, with the lengthscale $\iota$ set to $0.1$.

Each observation in the preference data design is generated similarly, with some changes. The covariates $X_i$ are a uniform random permutation of $\set{1, 2, \cdots, 7}$. We use the preference kernel $k(x,x') = \exp\{-N(x,x')/(2\iota^2)\}$, where $N$ is Kendall's (unnormalized) rank correlation. We employ the standard heuristic and choose $\iota$ to be the median of $N(X_i,X_j)$ for independent data draws.

The RKHS for this preference kernel is automatically well-specified in a certain sense.

\begin{lemma}[Universal kernel; {\citealp[Theorem 5]{mania2018kernel}}]
The stated preference kernel is universal. Formally, if $\mathbb{P}$ is a probability distribution over the set of possible preferences, then the RKHS for this kernel contains any function
$f:\mr{supp}(\mathbb{P}) \to \bb{R}$.
\end{lemma}

Conventionally, a kernel is universal if and only if $H$ is dense in $L^2(\mathbb{P})$. In the case of preferences, a stronger statement is possible because the set of preferences is finite.

Throughout, we implement Algorithm~\ref{algo:krr} as the point estimate, Algorithm~\ref{algo:incremental} as the $H$ norm band, and Algorithm~\ref{algo:variable} as the $\sup$ norm band. The regularization is typically $\lambda=n^{-1/2}$, following Theorem~\ref{thm:near-minimax-band}(b), unless specified otherwise.

\paragraph{\textcolor{black}{Metrics beyond coverage.}}
For the standard data and preference data designs, we fix $n=500$ and $\lambda=n^{-1/2}$ and examine additional metrics of confidence band performance beyond coverage. Table~\ref{tab:detail_all} records bias and width of the confidence bands across samples from the designs. We find that the $H$ norm bands have more bias and width than the $\sup$ norm bands. Both types of bands achieve nominal coverage for both $f_0$ and $f_{\lambda}$. Since we use the same band for both $f_0$ and $f_{\lambda}$, the width is the same for both quantities, but the bias for $f_0$ reflects $f_{\lambda}-f_0$.

\begin{table}
\captionsetup[subtable]{justification=Centering}
\caption{\label{tab:detail_all} A detailed look: Coverage, bias, and width. We fix $n=500$ and set $\lambda=n^{-1/2}$.}
\begin{subtable}[t]{0.42\textwidth}
    \centering
    \resizebox{\textwidth}{!}{%
    \begin{tabular}{ccccc}
         \multirow{2}{*}{metric} & \multicolumn{2}{c}{$\sup$ norm} &\multicolumn{2}{c}{$H$ norm } \\
         \cmidrule(lr){2-3}\cmidrule(lr){4-5}
         & true & pseudo & true & pseudo \\
         \hline
 coverage &      0.985 &        0.985 &    0.953 &      0.955 \\
 bias     &      0.015 &        0.000 &    0.019 &      0.000 \\
 width    &      0.324 &        0.324 &    1.395 &      1.395 \\
        \hline
    \end{tabular}
    }
    \caption{\label{tab:cov_detail} \footnotesize \textsc{Standard data.}}
\end{subtable} \quad\quad\quad
\begin{subtable}[t]{0.42\textwidth}
    \centering
    \resizebox{\textwidth}{!}{%
    \begin{tabular}{ccccc}
         \multirow{2}{*}{metric} & \multicolumn{2}{c}{$\sup$ norm} &\multicolumn{2}{c}{$H$ norm } \\
         \cmidrule(lr){2-3}\cmidrule(lr){4-5}
        & true & pseudo & true & pseudo \\
        \hline
 coverage &      0.955 &        0.965 &    0.963 &      0.965 \\
 bias     &      0.019 &        0.000 &    0.040 &      0.000 \\
 width    &      0.328 &        0.328 &    2.255 &      2.255 \\
        \hline
        \end{tabular}
    }
    \caption{\label{tab:cov_detail_ranking} \footnotesize \textsc{Preference data.}}
    \end{subtable}
\end{table}

\paragraph{\textcolor{black}{Robust performance in Sobolev spaces.}}\label{sec:sobolev}
We present additional results in Sobolev space, which corresponds to the Matern kernel. Each observation is generated as in the standard data design, replacing $f_0$ with a weighted average of the initial five eigenfunctions of the appropriate Matern kernel.
We consider  $f_0\in \mathbb{H}_2^1$ or $f_0\in \mathbb{H}_2^2$.
Figures~\ref{fig:cdf_sob1} and~\ref{fig:cdf_sob2} compare distributions that Proposition~\ref{prop:inf-intro} proves to be close. Tables~\ref{tab:cov_n_sob1},~\ref{tab:cov_n_sob2},~\ref{tab:cov_lambda_sob1}, and~\ref{tab:cov_lambda_sob2} verify nominal coverage across smoothness degrees, sample sizes, and regularization values. Tables~\ref{tab:cov_detail_sob1} and~\ref{tab:cov_detail_sob2} record additional metrics beyond coverage.

\begin{figure}[ht]
\captionsetup[subfigure]{justification=centering}
\begin{subfigure}{.32\textwidth}
  \centering
  \includegraphics[width=\linewidth]{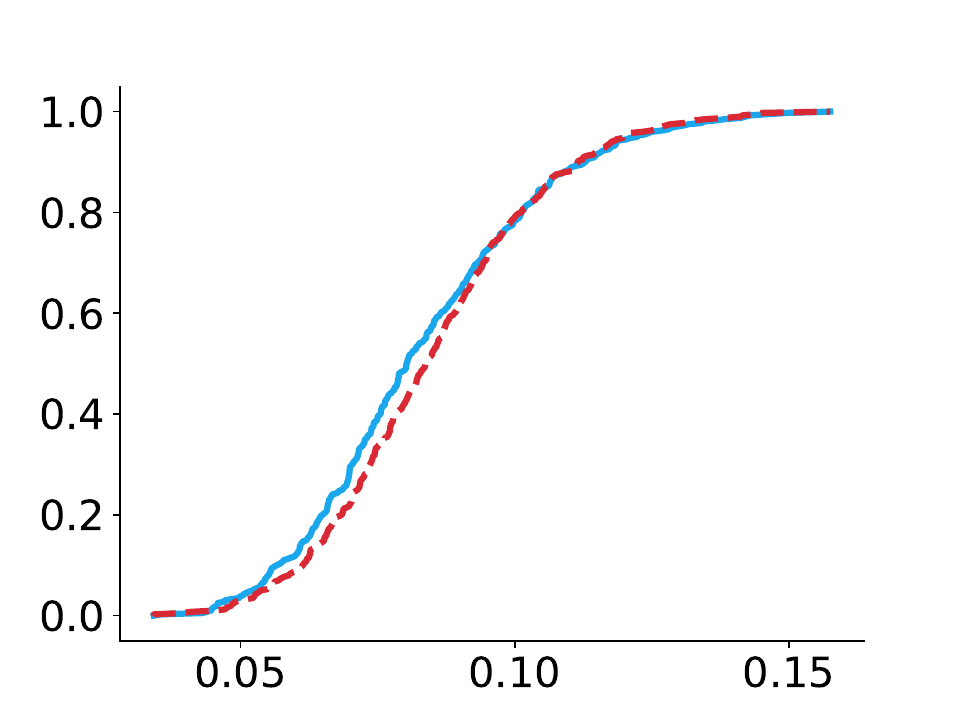}
  \caption{\label{fig:cdf_sob1} \footnotesize \textsc{Sobolev space $\mathbb{H}_2^1$.}}
\end{subfigure}
\begin{subfigure}{.32\textwidth}
  \centering
  \includegraphics[width=\linewidth]{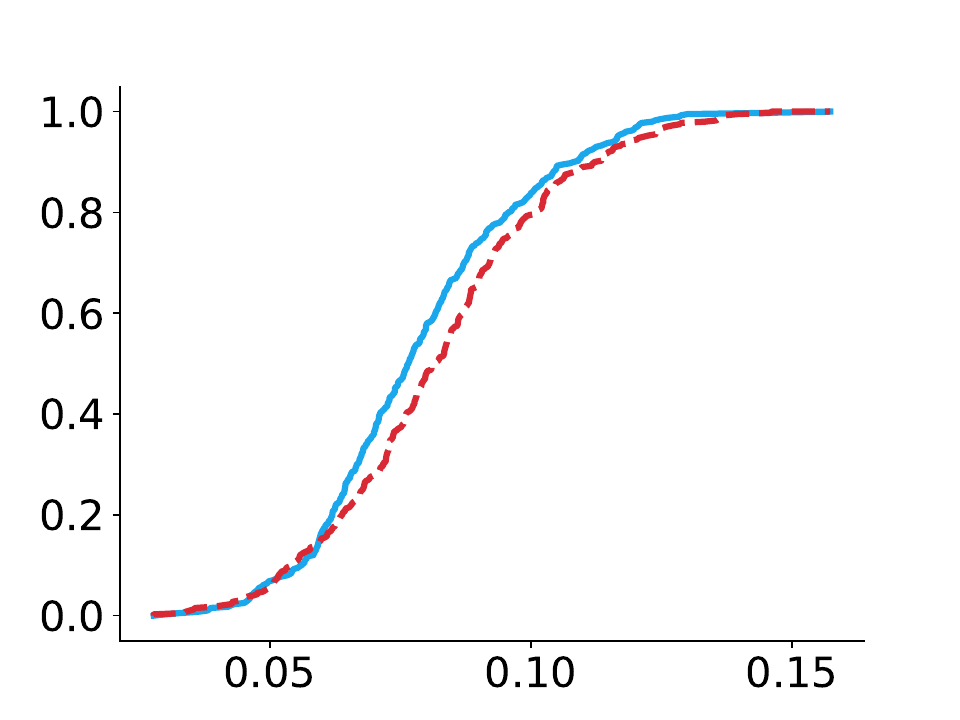}
  \caption{\label{fig:cdf_sob2} \footnotesize \textsc{Sobolev space $\mathbb{H}_2^2$.}}
\end{subfigure}
\begin{subfigure}{.32\textwidth}
  \centering
  \includegraphics[width=\linewidth]{img_ms/cdf_plot_misspec_data.pdf }
  \caption{\label{fig:cdf_hard} \footnotesize \textsc{Mis-specification.}}
\end{subfigure}
\caption{We compare the distribution of $n^{1/2}\snorm{\hat f-f_0}_\infty$ (or $n^{1/2}\snorm{\hat f-f_{\lambda}}_\infty$ in Figure~\ref{fig:cdf_hard})  across many samples (dashed red), with the distribution of our proposal $\snorm{\mathfrak{B}}_\infty$ across many bootstrap iterations, conditional upon a single sample (solid blue).}
\end{figure}

\begin{table}
\captionsetup[subtable]{justification=Centering}
 \caption{Coverage is nominal across sample sizes. Across rows, we vary $n$ and set $\lambda=n^{-1/2}$.}
\begin{subtable}[t]{0.31\textwidth}
    \centering
    \resizebox{\textwidth}{!}{%
    \begin{tabular}{ccccc}
         \multirow{2}{*}{sample} & \multicolumn{2}{c}{$\sup$ norm} &\multicolumn{2}{c}{$H$ norm } \\
         \cmidrule(lr){2-3}\cmidrule(lr){4-5}
         & true & pseudo & true & pseudo \\
         \hline
  100 &               0.960 &                 0.960 &             0.978 &               0.978 \\
  250 &               0.942 &                 0.942 &             0.978 &               0.972 \\
  500 &               0.938 &                 0.942 &             0.960 &               0.962 \\
 1000 &               0.945 &                 0.955 &             0.975 &               0.970 \\
        \hline
    \end{tabular}
    }
    \caption{\label{tab:cov_n_sob1} \footnotesize \textsc{Sobolev space $\mathbb{H}^1_2$.}}
\end{subtable} \quad
\begin{subtable}[t]{0.31\textwidth}
    \centering
    \resizebox{\textwidth}{!}{%
    \begin{tabular}{ccccc}
         \multirow{2}{*}{sample} & \multicolumn{2}{c}{$\sup$ norm} &\multicolumn{2}{c}{$H$ norm } \\
         \cmidrule(lr){2-3}\cmidrule(lr){4-5}
        & true & pseudo & true & pseudo \\
        \hline
  100 &               0.972 &                 0.975 &             0.878 &               0.882 \\
  250 &               0.958 &                 0.952 &             0.972 &               0.972 \\
  500 &               0.950 &                 0.965 &             0.962 &               0.962 \\
 1000 &               0.970 &                 0.972 &             0.965 &               0.965 \\
        \hline
        \end{tabular}
    }
    \caption{\label{tab:cov_n_sob2} \footnotesize \textsc{Sobolev space $\mathbb{H}^2_2$.}}
    \end{subtable}\quad
\begin{subtable}[t]{0.31\textwidth}
    \centering
    \resizebox{\textwidth}{!}{%
    \begin{tabular}{ccccc}
         \multirow{2}{*}{sample} & \multicolumn{2}{c}{$\sup$ norm} &\multicolumn{2}{c}{$H$ norm } \\
         \cmidrule(lr){2-3}\cmidrule(lr){4-5}
        & true & pseudo & true & pseudo \\
        \hline
  100 &               0.000 &                 0.980 &             0.000 &               0.992 \\
  250 &               0.000 &                 0.958 &             0.000 &               0.948 \\
  500 &               0.000 &                 0.948 &             0.000 &               0.952 \\
 1000 &               0.000 &                 0.958 &             0.000 &               0.978 \\
        \hline
        \end{tabular}
    }
    \caption{\label{tab:cov_n_hard} \footnotesize \textsc{Mis-specification.}}
    \end{subtable}
\end{table}

\begin{table}
\captionsetup[subtable]{justification=Centering}
\caption{Coverage is nominal across regularization values. Across rows, we fix $n=500$ and vary $\lambda$.}
\begin{subtable}[t]{0.31\textwidth}
    \centering
    \resizebox{\textwidth}{!}{%
    \begin{tabular}{ccccc}
         \multirow{2}{*}{reg.} & \multicolumn{2}{c}{$\sup$ norm} &\multicolumn{2}{c}{$H$ norm } \\
         \cmidrule(lr){2-3}\cmidrule(lr){4-5}
         & true & pseudo & true & pseudo \\
         \hline
  0.500 &               0.040 &                 0.975 &             0.802 &               0.953 \\
  0.100 &               0.958 &                 0.983 &             0.970 &               0.973 \\
  0.050 &               0.980 &                 0.990 &             0.968 &               0.970 \\
  0.010 &               0.975 &                 0.978 &             0.970 &               0.970 \\
  0.005 &               0.980 &                 0.975 &             0.960 &               0.963 \\
  0.001 &               0.945 &                 0.945 &             0.935 &               0.935 \\
        \hline
    \end{tabular}
    }
    \caption{\label{tab:cov_lambda_sob1} \footnotesize \textsc{ Sobolev space $\mathbb{H}^1_2$.}}
\end{subtable} \quad
\begin{subtable}[t]{0.31\textwidth}
    \centering
    \resizebox{\textwidth}{!}{%
    \begin{tabular}{ccccc}
         \multirow{2}{*}{reg.} & \multicolumn{2}{c}{$\sup$ norm} &\multicolumn{2}{c}{$H$ norm } \\
         \cmidrule(lr){2-3}\cmidrule(lr){4-5}
        & true & pseudo & true & pseudo \\
        \hline
  0.500 &               0.070 &                 0.978 &             0.810 &               0.948 \\
  0.100 &               0.917 &                 0.927 &             0.948 &               0.955 \\
  0.050 &               0.958 &                 0.968 &             0.975 &               0.970 \\
  0.010 &               0.968 &                 0.968 &             0.958 &               0.958 \\
  0.005 &               0.940 &                 0.935 &             0.960 &               0.960 \\
  0.001 &               0.920 &                 0.920 &             0.980 &               0.980 \\
        \hline
        \end{tabular}
    }
    \caption{\label{tab:cov_lambda_sob2} \footnotesize \textsc{ Sobolev space $\mathbb{H}^2_2$.}}
    \end{subtable}\quad
\begin{subtable}[t]{0.31\textwidth}
    \centering
    \resizebox{\textwidth}{!}{%
    \begin{tabular}{ccccc}
         \multirow{2}{*}{reg.} & \multicolumn{2}{c}{$\sup$ norm} &\multicolumn{2}{c}{$H$ norm } \\
         \cmidrule(lr){2-3}\cmidrule(lr){4-5}
        & true & pseudo & true & pseudo \\
        \hline
  0.500 &               0.000 &                 0.968 &             0.000 &               0.975 \\
  0.100 &               0.000 &                 0.935 &             0.000 &               0.983 \\
  0.050 &               0.000 &                 0.955 &             0.000 &               0.963 \\
  0.010 &               0.000 &                 0.968 &             0.000 &               0.935 \\
  0.005 &               0.000 &                 0.955 &             0.000 &               0.945 \\
  0.001 &               0.000 &                 0.950 &             0.000 &               0.965 \\
        \hline
        \end{tabular}
    }
    \caption{\label{tab:cov_lambda_hard} \footnotesize \textsc{Mis-specification.}}
    \end{subtable}
\end{table}

\begin{table}
\captionsetup[subtable]{justification=Centering}
\caption{A detailed look: Coverage, bias, and width. We fix $n=500$ and set $\lambda=n^{-1/2}$.}
\begin{subtable}[t]{0.31\textwidth}
    \centering
    \resizebox{\textwidth}{!}{%
    \begin{tabular}{ccccc}
         \multirow{2}{*}{metric} & \multicolumn{2}{c}{$\sup$ norm} &\multicolumn{2}{c}{$H$ norm } \\
         \cmidrule(lr){2-3}\cmidrule(lr){4-5}
         & true & pseudo & true & pseudo \\
         \hline
 coverage &      0.943 &        0.938 &    0.960 &      0.963 \\
 bias     &      0.017 &        0.000 &    0.023 &      0.000 \\
 width    &      0.319 &        0.319 &    1.672 &      1.672 \\
        \hline
    \end{tabular}
    }
    \caption{\label{tab:cov_detail_sob1} \footnotesize \textsc{Sobolev space $\mathbb{H}^1_2$.}}
\end{subtable} \quad
\begin{subtable}[t]{0.31\textwidth}
    \centering
    \resizebox{\textwidth}{!}{%
    \begin{tabular}{ccccc}
         \multirow{2}{*}{metric} & \multicolumn{2}{c}{$\sup$ norm} &\multicolumn{2}{c}{$H$ norm } \\
         \cmidrule(lr){2-3}\cmidrule(lr){4-5}
        & true & pseudo & true & pseudo \\
        \hline
 coverage &      0.960 &        0.973 &    0.973 &      0.970 \\
 bias     &      0.018 &        0.000 &    0.024 &      0.000 \\
 width    &      0.314 &        0.314 &    1.619 &      1.619 \\
        \hline
        \end{tabular}
    }
    \caption{\label{tab:cov_detail_sob2} \footnotesize \textsc{ Sobolev space $\mathbb{H}^2_2$.}}
    \end{subtable}\quad
\begin{subtable}[t]{0.31\textwidth}
    \centering
    \resizebox{\textwidth}{!}{%
    \begin{tabular}{ccccc}
         \multirow{2}{*}{metric} & \multicolumn{2}{c}{$\sup$ norm} &\multicolumn{2}{c}{$H$ norm } \\
         \cmidrule(lr){2-3}\cmidrule(lr){4-5}
        & true & pseudo & true & pseudo \\
        \hline
 coverage &      0.000 &        0.963 &    0.000 &      0.980 \\
 bias     &      0.989 &        0.000 &   $\infty$     &      0.000 \\
 width    &      0.392 &        0.392 &    1.485 &      1.485 \\
        \hline
        \end{tabular}
    }
    \caption{\label{tab:cov_detail_hard} \footnotesize \textsc{Mis-specification.}}
    \end{subtable}
\end{table}

\paragraph{\textcolor{black}{Pseudo true coverage under mis-specification.}}\label{sec:mis-spec}
\textcolor{black}{We present additional results for the mis-specified Gaussian kernel. Each observation is generated as in the standard data design, replacing $f_0$ with the step function $f_0(x)=\mathbbm{1}\{x \ge 1/2\}$. Figure~\ref{fig:mis_band} visualizes this design, where $f_0\not \in H$. Figure~\ref{fig:cdf_hard} compares distributions
that arise when using our procedure in a misspecified setting.
While coverage for the true parameter $f_0 \not\in H$ breaks down, coverage for the pseudo true parameter $f_{\lambda}\in H$ remains nominal across sample sizes and regularization values in Tables~\ref{tab:cov_lambda_hard} and~\ref{tab:cov_detail_hard}. These simulations
illustrate that our procedure appears effective
without any
control of
the bias $f_{\lambda}-f_0$. Again, we document additional metrics such as bias and width of the confidence bands in Table~\ref{tab:cov_detail_hard}.}

\section{Application details}\label{sec:application_details}

\paragraph{\textcolor{black}{Semi-synthetic preference data.}}
We generate synthetic data which resemble real preferences of Boston Public School students. Though the underlying student-level micro data are not publicly available, \citet{pathak2021well} report coefficients from models estimated using real student preferences. Their model reproduces the distribution of student preferences in subsequent years with high accuracy, motivating this procedure.

We generate each student preference as follows. First, we draw a location and covariate vector for that student from the distribution represented in the American Community Survey.  We obtain school locations and other covariates from the Massachusetts Department of Elementary and Secondary Education. Next, we compute walking distances from that student's location to various schools, using the Google Maps API for distances. We then use the random utility model of \citet{pathak2021well} to generate the preference list, either with or without match effects, as detailed below. Finally, we use these preference lists and the standard random serial dictatorship mechanism to assign students to schools.

\paragraph{\textcolor{black}{Random utility model.}}
\cite{pathak2021well} model the utility of student $i$ at school $s$ as
 $
 u_{is} = \hat\a_s + \hat{\beta}_1^{\top}W_1 + \beta_2^{\top}W_2 + \nu_{is}.
$
Here, $\hat \a_s$ is a school fixed effect, and $(W_1,W_2)$ are covariates about the student, school, and their interaction: the student's race and English language learner status; the school's demographics; the student's walking distance to the school, and the availability of English instructional programs in the student's language.
While $\hat{\beta}_1$ is an estimated coefficient, $\beta_2 \sim \mathcal{N}(\hat \mu_2, \hat \Sigma_2)$ is a random coefficient fit using the mixed multinomial logit procedure outlined in their paper. Finally, $\nu_i$ is an independent standard Gumbel random variable. In our replication, we use the coefficients that the authors are able to share.

In the no match effect design, the conditional average treatment effect is a flat function of student preferences, i.e. $\textsc{cate}(X)=15$. We generate counterfactual outcomes as $Y_{is} = 50 + 10\d_i + 15 D_s,$ where $\d_i$ is a standard normal random variable and $D_s$ indicates whether $s$ is a pilot school. To resemble test scores, we truncate counterfactual outcomes $Y_{is}$ to lie in $[0,100]$.\footnote{Formally, then, $\textsc{cate}(X)$ is slightly less than $15$.}

In the match effect design, the conditional average treatment effect is a nontrivial function of student preferences. Here, $\chi_i$ is a student's latent taste for pilot sector schools drawn as Laplacian with mean zero and scale two. %

In the match effect design, students who prefer pilot sector schools also benefit more from them. To ensure this, we generate student preferences and counterfactual outcomes via
$u_{is} = \hat\a_s + \hat{\beta}_1^{\top}W_1 + \beta_2^{\top}W_2 + \chi_iD_s + \nu_i$ and $Y_{is} = 40 + \d_i + 15 \chi_iD_s.$
Students know in advance if they will benefit from pilot schools or not, and by how much, due to the private information $\chi_i$.
We truncate counterfactual outcomes as before.

\paragraph{\textcolor{black}{Spectrum.}} Figure~\ref{fig:eigen} supports our main assumption: the semi-synthetic student preferences, when passed through the preference kernel, have a low effective dimension. In particular, the eigenvalues decay quickly. Now, we visualize and interpret the initial eigenfunctions.

Each eigenfunction is a mapping from $25!$ possible preferences to real numbers, so we summarize them in a similar manner to Figure~\ref{fig:sub_all}. We define four binary characteristics: is the school a 90 minute commute for the student, does the school offer an English language learner program, is the school in the top half of schools in terms of its average MCAS score, and is the school majority white or Asian. For a given characteristic, we define 25 groups of student preferences $S_1,...S_{25}$. A student belongs to $S_{\rho}$ if the highest ranking a student assigns to a school with that characteristic is $\rho$. Based on these group definitions, we report group averages.

Figures~\ref{fig:e1} and~\ref{fig:e2} visualize variation of the first and second eigenfunctions via group averages. Student preferences are well approximated by relatively few types of preferences. The initial two types of preferences are characterized well by these underlying characteristics.

\begin{figure}
 \centering
\begin{subfigure}{.24\textwidth}
  \centering
  \includegraphics[width=\linewidth]{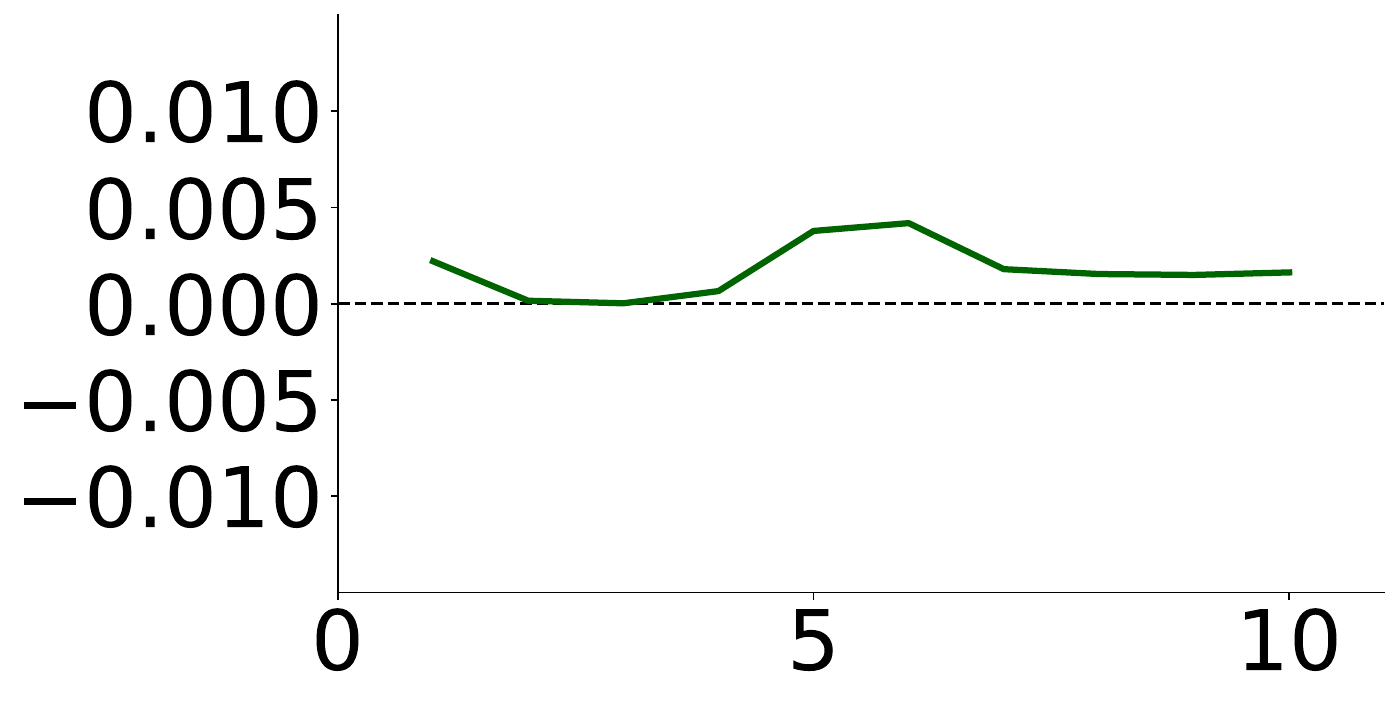}
  \caption{Commute}
  \label{fig:e1a}
\end{subfigure}
\begin{subfigure}{.24\textwidth}
  \centering
  \includegraphics[width=\linewidth]{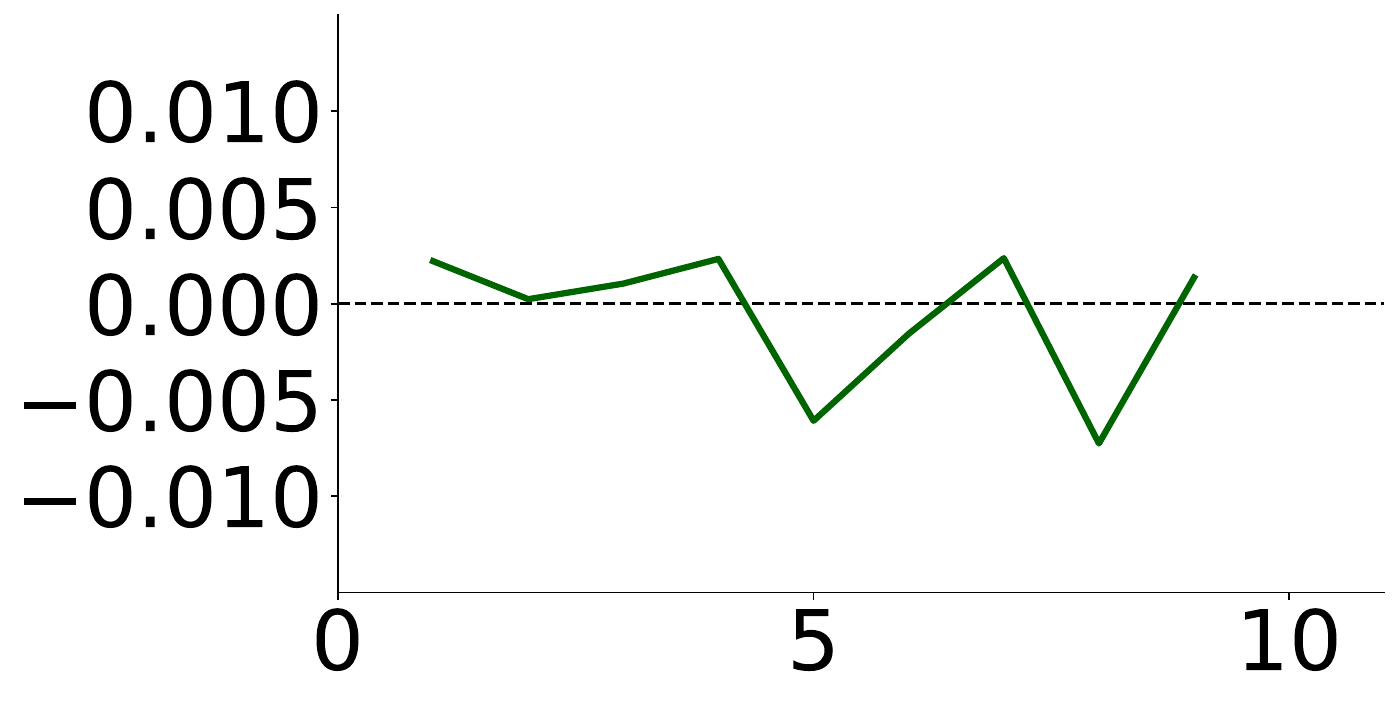}
  \caption{English learner}
  \label{fig:e1b}
\end{subfigure}
\begin{subfigure}{.24\textwidth}
  \centering
  \includegraphics[width=\linewidth]{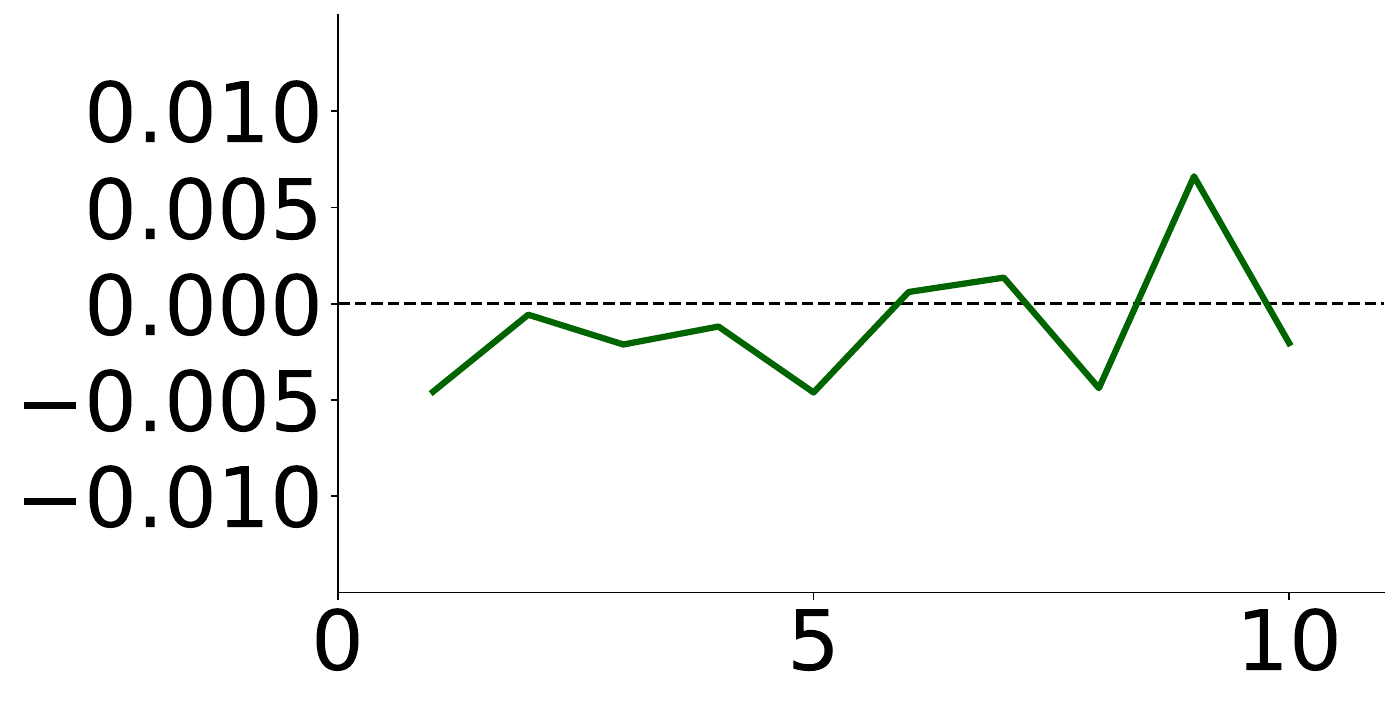}
  \caption{MCAS}
  \label{fig:e1c}
\end{subfigure}
\begin{subfigure}{.24\textwidth}
  \centering
  \includegraphics[width=\linewidth]{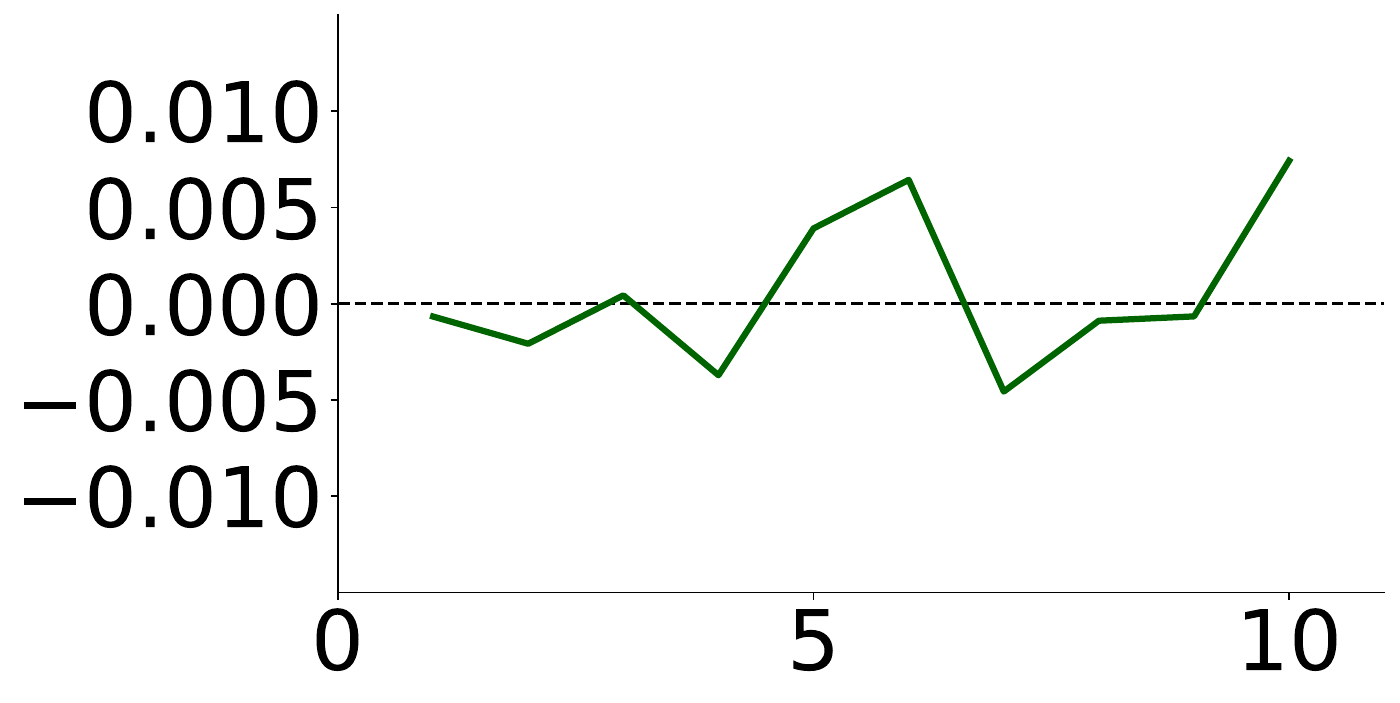}
  \caption{White/Asian}
  \label{fig:e1d}
\end{subfigure}
\caption{We visualize interpretable variation in the first eigenfunction.
}\label{fig:e1}
\end{figure}

\begin{figure}
 \centering
\begin{subfigure}{.24\textwidth}
  \centering
  \includegraphics[width=\linewidth]{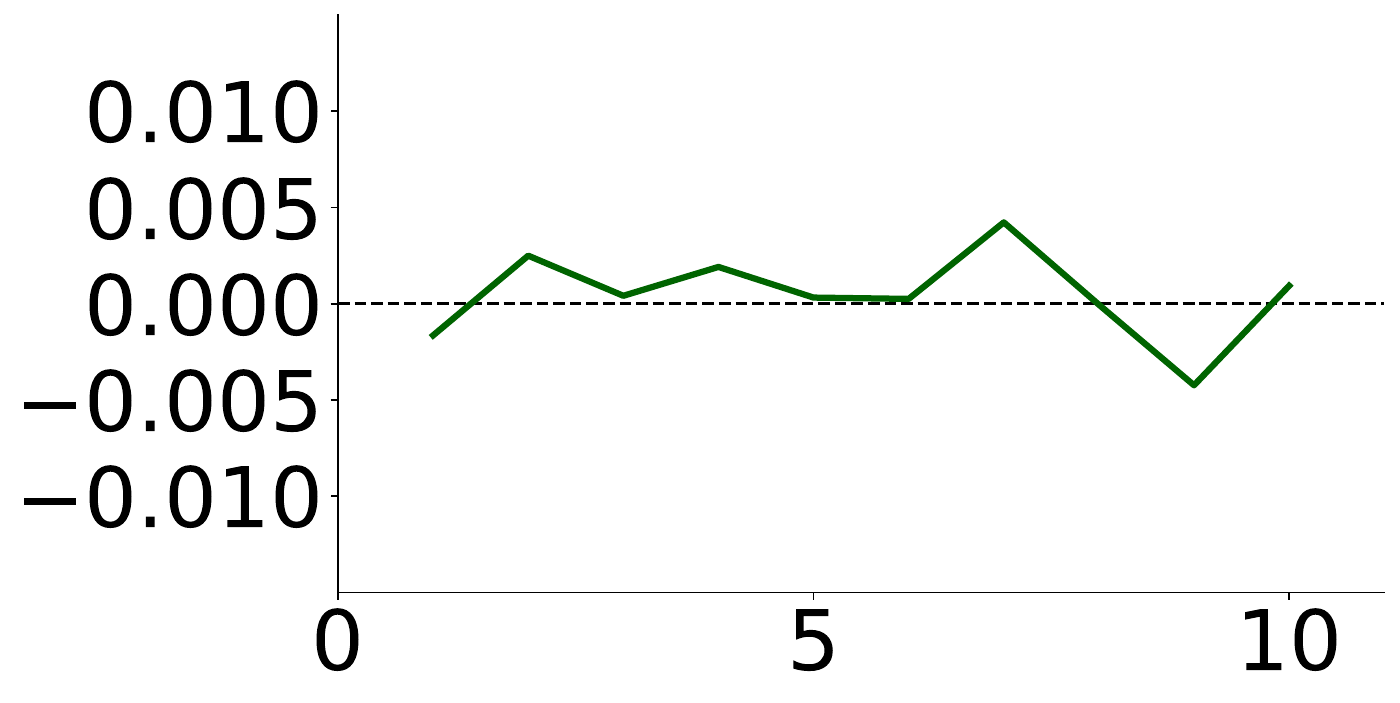}
  \caption{Commute}
  \label{fig:e2a}
\end{subfigure}
\begin{subfigure}{.24\textwidth}
  \centering
  \includegraphics[width=\linewidth]{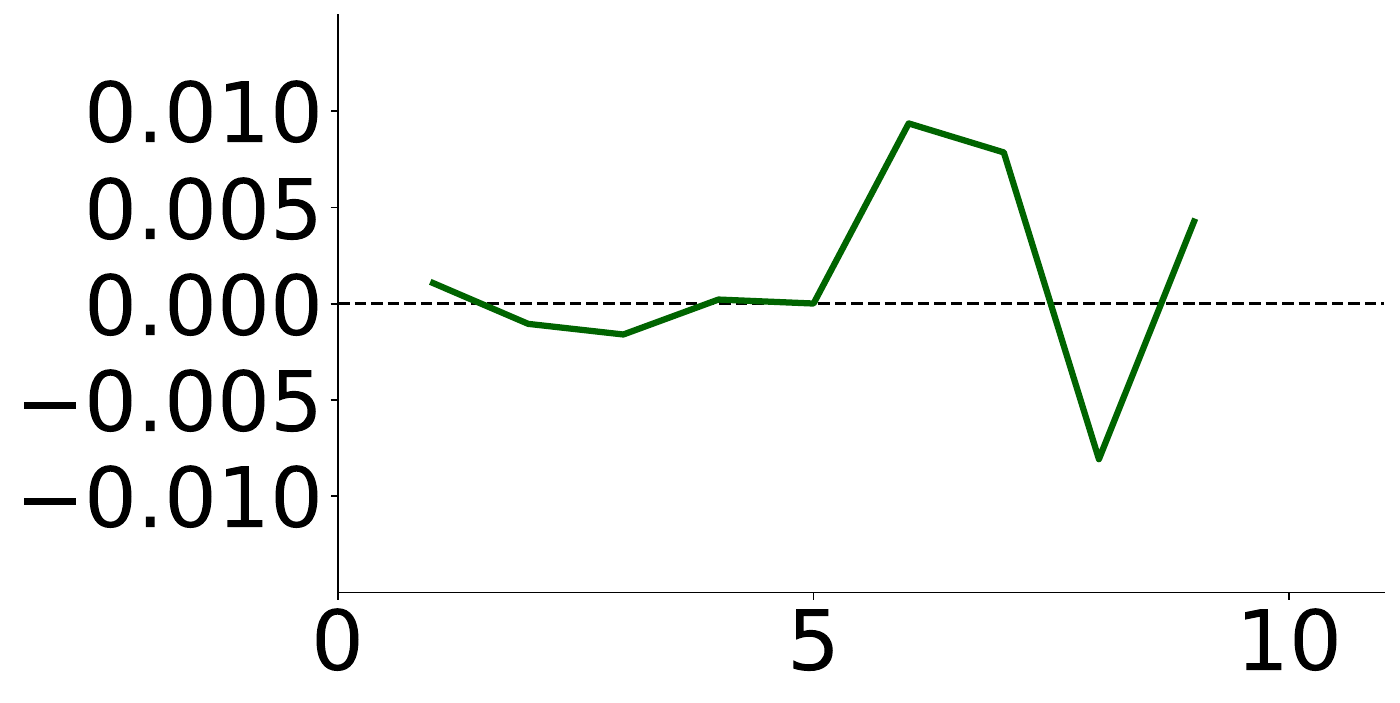}
  \caption{English learner}
  \label{fig:e2b}
\end{subfigure}
\begin{subfigure}{.24\textwidth}
  \centering
  \includegraphics[width=\linewidth]{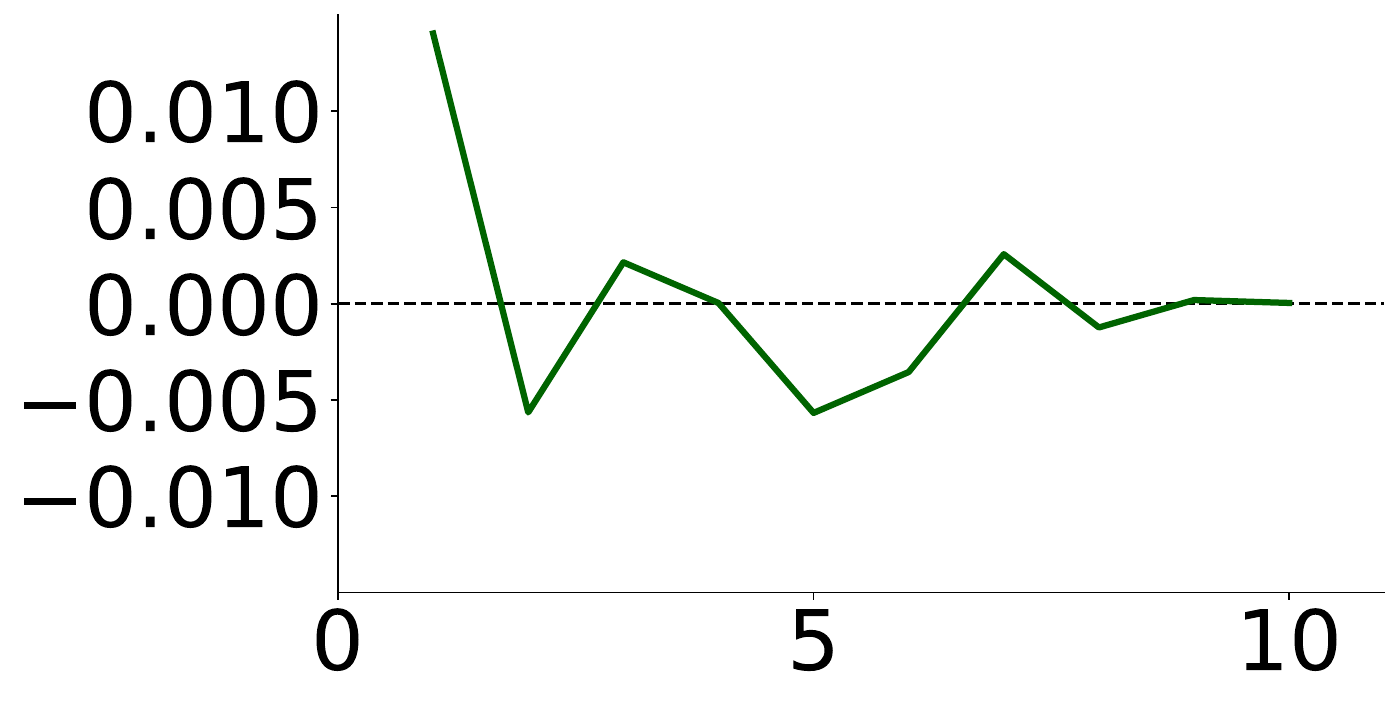}
  \caption{MCAS}
  \label{fig:e2c}
\end{subfigure}
\begin{subfigure}{.24\textwidth}
  \centering
  \includegraphics[width=\linewidth]{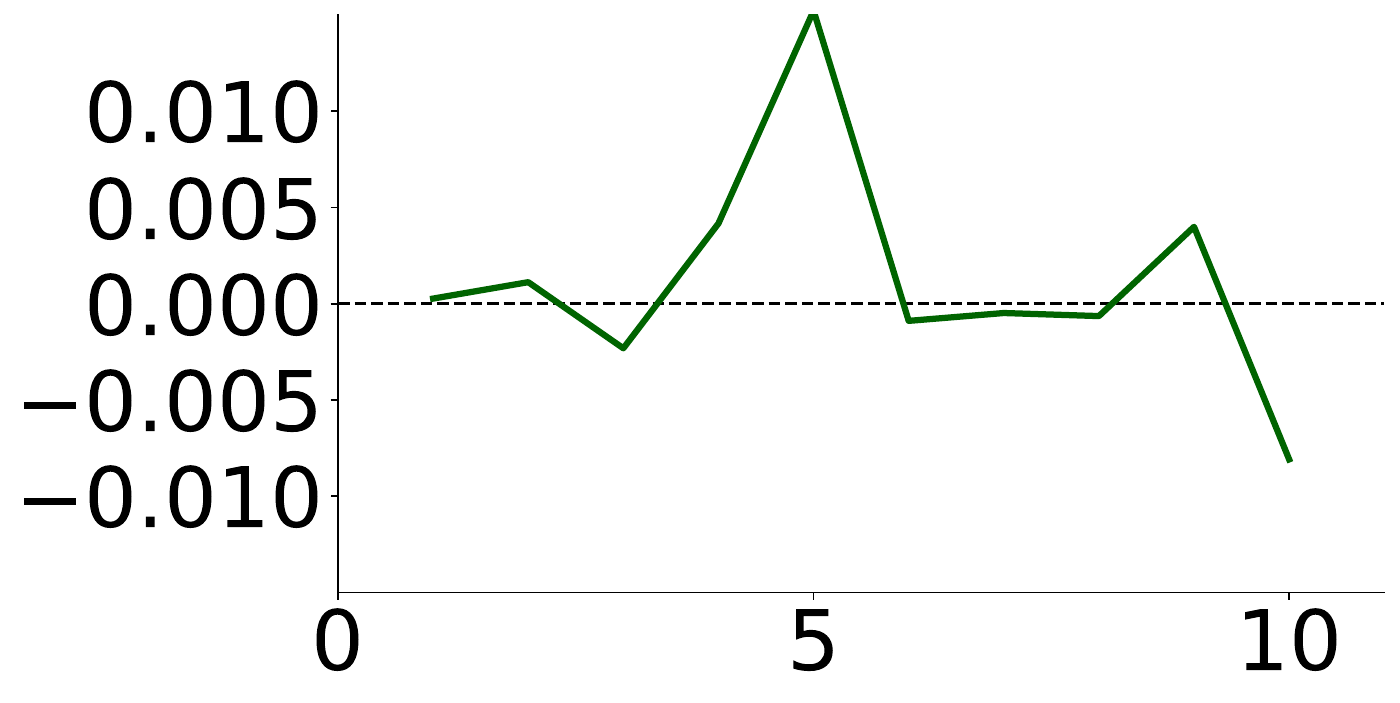}
  \caption{White/Asian}
  \label{fig:e2d}
\end{subfigure}
\caption{We visualize interpretable variation in the second eigenfunction.
}\label{fig:e2}
\end{figure}

\begin{center}
{\large\bf TERTIARY APPENDIX}
\end{center}

Appendix~\ref{sec:spectrum} supports Appendix~\ref{sec:zaitsev}, by characterizing Assumption~\ref{assumption:width}.
Appendix~\ref{sec:cov_appendix} supports Appendix~\ref{sec:cov}, by characterizing covariance estimation.
Appendix~\ref{sec:explicit} supports Appendix~\ref{sec:anti}, by characterizing Assumption~\ref{assumption:light-main}.
Appendix~\ref{sec:band} justifies variable width estimation.%
\section{Gaussian coupling details (Appendix~\ref{sec:zaitsev})}\label{sec:spectrum}

We characterize the behavior of
$
\sigma^2(m)= \sum_{s > m} \nu_s$ and $\N(\lambda)=\tr(T_{\lambda}^{-2}T)
$, where $T_{\lambda}=T+\lambda$,
under various spectral regimes. Such regimes can be deduced directly from regularity of the kernel function $k$ \citep{belkin2018approximation}.

\textbf{Spectral decay.} Define the upper incomplete gamma function $\Gamma(z,x) = \int_x^\infty u^{z-1}e^{-u}\,du.$

\begin{lemma}[{\citealt[Sec 3.1 and eq. 3.5]{natalini2000inequalities}}]\label{lem:gamma}
For $x > z > 1$  we have
$x^{z-1}e^{-x} < \Gamma(z,x) \le \frac{x^z e^{-x}}{x-z}.$
\end{lemma}

\begin{proposition}[Local width upper bound]\label{prop:spectral-integral}
Suppose $\nu_s \le \bar\nu(s): \bb{R} \to \bb{R}$ for some non-increasing positive function $\bar\nu(s)$. Then, it holds that
$
    \sigma^2(m) = \sum_{s=m+1}^{\infty} \nu_s \le \int_{m}^\infty \bar\nu(s)\,ds.
$
In particular, if $\nu_s \le \omega s^{-\beta}$ for $\beta > 1$, then we have
$
    \sigma^2(m) \le \frac{\omega m^{1-\beta}}{\beta-1}.
$
Similarly, if $\nu_s \le \omega\exp(-\alpha s^\gamma)$ for $\alpha, \omega> 0$ and $\gamma\in(0,1)$, then we have
$
    \sigma^2(m) \le C(\gamma,\alpha)\omega  m^{1-\gamma}\exp(-\alpha m^\gamma),
$
    for a constant $C(\gamma,\alpha)$ which depends only on $\alpha$ and $\gamma$ (and not on $m$).
\end{proposition}

\begin{proof}
We upper bound
$
\sigma^2(m)
    = \sum_{s=m+1}^\infty \nu_s \le \int_m^\infty \bar\nu(s)\,ds,
$ where $\bar\nu : \bb{R} \to \bb{R}$ is any non-increasing function with $\bar\nu(s) \ge \nu_s$. The first bound follows from taking $\bar\nu(s) = \omega s^{-\beta}$ and computing the resulting integral exactly. For the second bound, we take $\bar\nu(s) = \omega\exp(-\alpha s^\gamma)$ and obtain
$
    \frac{\sigma^2(m)}{\omega} \le \int_{m}^\infty \exp(-\alpha s^\gamma)\,ds.
$
To bound the right hand side, we proceed in steps.
\begin{enumerate}
    \item Making the substitution $u =  \alpha s^\gamma$ gives us
$$
\int_{m}^\infty \exp(-\alpha s^\gamma)\,ds =\gamma^{-1}\alpha^{\frac{-1}{\gamma}} \int_{\alpha m^\gamma}^\infty \exp(-u) u^{\nicefrac{(1-\gamma)}{\gamma}} \,du.
$$
\item We have
$
\int_{\alpha m^\gamma}^\infty \exp(-u) u^{\nicefrac{(1-\gamma)}{\gamma}} \,du=\Gamma(1/\gamma, \alpha m^\gamma)
$. By Lemma \ref{lem:gamma},
$
\Gamma(z,x) \leq \frac{x^ze^{-x}}{x-z}
$:
\begin{align*}
 \frac{\sigma^2(m)}{\omega}
 \le
 \gamma^{-1}\alpha^{\frac{-1}{\gamma}}
 \Gamma\Big(1/\gamma, \alpha m^\gamma\Big) \le
    \gamma^{-1}\alpha^{\frac{-1}{\gamma}}
\frac{(\alpha m^\gamma)^{1/\gamma}e^{-\alpha m^\gamma}}{\alpha m^\gamma-1/\gamma}
    =
    \frac{m\exp(-\alpha m^\gamma)}{\alpha\gamma m^\gamma-1}.
\end{align*}

\item Finally we absorb constants.
When $m \ge \{2/(\alpha\gamma)\}^{1/\gamma}$, we have that $\a\gamma m^\gamma \ge 2$ and hence $\a\gamma m^\gamma - 1 \ge \a\gamma m^\gamma/2$. Thus, the final expression is at most $(2/\a\gamma)m^{1-\gamma}\exp(-\alpha m^\gamma)$.
For $m < \{2/(\alpha\gamma)\}^{1/\gamma}$, we decompose the sum as
$
\sum_{s=m+1}^{\lfloor\{2/(\alpha\gamma)\}^{1/\gamma}\rfloor} \nu_s+\sum_{s=\lceil\{2/(\alpha\gamma)\}^{1/\gamma}\rceil}^{\infty} \nu_s.
$
The latter sum is bounded by the previous case. Focusing on the former sum,
$$
\sum_{s=m+1}^{\{2/(\alpha\gamma)\}^{1/\gamma}} \nu_s
\leq \sum_{s=1}^{\{2/(\alpha\gamma)\}^{1/\gamma}} \nu_s
\leq \{2/(\alpha\gamma)\}^{1/\gamma} \nu_1\leq \{2/(\alpha\gamma)\}^{1/\gamma} \cdot  \omega\exp(-\alpha)
$$
which absorbs into a constant depending only on $\omega, \alpha, \gamma$ that is linear in $\omega$.\qedhere
\end{enumerate}
\end{proof}

\paragraph{Complexity measures.} Recall the definition of the ellipse $\mathfrak{E} = \Sigma^{\f 1 2}B$, where $B$ is the unit ball in $L^2(\bb P)$. We use the results of \cite{wei2020gauss} to briefly sketch the claim that, under suitable regularity conditions, the local width $\sigma(\Sigma,m)$ is roughly comparable to the local Gaussian complexity of $\mathfrak{E}$ at scale $\d = \mathrm{ent}_m(\mathfrak{E})$, which we denote by $\mathcal{G}(\mathfrak{E} \cap \d B)$, defined below.

First, some technicalities. For $S_m \subset \bb{R}^m$, the Gaussian complexity is given by
 $\c G(S_m) = \bb{E}\,\sup_{t \in S_m}|\bk{g}{s}|,$ respectively, where $g$ is an isotropic Gaussian vector. This generalizes to compact subsets of a separable Hilbert space by finite-dimensional approximation. Moreover, when $S_m$ contains the origin (as in our setting) they are equivalent up to a constant \citep[Chapter 7]{vershynin2018high}. Thus, as is standard, we can use the results of \citet{wei2020gauss}, who study the Gaussian width of finite-dimensional sets, to characterize the Gaussian complexity in our setting.

\begin{lemma} Suppose either (i) $\nu_s(\Sigma) \asymp s^{-\beta}$ for $\beta > 1$; or
    (ii) $\nu_s(\Sigma) \asymp \exp(-\alpha s^\gamma)$ for $\alpha > 0$ and $\gamma \in (0,1)$.
For sufficiently small $\d = \mathrm{ent}_m(\mathfrak{E})$, we have $\sigma(\Sigma,m) \lesssim \c G(\mathfrak{E} \cap \d B)$.
\end{lemma}
\begin{proof}
    Both cases were covered by \citet[Example 4]{wei2020gauss}, and the subsequent discussion in Section 4 of that paper.

    In particular, it was shown that given a scale parameter $\d$, one can compute the ``critical dimension'' $m^*$ as the smallest $m$ such that $\sqrt{\nu_m(\Sigma)} \le 9\d/10$. Using regularity conditions which were checked in the aforementioned Example 4, \citet[Theorem 2]{wei2020gauss} shows that for small enough $\d$, it holds that
    $\d\sqrt{m^*} \lesssim \c G (\mathfrak{E} \cap \d B).$
    Using our bounds in Proposition \ref{prop:spectral-integral},
    $\d\sqrt{m^*} \gtrsim \sqrt{m^* \nu_{m^*}(\Sigma)} \gtrsim \sigma(\Sigma,m^*)$
    in both of the above examples.

    Next, \citet[Corollary 2]{wei2020gauss} implies, after inverting the entropy function, that in the same setting
    $\d \lesssim \mathrm{ent}_{m^*}(\mathfrak{E} \cap \d B) \le \mathrm{ent}_{m^*}(\mathfrak{E}) = \d'.$
    Our sketch is then complete by monotonicity of $\c G$ since
    $\sigma(\Sigma,m^*) \lesssim \c G (\mathfrak{E} \cap \d B) \le \c G(\mathfrak{E} \cap \d' B).$
\end{proof}

\paragraph{Effective dimension.} In this subsection, we will derive bounds on the quantities
$\psi(m,c) = \sum_{s=m+1}^\infty \frac{\nu_s}{(\nu_s+\lambda)^c},$ where $c \ge 1$. Such quantities arise frequently in studies of ridge regression, e.g. the ``effective dimension'' in \cite{caponnetto2007optimal}. According to our definition, $\N(\lambda) = \psi(0,2)$. We provide matching upper and lower bounds.

\begin{proposition}[Upper bound on effective dimension]
    \label{prop:effdim}
    Suppose $\nu_s \le \bar\nu(s): \bb{R} \to \bb{R}$ for some non-increasing positive function $\bar\nu(s)$, and let $c > 1$ be given. Then
   $$
       \psi(m,c) \le \lambda^{1-c} \inf_{s \ge m} \left\{ (s-m+1) + \frac{1}{\lambda}\int_{s}^\infty \bar\nu(t)\,dt\right\}.
  $$
    In particular, if $\nu_s \le \omega s^{-\beta}$ for some $\omega > 0$ and $\beta > 1$, and if $\lambda \leq \omega$,
   $$
        \psi(m,c) \le \frac{(\omega\vee \omega^{1/\beta}) (2\beta-1)}{\beta-1}\left(\frac{1}{\lambda^{c+1/\beta-1}}\wedge \frac{ m^{1-\beta}}{\lambda^c}\right).
   $$
    If $\nu_s \le \omega e^{-\alpha s^\gamma}$ for some  $\alpha, \omega, \gamma > 0$, and if $\lambda\leq \omega/e^{\alpha}$, then
   $$
        \psi(m,c) \le C(\alpha,\omega,\gamma) \left[ \frac{1}{\lambda^{c-1}} \left\{\frac{\log(\omega/\lambda)}{\alpha}\right\}^{1/\gamma} \wedge \frac{m^{1-\gamma}\exp(-\alpha m^\gamma)}{\lambda^c}\right]
   $$
    for some constant $C(\alpha,\omega,\gamma)$ which depends only on $\alpha$, $\omega$, and $\gamma$.
\end{proposition}

\begin{proof}
We proceed in steps.

\begin{enumerate}
    \item First we note
    $
    \frac{\nu_s}{(\nu_s + \lambda)^c} \le \frac{\nu_s}{\lambda^c}$ and $ \frac{\nu_s}{(\nu_s + \lambda)^c} \le \lambda^{1-c}.
    $
    The former holds since $\nu_s \ge 0$.
    The latter follows from maximizing the function $f(u)=\frac{u}{(u+ \lambda)^c}$ at $u=\lambda/(c-1)$.
    \item Combining the two, we have for any $k \geq m$ that
    \begin{align*}
       \psi(m,c)
       &=\sum_{s=m+1}^\infty \frac{\nu_s}{(\nu_s + \lambda)^c}
       \le \sum_{s=m+1}^k \lambda^{1-c} + \sum_{s=k+1}^\infty \frac{\nu_s}{\lambda^c}
       \le \lambda^{1-c} \left\{ (k - m ) + \frac{1}{\lambda}\int_{k}^\infty \bar\nu(t)\,dt\right\}.
    \end{align*}
    Since $k$ was arbitrary, we can further deduce that
    \begin{align*}
      \psi(m,c)
        &\le \lambda^{1-c} \inf_{k \geq m} \left\{ (k - m ) + \frac{1}{\lambda}\int_{k}^\infty \bar\nu(t)\,dt\right\}
       =\lambda^{1-c} \inf_{s \geq m } \left\{ (\lceil s \rceil - m ) + \frac{1}{\lambda}\int_{\lceil s \rceil}^\infty \bar\nu(t)\,dt\right\}
        \\
        &\le\lambda^{1-c} \inf_{s \geq m } \left\{ (\lceil s \rceil-m) + \frac{1}{\lambda}\int_{s}^\infty \bar\nu(t)\,dt\right\},
    \end{align*}
    where the latter infimum is over real numbers $s$, and $\lceil s \rceil$ rounds $s$ up to the nearest integer.

    \item Optimizing the bounds. In the infimum, consider $s=m$. Then
    $
    \psi(m,c) \leq \lambda^{-c} \int_{m}^\infty \bar\nu(t)\,dt.
    $
    Alternatively, consider the choice $s^*=\bar\nu^{-1}(\lambda) \vee m$. Then
    $$
    \psi(m,c) \le \lambda^{1-c}\left\{1 + s^* - m + \frac{1}{\lambda}\int_{s^*}^\infty \bar\nu(t)\,dt\right\}.
    $$
    In particular, choosing $m=0$ and $s^*=\bar\nu^{-1}(\lambda)$, we see
     $$
    \psi(m,c) \le \psi(0,c) \le \lambda^{1-c}\left\{ 1 + \bar\nu^{-1}(\lambda) + \frac{1}{\lambda}\int_{\bar\nu^{-1}(\lambda)}^\infty \bar\nu(t)\,dt\right\}.
    $$
     Combining both bounds,
    $$\psi(m,c)  \le \lambda^{1-c}\left\{1 + \bar\nu^{-1}(\lambda)+ \frac{1}{\lambda}\int_{\bar\nu^{-1}(\lambda)}^\infty \bar\nu(t)\,dt\right\} \wedge  \frac{1}{\lambda^c}\int_m^\infty \bar\nu(t)\,dt.$$

    \item Polynomial decay. Now, if $\bar\nu(s) = \omega s^{-\beta}$ then we have  $\bar\nu^{-1}(\lambda) = (\omega /\lambda)^{1/\beta}$, and as argued in Proposition~\ref{prop:spectral-integral},
    $\int_s^\infty \bar\nu(t)\,dt = \frac{\omega s^{1-\beta}}{\beta-1}.$
Plugging this into the first term in the optimized bound,
$$
    \frac{1}{\lambda}\int_{\bar\nu^{-1}(\lambda)}^\infty \bar\nu(t)\,dt
    = \frac{1}{\lambda} \frac{\omega \{(\omega /\lambda)^{1/\beta}\}^{1-\beta}}{\beta-1}
    =\lambda^{-1-(1-\beta)/\beta} \omega^{1+(1-\beta)/\beta} \frac{1}{\beta-1}
    =\lambda^{-1/\beta}\cdot \omega^{1/\beta}\cdot (\beta-1)^{-1}
    $$
    so that, when $\omega\geq\lambda$,
    \begin{align*}
        &\lambda^{1-c}\left\{ 1+\bar\nu^{-1}(\lambda)+ \frac{1}{\lambda}\int_{\bar\nu^{-1}(\lambda)}^\infty \bar\nu(t)\,dt\right\}
        =\lambda^{1-c}\left\{ 1+(\omega /\lambda)^{1/\beta} + \lambda^{-1/\beta}\cdot \omega^{1/\beta}\cdot (\beta-1)^{-1} \right\} \\
        &\leq \lambda^{1-c}\left\{ 2(\omega /\lambda)^{1/\beta} + \lambda^{-1/\beta}\cdot \omega^{1/\beta}\cdot (\beta-1)^{-1} \right\}
        =\left(2\omega^{1/\beta} + \frac{\omega^{1/\beta}}{{\beta-1}}\right)\lambda^{1-c-1/\beta} \\
        &=\frac{2\beta-1}{\beta-1}\omega^{1/\beta}\lambda^{1-c-1/\beta}.
    \end{align*}
Therefore the overall bound is
    $$\frac{2\beta-1}{\beta-1}\omega^{1/\beta}\lambda^{1-c-1/\beta} \wedge \frac{\omega m^{1-\beta}}{\lambda^c(\beta-1)}\leq \frac{(\omega\vee \omega^{1/\beta}) (2\beta-1)}{\beta-1}\left(\frac{1}{\lambda^{c+1/\beta-1}}\wedge \frac{ m^{1-\beta}}{\lambda^c}\right).$$

    \item Exponential decay. If $\bar\nu(t) \le \omega\exp(-\alpha t^{\gamma})$ then we have $\bar\nu^{-1}(\lambda) = \{\log(\omega/\lambda)/\alpha\}^{1/\gamma}$ and, as argued in Proposition~\ref{prop:spectral-integral},
    $
    \int_s^\infty \bar\nu(t)\,dt \leq  C \omega s^{1-\gamma}\exp(-\alpha s^\gamma).
    $
    Plugging this into the first term in the optimized bound,
    \begin{align*}
        &\frac{1}{\lambda}\int_{\bar\nu^{-1}(\lambda)}^\infty \bar\nu(t)\,dt
    =
    \frac{1}{\lambda}
    C\omega [\{\log(\omega/\lambda)/\alpha\}^{1/\gamma}]^{1-\gamma} \cdot \exp(-\alpha[\{\log(\omega/\lambda)/\alpha\}^{1/\gamma}]^\gamma) \\
    &=\frac{1}{\lambda} C\omega \{\log(\omega/\lambda)/\alpha\}^{(1-\gamma)/\gamma} \cdot  \exp(-\alpha\{\log(\omega/\lambda)/\alpha\})  \\
    &= \frac{1}{\lambda} C\omega \{\log(\omega/\lambda)/\alpha\}^{(1-\gamma)/\gamma} \cdot \frac{\lambda}{\omega}
    =C\{\log(\omega/\lambda)/\alpha\}^{(1-\gamma)/\gamma}
      \end{align*}
    so that, when $\omega/e^{\alpha}\geq\lambda$,
    \begin{align*}
        &\lambda^{1-c}\left\{1+ \bar\nu^{-1}(\lambda)+ \frac{1}{\lambda}\int_{\bar\nu^{-1}(\lambda)}^\infty \bar\nu(t)\,dt\right\}
        =\lambda^{1-c}\left[ 1+\{\log(\omega/\lambda)/\alpha\}^{1/\gamma} + C\{\log(\omega/\lambda)/\alpha\}^{(1-\gamma)/\gamma} \right] \\
        &\leq \lambda^{1-c}\left[ 2\{\log(\omega/\lambda)/\alpha\}^{1/\gamma} + C\{\log(\omega/\lambda)/\alpha\}^{(1-\gamma)/\gamma} \right]
        \leq C'\lambda^{1-c}\{\log(\omega/\lambda)/\alpha\}^{1/\gamma}
    \end{align*}
    since $\gamma>0$ implies $1/\gamma> (1-\gamma)/\gamma$. Therefore the overall bound is
    \begin{align*}
        &C'\lambda^{1-c}\{\log(\omega/\lambda)/\alpha\}^{1/\gamma}  \wedge \frac{1}{\lambda^c} C \omega m^{1-\gamma}\exp(-\alpha m^\gamma) \\
        &\leq  C'\omega \left[ \frac{\{\log(\omega/\lambda)/\alpha\}^{1/\gamma}}{\lambda^{c-1}} \wedge \frac{m^{1-\gamma}\exp(-\alpha m^\gamma)}{\lambda^c} \right]. & \qedhere
    \end{align*}
\end{enumerate}
\end{proof}

\begin{proposition}[Lower bound on effective dimension]\label{prop:eff_dim_lb}
Suppose $\nu_s \asymp \underline{\nu}(s): \bb{R} \to \bb{R}$ for some non-increasing positive function $\underline{\nu}(s)$. Then, if $s^*$ is the smallest positive integer with $\underline\nu(s^*) \le \lambda$,
\[\psi(m,2) \gtrsim \int_{ s^*  \vee (m+1)}^\infty \frac{\underline \nu(s)}{(\underline \nu(s)+\lambda)^2}\,ds \\
\ge \frac{1}{4\lambda^2}  \int_{s^* \vee (m+1)}^\infty \underline \nu(s)\,ds.\]
Moreover, if we have $\underline{\nu}(s) = \omega s^{-\beta}$ for $\beta>1$ (polynomial decay), then, whenever $\lambda \le \omega$,
\[\psi(0,2) \gtrsim \frac{1}{c({\beta})}\left(\frac{\omega^{1/\beta}}{\lambda^{1+1/\beta}}\right).\]
If $\underline\nu(s) = \omega \exp(-\alpha s^{\gamma})$ for $\alpha>0$ and $\gamma\in(0,1)$ (exponential decay), then, whenever $\lambda\leq \omega/e^{\alpha}$,
\[\psi(0,2) \gtrsim \frac{1}{c({\a,\gamma})} \left\{\frac{\log(\omega/\lambda)^{(1-\gamma)/\gamma}}{\lambda}\right\}.\]
\end{proposition}

\begin{proof}
We proceed in steps.

\begin{enumerate}
    \item Note that the function $t \mapsto (t+\lambda)^{-2}t$ is increasing for all $t < \lambda$ while $\underline\nu(s)$ is non-increasing, so their composition is non-increasing. Thus, for each integer $s$ such that $\underline \nu(s) \le \lambda$,
\[\frac{\nu_s}{(\nu_s+\lambda)^2} \gtrsim \int_{s}^{s+1} \frac{\underline \nu(t)}{\{\underline \nu(t)+\lambda\}^2}\,dt.\]
So we may bound the sum by an integral to show
\begin{equation*}
\psi(m,2) \gtrsim \sum_{s=m+1}^\infty \frac{\nu_s}{(\nu_s+\lambda)^2}
 \ge \int_{ s^*  \vee (m+1)}^\infty \frac{\underline \nu(t)}{\{\underline \nu(t)+\lambda\}^2}\,dt
 \ge \frac{1}{4\lambda^2}  \int_{s^* \vee (m+1)}^\infty \underline \nu(t)\,dt,
\end{equation*}
where $s^*$ is the smallest positive integer such that $\underline\nu(s) \le \lambda$.

    \item     Polynomial decay.
If $\underline\nu(s) = \omega s^{-\beta}$ then $s^* \le (\lambda/\omega)^{-1/\beta} + 1$.
Computing the integral as in the proof of Proposition~\ref{prop:spectral-integral},
\begin{align*}
\psi(0,2)
&\gtrsim \frac{1}{4\lambda^2} \int_{s^*}^\infty \omega s^{-\beta}\,ds = \frac{\omega(\beta-1)^{-1}}{4\lambda^2}(s^*)^{1-\beta}
\ge \frac{\omega(\beta-1)^{-1}}{4\lambda^2}\{(\lambda/\omega)^{-1/\beta} + 1\}^{1-\beta} \\
 &\ge \frac{(\beta-1)^{-1}\omega^{1/\beta}}{2^{1+\beta}\lambda^{1+1/\beta}}
 \ge \frac{\omega^{1/\b}}{c(\beta)}\left(\frac{1}{\lambda^{1+1/\beta}}\right)
\end{align*}
since $\lambda \le \omega$ implies that the bracketed term is at most $2(\lambda/\omega)^{-1/\beta}$.
    \item Exponential decay. If $\underline\nu(s) = \omega \exp(-\alpha s^{\gamma})$ then we have
\begin{equation*}
  \{\log(\omega/\lambda)/\alpha\}^{1/\gamma}
  \le
  s^*
  \le
  \{\log(\omega/\lambda)/\alpha\}^{1/\gamma} + 1
   \le
  \{\log(\omega/\lambda)/\alpha+1\}^{1/\gamma}
\end{equation*} following from the fact that $(a + b)^{1/\gamma} \ge a^{1/\gamma} + b^{1/\gamma}$ by Jensen's inequality for $\gamma \in (0,1)$.
Making the substitution $u = \alpha s^\gamma$ as in Proposition~\ref{prop:spectral-integral}, using $\Gamma(z,x) \ge x^{z-1}e^{-x}$ for $z > 1$ from Lemma~\ref{lem:gamma}, and appealing to the bounds on $s^*$,
\begin{align*}
&4\lambda^2 \cdot \psi(0,2)
\ge \int_{s^*}^\infty \omega \exp(-\alpha s^{\gamma}) \,ds
=\omega \gamma^{-1}\alpha^{\frac{-1}{\gamma}} \int_{\alpha (s^*)^\gamma}^\infty \exp(-u) u^{\nicefrac{(1-\gamma)}{\gamma}} \,du  \\
&= \omega \gamma^{-1}\alpha^{\frac{-1}{\gamma}} \Gamma\{1/\gamma, \alpha (s^*)^\gamma\}
\ge \omega \gamma^{-1}\alpha^{\frac{-1}{\gamma}} \{\alpha (s^*)^\gamma\}^{1/\gamma-1}e^{-\alpha (s^*)^\gamma} \\
&=\omega \gamma^{-1} \alpha^{-1} (s^*)^{1-\gam} \exp\{-\alpha(s^*)^\gamma\} \\
&\geq \omega \gamma^{-1} \alpha^{-1} \left[\{\log(\omega/\lambda)/\alpha\}^{1/\gamma} \right]^{1-\gam} \exp(-\alpha[\{\log(\omega/\lambda)/\alpha + 1\}^{1/\gamma}]^\gamma)\\
&=   \omega \gamma^{-1}\alpha^{-1}  \{\log(\omega/\lambda)/\alpha\}^{(1/\gamma) - 1} \frac{\lambda}{\omega}e^{-\alpha}
=\lambda \gamma^{-1} \alpha^{-1} e^{-\alpha}  \{\log(\omega/\lambda)/\alpha\}^{(1/\gamma) - 1}.
\end{align*}
After rearranging, we obtain
$\psi(0,2) \gtrsim \frac{1}{c({\a,\gamma})} \left\{\frac{\log(\omega/\lambda)^{(1-\gamma)/\gamma}}{\lambda}\right\}.$ \qedhere
\end{enumerate}
\end{proof}
A nearly identical argument bounds a quantity useful for anti-concentration.

\begin{lemma}[Lower bound on Hilbert-Schmidt  norm]\label{lem:frob-lower} With the same conditions and notation as Proposition~\ref{prop:eff_dim_lb} we have
\[\tilde\psi \coloneqq\sum_{s=1}^\infty \left\{\frac{\nu_s}{(\nu_s + \lambda)^2}\right\}^2 \ge \frac{1}{16\lambda^4}\int_{s^*}^\infty \underline{\nu}(t)^2\,dt.\]
Under polynomial decay, this implies $\tilde\psi \gtrsim c(\beta,\omega)\lambda^{-2-1/\beta}$, while under exponential decay we have $\tilde\psi \gtrsim c(\alpha,\gamma,\omega) \lambda^{-2}\log(\omega/\lambda)^{1/\gamma - 1}$.
\end{lemma}
\begin{proof}
Since $(t+\lambda)^{-4}t^2$ increases if and only if $(t+\lambda)^{-2}t$ does, the first statement follows exactly as before. We then carry out the computation in both leading cases.
\begin{enumerate}

    \item     Polynomial decay.
If $\underline\nu(s) = \omega s^{-\beta}$ then $s^* \le (\lambda/\omega)^{-1/\beta} + 1$.
Computing the integral as in the proof of Proposition~\ref{prop:spectral-integral},
\begin{align*}
\tilde\psi
&\gtrsim \frac{1}{16\lambda^4} \int_{s^*}^\infty \omega^2 s^{-2\beta}\,ds = \frac{\omega^2(2\beta-1)^{-1}}{16\lambda^4}(s^*)^{1-2\beta}
\ge \frac{c(\omega,\beta)}{\lambda^4}\{(\lambda/\omega)^{-1/\beta} + 1\}^{1-2\beta} \ge \frac{c(\omega,\beta)}{\lambda^{2+1/\beta}}
\end{align*}
since $\lambda \le \omega$ implies that the bracketed term is at most $2(\lambda/\omega)^{-1/\beta}$.

    \item Exponential decay. If $\underline\nu(s) = \omega \exp(-\alpha s^{\gamma})$, making the substitution $u = 2\alpha s^\gamma$ as in Proposition~\ref{prop:spectral-integral}, using $\Gamma(z,x) \ge x^{z-1}e^{-x}$ for $z > 1$ from Lemma~\ref{lem:gamma}, and appealing to the bounds on $s^*$ given in the proof of Proposition~\ref{prop:eff_dim_lb},
\begin{align*}
16\lambda^4\tilde\psi
&\gtrsim \int_{s^*}^\infty \omega^2 \exp(-2\alpha s^{\gamma}) \,ds
=\omega^2 \gamma^{-1}(2\alpha)^{\frac{-1}{\gamma}} \int_{2\alpha (s^*)^\gamma}^\infty \exp(-u) u^{\nicefrac{(1-\gamma)}{\gamma}} \,du  \\
&= \omega^2 \gamma^{-1}(2\alpha)^{\frac{-1}{\gamma}} \Gamma\{1/\gamma, 2\alpha (s^*)^\gamma\}
\ge \omega^2 \gamma^{-1}(2\alpha)^{\frac{-1}{\gamma}} \{2\alpha (s^*)^\gamma\}^{1/\gamma-1}e^{-2\alpha (s^*)^\gamma} \\
&=c(\omega, \gamma,\alpha) (s^*)^{1-\gam} \exp\{-2\alpha(s^*)^\gamma\} \\
&\geq c'(\omega, \gamma,\alpha) \left[\{\log(\omega/\lambda)/\alpha\}^{1/\gamma} \right]^{1-\gam} \exp(-2\alpha[\{\log(\omega/\lambda)/\alpha+1\}^{1/\gamma}]^\gamma)\\
&=  c''(\omega, \gamma,\alpha)  \{\log(\omega/\lambda)/\alpha\}^{(1/\gamma) - 1} \frac{\lambda^2}{\omega^2}e^{-2\alpha}.  & \qedhere
\end{align*}
\qedhere
\end{enumerate}
\end{proof}

\textbf{Leading cases.}
\textcolor{black}{Suppose the eigenvalues satisfy $\nu_m \le \omega m^{-\beta}$ for some $\beta > 1$. The following results specialize the bounded-summand coupling.}

In proving these results, we freely use the following lemma to simplify calculations by optimizing over real numbers instead of integers.
\begin{lemma}
Let $\underline{m} \le m \le \bar m$ be positive integers, and let $\psi_{\uparrow}(u)$ and $\psi_{\downarrow}(u)$  be increasing and decreasing functions of $u \in \mathbb{R}$. Finally, suppose $\psi_{\uparrow}(u+1) \lesssim  \psi_{\uparrow}(u)$ for $u \ge1$. Then
\[\min_{\underline{m} \le m \le \bar m} A_1 \psi_{\downarrow}(m) + A_2 \psi_{\uparrow}(m) \lesssim \inf_{t \in [\underline{m} ,\bar m]} A_1 \psi_{\downarrow}(t) + A_2 \psi_{\uparrow}(t).\]
\end{lemma}
\begin{proof}
Suppose $t^*$ attains the infimum on the right hand side. Then
\begin{align*}
    &\min_{\underline{m} \le m \le \bar m} A_1 \psi_{\downarrow}(m) + A_2 \psi_{\uparrow}(m)
    \le  A_1 \psi_{\downarrow}(\lceil t^* \rceil) + A_2 \psi_{\uparrow}(\lceil t^* \rceil) \\
    & \qquad \lesssim A_1 \psi_{\downarrow}(\lceil t^* \rceil) + A_2\psi_{\uparrow}(\lceil t^* \rceil-1)
    \le A_1 \psi_{\downarrow}(t^* ) + A_2\psi_{\uparrow}( t^*)
\end{align*}
follows by using $t \in [\underline{m} ,\bar m]$, $\psi_{\uparrow}(u+1) \lesssim  \psi_{\uparrow}(u)$, and monotonicity.
\end{proof}
\begin{corollary}[Sobolev RKHS, bounded data]\label{cor:1}
{\color{black}
Suppose that $\nu_m \le \omega m^{-\beta}$ for some $\beta > 1$ and $\|U_i\|\leq a$. Then there exists a Gaussian $Z$ with covariance $\Sigma$ such that, with probability at least $1-\eta$,
\[
\left\|\left(\frac{1}{\sqrt{n}} \sum_{i=1}^n U_i\right)-Z\right\|
\lesssim_\beta
\omega^{\frac{1}{2\beta}}a^{\frac{\beta-1}{\beta}}
\{\log(6/\eta)\}^{\frac{1}{2\beta}}
\left\{\frac{32+2\log_2(n)}{n\eta}\right\}^{\frac{\beta-1}{2\beta}}.
\]
}
\end{corollary}

\begin{proof}
\color{black}
In this case $\sigma^2(m) \lesssim_\beta \omega m^{1-\beta}$ by Proposition~\ref{prop:spectral-integral}. Theorem~\ref{thm:technical} therefore gives, for any $m$,
\[
\left\|\left(\frac{1}{\sqrt n}\sum_{i=1}^nU_i\right)-Z\right\|
\lesssim_\beta
\sqrt{\log(6/\eta)}\,\omega^{1/2}m^{(1-\beta)/2}
+a\sqrt{\frac{m\{32+2\log_2(n)\}}{n\eta}}.
\]
Balancing the two terms gives
\[
m=\left\{\frac{\omega n\eta\log(6/\eta)}
{a^2\{32+2\log_2(n)\}}\right\}^{1/\beta}.
\]
Substitution, together with the preceding integer-rounding lemma, gives the stated rate.
\end{proof}

Now, suppose $\sigma(m)$ is exponentially decaying. We obtain a $\sqrt{n}$ rate of Gaussian approximation up to logarithmic factors. Notably, this includes the case where $U_i = k_{X_i}$ for $k$ a smooth, radial kernel on $\bb{R}^d$ \citep{belkin2018approximation, wendland2004scattered}.

\begin{corollary}[Gaussian RKHS, bounded data]\label{cor:3}
{\color{black}
Suppose that $\nu_m \leq \omega \exp(-\a m^\gamma)$ for some $\a > 0$, $\gamma\in (0,1)$, and $\|U_i\|\leq a$ almost surely. Then there exists a Gaussian $Z$ with covariance $\Sigma$ such that, with probability at least $1-\eta$,
\[
\left\|\frac{1}{\sqrt{n}} \sum_{i=1}^nU_i-Z\right\|
\lesssim_{\alpha,\gamma}
a\sqrt{\frac{32+2\log_2(n)}{n\eta}}
\left[\frac{1}{\a}\log\left\{
\frac{\omega n\eta\log(6/\eta)}
{a^2\{32+2\log_2(n)\}} \vee 1\right\}\right]^{\frac{1}{2\gamma}}.
\]
}
\end{corollary}
\begin{proof}
\color{black}
In this case $\sigma^2(m)\lesssim_{\alpha,\gamma}\omega m^{1-\gamma}\exp(-\a m^\gamma)$ by Proposition~\ref{prop:spectral-integral}. Theorem~\ref{thm:technical} gives
\[
\left\|\frac{1}{\sqrt n}\sum_{i=1}^nU_i-Z\right\|
\lesssim_{\alpha,\gamma}
\sqrt{\log(6/\eta)}\,\omega^{1/2}m^{(1-\gamma)/2}e^{-\a m^\gamma/2}
+a\sqrt{\frac{m\{32+2\log_2(n)\}}{n\eta}}.
\]
Rounding up to the nearest integer, choose
\[
m=\left[\frac{1}{\a}\log\left\{
\frac{\omega n\eta\log(6/\eta)}
{a^2\{32+2\log_2(n)\}} \vee 1\right\}\right]^{1/\gamma} \vee 1.
\]
The first term is then bounded by the second because $m^{(1-\gamma)/2}\leq m^{1/2}$.
\end{proof}

 \section{Bootstrap coupling details (Appendix~\ref{sec:cov})}\label{sec:cov_appendix}

The abstract bound is in terms of
$    \Delta_1=\|\hat{\Sigma}-\Sigma\|_{\HS}$ and $
    \Delta_2= \tr\{\Pi_m^\perp (\hat \Sigma - \Sigma) \Pi_m^\perp\} \vee 0,$
where
$$
\Sigma=\mathbb{E}(U_i\otimes U_i^*),\quad \hat{\Sigma}=\mathbb{E}_n(U_i\otimes U_i^*)-\mathbb{E}_n(U_i)\otimes \{\mathbb{E}_n(U_i)\}^*.
$$
\textcolor{black}{We bound these key quantities under boundedness. The bounds hold with high probability using randomness in the sampled data, $U$.}
Throughout, we assume that $\|U_i\|\leq a$ almost surely.

\begin{lemma}[Covariance estimation]\label{lemma:bounded-concentration-cov} If $\norm{U_i} \le a$ almost surely, then w.p. $1-2\eta$:
$
    \snorm{\hat \Sigma - \Sigma}_{\HS} \le 2\log(2/\eta)^2\left\{\sqrt{\frac{a^2\sigma^2(0)}{n}} \vee \frac{4a^2}{n} \vee \frac{8a^2}{n^2}\right\}.
$
\end{lemma}

\begin{proof}
We proceed in steps.
\begin{enumerate}
    \item Decomposition. Write
    \begin{align*}
    \hat \Sigma - \Sigma
    &=\mathbb{E}_n(U_i\otimes U_i^*)-\mathbb{E}_n(U_i)\otimes \{\mathbb{E}_n(U_i)\}^*-\Sigma \\
    &=\left\{\mathbb{E}_n(U_i\otimes U_i^*)-\mathbb{E}(U_i\otimes U_i^*)\right\}-\left[\mathbb{E}_n(U_i)\otimes \{\mathbb{E}_n(U_i)\}^*\right].
\end{align*}
    \item Focusing on the former term, write $\Sigma_i = U_i \otimes U_i^*$. Since
$
\norm{\Sigma_i}_{\HS} =\|U_i\|^2\le a^2
$
and
$$
\mathbb{E}\|\Sigma_i\|^2_{\HS}=\int \tr(\Sigma_i^* \Sigma_i) \mathrm{d}\mathbb{P}\leq \int \|\Sigma_i\|_{\op} \tr(\Sigma_i) \mathrm{d}\mathbb{P}\leq a^2 \tr(\Sigma)=a^2 \sigma^2(0),
$$
by Bernstein inequality (Lemma~\ref{lemma:c_dv}), w.p. $1-\eta$,
$$
\snorm{\mathbb{E}_n(U_i\otimes U_i^*)-\mathbb{E}(U_i\otimes U_i^*)}_{\HS} \le 2\log(2/\eta)\left\{\sqrt{\frac{a^2 \sigma^2(0)}{n}} \vee \frac{2a^2}{n}\right\}.
$$
\item Focusing on the latter term,
$
    \| \mathbb{E}_n(U_i)\otimes \{\mathbb{E}_n(U_i)\}^*\|_{\HS}
    = \|\mathbb{E}_n(U_i)\|^2
    =\|\mathbb{E}_n(U_i)-0\|^2,
$
 where
$
\|U_i\|\leq a$ and $
\mathbb{E}\|U_i\|^2 \leq a^2.
$
By Bernstein inequality (Lemma~\ref{lemma:c_dv}), w.p. $1-\eta$,
$
\|\mathbb{E}_n(U_i)-0\|\leq 2\log(2/\eta)\left\{\sqrt{\frac{a^2}{n}} \vee \frac{2a}{n}\right\}
$
and therefore
$$
\| \mathbb{E}_n(U_i)\otimes \{\mathbb{E}_n(U_i)\}^*\|_{\HS} \leq 4\log(2/\eta)^2\left\{\frac{a^2}{n} \vee \frac{4a^2}{n^2}\right\}=2\log(2/\eta)^2\left\{\frac{2a^2}{n} \vee \frac{8a^2}{n^2}\right\}.
$$
\item In summary
\begin{align*}
    &\snorm{\hat \Sigma - \Sigma}_{\HS}
    \le
     \snorm{\mathbb{E}_n(U_i\otimes U_i^*)-\mathbb{E}(U_i\otimes U_i^*)}_{\HS} + \| \mathbb{E}_n(U_i)\otimes \{\mathbb{E}_n(U_i)\}^*\|_{\HS} \\
     &\leq 2\log(2/\eta)\left\{\sqrt{\frac{a^2 \sigma^2(0)}{n}} \vee \frac{2a^2}{n}\right\}+ 2\log(2/\eta)^2\left\{\frac{2a^2}{n} \vee \frac{8a^2}{n^2}\right\}. & \qedhere
\end{align*}
\end{enumerate}
\end{proof}

\begin{lemma}\label{lemma:bounded-concentration-cov2}
If $\norm{U_i} \le a$ almost surely, then w.p. $1-\eta$:
$$
    \tr \{\Pi_m^\perp (\hat \Sigma - \Sigma) \Pi_m^\perp\} \vee 0 \le 2\log(2/\eta)\left\{\sqrt{\frac{a^2\sigma^2(m)}{n}} \vee \frac{2a^2}{n}\right\}.
$$
\end{lemma}

\begin{proof}
We proceed in steps.
\begin{enumerate}
    \item Decomposition. As before,
    $$
    \Pi_m^\perp (\hat \Sigma - \Sigma) \Pi_m^\perp = \Pi_m^\perp\left\{\mathbb{E}_n(U_i\otimes U_i^*)-\mathbb{E}(U_i\otimes U_i^*)\right\}\Pi_m^\perp-\Pi_m^\perp\left[\mathbb{E}_n(U_i)\otimes \{\mathbb{E}_n(U_i)\}^*\right]\Pi_m^\perp.
    $$
    \item In the former term,
    $
\tr [\Pi_m^\perp\left\{\mathbb{E}_n(U_i\otimes U_i^*)-\mathbb{E}(U_i\otimes U_i^*)\right\}\Pi_m^\perp] =\bb{E}_n\xi_i-\bb{E}\xi_i
$
 where
\begin{align*}
\xi_i&=\tr \Pi_m^\perp \Sigma_i \Pi_m^\perp =\tr \{\Pi_m^\perp U_i\otimes U_i^* \Pi_m^\perp\}=\|\Pi_m^\perp U_i\|^2\leq \|U_i\|^2\leq a^2, \\
\bb{E} \xi_i^2&\leq a^2 \bb{E} \xi_i=a^2 \int \tr \Pi_m^\perp \Sigma_i \Pi_m^\perp \mathrm{d}\mathbb{P}=a^2\int \tr \Pi_m^\perp \Sigma_i \mathrm{d}\mathbb{P}=a^2\tr \Pi_m^\perp \Sigma=a^2\sigma^2(m).
\end{align*}
By Bernstein inequality (Lemma~\ref{lemma:c_dv}), w.p. $1-\eta$,
$$
|\tr [\Pi_m^\perp\left\{\mathbb{E}_n(U_i\otimes U_i^*)-\mathbb{E}(U_i\otimes U_i^*)\right\}\Pi_m^\perp]|  \leq 2\log(2/\eta)\left\{\sqrt{\frac{a^2 \sigma^2(m)}{n}} \vee \frac{2a^2}{n}\right\}.
$$
    \item Focusing on the latter term,
    $
    \tr(\Pi_m^\perp\left[\mathbb{E}_n(U_i)\otimes \{\mathbb{E}_n(U_i)\}^*\right]\Pi_m^\perp)=\|\Pi_m^\perp\mathbb{E}_n(U_i)\|^2 \geq 0.
    $
    \item In summary.
    \begin{align*}
       \tr \Pi_m^\perp (\hat \Sigma - \Sigma) \Pi_m^\perp
        &\leq \tr [\Pi_m^\perp\left\{\mathbb{E}_n(U_i\otimes U_i^*)-\mathbb{E}(U_i\otimes U_i^*)\right\}\Pi_m^\perp] \\
        &\leq 2\log(2/\eta)\left\{\sqrt{\frac{a^2 \sigma^2(m)}{n}} \vee \frac{2a^2}{n}\right\}.  & \qedhere
    \end{align*}
\end{enumerate}

\end{proof}

\section{Uniform confidence band details (Appendix~\ref{sec:anti})}\label{sec:explicit}

We collect constants with the abbreviations $\bar M = (\kappa\vee 1)^2(\norm{f_0} \vee \bar{\sigma} \vee 1)^2$, $\tilde{M}=\frac{1 }{\underline{\sigma}}(\kappa\|f_0\|+\bar{\sigma})$, and $M=\bar M \vee \tilde{M}$. We abbreviate  $\N(\lambda)=\tr(T_{\lambda}^{-2}T)$ where $T_{\lambda}=T+\lambda$.
To simplify high probability statements, we impose that $n$ is sufficiently large.

\begin{assumption}[Rate condition]\label{assumption:rate-condition}
The sample size $n$, regularization parameter $\lambda$, and kernel $k$, are such that
\textcolor{black}{$n \ge 16\kappa^2\ln(24/\eta^2)^2\{\N(\lambda) \vee \lambda^{-1} \vee \lambda^{-2}\N(\lambda)^{-1}\}$}.
\end{assumption}

\textcolor{black}{For fixed $\eta$, the fundamental condition within Assumption~\ref{assumption:rate-condition} is $\N(\lambda)/n \downarrow 0$, which is necessary for KRR to be consistent in $H$-norm. If $\eta=n^{-\xi}$, the corresponding condition is $\N(\lambda)\log(n)^2/n\downarrow0$, up to constants. Under the weak regularity condition $\N(\lambda)\geq \lambda^{-1}$, this also controls the remaining terms in Assumption~\ref{assumption:rate-condition}.}

\subsubsection*{B and L for general data.}

\begin{lemma}[Bias upper bound; Theorem 4 of \cite{smale2005shannon}]\label{lemma:bias}
    Under Assumption~\ref{assumption:source}$, \|f_{\lambda}-f_0\| \leq \kappa^{1-r} \lambda^{(r-1)/2}\|f_0\|$.
\end{lemma}

\begin{lemma}[Variance lower bound]\label{lemma:hnorm-lb} Let $Z$ be a Gaussian random element of $H$ with covariance $\Sigma$, and suppose $\bb{E}(\ep_i^2|X_i) \ge \underline{\sigma}^2$ almost surely. Then w.p. $1-\eta$,
 $\norm{Z} \ge \sqrt{\underline\sigma^2\N(\lambda)} - \Big\{2 + \sqrt{2\log(1/\eta)}\Big\}\sqrt{M/\lambda}
.$
\end{lemma}

\begin{proof}
We lower bound $\bb{E}\snorm{Z}$ via the identity
$
    \{\bb{E}(\snorm{Z})\}^2 = \bb{E}(\snorm{Z}^2) - \bb{E}\{\snorm{Z} - \bb{E}(\snorm{Z})\}^2$
then appeal to Borell's inequality (Lemma~\ref{lemma:borell}). Let $B_H$ be the unit ball in $H$.

\begin{enumerate}
    \item To upper bound $ \bb{E}(\snorm{Z} - \bb{E}\snorm{Z})^2$, we express $\norm{Z}$ as the supremum of a Gaussian process: $\norm{Z}=\sup_{t\in B_H}\bk{Z}{t}=\sup_{t\in B_H}G_t$. By Lemma~\ref{lemma:krr-bahadur-covariance} and by  maximizing $s \mapsto (s+\lambda)^{-2}s$,
$
\sigma^2_T=\sup_{t \in B_H} \bb{E} \bk{Z}{t}^2 = \norm{\Sigma}_{\op} \le M \norm{T_{\lambda}^{-2}T}_{\op} \le M/\lambda.
$
Similarly,
$\bb{E}\norm{Z}^2 = \tr \Sigma \le M \tr(T_{\lambda}^{-2}T) = M\N(\lambda)$, so by Markov's inequality $\bk{Z}{t}$ is a.s.~bounded on $B_H$. Thus, by combining the two inequalities of Borell's inequality (Lemma~\ref{lemma:borell}) with a union bound we have
\begin{align*}
\mathbb{P}\left\{ (\norm{Z} - \bb{E}\norm{Z})^2 \ge u \right\} &= \mathbb{P}\left( |\norm{Z} - \bb{E}\norm{Z}| \ge \sqrt u \right) \le 2\exp\left(-\frac{u}{2M/\lambda}\right).
\end{align*}
By integrating the tail,
\begin{align*}
\bb{E}(\norm{Z} - \bb{E}\norm{Z})^2
&= \int_{0}^\infty \mathbb{P}\left\{ (\norm{Z} - \bb{E}\norm{Z})^2 \ge u \right\}\,du
\le \int_0^\infty 2\exp\left(-\frac{u}{2M/\lambda}\right) \, du
= 4 M/\lambda.
\end{align*}

    \item We lower bound $\bb{E}(\snorm{Z}^2)$ by Lemma~\ref{lemma:krr-bahadur-covariance2}: $\bb{E}\snorm{Z}^2 = \tr \Sigma \ge \underline{\sigma}^2 \tr(T_{\lambda}^{-2}T) = \underline\sigma^2 \N(\lambda).$
    \item Combining the upper and lower bounds with the identity,
$
\bb{E}\snorm{Z}
\ge \sqrt{\underline\sigma^2\N(\lambda) - 4M/\lambda}.
$ Using the second tail bound of Borell's inequality (Lemma~\ref{lemma:borell}) in inverted form, as well as $\sqrt{a-b} \ge \sqrt{a} - \sqrt{b}$ for $a\geq b \geq 0$, we conclude that w.p.~$1-\eta$,
\begin{align*}
    \norm{Z} &\ge \bb{E}\snorm{Z}
 - u \ge \sqrt{\underline\sigma^2\N(\lambda) - 4M/\lambda}
 - \sqrt{2M\log(1/\eta)/\lambda} \\
 &\ge \sqrt{\underline\sigma^2\N(\lambda)} - \Big\{2 + \sqrt{2\log(1/\eta)}\Big\}\sqrt{M/\lambda}. & \qedhere
\end{align*}
\end{enumerate}
\end{proof}

\paragraph{\textcolor{black}{Anti-concentration for general data.}}

\begin{proposition}[\textcolor{black}{Anti-concentration}]\label{prop:hnorm-anti}
{\color{black}
Suppose that the eigenvalues $\nu_s(T)$ of the integral operator $T$ satisfy either (i) polynomial decay, i.e.~$\nu_s(T) \asymp \omega s^{-\beta}$ for $\beta>1$, or
(ii) exponential decay, i.e.~$\nu_s(T) \asymp \omega \exp(-\a s^{\gamma})$ for $\gamma\in(0,1)$. Define
\[
\bar\zeta_\lambda
=\frac{\sqrt{M\N(\lambda)}}{\underline{\sigma}^2}
\|\,T_\lambda^{-2}T\|_{\HS}^{-1}
\quad\text{and}\quad
\mathfrak{a}_{\|Z\|}(\delta)
=\sup_{t\in\bb R}\bb{P}\{|\|Z\|-t|\leq\delta\}.
\]
There are universal constants $c,C>0$ such that, for all sufficiently small $\lambda$, whenever
$x=\delta\bar\zeta_\lambda\leq c$,
\[
\mathfrak{a}_{\|Z\|}(\delta)
\leq Cx\sqrt{\log(e/x)}.
\]
Under polynomial decay,
\[
\bar\zeta_\lambda
\lesssim_{\beta,\omega}
\frac{\sqrt{M\lambda}}{\underline{\sigma}^2},
\]
while under exponential decay,
\[
\bar\zeta_\lambda
\lesssim_{\alpha,\gamma,\omega}
\frac{\sqrt{M\lambda}}{\underline{\sigma}^2}
\log(\omega/\lambda)^{1/2}.
\]
}
\end{proposition}

\begin{proof} We proceed in steps.
{\color{black}
\paragraph{Step 1: $\chi^2$ anti-concentration for small $t$}
Put $A_\lambda=T_\lambda^{-2}T$. Its eigenvalues are $\nu_j/(\nu_j+\lambda)^2$, so $\|A_\lambda\|_{\HS}^2=\tilde\psi$ in the notation of Lemma~\ref{lem:frob-lower}. Since $\|A_\lambda\|_{\op}\leq(4\lambda)^{-1}$, that lemma gives, under either spectral-decay condition,
\(
\|A_\lambda\|_{\op}^2=o(\|A_\lambda\|_{\HS}^2).
\)
Moreover, Lemma~\ref{lemma:krr-bahadur-covariance2} and eigenvalue monotonicity imply
\[
\|\Sigma\|_{\HS}\geq\underline\sigma^2\|A_\lambda\|_{\HS},
\qquad
\|\Sigma\|_{\HS}^2-\|\Sigma\|_{\op}^2
\geq\underline\sigma^4\{\|A_\lambda\|_{\HS}^2-\|A_\lambda\|_{\op}^2\}.
\]
Thus the coefficient in Theorem~2.7 of \citet{gotze2019large} satisfies, for all sufficiently small $\lambda$,
\[
b\coloneqq
\left[\|\Sigma\|_{\HS}
\{\|\Sigma\|_{\HS}^2-\|\Sigma\|_{\op}^2\}^{1/2}\right]^{-1/2}
\lesssim\underline\sigma^{-2}\|A_\lambda\|_{\HS}^{-1}
=\frac{\bar\zeta_\lambda}{\sqrt{M\N(\lambda)}}.
\]
In particular, the cited theorem gives, for any $w > 0$, that
\[
\sup_{q>0}\bb{P}\{q<\|Z\|^2<q+w\}\lesssim bw.
\]

\paragraph{Step 2: concentration for large $t$. }Fix $\delta,u>0$. If $t\leq\bb{E}\|Z\|+u+\delta$ (the claim being trivial when $t+\delta\leq0$), the corresponding interval for $\|Z\|^2$ has width at most $4\delta(t+\delta)$, and hence, by a limiting argument when its lower endpoint is zero,
\[
\bb{P}(|\|Z\|-t|\leq\delta)
\lesssim b\delta(\bb{E}\|Z\|+u+\delta).
\]
If $t>\bb{E}\|Z\|+u+\delta$, Borell's inequality (Lemma~\ref{lemma:borell}) instead gives
\[
\bb{P}(|\|Z\|-t|\leq\delta)
\leq\exp\left(-\frac{u^2}{2\|\Sigma\|_{\op}}\right).
\]
Combining the two cases gives
\begin{equation}\label{eq:ac-primitive}
\mathfrak{a}_{\|Z\|}(\delta)
\lesssim b\delta(\bb{E}\|Z\|+u+\delta)
+\exp\left(-\frac{u^2}{2\|\Sigma\|_{\op}}\right).
\end{equation}

Now let $x=\delta\bar\zeta_\lambda\leq c$ and take $u=\|\Sigma\|_{\op}^{1/2}\sqrt{2\log(1/x)}$. Jensen's inequality gives $\bb{E}\|Z\|\leq(\tr\Sigma)^{1/2}$, while $\|\Sigma\|_{\op}^{1/2}\leq(\tr\Sigma)^{1/2}$ and the definition of $b$ gives $b\tr\Sigma\geq1$. Since $\tr\Sigma\leq M\N(\lambda)$ by Lemma~\ref{lemma:krr-bahadur-covariance}, the preceding comparison of $b$ and $\bar\zeta_\lambda$ yields
\[
b\delta\bb{E}\|Z\|\lesssim x,
\qquad
\frac{\delta}{(\tr\Sigma)^{1/2}}\lesssim x,
\qquad
b\delta u\lesssim x\sqrt{\log(1/x)},
\qquad
b\delta^2\lesssim x^2.
\]
Substituting these bounds into \eqref{eq:ac-primitive} and absorbing constants then proves
\(
\mathfrak{a}_{\|Z\|}(\delta)
\lesssim x\sqrt{\log(e/x)}.
\)

\paragraph{Step 3: decay rates.} Under polynomial decay, Proposition~\ref{prop:effdim} and Lemma~\ref{lem:frob-lower} give
\[
\sqrt{\N(\lambda)}\lesssim_{\beta,\omega}
\lambda^{-\f12-\f{1}{2\beta}},
\qquad
\|\,T_\lambda^{-2}T\|_{\HS}^{-1}
\lesssim_{\beta,\omega}\lambda^{1+\f{1}{2\beta}},
\]
which proves the stated polynomial bound on $\bar\zeta_\lambda$. Under exponential decay, the same results give
\[
\sqrt{\N(\lambda)}\lesssim_{\alpha,\gamma,\omega}
\lambda^{-\f12}\log(\omega/\lambda)^{1/(2\gamma)},
\qquad
\|\,T_\lambda^{-2}T\|_{\HS}^{-1}
\lesssim_{\alpha,\gamma,\omega}
\lambda\log(\omega/\lambda)^{1/2-1/(2\gamma)}. \qedhere\]
}
\end{proof}

\paragraph{\textcolor{black}{Q and R for bounded data.}}

\begin{proposition}[Gaussian approximation]\label{prop:gaussian-error-apx}
If Assumption~\ref{assumption:rate-condition} holds then \textcolor{black}{there exists a Gaussian $Z$ in $H$, with covariance $\Sigma$,} such that w.p. $1-\eta$,
$$
\textcolor{black}{\norm{\sqrt{n}(\hat f - f_\lambda) - Z} \lesssim Q_{\mathrm{bd}}(T,n,\lambda,\eta) M + \frac{\N(\lambda)}{\sqrt{n}}M\log(8/\eta)^2,}
$$
where
\textcolor{black}{$Q_{\mathrm{bd}}(T,n,\lambda,\eta) = \frac{1}{\lambda}\inf_{m\ge 1} \left[\sqrt{\log(6/\eta)}\,\sigma(T,m) + \sqrt{\frac{m\{32+2\log_2(n)\}}{n\eta}}\right].$}
\end{proposition}
\begin{proof}
By Theorem \ref{thm:bahadur}, w.p. $1-\eta$,
$\norm{\sqrt{n}(\hat f - f_\lambda) - \frac{1}{\sqrt n} \sum_{i=1}^n U_i} \lesssim \frac{\N(\lambda)}{\sqrt{n}}M\log(4/\eta)^2$
since
$
\frac{2\kappa}{n\lambda} \leq \sqrt{\frac{\N(\lambda)}{n}} \iff n\geq \frac{4\kappa^2}{\N(\lambda)\lambda^2}.
$
Then, using Theorem~\ref{thm:technical} together with the bounds deduced in Lemmas \ref{lemma:inference-rv-bounds} and~\ref{lemma:inference-rv-bounds-bounded}, we deduce that w.p. $1-\eta$,
\begin{align*}
    \textcolor{black}{\left\|\frac{1}{\sqrt{n}} \sum_{i=1}^n U_i-Z\right\|}
    & \textcolor{black}{\lesssim \inf_{m\ge 1}  \left\{\sqrt{\log(6/\eta)} \sigma(\Sigma,m) + a\sqrt{\frac{m\{32+2\log_2(n)\}}{n\eta}}  \right\}} \\
    & \textcolor{black}{\lesssim  \frac{M}{\lambda}\inf_{m\ge 1} \left\{ \sigma(T,m)\sqrt{\log(6/\eta)} + \sqrt{\frac{m\{32+2\log_2(n)\}}{n\eta}}\right\}.}
\end{align*}
A union bound gives the desired result,
after consolidating log factors.
\end{proof}

\begin{proposition}[Bootstrap approximation]\label{prop:krr-bs-coupling}
If Assumption~\ref{assumption:rate-condition} holds then there exists a random variable $Z$ whose conditional distribution given $D$ is Gaussian with covariance $\Sigma$, such that w.p. $1-\eta$
$$
\mathbb{P}\left[
\beef \norm{\mathfrak{B} - Z}
\lesssim M\log(12/\eta)^2\left\{ R_{\mathrm{bd}}(T,n,\lambda) + \frac{\kappa \N(\lambda)}{\sqrt n}\right\} \middle| D \right] \ge 1-\eta
$$
where
$R_{\mathrm{bd}}(T,n,\lambda) = \inf_{m\geq 1} \left\{\left(\frac{m \N(\lambda)}{\lambda^2 n}
+ \frac{m}{\lambda^4 n^2} \right)^{\f 1 4} + \frac{1}{\lambda}\sigma(T,m) \right\}.$
\end{proposition}

\begin{proof}
{\color{black}
Under Assumption \ref{assumption:rate-condition}, Theorem \ref{thm:bahadur2}, applied with failure parameter $\eta^2/2$, implies that
$
    \left\|\mathfrak{B}-Z_{\mathfrak{B}} \right\|
    \lesssim \frac{\N(\lambda)}{\sqrt n} \kappa M \ln(24/\eta^2)^2
$
w.p. $1-\eta^2/2$ since $n\geq \frac{16\kappa^2}{\N(\lambda)\lambda^2}$ implies
$
\frac{4\kappa}{n\lambda} \leq \sqrt{\frac{\N(\lambda)}{n}}
$, $
\sqrt{\N(\lambda)} \geq \frac{1}{\sqrt{n}\lambda}
$, and $
\sqrt{\N(\lambda)} \geq \frac{\N(\lambda)^{1/4}}{n^{1/4}\lambda^{1/2}}.
$
We use Corollary \ref{cor:bootstrap-apx} along with the bounds in Lemmas~\ref{lemma:inference-rv-bounds} and~\ref{lemma:inference-rv-bounds-bounded}. In particular, set $W=Z_{\mathfrak{B}}$, $W'=\mathfrak{B}$, and
$\delta_\eta=\frac{\N(\lambda)}{\sqrt n} \kappa M \ln(24/\eta^2)^2$.
Then there must exist $Z$ with the desired conditional distribution, such that w.p. $1-\eta$, the $\sigma(D)$-conditional probability of the event
\begin{align*}
\norm{Z - \mathfrak{B}}
&\lesssim C' \log(6/\eta)^{3/2} \inf_{m\geq 1} \left[m^{\f 1 4}\left\{\frac{a^2\sigma^2(0)}{n} + \frac{a^4}{n^2}\right\}^{\f 1 4} + \sigma(m)\right] + \delta_\eta \\
&\lesssim  \log(6/\eta)^{3/2} \inf_{m\geq 1} \bigg\{m^{\f 1 4}\left(\frac{ M^2 \cdot M^2 \N(\lambda)}{\lambda^2 n} + \frac{M^4}{\lambda^4 n^2} \right)^{\f 1 4}
 + \frac{M}{\lambda}\sigma(T,m) \bigg\} + \frac{\N(\lambda)}{\sqrt n} \kappa M \log(24/\eta^2)^2 \\
&\lesssim M\log(12/\eta)^2\left[ \inf_{m\geq 1} \left\{m^{\f 1 4}\left(\frac{ \N(\lambda)}{\lambda^2 n} + \frac{1}{\lambda^4 n^2} \right)^{\f 1 4} + \frac{\sigma(T,m)}{\lambda} \right\} + \frac{\kappa \N(\lambda)}{\sqrt n}\right]
\end{align*}
is at least $1-\eta$ when $n\geq 2$.
}
\end{proof}

\paragraph{\textcolor{black}{Q and R for leading cases.}}
We now suppress dependence on $T$ in the notation. \textcolor{black}{Write $c_n=32+2\log_2(n)$.}
\begin{align*}
\textcolor{black}{Q_{\mathrm{bd}} (n, \lambda,\eta)} &= \textcolor{black}{\f 1 \lambda \inf_{m\ge 1} \left\{\sqrt{\log(6/\eta)}\,\sigma(m) + \sqrt{\frac{m c_n}{n\eta}} \right\}},\\
\textcolor{black}{R_{\mathrm{bd}} (n, \lambda)} &= \textcolor{black}{\inf_{m\geq 1} \left[\left\{\frac{m \,\N(\lambda)}{\lambda^2 n}  + \frac{m}{\lambda^4n^2
} \right\}^{\f 1 4} + \frac{\sigma(m)}{\lambda} \right]}.
\end{align*}

\textcolor{black}{For the bounded-data rates below, we take $\eta=n^{-\xi}$ for a fixed $\xi\in(0,1)$ and suppress logarithmic factors.}
We simplify each expression under two different assumptions on the spectrum of $T$, namely
(i) polynomial decay, i.e.~$\nu_s(T) \le \omega s^{-\beta}$ (``Sobolev type RKHS'');
(ii) exponential decay, i.e.~$\nu_s(T) \le \omega \exp(-\a s^{\gamma})$ (``Gaussian type RKHS'').
Table~\ref{tab:inference-rates} is our claim.

\begin{table}[H]
    \centering
    {\color{black}
    \begin{tabular}{ccc}
        & Poly. : $\nu_s \le \omega s^{-\beta}$ & Exp. : $\nu_s \le \omega \exp(-\alpha s^\gamma)$ \\
        \cmidrule(lr){2-2}\cmidrule(lr){3-3}
        \hline
        \rule{0pt}{3ex} $Q_{\bullet}$ &
            $\lambda^{-1}n^{\f{(1-\xi)(1-\beta)}{2\beta}}$ &
            $\lambda^{-1} n^{-\frac{1-\xi}{2}}$ \\
        \rule{0pt}{1ex} $R_{\bullet}$ &
            $\left({\lambda^{3+\f{1}{\beta}+\f{2}{\beta-1}}n}\right)^{\f{1-\beta}{4\beta-2}}$ &
            $\lambda^{-3/4}n^{-1/4}$ \\
        \hline
    \end{tabular}}
    \caption{\textcolor{black}{Summary of bounded-data results under different spectral assumptions (suppressing log factors).}}
    \label{tab:inference-rates}
\end{table}

By the results above, $Q\lesssim Q_{\bullet}+\frac{\N(\lambda)}{\sqrt{n}}$ and $R\lesssim R_{\bullet}+\frac{\N(\lambda)}{\sqrt{n}}$.

\subsubsection*{Polynomial decay}

If $\nu_s \le \omega s^{-\beta}$ then by Proposition \ref{prop:spectral-integral} we have $\sigma(m) \lesssim_{\beta,\omega}  m^{1/2-\beta/2}$, and by Proposition \ref{prop:effdim} we have $\N(\lambda) \lesssim_{\beta,\omega} \lambda^{-1-1/\beta}$.

\textcolor{black}{First we study $Q_{\mathrm{bd}}(n,\lambda,\eta)$ with $\eta=n^{-\xi}$. Equating the main terms while suppressing logarithmic factors gives
\[
m^{1/2-\beta/2}=\sqrt{m/(n\eta)}\iff m=n^{(1-\xi)/\beta}.
\]
This value of $m$ gives
\[
Q_{\mathrm{bd}}(n,\lambda,\eta)
\lesssim_{\beta,\omega}\lambda^{-1}n^{(1-\xi)(1-\beta)/(2\beta)}
\quad\text{up to logarithmic factors.}
\]}

Within $R_{\mathrm{bd}}(n, \lambda)$,
$
\frac{m\N(\lambda)}{\lambda^2 n}  \geq  \frac{m}{\lambda^4n^2} \iff n\geq \frac{1}{\lambda^2 \N(\lambda)}
$
which is implied by Assumption~\ref{assumption:rate-condition}.
    Equating main terms gives
$
     m^{\f 1 4}\left(\frac{\lambda^{-1-1/\beta}}{\lambda^2n}\right)^{\f 1 4} = \frac{m^{1/2-\beta/2}}{\lambda}
    \iff   \left(\frac{\lambda^{1-1/\beta}}{n}\right)^{\f{1}{1-2\beta}} = m.
$
This value of $m$ gives the bound
\begin{align*}
    R_{\mathrm{bd}} (n, \lambda) &\lesssim_{\beta,\omega } \left(\frac{m^{1/2-\beta/2}}{\lambda}  \right)
    = \frac{1}{\lambda}\left(\frac{\lambda^{1-1/\beta}}{n}\right)^{\f{1-\beta}{2-4\beta}}
    = \left(\frac{\lambda^{1-\frac{1}{\beta}+\frac{4\beta-2}{1-\beta}}}{n}\right)^{\f{1-\beta}{2-4\beta}}
    =\left\{\frac{1}{\lambda^{3+1/\beta+2/(\beta-1)}n}\right\}^{\f{1-\beta}{2-4\beta}} .
\end{align*}

\subsubsection*{Exponential decay}

If $\nu_s \le \omega \exp(-\alpha s^\gamma)$ then by Proposition \ref{prop:spectral-integral} we have $\sigma(m) \lesssim_{\,\omega,\alpha,\gamma} m^{1/2-\gamma/2} \exp(-\alpha m^\gamma/2)$, and by Proposition \ref{prop:effdim} we have $\N(\lambda) \lesssim_{\,\omega,\alpha,\gamma} \lambda^{-1}\log(1/\lambda)^{1/\gamma}$.

\textcolor{black}{For $Q_{\mathrm{bd}}(n,\lambda,\eta)$, take $\eta=n^{-\xi}$ and
$m=\{(1-\xi)\log(n)/\alpha\}^{1/\gamma}$. Then
\begin{align*}
\sqrt{\log(6/\eta)}\,\sigma(m)
&\lesssim_{\omega,\alpha,\gamma}
n^{-(1-\xi)/2}\left\{\frac{(1-\xi)\log(n)}{\alpha}\right\}^{\frac{1-\gamma}{2\gamma}}
\sqrt{\log(6n^\xi)},\\
\sqrt{\frac{m\{32+2\log_2(n)\}}{n\eta}}
&=n^{-(1-\xi)/2}\left\{\frac{(1-\xi)\log(n)}{\alpha}\right\}^{\frac{1}{2\gamma}}
\sqrt{32+2\log_2(n)}.
\end{align*}
The second display dominates up to constants, and therefore
\[
Q_{\mathrm{bd}}(n,\lambda,\eta)
\lesssim_{\omega,\alpha,\gamma}
\lambda^{-1}n^{-(1-\xi)/2}\log(n)^{1/2+1/(2\gamma)}.
\]}

Within $R_{\mathrm{bd}}(n, \lambda)$,
$
\frac{m\N(\lambda)}{\lambda^2 n}  \geq  \frac{m}{\lambda^4n^2} \iff n\geq \frac{1}{\lambda^2 \N(\lambda)}
$
which is implied by Assumption~\ref{assumption:rate-condition}.
    Equating main terms gives
$
     \left\{\frac{\lambda^{-1}\log(1/\lambda)^{1/\gamma}}{\lambda^2n}\right\}^{\f 1 4}
    \iff    \left[\frac{1}{2\alpha}\log\left\{ \frac{n}{\lambda\log(1/\lambda)^{1/\gamma}}\right\}\right]^{\f 1 \gamma} = m.$
 This value of $m$ gives the bounds
\begin{align*}
    \frac{\sigma(m)}{\lambda} &\lesssim_{\omega,\alpha,\gamma}
    \frac{1}{\lambda} \left[\frac{1}{2\alpha}\log\left\{\frac{n}{\lambda\log(1/\lambda)^{1/\gamma}}\right\}\right]^{\frac{1-\gamma}{2\gamma}} \exp\left[-\alpha \frac{1}{2\alpha}\log\left\{\frac{n}{\lambda\log(1/\lambda)^{1/\gamma}}\right\} /2\right] \\
    &=\frac{1}{\lambda} \left[\frac{1}{2\alpha}\log\left\{\frac{n}{\lambda\log(1/\lambda)^{1/\gamma}}\right\}\right]^{\frac{1-\gamma}{2\gamma}}
    \left\{\frac{n}{\lambda\log(1/\lambda)^{1/\gamma}} \right\}^{-\frac{1}{4}}, \\
    \left\{\frac{m \N(\lambda)}{n\lambda^2}\right\}^{\frac{1}{4}}
    &\lesssim_{\omega,\gamma,\alpha } \left(\frac{1}{n\lambda^2}   \left[\frac{1}{2\alpha}\log\left\{\frac{n}{\lambda\log(1/\lambda)^{1/\gamma}}\right\}\right]^{\f 1 \gamma}  \lambda^{-1}\log(1/\lambda)^{1/\gamma}   \right)^{\frac{1}{4}}.
\end{align*}
Therefore since $\gamma>0$ implies $\frac{1}{2\gamma}>\frac{1}{4\gamma}$ and $\frac{1}{2\gamma}>\frac{1-\gamma}{2\gamma}$, we conclude that
$$
R_{\mathrm{bd}}(n,\lambda) \lesssim_{\omega,\gamma,\alpha} \left[\log\left\{\frac{n}{\lambda\log(1/\lambda)^{1/\gamma}}\right\}\right]^{\frac{1}{2\gamma}}
    \left\{\frac{n\lambda^3}{\log(1/\lambda)^{1/\gamma}} \right\}^{-\frac{1}{4}}.
$$
\paragraph{\textcolor{black}{Summary.}}
We collect the results in this section to verify Table~\ref{tab:rate-all}.
Lemma~\ref{lemma:bias} gives $B\lesssim n^{1/2}\lambda^{(r-1)/2}$.

Lemma~\ref{lemma:hnorm-lb} implies $L\asymp \N(\lambda)^{1/2}-\lambda^{-1/2}$. Since $\N(\lambda)=\sigma^2\{(T+\lambda)^{-2}T,0\}=\psi(0,2)$, we appeal to Proposition~\ref{prop:eff_dim_lb}.
For polynomial decay, $\psi(0,2) \gtrsim_{\omega,\beta} \lambda^{-1-1/\beta}$. Hence, for $\beta>1$,
$L \gtrsim \lambda^{-\frac{1}{2}-\frac{1}{2\beta}}-\lambda^{-1/2}
=\lambda^{-1/2}(\lambda^{-\frac{1}{2\beta}}-1)
\gtrsim \lambda^{-\frac{1}{2}-\frac{1}{2\beta}}.$
For exponential decay,$\psi(0,2) \gtrsim_{\omega,\gamma,\alpha} \lambda^{-1}\log(\omega/\lambda)^{(1-\gamma)/\gamma}$. Hence
$L \gtrsim \lambda^{-\frac{1}{2}}\log(\omega/\lambda)^{(1-\gamma)/2\gamma}-\lambda^{-\frac{1}{2}} \gtrsim \lambda^{-\frac{1}{2}},$ suppressing log terms.

Next we analyze the condition $B\ll L$. For polynomial decay,
    $
    n\lambda^{r-1} \ll \lambda^{-1-1/\beta} \iff \lambda \ll n^{-1/(r+1/\beta)}.
    $
    For exponential decay
    $
    n\lambda^{r-1} \ll \lambda^{-1} \iff \lambda \ll n^{-1/r}.
    $

Table~\ref{tab:inference-rates} gives $Q_{\bullet},R_{\bullet}$ for various cases. For each case, we now analyze the condition $Q+R \ll L$ for which it suffices to study $ Q_{\bullet}+R_{\bullet}+\frac{\N(\lambda)}{n^{1/2}} \ll L$.

\subsubsection*{Polynomial decay, bounded data}

\textcolor{black}{First, we have
\[
Q_{\bullet}\ll L
\iff \lambda^{-2}n^{-(1-\xi)(\beta-1)/\beta}\ll\lambda^{-1-1/\beta}
\iff \lambda\gg n^{-(1-\xi)}.
\]}
Second, we have
$
R_{\bullet} \ll L \iff \{\lambda^{3+\frac{1}{\beta}+\frac{2}{\beta-1}} n\}^{\frac{1-\beta}{2\beta-1}} \ll \lambda^{-1-1/\beta} \iff \lambda \gg n^{-1}.
$
Third, we have
$
\frac{\N(\lambda)}{n^{1/2}} \ll L \iff \frac{\lambda^{-2-2/\beta}}{n} \ll \lambda^{-1-1/\beta}\iff \lambda \gg  n^{\frac{-\beta}{1+\beta}}.
$
\textcolor{black}{If $\xi<1/(\beta+1)$, then $1-\xi>\beta/(\beta+1)$, so the residual condition binds. Combining $Q_{\bullet}+R_{\bullet}+\N(\lambda)/n^{1/2}\ll L$ with $B\ll L$ gives
\[
n^{-\beta/(\beta+1)}\ll\lambda\ll n^{-\beta/(r\beta+1)}
\iff r>1.
\]}

\subsubsection*{Exponential decay, bounded data}

\textcolor{black}{First, we have $Q_{\bullet}\ll L\iff \lambda^{-2}n^{-(1-\xi)}\ll\lambda^{-1}\iff\lambda\gg n^{-(1-\xi)}$.}
Second, we have $R_{\bullet} \ll L \iff  \lambda^{-\frac{3}{2}}n^{-\frac{1}{2}} \ll \lambda^{-1} \iff \lambda \gg n^{-1}$.
Third, we have $\frac{\N(\lambda)}{n^{1/2}}\ll L \iff \frac{\lambda^{-2}}{n}\ll \lambda^{-1}\iff \lambda \gg n^{-1}$.
\textcolor{black}{The Gaussian coupling condition binds. Combining $Q_{\bullet}+R_{\bullet}+\N(\lambda)/n^{1/2}\ll L$ with $B\ll L$ gives
\[
n^{-(1-\xi)}\ll\lambda\ll n^{-1/r},
\]
which is feasible whenever $\xi<1-1/r$.}

\section{Variable width band}\label{sec:band}

We have shown via the Bahadur representation that $\sqrt{n}(\hat{f}-f_{\lambda})$ is approximated by $\sqrt{n}\mathbb{E}_n(U_i)$. We have constructed a Gaussian coupling for this partial sum with covariance $\Sigma=\mathbb{E}(U_i\otimes U_i^*)$. Therefore the standard deviation of the approximating Gaussian for $\sqrt{n}\{\hat{f}(x)-f_{\lambda}(x)\}$ is $\bk{k_x} {\Sigma k_x}^{1/2}$, which we assume is strictly positive for all $x \in S$. We propose standard error estimates $\s(x)$ and $\tilde\s (x)$. This appendix gives conditions under which the estimation error of $\s(x)$ and $\tilde s(x)$ is negligible in the sense that
\[\sup_{x \in S} \left| \frac{\s(x)}{\bk{k_x} {\Sigma k_x}^{1/2}} - 1 \right| = o_p(1),\] justifying variable width confidence bands with estimated widths (analogous to  \citealp[Condition H4]{chernozhukov2014anticoncentration}).

\paragraph{\textcolor{black}{Algorithm details.}} To begin, we expand the statement of Algorithm~\ref{algo:variable}.

\begin{algorithm}[Variable width confidence band]\label{algo:variable_long}
    Given a sample $D = \{(X_i,Y_i)\}_{i=1}^n$, a kernel $k$, and regularization parameter $\lambda>0$:
\begin{enumerate}
    \item Compute the kernel matrix $K \in \mathbb{R}^{n\times n}$ with entries $K_{ij} = k(X_i,X_j)$ and the kernel vector $K_x\in\mathbb{R}^{1\times n}$ with entries $k(x,X_i)$. Set $v_x^{\top}=K_x(K+n\lambda  I )^{-1}\in\mathbb{R}^{1\times n}$.
    \item Estimate KRR as $\hat{f}(x)=v_x^{\top}Y$ and compute the residual vector $\hat{\varepsilon} \in \mathbb{R}^n$ with entries $\hat{\varepsilon}_i=Y_i-\hat{f}(X_i)$. Set $\s^2(x) = n\norm{v_x^{\top}\diag(\hat\ep)}^2_{\bb{R}^n}.$
    \item For each bootstrap iteration,
    \begin{enumerate}
        \item draw multipliers $q\in\R^n$ from $\mathcal{N}(0,I-\boldsymbol{1}\boldsymbol{1}^{\top}/n)$, where $\boldsymbol{1}\in\R^n$ has $\boldsymbol{1}_i=1$;
        \item set $\mathfrak{B}(x)=n^{1/2}v_x^{\top}\diag(\hat\ep)q$;
        \item compute $M=\sup_{x \in S} \left|\s(x)^{-1}\mathfrak{B}(x)\right|$.
    \end{enumerate}
    \item Across bootstrap iterations, compute the \textcolor{black}{$(1-\alpha)$}-quantile, $\hat{t}_{\alpha}$, of $M$.
    \item Calculate the band $\hat C_\a(x)=\hat f(x) \pm \hat{t}_\a \cdot  n^{-1/2}\s(x) $ for $x\in S$.
\end{enumerate}
Alternatively, replace $\s^2(x)$ with $\tilde{\s}^2(x)=\E_q\{\mathfrak{B}(x)^2\}$ by averaging across iterations.
\end{algorithm}

\paragraph{\textcolor{black}{Decomposition.}} We decompose the error into numerator and denominator terms. We then lower bound the denominator under a support condition. We upper bound the numerator using arguments from Appendix \ref{sec:bahadur2}.

\begin{lemma}\label{lemma:s_decomp}
 Almost surely, the following inequalities hold:
 \begin{align*}
    \left|\frac{\s(x)}{\bk{k_x} {\Sigma k_x}^{1/2}}-1\right| &\le \bk{k_x} {\Sigma k_x}^{-1} \left\{|\s^2(x)-\tilde{\s}^2(x)|
    +
    \left|\tilde{\s}^2(x)-\bk{k_x} {\hat{\Sigma} k_x}\right|
    +
    \left|\bk{k_x} {\hat{\Sigma} k_x}-\bk{k_x} {\Sigma k_x}\right|\right\}; \\
  \left|\frac{\tilde{\s}(x)}{\bk{k_x} {\Sigma k_x}^{1/2}}-1\right| &\leq
  \bk{k_x} {\Sigma k_x}^{-1} \left\{
    \left|\tilde{\s}^2(x)-\bk{k_x} {\hat{\Sigma} k_x}\right|
    +
    \left|\bk{k_x} {\hat{\Sigma} k_x}-\bk{k_x} {\Sigma k_x}\right|\right\}.
 \end{align*}
\end{lemma}

\begin{proof}
    For $a,b>0$,
    $
    \left|\frac{a}{b}-1\right|=\frac{|a-b|}{b}=\frac{|a^2-b^2|}{(a+b)b} \leq \frac{|a^2-b^2|}{b^2}.
    $
    Therefore
    $
    \left|\frac{\s(x)}{\bk{k_x} {\Sigma k_x}^{1/2}}-1\right| \leq \frac{|\s(x)^2-\bk{k_x} {\Sigma k_x}|}{\bk{k_x} {\Sigma k_x}}.
    $
    Finally, we apply the triangle inequality.
\end{proof}

\paragraph{\textcolor{black}{High probability events.}}

\begin{lemma}\label{lemma:s_prob}
\textcolor{black}{By the randomness of the data, w.p. $1-2\eta$, $\|\hat{\Sigma}-\Sigma\|_{\HS}\leq \Delta$, where
$ \Delta=8M^2 \log(2/\eta)^2 \left\{\frac{ \N(\lambda)^{1/2}}{n^{1/2}\lambda} \vee \frac{1}{n\lambda^2}\right\}$.}
\end{lemma}

\begin{proof}
    \textcolor{black}{By Lemma~\ref{lemma:bounded-concentration-cov}, w.p. $1-2\eta$ when $n\geq 2$,}
    $$
    \snorm{\hat \Sigma - \Sigma}_{\HS}
    \le 2\log(2/\eta)^2\left\{\sqrt{\frac{a^2\sigma^2(\Sigma,0)}{n}} \vee \frac{4a^2}{n}\right\}
    \leq 8\log(2/\eta)^2\left\{\frac{M\cdot M^{1/2} \N(\lambda)^{1/2}}{n^{1/2}\lambda} \vee \frac{M^2}{n\lambda^2}\right\}
$$
since by Lemma~\ref{lemma:inference-rv-bounds}, $\sigma(\Sigma,0)\leq (\kappa\norm{f_0} + \bar\sigma)\sqrt{\N(\lambda)}\leq \sqrt{M \N(\lambda)}$, and by Lemma~\ref{lemma:inference-rv-bounds-bounded},  $a=\left(\frac{\kappa^2\norm{f_0} + \kappa\bar\sigma}{\lambda}\right)\leq M/\lambda$.
\end{proof}

\paragraph{\textcolor{black}{Denominator term.}} We restrict the support $x\in S$ via the condition $\mathbb{E}\{k(X,x)^2\}\geq\chi^2$.

\begin{lemma}\label{lemma:s0}
    If $\mathbb{E}(\ep^2|X)\geq\underline{\sigma}^2$, $\mathbb{E}\{k(X,x)^2\}\geq\chi^2$, $\lambda\leq \kappa^2$ then
$\bk{k_x} {\Sigma k_x}^{-1}\leq \frac{4\kappa^4}{\underline{\sigma}^2 \chi^2}$.
\end{lemma}

\begin{proof}
    To being, we show that, for any $g \in H$,
    $\bk{g} {\Sigma g}
    \geq \underline{\sigma}^2 (2\kappa^2)^{-2} \mathbb{E}\{g(X)^2\}.
    $
    By Lemma~\ref{lemma:krr-bahadur-covariance2}, $\bk{g} {\Sigma g} \geq \underline{\sigma}^2 \bk{g} {T_{\lambda}^{-2}T g}$. Let $(e_s,\nu_s)$ be the spectrum of $T$. In this basis,
$$
g=\sum_s\bk{g}{e_s} e_s,\quad T g=\sum_s  \nu_s \bk{g}{e_s} e_s,\quad T_{\lambda}^{-2}T g=\sum_s  \frac{\nu_s}{(\nu_s+\lambda)^2}\bk{g}{e_s} e_s.
    $$
   Using these expansions, we write
    \begin{align*}
    \bk{g} {T_{\lambda}^{-2}T g}
    &=\sum_s  \frac{\nu_s}{(\nu_s+\lambda)^2}\bk{g}{e_s}^2
    \geq (\nu_1+\lambda)^{-2} \sum_s  \nu_s \bk{g}{e_s}^2
    =(\nu_1+\lambda)^{-2}  \bk{g} {T g}.
    \end{align*}
    Since $\nu_1=\|T\|_{\op}\leq \kappa^2$ and $\lambda\leq \kappa^2$, $(\nu_1+\lambda)^{-2} \geq (2\kappa^2)^{-2}$. Moreover, $\bk{g} {T g}=\mathbb{E}\{g(X)^2\}$.

    Taking $g=k_x$,
$
\bk{k_x} {\Sigma k_x}
    \geq \underline{\sigma}^2 (2\kappa^2)^{-2} \mathbb{E}\{k(x,X)^2\}\geq \underline{\sigma}^2 (2\kappa^2)^{-2} \chi^2.
$
\end{proof}

\paragraph{\textcolor{black}{First numerator term.}} If we use $\tilde{\s}^2(x)$ as our standard error, then this term does not exist.

\begin{lemma}\label{lemma:ss}
    $
\sup_x|\s^2(x)-\tilde{\s}^2(x)|  \leq \kappa^2 \|\hat{T}_{\lambda}^{-1} \mathbb{E}_n (\hat{\ep}_ik_{X_i})\|^2.
$
\end{lemma}

\begin{proof}
    By Propositions~\ref{prop:s_tilde} and~\ref{prop:b},
    \begin{align*}
          &\s^2(x)-\tilde{\s}^2(x)
          =
   n K_x(K+n\lambda)^{-1}\diag(\hat\ep)(\boldsymbol{1}\boldsymbol{1}^{\top}/n) \diag(\hat\ep)(K+n\lambda)^{-1}K_x^{\top} \\
   &=\{\boldsymbol{1}^{\top}\diag(\hat\ep)(K+n\lambda)^{-1}K_x^{\top}\}^2
   =\{(\hat{\ep}_1,...,\hat{\ep}_n)\Phi(\Phi^*\Phi+n\lambda)^{-1}k_x\}^2 \\
   &=\{\mathbb{E}_n (\hat{\ep}_ik^*_{X_i}) \hat{T}_{\lambda}^{-1}k_x\}^2
   =\bk{\hat{T}_{\lambda}^{-1} \mathbb{E}_n (\hat{\ep}_ik_{X_i})}{k_x}^2.
    \end{align*}
    Finally appeal to the Cauchy-Schwarz inequality.
\end{proof}

\begin{lemma}\label{lemma:ss2}
  Suppose  $\|T_{\lambda}^{-1}(\hat T - T)\|_{\HS} \le \d \le \f 1 2$ and $\norm{T_{\lambda}^{-1}\bb{E}_n(k_{X_i}\ep_i)} \le \gamma$.
  Then $\|\hat{T}_{\lambda}^{-1} \mathbb{E}_n (\hat{\ep}_ik_{X_i})\|\leq 8\gamma+6\delta \|f_0\|+3\|f_{\lambda}-f_0\|$.
\end{lemma}

\begin{proof}
    By the triangle inequality and Lemma~\ref{lemma:first-order-denom},
  \begin{align*}
    &\|\hat{T}_{\lambda}^{-1} \mathbb{E}_n (\hat{\ep}_ik_{X_i})\|
    \leq
    \|(\hat{T}_{\lambda}^{-1}-T_{\lambda}^{-1}) \mathbb{E}_n (\hat{\ep}_ik_{X_i})\|
    +
    \|T_{\lambda}^{-1}\mathbb{E}_n (\hat{\ep}_ik_{X_i})\| \\
    &\leq 2\delta \|T_{\lambda}^{-1}\mathbb{E}_n (\hat{\ep}_ik_{X_i})\|+\|T_{\lambda}^{-1}\mathbb{E}_n (\hat{\ep}_ik_{X_i})\| \leq 2\|T_{\lambda}^{-1}\mathbb{E}_n (\hat{\ep}_ik_{X_i})\|.
  \end{align*}
Since $\hat{\ep}_ik_{X_i}=\{\ep_i+f_0(X_i)-\hat{f}(X_i)\}k_{X_i}=\ep_ik_{X_i}+T_i(f_0-\hat{f})$,
$$
\|T_{\lambda}^{-1}\mathbb{E}_n (\hat{\ep}_ik_{X_i})\|
\leq
\|T_{\lambda}^{-1}\mathbb{E}_n (\ep_ik_{X_i})\|
+\|T_{\lambda}^{-1}\mathbb{E}_n (T_i)(f_0-\hat{f})\|\leq \gamma+\|T_{\lambda}^{-1}\mathbb{E}_n (T_i)(f_0-\hat{f})\|.
$$
  We bound the latter term by $\|f_0-\hat{f}\|\cdot \|T_{\lambda}^{-1}\mathbb{E}_n (T_i)\|_{\op}$. Then
 $$
  \|T_{\lambda}^{-1}\mathbb{E}_n (T_i)\|_{\op}
  \leq \|T_{\lambda}^{-1}\mathbb{E}_n (T_i-T)\|_{\op}
  + \|T_{\lambda}^{-1}T\|_{\op}
  \leq \delta+1 \leq 3/2.
 $$
 In summary, $\|\hat{T}_{\lambda}^{-1} \mathbb{E}_n (\hat{\ep}_ik_{X_i})\|\leq 2\gamma+3\|\hat{f}-f_0\|$. By the triangle inequality and Lemma~\ref{lemma:delta2_latter}, $\|\hat{f}-f_0\|\leq 2(\gamma+\delta\|f_0\|)+\|f_{\lambda}-f_0\|$.
\end{proof}

\begin{lemma}\label{lemma:s1}
    Suppose  $\|T_{\lambda}^{-1}(\hat T - T)\|_{\HS} \le \d \le \f 1 2$ and $\norm{T_{\lambda}^{-1}\bb{E}_n(k_{X_i}\ep_i)} \le \gamma$. Then
    $
\sup_x|\s^2(x)-\tilde{\s}^2(x)|  \leq 3\kappa^2 (64\gamma^2+36\delta^2\|f_0\|^2+9\|f_{\lambda}-f_0\|^2).
$
\end{lemma}

\begin{proof}
    We use Lemmas~\ref{lemma:ss} and~\ref{lemma:ss2} with $(a+b+c)^2\leq 3(a^2+b^2+c^2)$.
\end{proof}

\paragraph{\textcolor{black}{Second numerator term.}}

\begin{lemma}\label{lemma:s2}
    Suppose \textcolor{black}{$|\sbk{k_x}{(\hat{\Sigma} - \Sigma)k_x}| \leq \tilde\delta$} and the following events hold:
    $$
\mathbb{E}_n\|  T_{\lambda}^{-1} T_i\|_{\HS}^2  \leq \delta'
    \quad \mathbb{E}_n\|  T_{\lambda}^{-1} \ep_i k_{X_i}\|^2 \leq \gamma'
    \quad \|T_{\lambda}^{-1}(\hat T - T)\|_{\HS} \le \d \le \f 1 2
    \quad \norm{T_{\lambda}^{-1}\bb{E}_n(k_{X_i}\ep_i)} \le \gamma.
    $$
    From these, define $\Delta'=32 \delta^2\gamma'+(96\gamma^2+288\snorm{f_0}^2\d^2)\delta'$. Then
    $$
    \sup_x \left|\tilde{\s}^2(x) - \bk{k_x} {\hat{\Sigma} k_x}\right| \leq
    \kappa^2 \Delta'+2\kappa^2 \sqrt{\Delta'}\textcolor{black}{\sqrt{\sbk{k_x}{\Sigma k_x}+\tilde\delta}}
    $$
\end{lemma}

\begin{proof}
We proceed in steps.

\begin{enumerate}
    \item  Conditional on data, $\bk{k_x} {\hat{\Sigma} k_x}=\mathbb{E}_h(\bk{Z_{\BS}}{k_x}^2)=\mathbb{E}_h\{Z_{\BS}(x)^2\}$. Therefore
   \begin{align*}
&\left|\tilde{\s}^2(x) - \bk{k_x} {\hat{\Sigma} k_x}\right|
   =\left|\mathbb{E}_h\{\BS(x)^2-Z_{\BS}(x)^2\} \right|\\
   &=  \left|\mathbb{E}_h\left[ \{\BS(x)-Z_{\BS}(x)\}^2 + 2Z_{\BS}(x)\{\BS(x)-Z_{\BS}(x)\}\right]\right| \\
    &\leq \mathbb{E}_h[ \{\BS(x)-Z_{\BS}(x)\}^2] +  2\mathbb{E}_h\{|Z_{\BS}(x)|\cdot |\BS(x)-Z_{\BS}(x)|\} \\
   &\leq \kappa^2 \mathbb{E}_h(\|\BS-Z_{\BS}\|^2)+2\kappa \textcolor{black}{\{\mathbb{E}_h(|Z_{\BS}(x)|^2)\}^{1/2}}\{\mathbb{E}_h(\|\BS-Z_{\BS}\|^2)\}^{1/2}.
  \end{align*}

    \item By Lemma~\ref{lemma:decomp_bahadur2}, for $\Delta_1$ and $\Delta_2$ defined there,
    \begin{align*}
        \mathbb{E}_h(\|\BS-Z_{\BS}\|^2)
        &=
        \mathbb{E}_h(\|\Delta_1+\Delta_2\|^2)
        \leq 2\left\{\mathbb{E}_h(\|\Delta_1\|^2)+\mathbb{E}_h(\|\Delta_2\|^2)\right\}.
    \end{align*}

Consider the former term. As argued in Lemma~\ref{lemma:delta1}, $\|\Delta_1\|\leq 2\delta \norm{T_{\lambda}^{-1}u_1}$ for $u_1$ defined therein. Moreover, $\bb{E}_h\norm{T_{\lambda}^{-1} u_1}^2\leq 4(\gamma'+\|\hat{f}-f_0\|^2\delta')$ by the arguments given there. In summary, $\mathbb{E}_h(\|\Delta_1\|^2)\leq 16\delta^2(\gamma'+\|\hat{f}-f_0\|^2\delta')$.

Consider the latter term. As argued in Lemma~\ref{lemma:delta2}, $ \|\Delta_2\|\leq \left\|T_{\lambda}^{-1}u_2\right\|_{\HS}\cdot \|\hat{f}-f_{\lambda}\|$ for $u_2$ defined in the lemma statement. Moreover, $\bb{E}_h\norm{T_{\lambda}^{-1} u}_{\HS}^2\leq 2\delta'$ by the arguments given there. In summary, $\mathbb{E}_h(\|\Delta_2\|^2)\leq 2\delta'\|\hat{f}-f_{\lambda}\|^2$.

Collecting results and appealing to Lemma~\ref{lemma:delta2_latter},
\begin{align*}
&\mathbb{E}_h(\|\BS-Z_{\BS}\|^2)
\leq
2\left\{16\delta^2(\gamma'+\|\hat{f}-f_0\|^2\delta')+2\delta'\|\hat{f}-f_{\lambda}\|^2\right\} \\
&\leq 2\left[16\delta^2\{\gamma'+4(\gamma +\snorm{f_0})^2\delta'\}+2\delta'4(\gamma + \d\snorm{f_0})^2\right]\\
&= 32\left[\delta^2\{\gamma'+8(\gamma^2 +\snorm{f_0}^2)\delta'\}+\delta'(\gamma^2 + \d^2\snorm{f_0}^2)\right]
\leq 32 \delta^2\gamma'+(96\gamma^2+288\snorm{f_0}^2\d^2)\delta' = \Delta'.
\end{align*}
    \item \textcolor{black}{Finally,
    \begin{align*}
        \mathbb{E}_h(|Z_{\BS}(x)|^2)
        =\bk{k_x}{\hat \Sigma k_x}
        = \bk{k_x}{\Sigma k_x} + \bk{k_x}{(\hat\Sigma - \Sigma) k_x}
        = \bk{k_x}{\Sigma k_x} + \tilde\delta. \qquad  \qedhere
    \end{align*}}
\end{enumerate}
\end{proof}

\paragraph{\textcolor{black}{Third numerator term.}}

\begin{lemma}\label{lemma:s3}
   \textcolor{black}{If $\|\hat{\Sigma} - \Sigma\|_{\HS}\leq \Delta$ then
    $
    \sup_x\left|\bk{k_x} {\hat{\Sigma} k_x}-\bk{k_x} {\Sigma k_x}\right|\leq \kappa^2\Delta.
    $}
\end{lemma}

\begin{proof}
   By the Cauchy-Schwarz inequality,
        $$
        \left|\bk{k_x} {(\hat{\Sigma} - \Sigma) k_x}\right|
        \leq \|k_x\| \|(\hat{\Sigma} - \Sigma) k_x\|
        \leq \|k_x\|^2 \|\hat{\Sigma} - \Sigma\|_{\op}. \qedhere
        $$
\end{proof}

\paragraph{\textcolor{black}{Main result.}}

\begin{theorem}
    Suppose Assumption~\ref{assumption:rate-condition} holds, $\mathbb{E}(\ep^2|X)\geq\underline{\sigma}^2$, $\underline{S}=[x\in S: \mathbb{E}\{k(X,x)^2\}\geq\chi^2]$, and $\lambda\leq \kappa^2$.
    \textcolor{black}{Suppose $\frac{\N(\lambda)}{n\lambda^2}=o(1)$.}
  Then $\sup_{x\in \underline{S}}\frac{\tilde{\s}(x)}{\langle k_x,\Sigma k_x \rangle^{1/2}}=1+o_p(1)$ and
  $\sup_{x\in \underline{S}}\frac{\langle k_x,\Sigma k_x \rangle^{1/2}}{\tilde{\s}(x)}=1+o_p(1)$.
    If in addition $\|f_{\lambda}-f_0\|=o(1)$ then the same conclusions hold for $\s(x)$ instead of $\tilde{\s}(x)$.
\end{theorem}

\begin{proof}
    We prove the second result. The first follows from the same argument and bounds.
    \begin{enumerate}
        \item We prove $\sup_{x\in S}\left|\frac{\s(x)}{\bk{k_x} {\Sigma k_x}^{1/2}}-1\right|=o_p(1)$.  By Lemma~\ref{lemma:s_decomp}, it suffices to study
        $$
          \bk{k_x} {\Sigma k_x}^{-1} \left\{|\s^2(x)-\tilde{\s}^2(x)|
    +
    \left|\tilde{\s}^2(x)-\bk{k_x} {\hat{\Sigma} k_x}\right|
    +
    \left|\bk{k_x} {\hat{\Sigma} k_x}-\bk{k_x} {\Sigma k_x}\right|\right\}.
        $$
        By Lemma~\ref{lemma:s0}, the initial factor is $O(1)$. By Lemma~\ref{lemma:s1}, the first term is $O(\gamma^2+\delta^2+\|f_{\lambda}-f_0\|^2)$. \textcolor{black}{By Lemma~\ref{lemma:s2} and Lemma~\ref{lemma:s3}, the latter of which shows $\tilde \delta \le \kappa^2\Delta$ the second term is $O\{\Delta'+\sqrt{\Delta'}\sqrt{1 + \Delta}\}$. By Lemma~\ref{lemma:s3}, the third term is $O(\Delta)$.}

Therefore, \textcolor{black}{omitting logarithmic factors}, it suffices to show $\gamma^2+\delta^2=o_p(1)$, $\|f_{\lambda}-f_0\|^2=o(1)$, $\Delta=o_p(1)$, and $\Delta'=o_p(1)$. Since
        $\Delta'=O\{\delta^2\gamma'+(\gamma^2+\delta^2)\delta'\}$, the first condition does not bind. The second holds by hypothesis.

        \textcolor{black}{For $\Delta=o_p(1)$, by Lemma \ref{lemma:s_prob},} we require $\left\{\frac{ \N(\lambda)^{1/2}}{n^{1/2}\lambda} \vee \frac{1}{n\lambda^2}\right\}=o(1)$. Assumption~\ref{assumption:rate-condition} implies the former term dominates, so the condition simplifies to $\frac{\N(\lambda)}{n\lambda^2}=o(1)$.

        \textcolor{black}{For $\Delta'=o_p(1)$, we require by Lemmas~\ref{lemma:numerator-concentration},~\ref{lemma:denominator-concentration},~\ref{lemma:high1}, and~\ref{lemma:high2},
        $$
\{\delta^2\gamma'+(\gamma^2+\delta^2)\delta'\}
\lesssim
\left\{\sqrt{\frac{\N(\lambda)}{n}} \vee \frac{1}{n\lambda}\right\}^2
\cdot \left[\N(\lambda)+\left\{\frac{1}{n\lambda^2}+\sqrt{\frac{\N(\lambda)}{n\lambda^2}}\right\}\right]
=o(1).
$$
Assumption~\ref{assumption:rate-condition} implies $\sqrt{\frac{\N(\lambda)}{n}}$ is the dominant term in the first factor and $\N(\lambda)$ is the dominant term in the middle factor. Therefore the condition simplifies to $\frac{\N(\lambda)^{2}}{n}=o(1)$.}
\textcolor{black}{Since $\N(\lambda) = O(\lambda^{-2})$, the sufficient condition $\frac{\N(\lambda)}{n\lambda^2}=o(1)$ covers both cases.}
        \item Lighten notation as $\sigma(x)=\langle k_x,\Sigma k_x \rangle^{1/2}$. We have shown that, for all $x$, $\left|\frac{\s(x)}{\sigma(x)}-1\right|=o_p(1)$, which clearly implies the same for $\left|\frac{\sigma(x)}{\s(x)}-1\right|$. To finish the argument, write
$$
\frac{\s(x)}{\sigma(x)}=1+\left\{\frac{\s(x)}{\sigma(x)}-1\right\},\quad \frac{\sigma(x)}{\s(x)}=1+\left\{\frac{\sigma(x)}{\s(x)}-1\right\}.\qedhere
$$
    \end{enumerate}

\end{proof}

\end{document}